\documentclass[10pt]{amsart}
\usepackage{amsmath}
\usepackage{amssymb}
\usepackage{enumerate}
\usepackage{amsbsy}
\usepackage{amsfonts}
\usepackage{color}
\usepackage{upgreek}
\usepackage{slashed}
\usepackage{mathrsfs}
\usepackage[notcite,notref]{showkeys}
\newtheorem{thm}{Theorem}[section]
\newtheorem{lem}[thm]{Lemma}
\newtheorem{prop}[thm]{Proposition}
\newtheorem{cor}[thm]{Corollary}
\newtheorem{defn}[thm]{Definition}
\newtheorem{rem}[thm]{Remark}

\numberwithin{equation}{section}

\begin{document}

	\title[Geometric analysis of nonlinear Dirac equations]{A spinor-adapted geometric approach for nonlinear Dirac systems and its application to a tensorial wave-Dirac system near Minkowski spacetime}

    \author[S. Hong]{Seokchang Hong}
    \address{Universit\"at M\"unster, Mathematisches Institut,
Einsteinstrasse 62 48149 M\"unster, Germany}
    \email{seokchangh.11@uni-muenster.de}


	\begin{abstract}
		We study a nonlinear tensorial wave-Dirac system on $(1+3)$-dimensional asymptotically flat spacetimes as a semilinear model motivated by the Maxwell-Dirac and Einstein-Dirac systems. The purpose of this model is to isolate the interaction between the null geometry of antisymmetric tensor fields and the intrinsic first-order geometry of the Dirac equation while avoiding the derivative loss mechanism of the Einstein equations and the gauge structure of the Maxwell equations.

Our analysis preserves the first-order nature of the Dirac equation throughout the nonlinear argument. The Dirac current provides the fundamental energy identity, while quantitative spacetime estimates are obtained from the wave equation arising from the squared Dirac operator. Combining these ingredients with integrated local energy decay estimates and $r^p$-weighted energy hierarchies, we establish a coupled energy method for the tensorial and spinorial components.
A key observation is that the Clifford algebra is compatible with the null decomposition of antisymmetric tensor fields and excludes the most singular nonlinear interactions.

As a consequence, we establish the global existence of small-data solutions together with quantitative weighted energy and decay estimates. Remarkably, combining the null structure with dyadic argument, weak decay such as $t^{-\frac12-\delta}$ is sufficient to obtain nonlinear stability of the system, in the spirit of \cite{DHRT}.

We expect that the geometric ideas developed in this paper provide a useful starting point for the study of more general nonlinear Dirac systems, including the Maxwell-Dirac and Einstein-Dirac equations.

	\end{abstract}

		\maketitle

\tableofcontents

\section{Introduction}\label{sec:Intro}
We consider the tensorial wave-Dirac system, given by
\begin{align}\label{eq-tensor-dirac}
    \begin{aligned}
        i\gamma^\mu\nabla_\mu\psi = iF_{\mu\nu}\gamma^\mu\gamma^\nu \psi, \\
        \nabla^\mu F_{\mu\nu} = \langle \psi, \gamma_{\nu}\psi\rangle ,
    \end{aligned}
\end{align}
on a $(1+3)$-dimensional time-oriented Lorentzian manifold $(\mathcal M,g)$, where $g$ is sufficiently close to the Minkowski metric so that the presence of a trapping region is excluded. Here $\psi$ is a spinor field on $\mathcal M$ whose value is represented by a column vector in $\mathbb C^N$, and $F$ is a closed $2$-form on $\mathcal M$. The manifold $\mathcal M$ is foliated by a family $\{ \Sigma_\tau \}_{\tau\ge0}$, where $\Sigma_\tau$ is outgoing null for large $r$. The main result of this paper is the following:
\begin{thm}[Informal version]\label{main-thm-informal}
    For the above system, given sufficiently smooth and small initial data, with an appropriate decay assumption, we have global existence and decay of solutions to \eqref{eq-tensor-dirac}.
\end{thm}
The toy model \eqref{eq-tensor-dirac} is motivated by the Maxwell-Dirac system, and global nonlinear stability problems for the Einstein-Dirac system. However, the main purpose of this paper is to develop a spinorial geometric method for global analysis of nonlinear stability of the Dirac-type equations. 

\subsection{Motivation}
The Dirac equation, introduced by P. A. M. Dirac in 1928 \cite{dirac1928}, provided the first relativistic description of spin-$\frac12$ particles and has become a fundamental object in modern theoretical physics. By reconciling quantum mechanics with special relativity, Dirac's formulation revealed that relativistic fermionic matter is naturally described not by scalar or tensorial fields, but by spinor fields. The equation also predicted the existence of antimatter and established spinors as indispensable geometric objects in relativistic field theory.

From a mathematical viewpoint, the Dirac equation occupies a distinguished position among hyperbolic partial differential equations. In contrast to scalar wave or Klein-Gordon equations, it is intrinsically first-order, and its causal structure is encoded through Clifford multiplication and spin geometry \cite{lawson,wernli}. Its natural conserved quantity is not derived from a positive-definite energy-momentum tensor in the usual wave-equation sense, but rather from the Dirac current, whose causal positivity reflects the Lorentzian geometry of spacetime itself. This geometric structure becomes particularly significant in the study of null propagation, dispersion, and nonlinear interactions.

The role of Dirac fields becomes even more prominent in general relativity, where fermionic matter propagates on curved Lorentzian spacetimes through the spin connection. Besides providing the natural relativistic model for spin-$\frac12$ particles, Dirac fields also appear in several fundamental problems of gravitational physics. They play an important role in quantum field theory on curved spacetimes, and arise naturally in coupled Einstein-matter models where gravity interacts with spin-$\frac12$ particles. Therefore, the mathematically rigorous study of Dirac equations on curved spacetimes is of great importance not only for understanding relativistic fermionic matter itself, but also as a fundamental step toward nonlinear Einstein-matter systems.

From the mathematical point of view, the analysis of Dirac equations has been developed along several complementary directions.

A large part of the existing analysis of nonlinear Dirac equations proceeds by relating them to wave- or Klein-Gordon-type equations. For Cauchy problems on the Minkowski background, a standard and very effective approach is to use Fourier-multiplier projections in order to decompose the Dirac equation into half-wave or half-Klein-Gordon components. This makes it possible to apply scalar dispersive techniques from Fourier analysis, and has led to strong results, especially in low-regularity problems. See \cite{behe,machihara,canhe2} and the references therein.

Another fruitful approach is based on vector-field methods \cite{lizhang}. This perspective is particularly well suited to problems in which the causal geometry of spacetime plays a central role, since commutation, multiplier identities, and null decompositions can be expressed in an invariant way. For Dirac equations, such methods also have the advantage of interacting naturally with the first-order character of the equation and with the geometry of the Dirac current.

On curved backgrounds, Dirac equations have been investigated in various settings, including the characteristic Cauchy problem, scattering theory, black hole spacetimes, and Einstein-Dirac systems \cite{treude,hafnernicolas,leflochmazhang,zhaowu}.
One approach is based on reducing the Dirac equation to wave-type equations, in particular to spin-$\frac12$ Teukolsky-type equations \cite{mazhang}. Such reductions allow one to apply geometric and analytic techniques originally developed for wave equations.

Building upon these developments, the purpose of this paper is to develop a spinor-adapted geometric approach for nonlinear Dirac systems near Minkowski spacetime. Our approach is to combine the intrinsic first-order structure of the Dirac equation with wave-type estimates obtained after squaring the Dirac operator. In this way, one can import elements of the geometric energy method into a spinorial setting while retaining the distinguished role of the Dirac current. This method is then applied to a tensorial wave-Dirac system, where the coupling between tensorial wave components and spinorial Dirac components requires estimates that are compatible with both null geometry and Clifford multiplication.

More precisely, we would like to emphasise the following remarks throughout this paper:
\begin{enumerate}
    \item (First-order) \\
    First, the Dirac operator is fundamentally different from the wave operator.
    Unlike the wave equation, the Dirac equation is intrinsically first-order and spinorial.

    \item (Coercivity) \\
    For Dirac fields, the energy-momentum tensor does not provide a canonical positive energy.
Instead, the natural object controlling both causality and energy propagation is the Dirac current.
 The Dirac current is a causal vector field, and its contraction with any future-directed timelike vector yields a positive quantity.
This current plays an important role in the energy framework.
    \item (Modification of the classical vector fields) \\
    Working with spinor fields requires a geometric formalism that is invariant under changes of coordinates and frames.
This forces a formulation of the Dirac operator at the level of the spinor bundle and Clifford algebra, which in turn dictates the structure of admissible vector field methods.

In particular, the commutation vector fields must respect the underlying Clifford structure.
We study global analysis via vector field methods. However, the usual Lie derivatives do not, in general, preserve the Clifford structure when acting on spinors. Instead, one has to employ the Kosmann Lie derivative.
For instance, the rotation vector fields $\Omega_{ij}$ do not commute with the Dirac operator even in the flat background; the modified vector field $\Omega_{ij}-\frac12\gamma^i\gamma^j$ commutes with the Dirac operator.
which in turn dictates the structure of admissible vector field methods.


    \item (Double null foliation) \\
    We adopt the double null foliation, in which the Dirac equation is a coupled system of transport-type equations along null directions. This null-adapted foliation is useful in that it reveals the null structure in the nonlinearity induced by the Clifford algebra.


\end{enumerate}

\subsection{Geometric features of the linear Dirac equation}
Given a $(1+3)$-dimensional Lorentzian manifold $(\mathcal M,g)$ with a spin structure\footnote{We refer the reader to Section \ref{sec:Dirac} for the definition of the spin structure and the spinor bundle.}, the linear Dirac equation is given by
\begin{align}
(i\gamma^\mu\nabla^{\mathbf S}_\mu-m)	\psi =0,
\end{align}
where $\psi$ is a spinor field represented locally by a column vector in $\mathbb C^N$ for some $N\ge2$, and $m\ge0$ is the mass of a spin-$\frac12$ particle.\footnote{In fact, it is well known that the spinor field $\psi$ takes values in $\mathbb C^4$ when $\mathcal M$ is a $(1+3)$-dimensional Lorentzian manifold. However, the explicit dimension of the target space and the choice of representation are not important.} The covariant derivative $\nabla^{\mathbf S}_\mu$ is the spin connection induced by the Levi-Civita connection on $(\mathcal M,g)$, and $\gamma^\mu$ denotes a gamma matrix representation satisfying the following algebraic relation, known as the Clifford algebra:
\begin{align}\label{intro-clifford-algebra}
    \gamma(e_\mu)\gamma(e_\nu) + \gamma(e_\nu)\gamma(e_\mu) = -2g_{\mu\nu}I,
\end{align}
where $I$ is the identity matrix and $\{e_\mu\}$ is a local frame of the tangential bundle of the manifold $\mathcal M$. Thus the gamma matrices are considered to be an algebraic square root of the metric.

While one might study the Dirac equation as a time-evolution equation in the usual coordinate system $(t,x,y,z)$, its first-order structure also allows one to view it as a transport-type equation. Indeed, we introduce the null frame $\{L,\underline L,e_A\}$, where $e_A$, $A=1,2$, are orthonormal vector fields tangent to the $2$-sphere $S_{u,v}$, and we normalise the null pair by $g(L,\underline L)=-2$. Then the Dirac equation is rewritten as
\begin{align}
    i\gamma^L\nabla_L\psi + i\gamma^{\underline L}\nabla_{\underline L}\psi+ i\gamma^{e_A}\nabla_{e_A}\psi-m\psi = 0,
\end{align}
where we put $\gamma^L=\gamma(L)$ for simplicity.

By the Clifford algebra \eqref{intro-clifford-algebra}, we have $\gamma(L)^2=\gamma(\underline L)^2=0$, i.e., $\gamma^L$ and $\gamma^{\underline L}$ are nilpotent. We now introduce the following projection operators:
\begin{align}
    \Pi_+ = \frac14 \gamma^{\underline L}\gamma^L, \quad \Pi_- = \frac14 \gamma^{L}\gamma^{\underline L},
\end{align}
and we decompose $\psi= \psi_++\psi_-$, where $\psi_+:= \Pi_+\psi$ and $\psi_-:=\Pi_-\psi$. Then we obtain the following null decomposition for the Dirac equation: (See also Proposition \ref{prop-null-decomp-dirac}.)
\begin{align}
\begin{aligned}
    \nabla_L\psi_+ \approx \slashed\nabla\psi_-+m\psi_-+ O(r^{-1})\psi, \\ 
    \nabla_{\underline L}\psi_- \approx \slashed\nabla\psi_+ +m\psi_++O(r^{-1})\psi.  
    \end{aligned}
\end{align}
From now on, we are only concerned with the massless case, i.e., $m=0$. We conclude that the Dirac equation can be reformulated as a coupled system of transport-type equations for $(\psi_+,\psi_-)$. Furthermore, the $\psi_+$ component is transported along the null direction $L$, whereas the $\psi_-$ component is transported along the $\underline L$-direction. We also note that two components $\psi_+$ and $\psi_-$ are coupled via the angular derivatives up to lower-order terms.

The above decomposition leads us to adopt the double null foliation. Indeed, the canonical energy for the spinor fields can be defined along the outgoing cone $\mathcal C_u$ and the ingoing cone $\underline{\mathcal C}_v$.

To see this, one has to observe that the natural energy for the Dirac equation is not derived from the energy-momentum tensor. While the energy-momentum tensor for wave or Klein-Gordon fields gives a coercive quantity, that of a spinor field does not. Instead, there is another natural quantity, the so-called Dirac current $J^\mu$, given by
\begin{align}
    J^\mu[\psi] := \langle \psi, \gamma^\mu\psi\rangle_D,
\end{align}
where $\langle\cdot,\cdot\rangle_D$ is the spinor inner product or pairing. We refer to Section \ref{sec:Dirac} for the definition. Although this inner product is not positive definite in general, the Dirac current $J^\mu$ is a causal vector field; that is, its contraction with a future-directed timelike or null vector yields a nonnegative quantity. Moreover, the Dirac current is divergence-free in the linear case. Then the divergence theorem gives
\begin{align}
   \int_{\partial \mathcal D}J^\mu n_\mu =   \int_{\mathcal D} \nabla_\mu J^\mu \,dV =0.
\end{align}
Here we choose the domain $\mathcal D$ so that the boundary $\partial \mathcal D$ is the union of two hypersurfaces $\Sigma_{\tau_2}$ and $\Sigma_{\tau_1}$, where $\{\Sigma_{\tau}\}_{\tau\ge0}$ is the time-slice for $r\le R$ and it is outgoing null for $r >R$. Then we have
\begin{align}
\begin{aligned}
   & \int_{\Sigma_{\tau_2}\cap \{r\le R\} } \langle \psi, \gamma(T)\psi\rangle \,d\sigma + \int_{ \substack{ u=\tau_2-R \\ v\ge \tau_2+R } } \langle \psi, \gamma(L)\psi\rangle\,d\sigma + \int_{ \mathcal I^{\tau_2-R}_{\tau_1-R} } \langle \psi, \gamma(\underline L)\psi\rangle\,d\sigma \\
   & =   \int_{\Sigma_{\tau_1}\cap \{r\le R\} } \langle \psi, \gamma(T)\psi\rangle \,d\sigma + \int_{ \substack{ u=\tau_1-R \\ v\ge \tau_1+R } } \langle \psi, \gamma(L)\psi\rangle\,d\sigma,
   \end{aligned}
\end{align}
where $\mathcal I^{\tau_2-R}_{\tau_1-R} $ is the null infinity and the vector $T$ given by $T=\frac12(L+\underline L)$ is the future-directed unit normal to the time-slice, which is equal to $\partial_t$ on the Minkowski spacetime. In Section \ref{sec:Dirac}, the Dirac pairing turns out to be locally represented by $\langle \psi,\phi\rangle_D \approx \psi^\dagger \gamma^T \phi$ and hence
\begin{align}
    \langle \psi, \gamma^L\psi\rangle = 2|\psi_+|^2, \quad \langle \psi,\gamma^{\underline L}\psi\rangle  = 2|\psi_-|^2,
\end{align}
where we put $\psi^\dagger\psi = |\psi|^2$ for simplicity. Therefore, the above energy identity implies that the energy of the $\psi_+$ component is defined along the outgoing cone $\mathcal C_u$, while the energy of the $\psi_-$ component is naturally given along the ingoing cone $\underline{\mathcal C}_v$.

For a moment we would like to compare the transport property of the Dirac equation $i\gamma^\mu\nabla_\mu\psi =0$ with the wave equation $\Box_g\phi=0$. While the transport of a scalar field $\phi$ is determined by its derivatives $L\phi$ or $\underline L\phi$, the direction of transport for the spinor field $\psi$ is intrinsically determined by its null components $\psi_+$ and $\psi_-$. This transport property of $(\psi_+,\psi_-)$ is therefore more closely related to the Maxwell components $(\alpha,\underline\alpha)$, in that the energy of $\alpha$ is defined on outgoing cones whereas the energy of $\underline\alpha$ is given along ingoing cones.

This explains why we adopt the double null foliation for the study of the Dirac equation.

This null-adapted foliation is useful to reveal the intrinsic property of the spinor fields.
Indeed, the foliation by time slices $\{t=\textrm{constant}\}$ is a favourable choice in that it naturally captures the time-evolution property of the equation. However, this foliation obscures the null decomposition for the spinor fields. The hyperboloidal foliation is also very effective, in that the Lorentz boosts are well adapted to this foliation and it is easier to extract a certain coercivity even in the massive case. However, the foliation restricts solutions to be supported inside the cones and does not reveal the transport properties of the spinor field.

We would like to highlight that, in the above discussion of the geometric features of the Dirac equation, we have not imposed any geometric assumptions on the underlying background. Indeed, the null frame is intrinsically defined, and hence the aforementioned null decomposition and transport properties are intrinsic geometric features of the Dirac equation. In other words, they do not require any specific symmetry of the background spacetime.
\subsection{The background spacetime}
We begin by briefly describing the class of background spacetimes considered in this paper.

We suppose that the underlying manifold $(\mathcal M,g)$ satisfies an asymptotic flatness without trapping region. Roughly speaking, the metric is assumed to be a sufficiently small asymptotically flat perturbation of the Minkowski metric, with the same long-range radial structure as the Kerr metric and without trapped null geodesics.

Following the geometric framework developed for wave equations on asymptotically flat spacetimes, we separate the long-range radial part from the faster decaying angular perturbation.
More precisely, we write the metric $g$ as
\begin{align}
    g = m+h,
\end{align}
where $m$ is the Minkowski metric. Accordingly, we decompose $h$ into the spherically symmetric part and angular part:
\begin{align}\label{background-assum1}
    h(t,r,\omega) = h^{\rm rad}(t,r) + h^{\rm ang}(t,r,\omega)
\end{align}
for $t,r>0$ and $\omega\in\mathbb S^2$. Now we impose the following decay assumptions on $h$:
\begin{align}\label{background-assum2}
    | \partial^{\alpha}h^{\rm rad}| \le c_\alpha \langle r\rangle^{-1-|\alpha|}, \quad | \partial^{\alpha}h^{\rm ang}|  \le c_\alpha' \langle r\rangle^{-2-|\alpha|} , \quad |\alpha| \le N+1,
\end{align}
where $N$ will be fixed later.
We assume throughout that the spacetime is non-trapping. Note that we impose an improved decay assumption on the non-spherically symmetric component $h^{\rm ang}$ in order to obtain better commutator estimates for the rotation vector fields. 
We further assume that the constants $c_\alpha$ and $c_{\alpha}'$ are sufficiently small so that the manifold $\mathcal M$ contains no trapping region.
\subsection{Linear analysis of the Dirac equation}
From now on, we discuss more quantitative features of spinor fields. Although our geometric perspective on the Dirac equation is fundamentally based on its first-order structure, many quantitative estimates are obtained after squaring the Dirac operator. This produces a wave-type equation with lower-order curvature terms, allowing us to import local energy decay, Morawetz estimates, and $r^p$-weighted energy hierarchies while preserving the intrinsic spinorial energy through the Dirac current.

\subsubsection{Commutation}
Since we are interested in the global analysis of the Dirac equation via the energy method, the first step is to establish a stable commutation theory. Since the Dirac operator acts on the spinor bundle rather than on scalar or tensor fields, the ordinary Lie derivatives do not preserve the Clifford structure. Instead, we commute the equation with a family of modified geometric vector fields based on the Kosmann Lie derivative. This produces a hierarchy of commuted equations whose error terms remain compatible with the spinorial geometry.

We recall the linear (massless) Dirac equation:
\begin{align*}
    i\gamma^\mu\nabla_\mu\psi = 0.
\end{align*}
Throughout this paper, we shall use the null-adapted vector fields $Z\in \{L,\underline L,\Omega\}$ as commutators, where $\Omega$ denotes the rotation vector fields on $\mathbb R^3$. Commuting the equation by means of the Kosmann Lie derivatives, we obtain the equation
\begin{align}
    i\gamma^\mu\nabla_\mu \mathscr L_Z^{\le k}\psi = [ i\gamma^\mu\nabla_\mu, \mathscr L_Z^{\le k} ] \psi.
\end{align}
We refer to Section \ref{sec:Dirac} for the explicit definition of the Kosmann Lie derivative $\mathscr L_Z$. Then we apply the divergence theorem to the Dirac current $J[\mathscr L_Z^{\le k}\psi ]$. The Dirac current is no longer divergence-free because of the commutator terms, and hence the energy identity presents an additional spacetime integral of the commutator terms as a source. The main analytical difficulty is therefore to control these spacetime error terms.

\subsubsection{Squaring the Dirac}
However, the problem is that the conserved energy alone from the Dirac current is insufficient to obtain a certain decay. In order to obtain spacetime decay estimates, we exploit the fact that squaring the Dirac operator yields a wave-type identity through the Lichnerowicz formula. This enables us to take advantage of local energy estimates, Morawetz estimates, and the $r^p$-weighted hierarchy from the wave equation while retaining the intrinsic first-order energy based on the Dirac current simultaneously.

From the commuted Dirac equation, an application of the Dirac operator $\gamma^\nu\nabla_\nu$ yields
\begin{align}
    g^{\mu\nu}\nabla_\mu\nabla_\nu \mathscr L_Z^{\le k}\psi = \gamma^\lambda\nabla_\lambda ( [ \gamma^\mu\nabla_\mu, \mathscr L_Z^{\le k} ] \psi ) + \mathcal R \cdot \mathscr L_Z^{\le k}\psi,
\end{align}
where $\mathcal R$ is the scalar curvature. Thus squaring the Dirac operator transfers the first-order system to a wave-type equation.

\subsubsection{Morawetz estimates}
An immediate benefit of squaring the Dirac operator is that we can now exploit the integrated local energy decay estimates or the Morawetz estimates. These estimates are now able to control spacetime norms of the spinor field and provide the basic coercive mechanism for the nonlinear analysis.

One remarkable difference from the scalar wave equation is that the natural energy-momentum tensor associated with the squared Dirac equation is not divergence-free.
We discuss it in detail in Section \ref{sec:linear-dirac}.

\subsubsection{$r^p$-method}
Based on the Morawetz estimate, we establish an $r^p$-weighted hierarchy in the spirit of Dafermos-Rodnianski \cite{DR}. The hierarchy yields quantitative energy decay through the dyadic argument together with the pigeonhole principle and forms the basis for the bootstrap scheme developed later.

\subsubsection{Teukolsky equation}
Another important approach to the analysis of linear Dirac equations is to reduce the first-order system to a second-order wave-type equation. On black hole backgrounds, this often leads to spin-$\frac12$ Teukolsky equations, for which robust decay and scattering theories have been developed. Such reductions allow one to exploit the extensive analytic machinery available for wave equations and have proved remarkably successful in the study of linear problems. We refer the reader to Ma--Zhang \cite{mazhang}.

The present work, however, is motivated by nonlinear coupled systems. For our purposes, it is advantageous to retain the original spinor field as the primary unknown throughout the analysis.

A crucial point is that the wave equation obtained by squaring the Dirac operator is still an equation for the original spinor field rather than for a derived scalar quantity. Consequently, the first-order spinorial structure and the Dirac current remain available throughout the analysis.

Therefore, the squared equation plays an auxiliary analytic role rather than replacing the original Dirac evolution.
With this linear framework at hand, we now turn to the nonlinear tensorial wave-Dirac system, where these geometric and analytic ingredients are combined within a bootstrap argument.



\subsection{Application to nonlinear tensorial wave-Dirac system}
Building upon the linear geometric framework described above, we now turn to the nonlinear tensorial wave-Dirac system \eqref{eq-tensor-dirac}. We explain how the geometric and analytic ingredients developed in the linear theory are incorporated into the nonlinear bootstrap argument.
\subsubsection{Motivation}
The present work is motivated by the long-term goal of understanding the nonlinear Einstein-Dirac system.
However, the Einstein-Dirac system contains genuinely quasilinear interactions. In particular, the nonlinear terms in the Einstein equations involve derivatives of the metric itself, leading to derivative loss, which substantially complicates the nonlinear analysis. 

As a first step, it is natural to isolate the spinorial aspects of the problem while removing the derivative-loss mechanism.

The Maxwell-Dirac system provides a natural intermediate model. The wave component no longer exhibits derivative loss, while the coupling still retains the essential interaction between tensorial fields and Dirac spinors.
Nevertheless, the Maxwell-Dirac system introduces additional gauge-theoretic aspects which are not the primary focus of the present work.

The tensorial wave-Dirac system studied here is introduced to isolate the geometric interaction between tensorial wave propagation and Dirac spinors without the additional complications coming from gauge invariance.

Despite its simplified form, this model preserves the essential geometric features relevant to the nonlinear analysis. In particular, it retains both the null structure associated with an antisymmetric two-form tensor and the intrinsic null decomposition of the Dirac equation. Consequently, it provides a natural setting in which the spinorial geometric framework developed in the linear theory can be combined with the vector-field method for tensorial wave equations.

\subsubsection{A mini proof sketch}

The main scheme of the proof is motivated by the physical-space approach underlying the black hole stability program. Following the idea developed by Dafermos, Rodnianski \cite{DR}, and Holzegel \cite{holzegel10} we combine integrated local energy decay with $r^p$-weighted energy hierarchies as the principal analytic mechanism for deriving quantitative decay estimates. In the present work, this strategy is adapted to the Dirac equation while preserving its intrinsic first-order spinorial structure.

More precisely, we introduce the null components $(\alpha,\underline\alpha,\rho,\sigma)$ for the tensor fields $F_{\mu\nu}$.
For the tensor field we employ two complementary weighted energy hierarchies.
The $r^p$-hierarchy is applied to the null components $(\underline\alpha,\rho,\sigma)$ with $0\le p \le2$. Then the second hierarchy with $r^q$ is encoded to the outgoing component $\alpha$ with $2\le q\le 3+\delta$, as well as the spinor fields in terms of the squared equation, with $0\le p\le2-\delta$. Here $\delta$ is chosen to be $0<\delta<\frac1{20}$.

Since $q$ here is at most $3+\delta$, the decay of $\alpha$ is relatively weaker.
This produces a difficulty in the control of the spacetime errors, since the outgoing component exhibits weaker decay than the remaining null components and the spinors.
To compensate this, we rewrite the equation of the tensor fields in terms of wave equation and view $\alpha$ and $\underline\alpha$ as solutions to wave equations. Then we apply the $r^p$-method to wave $(\alpha,\underline\alpha)$ with $p\le 2-\delta$.

The first genuinely nonlinear difficulty appears in the top-order weighted spinorial estimate.
Indeed, after commuting the Dirac equation and applying the $r^p$-weighted estimate to the squared equation, one encounters the spacetime integral arising from the nonlinear source.
Therefore, the first step to close the $r^p$-method for the spinor with $0\le p\le 2-\delta$, is to control the integral of the nonlinearity with the weight $r^{2-\delta}$, which is of the form:\footnote{Although the vector fields should be formulated in terms of the usual Lie derivative $\mathcal L_Z$ for the tensor fields and the Kosmamm Lie derivative $\mathscr L_Z$, for the sake of simplicity, we will use an abuse of notation and simply write $Z$ throughout this section.}
\begin{align}
    \int_{\mathcal D^{2\tau}_{\tau}} r^{2-\delta} \langle \gamma^\lambda\nabla_\lambda Z^{\le N-1} ( F_{\mu\nu}\gamma^\mu\gamma^\nu \psi), \gamma^T \nabla_L Z^{\le N-1}\psi\rangle \,dV,
\end{align}
where $\mathcal D^{2\tau}_{\tau}$ is the domain bounded by two null hypersurfaces $\Sigma_{2\tau}$, $\Sigma_\tau$ and the time-like surface $\{r=R\}$.

We point out that the Clifford algebra allows only a restricted combination of the interaction. For example, if $F=\underline\alpha$, i.e., $\mu= A$, $\nu = L$, then the integral vanishes for $\lambda=\underline L$, since $(\gamma^{\underline L})^2\equiv0$. Thus the bad derivative of the bad component such as $\nabla_{\underline L}\underline\alpha$ does not appear in the nonlinearity. Consequently, the most singular null interaction is structurally forbidden from the nonlinear error terms.

In the nonlinear estimates, we exploit a hierarchy according to the number of derivatives. The commuted fields are divided into higher-order and lower-order components. Here the gap between these two levels is chosen sufficiently large (in our argument, eight derivatives are sufficient), so that every nonlinear interaction necessarily contains at least one lower-order factor.

Consequently, the genuinely difficult interactions are those involving one higher-order and one lower-order factor. When both factors carry only intermediate numbers of derivatives, for instance around order $N-5$, they can be estimated symmetrically by applying the H\"older inequality together with $L^4$-weighted Sobolev estimates to both factors.

The large regularity assumption is used precisely to create a gap between the higher-order and lower-order regimes. This guarantees that every quadratic nonlinear interaction contains at least one factor with sufficiently many derivatives to supply the energy estimate and another with sufficiently few derivatives to provide a certain decay.

Therefore, the structural difficulty arises in the high-low and low-high interactions. In both cases, Sobolev embedding can be only applied to the lower-order factor, while the higher-order factor must be controlled purely through energy estimates. The nonlinear argument is designed to ensure that the lower-order component always enjoys a sufficient decay.

In view of this observation, the worst interaction occurs when $h=\underline\alpha$ is higher-order and the derivative $\nabla_L$ acts on the lower-order spinor $Z^{\le N-8}\psi$:
\begin{align}
    \int_{\mathcal D^{2\tau}_{\tau}} r^{2-\delta} \langle Z^{\le N-1}\underline\alpha \gamma^{\underline L}\gamma^{L}\gamma^{\underline L} \nabla_L Z^{\le N-8}\psi, \gamma^T \nabla_L Z^{\le N-1}\psi\rangle \,dV.
    \end{align}
Note that the weighted Sobolev inequality can be only applied to the lower-order spinor $\psi$. Unfortunately, it already involves the derivative $\nabla_L$, and hence one cannot expect further decay by using the Sobolev embedding.

To manage this problem, we first foliate the domain $\mathcal D^{2\tau}_{\tau}$ via the ingoing null cones $\underline{\mathcal C}_v$. Indeed, this foliation turns out to be an appropriate choice in geometric perspective, since the null component $\underline\alpha$ is aligned with the ingoing cone, and the Clifford algebra implies that the spinor fields in the integrand become $\psi_-$ components. Then an application of the weighted Sobolev inequality to $\psi_-$ gives the derivative $\nabla_{\underline L}$. Here, the null decomposition for the spinor shows $\nabla_{\underline L}\psi_-= \slashed\nabla\psi_+$ up to lower-order terms, which gives a gain $r^{-1}$ via an obvious inequality $|\slashed\nabla\psi_+| \lesssim r^{-1}|Z^{\le1}\psi|$. Then the remaining task is to invoke the integrated local energy decay estimates for the spinor so that the energy can be controlled along the outgoing null again.

A further important feature of the argument is that the top-order weighted wave energy for the extreme tensor components $\alpha$ and $\underline\alpha$ is never required.
The wave equations for these components are employed only below the top order. At the top order, whenever a derivative such as $\nabla_L Z^{\le N-1}\underline\alpha$ appears, the null structure equations for the tensor fields allow us to express it schematically by \(\nabla_{\underline L}Z^{N-1}\alpha\) together with scalar components. We then integrate by parts in the \(\underline L\)-direction. In this way the derivative is transferred away from the extreme component, and the estimate closes without invoking the top-order \(r^{2-\delta}\)-weighted wave energy for \(\alpha\) or \(\underline\alpha\).
At first glimpse, the integration by parts along the $\underline L$-direction appears to introduce an additional derivative loss, as the term $\nabla_{\underline L}\nabla_L Z^{\le N-1}\psi$ arises. However, this derivative loss turns out to be not harmful, since the wave equation, obtained from squaring the Dirac equation shows that $\nabla_{\underline L}\nabla_L\psi$ can be replaced by $\slashed\Delta\psi$ up to lower-order terms. Then we use the integration by parts with respect to the angular variables, thereby recovering the full derivative.

The preceding discussion illustrates a key idea of the proof. The first-order geometry of the Dirac equation is preserved throughout the nonlinear analysis, while the wave equation obtained by squaring the Dirac operator is used only as an auxiliary analytic tool for deriving spacetime estimates. The interaction between the tensorial null structure and the spinorial geometry ultimately allows the coupled energy hierarchy to close.


\subsection{Related problems}
Now we briefly discuss two related nonlinear systems, which motivate the present work.
\subsubsection{Einstein-Dirac system}
The system is given by
\begin{align}
\begin{aligned}
    \mathrm{Ric}_{\mu\nu}-\frac12 \mathcal R g_{\mu\nu} = T_{\mu\nu}(\psi), \\
    (i\gamma^\mu\nabla^{\mathbf S}_\mu -m) \psi = 0,
    \end{aligned}
\end{align}
where $\mathrm{Ric}_{\mu\nu}$ is the Ricci tensor and $\mathcal R$ is the scalar curvature and $m\ge0$ is a mass of the $\frac12$-spin particle. Here, $T_{\mu\nu}(\psi)$ is the energy-momentum tensor associated with the spinor field, given by
\begin{align}
    T_{\mu\nu}(\psi) & = \frac{i}{4} \left( \langle \psi, \gamma_\mu \nabla_\nu\psi\rangle - \langle \nabla_\nu\psi, \gamma_\mu\psi\rangle + \langle \psi,\gamma_\nu\nabla_\mu\psi\rangle - \langle \nabla_\mu\psi, \gamma_\nu\psi\rangle \right),
\end{align}
where we omit the superscript $\mathbf S$ for simplicity and $\langle \cdot,\cdot\rangle$ is the Dirac pairing.

The first global nonlinear stability result for the Einstein--Dirac system was obtained by Chen \cite{chen} in the massless case, where the scaling properties of the massless Dirac equation are compatible with the classical vector-field method.

More recently, LeFloch, Ma, and Zhang \cite{leflochmazhang} established the global nonlinear stability of Minkowski spacetime for the massive Einstein-Dirac system by developing a gauge-invariant treatment of spinor fields within the Euclidean-hyperboloidal foliation framework.
Both works address the full quasilinear Einstein-Dirac system and therefore face the derivative-loss mechanism inherent in Einstein's equations.

Unlike these works, our objective is not to establish the nonlinear stability of the Einstein--Dirac system itself. Instead, we deliberately remove the quasilinear derivative-loss mechanism in order to isolate the interaction between the null geometry of antisymmetric tensor fields and the intrinsic first-order geometry of the Dirac equation. The resulting semilinear model provides a natural setting for developing the spinorial geometric framework that we expect to be useful for more general quasilinear systems in the future.

\subsubsection{Maxwell-Dirac system}
The system is given by
\begin{align}
    \begin{aligned}
    \nabla^\mu F_{\mu\nu} = J_\nu, \\
       ( i\gamma^\mu\mathbf D_\mu-m)\psi =0,
    \end{aligned}
\end{align}
where $F$ is a closed $2$-form, i.e., an antisymmetric $2$-tensor satisfying $dF=0$.

This system has been extensively studied on Minkowski spacetime. A typical approach is to impose a specific gauge condition. In particular, the above system becomes a system of nonlinear wave-type equations under the Lorenz gauge, while it becomes a coupled wave-elliptic system under the Coulomb gauge. Following the global result under the Lorenz gauge by Georgiev \cite{geogiev}, Psarelli \cite{psarelli1} obtained global existence in a gauge-independent setting. Global results under the Coulomb gauge were established by Gavrus-Oh \cite{gavrusoh}. Lee \cite{klee} proved modified scattering for the system in $(1+4)$-dimensional spacetime. Modified scattering results were obtained by Herr-Ifrim-Spitz \cite{herrifrimspitz} and Cho-Lee \cite{choklee} under the Lorenz gauge condition. While the Minkowski problem has been studied extensively, the corresponding theory on curved spacetimes remains comparatively less developed. We refer the reader to Ginoux--M\"uller \cite{ginoux}.

\subsubsection{Other related works}
Beyond the Einstein-Dirac and Maxwell-Dirac systems, significant progress has also been made in the analysis of nonlinear Dirac equations, particularly on the Minkowski spacetime.
A typical approach is to reformulate the Dirac equation in terms of a half-Klein-Gordon equation and to exploit the dispersive properties of the Dirac spinor inherited from the half-Klein-Gordon operator. Following the work of Escobedo-Vega \cite{escobedovega}, Machihara-Nakamura-Nakanishi-Ozawa \cite{machihara} obtained global well-posedness for a cubic Dirac equation using endpoint Strichartz estimates for half-Klein-Gordon equations. Small-data scattering for the cubic Dirac equation with critical Sobolev data was subsequently proved by Bejenaru-Herr \cite{behe,behe1} and Bournaveas-Candy \cite{boucan}, using the null structure of the cubic nonlinearity. Small-data scattering for the Dirac-Klein-Gordon system was established in \cite{beherr}, and the regularity threshold was improved by Wang \cite{wang} under an additional angular regularity assumption. This was further improved, and conditional large-data scattering was obtained by Candy-Herr \cite{canhe2,canhe1,canhe}.

After eliminating the auxiliary fields, both the Dirac-Klein-Gordon and Maxwell-Dirac systems reduce to cubic Dirac equations with Hartree-type nonlinearities \cite{chagla,chagla1}. Small-data scattering results were obtained by Tesfahun \cite{tes,tes1} and Yang \cite{cyang}. The regularity was subsequently improved in \cite{chohonglee}. We refer the reader to Georgiev--Shakarov \cite{geosha}, Cho--Lee--Ozawa \cite{choleeoz}, Cloos \cite{cloos}, and Cho--Kwon--Lee--Yang \cite{chokwonleeyang} for related scattering and modified scattering results. We also refer to \cite{herrmaul} for decay estimates for Dirac equations via the virial identity.

Recently, the analysis of Dirac equations on curved spacetimes has also attracted considerable attention in a variety of geometric settings. On black hole backgrounds, decay, scattering, and asymptotic properties of massive Dirac fields have been investigated by H\"afner-Nicolas \cite{hanicolas}, and Finster-Kamran-Smoller-Yau \cite{finster} and by Batic \cite{batic}. More recently, Ma and Zhang \cite{mazhang} established sharp decay estimates for the massless Dirac equation on the Schwarzschild spacetime through a reduction to the spin-$\frac12$ Teukolsky equation. The characteristic Cauchy problem for Dirac fields on curved spacetimes was studied by H\"afner-Nicolas \cite{hafnernicolas}. On cosmological backgrounds, global existence for nonlinear Dirac equations on FLRW and de Sitter spacetimes was established by Galstian-Yagdjian \cite{galyag} and by \cite{yag}. We also mention the work of B\"ar \cite{bar} on the spectral and geometric analysis of the Dirac operator on hyperbolic manifolds, and the global existence theory for Dirac-wave maps developed by Brandling and Kr\"oncke \cite{brand}. In the context of the dispersive PDE, we refer the readers to \cite{artzcaccia,cacciasu,caccia1,caccia2} for linear estimates of the Dirac equations. More recently, endpoint Strichartz estimates \cite{strichartz} for half-Klein-Gordon equations on asymptotically flat spacetimes were established in \cite{herrh}, by a construction of phase-space parametrices, which provide an application to cubic Dirac equations on curved backgrounds.

\subsubsection{Historical remarks on geometric analysis of wave equations}
Alongside the development of the analysis of Dirac equations, another major direction of the study of hyperbolic equations has been the geometric analysis of wave equations on Lorentzian manifolds. Beginning with the vector-field method of Klainerman \cite{klai} and the nonlinear stability work of Christodoulou-Klainerman \cite{christoklai} on the Minkowski spacetime, geometric energy methods have become one of the principal tools for studying global existence and quantitative decay of waves on curved backgrounds. Subsequent developments by Dafermos, Rodnianski, Holzegel \cite{dahorod1,dahorod} and many others established the physical-space approach based on integrated local energy decay, red-shift estimates \cite{daferrod1}, and $r^p$-weighted energy hierarchies \cite{DR}, leading to a robust framework for the analysis of wave equations on black hole spacetimes \cite{daferrod} and, ultimately, the nonlinear stability problem.

\subsection{Discussion}
Here we give some remarks.
\subsubsection{On the role of the first-order Dirac structure}
A natural question is whether nonlinear Dirac equations can be analysed entirely through their associated second-order wave equations. Since the Dirac operator is a square root of the wave operator, one may square the Dirac equation to obtain a tensorial wave equation for the spinor field. This reformulation makes available many of the fundamental tools from the theory of wave equations, such as Morawetz estimates, the $r^p$-method, and the corresponding weighted energy hierarchy. In particular, for the massless Dirac equation, it may seem natural to expect that the resulting wave reduction alone might suffice for the nonlinear analysis.

The present work suggests that this is not generally the case, even for the simplified tensor-Dirac model \eqref{eq-tensor-dirac}. Although the wave formulation plays a fundamental role in deriving decay estimates, it is not sufficient to close the nonlinear argument, even for the simplified tensor-Dirac model \eqref{eq-tensor-dirac} considered here. The nonlinear estimates naturally require the propagation of the Dirac current energy, which provides $L^2$-control of the spinor itself and cannot be replaced by the derivative energies arising from the wave formulation alone. Moreover, the nonlinear analysis makes essential use of the null structure encoded in the first-order Dirac equation, which is not naturally reflected in the associated second-order wave equation.

Therefore, we suggest that the wave formulation can be viewed as an analytical tool rather than a complete replacement for the Dirac equation. The global argument relies on a genuinely coupled propagation of the tensorial wave energies and the Dirac current energies. The tensor-Dirac model \eqref{eq-tensor-dirac} studied in the present paper therefore illustrates that, in nonlinear problems involving Dirac fields, reducing the system to a second-order wave equation alone is not, in general, expected to provide a complete analytical framework. This observation may also help explain why nonlinear Dirac equations cannot, in general, be treated merely as nonlinear wave equations.

\subsubsection{Teukolsky spin-$\frac12$ equation}
The present approach also differs from the Teukolsky formalism for spin-$\frac12$ fields. In the latter, one derives decoupled second-order equations for suitable spinor components and studies these component equations directly. In contrast, we preserve the full Dirac spinor throughout the analysis. The associated second-order wave equation is introduced only as an auxiliary tool for deriving Morawetz estimates, the $r^p$-hierarchy, and decay estimates, while the first-order Dirac equation remains an essential part of the nonlinear argument. In particular, the propagation of the Dirac current energy and the null structure are carried out entirely at the level of the first-order equation.

\subsubsection{Application to the Einstein-Dirac equations}
In the harmonic gauge, the Einstein-Dirac system can be schematically viewed as a coupled quasilinear wave-Dirac system as follows:
\begin{align}
    \begin{aligned}
        \Box_g g_{\mu\nu}& = Q_{\mu\nu}(\partial g,\partial g) + \psi \cdot \nabla\psi, \\
        i\gamma^\mu\nabla_\mu\psi &= 0,
    \end{aligned}
\end{align}
where, schematically, the quadratic source terms $\psi\cdot\nabla\psi$ originate from the energy-momentum tensor of the Dirac field and $Q_{\mu\nu}$ arises from the nonlinear part of the Einstein equations.

From this perspective, the spinorial nonlinearities possess essentially the same first-order and null structures exploited in the present work. The principal additional difficulty lies in the quasilinear nature of the Einstein equation, where derivative loss and the nonlinear wave interactions must be handled simultaneously.

Consequently, the present work may be viewed as isolating the spinorial component of the Einstein-Dirac system, while removing the genuinely quasilinear difficulties associated with the Einstein equations. From this perspective, we believe that the geometric method developed in this paper provides a useful starting point for the study of more general nonlinear Dirac systems, such as the Einstein-Dirac equations.

\subsection{Organisation}
The rest of this paper is organised as follows. We end this section with the introduction of the notations, which will be used throughout this paper.


In Section \ref{sec:prelim-geometry}, we present some preliminaries on differential geometry. Here we review Lorentzian geometry, introduce the double null foliation, and recall the basic multiplier identities and energy estimates.

Sections \ref{sec:Dirac} and \ref{sec:linear-dirac} form the first main part of the present paper. In Section \ref{sec:Dirac}, we formulate the Dirac equation in terms of spin geometry. We also introduce the null decomposition of the Dirac equation, the Dirac current, and the Kosmann Lie derivative. Here we present the Kosmann Lie derivative and the commutator identity in terms of the deformation tensor. Section \ref{sec:linear-dirac} is devoted to the linear analysis of the massless Dirac equation. By squaring the Dirac equation, we obtain a wave-type equation and then establish quantitative estimates for the linear Dirac equation, including an $r^p$-weighted energy hierarchy.

Section \ref{sec:tensor-dirac} is the second main part of this paper. Here we state our main result concerning the nonlinear stability of the tensorial wave-Dirac system, and discuss the existence of local solutions via the Picard iteration with energy inequalities. We also present the bootstrap assumptions for the global analysis in this section. Then we introduce the $r^p$-weighted energy hierarchy for this nonlinear system in Section \ref{sec:rp-method}.
Sections \ref{sec:bdd-weight-spinor}, \ref{sec:bdd-weight-tensor}, \ref{sec:improv-spinor}, and \ref{sec:imp-energy-est} concern the proof of Theorem \ref{main-thm-formal}. Section \ref{sec:interior} concerns the decay of the energy in the interior domain, while the preceding sections concern the exterior domain.

In Section \ref{sec:iteration}, we recall the bootstrap assumption and list the improved estimates, which have been established in Section \ref{sec:bdd-weight-spinor}, \ref{sec:bdd-weight-tensor}, \ref{sec:improv-spinor}, \ref{sec:imp-energy-est}, and then show how the improved decay estimates can be further obtained via the iteration.

\subsection{Notations}
Throughout this paper, we consider a $(1+3)$-dimensional Lorentzian manifold $(\mathcal M,g)$ with $g=m+h$ and $m$ is the Minkowski metric and $h$ satisfies \eqref{background-assum1} and \eqref{background-assum2} for sufficiently small $c_\alpha$ and $c'_\alpha$.

We suppose that the manifold $\mathcal M$ admits a foliation $\{\Sigma_{\tau}\}_{\tau\ge0}$, where $\Sigma_{\tau}$ is the time-slice for $r\le R$, i.e, $\Sigma_\tau = \{t=\tau\}$, and $\Sigma_\tau$ is outgoing null hypersurface for $r\ge R$, i.e.,  $\Sigma_\tau = \{u=\tau-R\}\cap \{v\ge \tau+R\}$. Here $R\ge1$ is a fixed large number and $u,v$ are null coordinates.
We consider the domain $D^{\tau}_{\tau_0}$ bounded by two hypersurfaces $\Sigma_\tau$ and $\Sigma_{\tau_0}$. 

We denote by $\mathcal D^{\tau}_{\tau_0}$ the exterior part of the domain $D^{\tau}_{\tau_0}$, i.e., $\mathcal D^{\tau}_{\tau_0}= D^{\tau}_{\tau_0}\cap \{r \ge R\} $.

The energy of the spinor fields in terms of the Dirac current is defined as follows:
\begin{align}
    \mathscr E^D[ \psi ]^2(\tau) := \int_{\substack{u=\tau-R \\ v\ge \tau+R } } \langle \psi, \gamma(L)\psi\rangle \,d\sigma, \\
    \mathscr F^D[\psi]^2(v,\tau_0,\tau) := \int_{\underline{\mathcal C}_v\cap \mathcal D^{\,\tau}_{\tau_0} } \langle \psi, \gamma(\underline L)\psi\rangle \,d\sigma.
\end{align}
For the interior part, we define
\begin{align}
   {}_{\rm int} \mathscr E^D [\psi]^2(\tau) := \int_{\Sigma_\tau\cap \{r\le R\} } \langle \psi,\gamma(T)\psi\rangle\,d\sigma.
\end{align}
We also define the energy for the spinor fields in terms of wave:
\begin{align}
    \mathcal E^D[\psi]^2(\tau) := \int_{\Sigma_\tau} \langle \nabla_L\psi,\gamma(T)\nabla_L\psi\rangle+\langle\slashed\nabla\psi,\gamma(T)\slashed\nabla\psi\rangle+\iota_{r\le R}\langle\nabla_{\underline L}\psi,\gamma(T)\nabla_{\underline L}\psi\rangle + \frac{1}{r^2}\langle\psi,\gamma(T)\psi\rangle  \,d\sigma, \\
    \mathcal F^D[ \psi ]^2(v,\tau_0,\tau) := \int_{\underline{\mathcal C}_v\cap \mathcal D^{\,\tau}_{\tau_0}}  \langle\nabla_{\underline L}\psi,\gamma(T)\nabla_{\underline L}\psi\rangle+ \langle\slashed\nabla\psi,\gamma(T)\slashed\nabla\psi\rangle+  \frac{1}{r^2}\langle\psi,\gamma(T)\psi\rangle \,d\sigma.
\end{align}
We also define ${}_{\rm int}\mathcal E^D[\psi] $ in a similar way as ${}_{\rm int}\mathscr E^D[\psi]$.
We also invoke the following bulk terms:
\begin{align}
    \begin{aligned}
        {}_{(-1-\delta)}{\mathcal E^D}[\psi]^2 (\tau) := \int_{\Sigma_\tau}\frac1{r^{1+\delta}} (  \langle \nabla_L\psi,\gamma(T)\nabla_L\psi\rangle+\langle\slashed\nabla\psi,\gamma(T)\slashed\nabla\psi\rangle+\langle\nabla_{\underline L}\psi,\gamma(T)\nabla_{\underline L}\psi\rangle  )\,d\sigma, \\
        {}_{(-3-\delta)}{\mathcal E^D}[\psi]^2(\tau) := \int_{\Sigma_\tau}\frac1{r^{3+\delta}}\langle\psi,\gamma(T)\psi\rangle\,d\sigma.
    \end{aligned}
\end{align}
We define the weighted energy derived from the $r^p$-method \cite{DR}: for $1\le p<2$,
\begin{align}
\begin{aligned}
    {}^{(p)}\mathcal E^D[\psi]^2(\tau) := \int_{\Sigma_\tau} r^p \langle \nabla_L\psi,\gamma(T)\nabla_L\psi\rangle\,d\sigma, \\
    {}^{(p-1)}\mathcal E^D[\psi]^2(\tau) := \int_{\Sigma_\tau} r^{p-1} ( \langle \nabla_L\psi,\gamma(T)\nabla_L\psi\rangle+ \langle\slashed\nabla\psi,\gamma(T)\slashed\nabla\psi\rangle)\,d\sigma, \\
    {}^{(p)}\mathcal F^D[\psi]^2(v,\tau_0,\tau) := \int_{\underline{\mathcal C}_v \cap \mathcal D^{\tau}_{\tau_0} } r^p \langle \slashed\nabla\psi,\gamma(T)\slashed\nabla\psi\rangle\,d\sigma.
    \end{aligned}
\end{align}
We define
\begin{align}
   {}^{(p)} \mathfrak E^D[\psi]^2: =  {}^{(p)}\mathcal E^D[\psi]^2(\tau) + \int_{\tau_{i}}^{\tau_{i+1}} {}^{(p-1)}\mathcal E^D[\psi]^2(\tau)\,d\tau.
\end{align}
We let $0\le k\le N$. We define
\begin{align}
    \mathscr E^D_{\le k}[ \psi ] := \sum_{|I|\le k}\mathscr E^D[ \mathscr L_Z^I\psi ],
\end{align}
and we also define the analogue for $\mathscr F^D_{\le k},\mathcal E^D_{\le k}$, and so on.



\section{Preliminaries on Geometry}\label{sec:prelim-geometry}
\subsection{Geometry of the background spacetime}\label{subsec:background-geometry}
As we have mentioned in Section \ref{sec:Intro}, we consider a Lorentzian manifold $(\mathcal M,g)$, whose underlying differentiable manifold is $\mathbb R^4$ and Lorentzian metric $g$ is a sufficiently small perturbation of the Minkowski metric.

To be precise, we consider the underlying differential structure of the spacetime $\mathcal M$ to be given by $\mathbb R^4=\{(x^0,x^1,x^2,x^3)\}$. Accordingly, we define the coordinate vector fields $(\frac{\partial}{\partial x^0},\frac{\partial}{\partial x^1},\frac{\partial}{\partial x^2},\frac{\partial}{\partial x^3})$. Then the coordinate vector fields are globally regular. Now we define
\begin{align*}
    r = \sqrt{ (x^1)^2+ (x^2)^2+(x^3)^2}.
\end{align*}
We also define the spherical coordinates in the usual way:
\begin{align}
    (x^1,x^2,x^3) := ( r\sin\theta\cos\phi, r\sin\theta\sin\phi, r\cos\theta),
\end{align}
where $0\le\theta\le\pi$ and $0\le\phi\le2\pi$. We also put $t=x^0$.

We assume that $g$ is a time-oriented Lorentzian metric on $\mathcal M$ with the signature $(-,+,+,+)$.
We write
\begin{align}
    g = m+h,
\end{align}
and $h$ admits a decomposition:
\begin{align}
    h(t,r,\omega) = h^{\rm rad}(t,r)+ h^{\rm ang}(t,r,\omega),
\end{align}
where $\omega\in\mathbb S^2$ can be parametrised by $\omega=(\sin\theta\cos\phi,\sin\theta\sin\phi,\cos\theta)\in\mathbb S^2$. Here we refer to $h^{\rm rad}$ as the spherically symmetric component of $h$ and $h^{\rm ang}$ as the spherically non-symmetric component. Then $h^{\rm rad}$ and $h^{\rm ang}$ satisfy
\begin{align}
    |\partial^\alpha h^{\rm rad}(t,r)| \le  c_\alpha \langle r\rangle^{-1-|\alpha|}, \quad |\partial^\alpha h^{\rm ang}(t,r,\omega)| \le  c'_{\alpha}\langle r\rangle^{-2-|\alpha|},
\end{align}
where the constants $c_\alpha$ and $c'_\alpha$ are sufficiently small for all multi-indices $|\alpha|\le N+2$. Note that the metric $g$ is not necessarily stationary. Indeed, the above decay assumption implies $|\partial_t g|= O(r^{-2}) $.

Since the metric $g$ is sufficiently close to the Minkowski metric, the coordinate vector field $\partial_{x^0}=\partial_t$ is a globally future-directed time-like vector field.

We choose $R\ge1$ sufficiently large so that the asymptotically flat structure is already valid for $r\ge R$. We
consider the initial hypersurface $\Sigma_{\tau=0}$, where $\Sigma_{\tau=0}\cap \{r\le R\}$ is the time-slice hypersurface $\{t=0\}$ and $\Sigma_{\tau=0}\cap\{r\ge R\}$ is an outgoing null hypersurface. To define the foliation, we let $\phi_{\tau}$ be the one-parameter group of diffeomorphisms generated by $\partial_{x^0}$. We also denote by $J^+(\Sigma_{\tau=0})$ the causal future in the manifold $\mathcal M$ with respect to the metric $g$ of the surface $\Sigma_{\tau=0}$. Then we assume that
\begin{align}
    J^+(\Sigma_{\tau=0}) = \bigcup_{\tau\ge0}\Sigma_{\tau},
\end{align}
so that for any $\tau\ge0$, $\Sigma_\tau\cap \{r\le R\}$ is the time-slice and $\Sigma_\tau\cap \{r\ge R\}$ is an outgoing null hypersurface. 

\subsection{Basic differential geometric conventions}
We present the elementary geometric conventions used throughout the paper. Let $\nabla$ be the Levi-Civita connection of $g$. Thus $\nabla$ is torsion-free and metric-compatible:
\begin{align}
    \nabla_XY-\nabla_YX=[X,Y], \qquad \nabla g=0,
\end{align}
for all vector fields $X,Y$ on $\mathcal M$. If $T$ is a tensor field, its covariant derivative is denoted by $\nabla T$, and repeated Greek indices are contracted using the spacetime metric $g$.

\begin{defn}[Curvature convention]
The Riemann curvature tensor is defined by
\begin{align}
    R(X,Y)Z=\nabla_X\nabla_YZ-\nabla_Y\nabla_XZ-\nabla_{[X,Y]}Z.
\end{align}
In coordinates, we write
\begin{align}
    (\nabla_\mu\nabla_\nu-\nabla_\nu\nabla_\mu)V^\alpha
    ={R^\alpha}_{\beta\mu\nu}V^\beta.
\end{align}
The Ricci tensor and scalar curvature are given by
\begin{align}
    \mathrm{Ric}_{\mu\nu}={R^\alpha}_{\mu\alpha\nu}, \quad
    \mathcal R=g^{\mu\nu}\mathrm{Ric}_{\mu\nu}.
\end{align}
We also use the wave operator
\begin{align}
    \Box_g\phi=g^{\mu\nu}\nabla_\mu\nabla_\nu\phi
\end{align}
for scalar functions $\phi$.
\end{defn}

\begin{prop}[Ricci identity]
For any $(k,l)$-tensor field $T$, the commutator of covariant derivatives is determined by the curvature tensor. In particular,
\begin{align}
    [\nabla_\mu,\nabla_\nu]X^\alpha={R^\alpha}_{\beta\mu\nu}X^\beta,
    \qquad
    [\nabla_\mu,\nabla_\nu]\omega_\alpha=-{R^\beta}_{\alpha\mu\nu}\omega_\beta.
\end{align}
More generally, each contravariant index contributes a term with sign $+$ and each covariant index contributes a term with sign $-$.
\end{prop}

\begin{prop}[Commutation formula for the wave operator]
Let $X$ be a vector field and let $\phi$ be a scalar function. Then
\begin{align}\label{comm-wave-vectorfield}
    [\Box_g,X]\phi
    =2\nabla^\mu X^\nu\nabla_\mu\nabla_\nu\phi+(\Box_gX^\nu)\nabla_\nu\phi.
\end{align}
Equivalently, using the deformation tensor ${}^{(X)}\pi$, one has
\begin{align}
    [\Box_g,X]\phi
    = {}^{(X)}\pi_{\mu\nu}\nabla^\mu \nabla^\nu\phi + \nabla_\mu \pi^{\mu\nu} \nabla_\nu\phi -\frac12 \nabla^\mu ( \mathrm{Tr}\pi)\nabla_\mu\phi. 
\end{align}
\end{prop}
We refer the reader to \cite{alinhac} for the proof.
\begin{defn}[Volume forms and divergence]
Let $dV_g$ be the spacetime volume form induced by $g$. For a vector field $P$, its divergence is
\begin{align}
    \mathrm{div}\,P=\nabla_\mu P^\mu.
\end{align}
If $\mathcal D\subset \mathcal M$ is a spacetime domain with piecewise smooth boundary, we denote by $n$ the unit normal to the boundary $\partial\mathcal D$ of $\mathcal D$.
\end{defn}

\begin{prop}[Divergence theorem]
Let $P$ be a smooth vector field on a spacetime domain $\mathcal D$. Then
\begin{align}
    \int_{\mathcal D}\nabla_\mu P^\mu\,dV_g
    =\int_{\partial\mathcal D} P^\mu n_\mu\,d\sigma_{\partial\mathcal D},
\end{align}
with the usual interpretation of the boundary measure and normal on null hypersurfaces.
\end{prop}
 This identity will be applied to energy currents associated with wave, and Dirac-type equations.
\begin{defn}[Deformation tensor]
For a vector field $X$, its deformation tensor is
\begin{align}
    {}^{(X)}\pi_{\mu\nu}=\nabla_\mu X_\nu+\nabla_\nu X_\mu.
\end{align}
It measures the failure of $X$ to be Killing and appears in the divergence identity for vector-field multipliers.
\end{defn}
\begin{lem}
    We have
    \begin{align}
    \begin{aligned}
            {}^{(L)}\pi_{AB} = 2 \chi_{AB}, \quad {}^{(\underline L)}\pi_{AB} = 2\underline\chi_{AB},
    \end{aligned}
\end{align}
and all the remaining components for ${}^{(L)}\pi_{\mu\nu}$ and ${}^{(\underline L)}\pi_{\mu\nu}$  are zero, for $\mu,\nu\in\{e_A,L,\underline L\}$.
\end{lem}

\begin{defn}[Energy-momentum tensor for waves]
For a scalar field $\phi$, we define the energy-momentum tensor by
\begin{align}
    T_{\mu\nu}[\phi]
    =\nabla_\mu\phi\nabla_\nu\phi-\frac12 g_{\mu\nu}\nabla^\alpha\phi\nabla_\alpha\phi.
\end{align}
Given a vector field $X$, the associated current is
\begin{align}
   \mathcal J^X_\mu[\phi]=T_{\mu\nu}[\phi]X^\nu.
\end{align}
\end{defn}

\begin{prop}[Multiplier identity for waves]
Let $\phi$ be a scalar field and let $X$ be a vector field. Then
\begin{align}
    \nabla^\mu\mathcal J^X_\mu[\phi]
    = (\Box_g\phi)(X\phi)+\frac12 T^{\mu\nu}[\phi]{}^{(X)}\pi_{\mu\nu}.
\end{align}
In particular, if $\Box_g\phi=F$, then the bulk term generated by the multiplier $X$ is
\begin{align}
    F X\phi+\frac12 T^{\mu\nu}[\phi]{}^{(X)}\pi_{\mu\nu}.
\end{align}
\end{prop}
This identity is the basic starting point for the vector-field method and for the $r^p$-weighted estimates.
\subsection{Double null foliation}\label{subsec:double-null-foliation}
Define optical functions $(u,\underline u)$ to be solutions to the equations $g^{-1}(du,du)=0$ and $g^{-1}(d\underline u,d\underline u)=0$. Then the level sets
\begin{align*}
	\mathcal C_u = \{u=const\}, \ \underline{\mathcal C}_{v} = \{v=const\},
\end{align*}
are outgoing and ingoing null hypersurfaces, respectively.
 Then we define the normalised null pair $(L,\underline L)$ by
\begin{align*}
	L=-2\Omega^2 \nabla u, \ \underline L = -2\Omega^2 \nabla v,
\end{align*}
where $2\Omega^{-2}=-g^{-1}(du,dv)$, $\Omega>0$. Then we have
\begin{align*}
	g(L,L)=g(\underline L,\underline L)=0, \ g(L,\underline L)=-2.
\end{align*}
We set $S_{u,v}=\mathcal C_u\cap \underline{\mathcal C}_{v}$ to be the $2$-surface and choose $e_A$, $A=1,2$ orthonormal tangent vector fields on $S_{u,v}$. Then we consider the null frame $\{e_A, L,\underline L\}$.

\begin{defn}[Null frame decomposition]
Relative to the null frame $\{e_A,L,\underline L\}$, we first recall the following notation. For any vector field $X$, we denote by $X^\flat$ its metric dual $1$-form,
\begin{align}
    X^\flat(Y)=g(X,Y)
\end{align}
for all vector fields $Y$. In particular, $L^\flat$ and $\underline L^\flat$ denote the metric duals of the null generators $L$ and $\underline L$. The metric and inverse metric decompose as
\begin{align}
    g=-\frac12\left(L^\flat\otimes \underline L^\flat+\underline L^\flat\otimes L^\flat\right)+\sum_{A=1}^2 e^A\otimes e^A,
\end{align}
and
\begin{align}
    g^{-1}=-\frac12\left(L\otimes \underline L+\underline L\otimes L\right)+\sum_{A=1}^2 e_A\otimes e_A.
\end{align}
Consequently, for a scalar function $\phi$,
\begin{align}
    |\nabla\phi|^2=-L\phi\,\underline L\phi+\sum_{A=1}^2 |e_A\phi|^2.
\end{align}
For a $1$-form $\xi$, we write its null components as
\begin{align}
    \xi_L=\xi(L), \qquad \xi_{\underline L}=\xi(\underline L), \qquad \xi_A=\xi(e_A).
\end{align}
Similar notation will be used for higher-rank tensors by evaluating them on the null frame.
\end{defn}

\begin{defn}[Projected connection on $S_{u,v}$]
Let $\Pi$ be the orthogonal projection onto the tangent bundle of $S_{u,v}$. For an $S_{u,v}$-tangent tensor field $\theta$, we define
\begin{align}
    \slashed\nabla_X\theta:=\Pi(\nabla_X\theta), \qquad X\in TS_{u,v}.
\end{align}
In components, for an $S_{u,v}$-tangent $1$-form $\theta$, we have
\begin{align}
    \slashed\nabla_A\theta_B=e_A(\theta_B)-\theta_C\langle \nabla_{e_A}e_B,e_C\rangle.
\end{align}
We also set
\begin{align}
    \slashed{\mathrm{div}}\,\theta=\slashed\nabla^A\theta_A,
    \qquad
    \slashed{\mathrm{curl}}\,\theta=\epsilon^{AB}\slashed\nabla_A\theta_B,
    \qquad
    \slashed\Delta f=\slashed\nabla^A\slashed\nabla_A f.
\end{align}
Here indices $A,B$ are raised and lowered by the induced metric on $S_{u,v}$.
\end{defn}

\begin{defn}[Hodge operators on $S_{u,v}$]
Let $\epsilon_{AB}$ be the volume form of $(S_{u,v},\slashed g)$, normalised by $\epsilon_{12}=1$ with respect to an oriented orthonormal frame. For an $S_{u,v}$-tangent $1$-form $\theta$, its Hodge dual is
\begin{align}
    (*\theta)_A={\epsilon_A}^{B}\theta_B.
\end{align}
For an $S_{u,v}$-tangent $2$-form $F$, we identify $F$ with the scalar $*F$ defined by
\begin{align}
    *F=\frac12\epsilon^{AB}F_{AB}.
\end{align}
If $\theta$ and $\zeta$ are $S_{u,v}$-tangent $1$-forms, we use the notation
\begin{align}
    \theta\wedge\zeta=\epsilon^{AB}\theta_A\zeta_B.
\end{align}
These conventions fix the signs in the angular divergence and curl identities used below.
\end{defn}

\begin{defn}[Flux measures on null hypersurfaces]
The induced measure on $S_{u,v}$ is denoted by $d\sigma_{S_{u,v}}$. With the normalisation $g(L,\underline L)=-2$, the natural flux measures on the null hypersurfaces are written as
\begin{align}
    d\sigma_{\mathcal C_u}=dv\,d\sigma_{S_{u,v}},
    \qquad
    d\sigma_{\underline{\mathcal C}_v}=du\,d\sigma_{S_{u,v}}.
\end{align}
In the asymptotically flat region, one has $d\sigma_{S_{u,v}}\simeq r^2d\omega$, and hence
\begin{align}
    d\sigma_{\mathcal C_u}\simeq r^2dv\,d\omega,
    \qquad
    d\sigma_{\underline{\mathcal C}_v}\simeq r^2du\,d\omega,
\end{align}
where $d\omega$ denotes the standard measure on the unit sphere. These equivalences are used only at the level of uniform constants depending on the asymptotic flatness assumptions.
\end{defn}

In addition, we fix the gauge $\Omega=1$ on $\mathcal C_u$ and $\underline{\mathcal C}_{v}$ so that $\nabla_LL=\nabla_{\underline L}\underline L=0$. This choice is standard in the null frame formalism and simplifies the commutator computations without loss of generality.

Now we introduce the connection coefficients:
\begin{align}\label{conn-coeff}
    \begin{aligned}
        \chi_{AB} = \langle \nabla_{e_A}L,e_B\rangle, \ \underline\chi_{AB} = \langle \nabla_{e_A}\underline L,e_B\rangle, \\
        2\eta_A = \langle \nabla_{\underline L}L, e_A\rangle, \ 2\underline\eta_A = \langle \nabla_L\underline L,e_A\rangle, 2\xi_A =\langle\nabla_LL,e_A\rangle, \ 2\underline\xi_A = \langle\nabla_{\underline L}\underline L,e_A\rangle, \\
        4\omega = \langle\nabla_LL,\underline L\rangle, \ 4\underline\omega = \langle \nabla_{\underline L}\underline L,L\rangle.
        \end{aligned}
\end{align}
We note that $\xi=\underline\xi=\omega=\underline\omega=0$ under the aforementioned gauge choice.
The tensors $\chi$ and $\underline\chi$ are the null second fundamental forms of the spheres $S_{u,v}$ along the outgoing and incoming null directions, respectively. We decompose them into trace and trace-free parts by
\begin{align}
    \chi_{AB}=\hat\chi_{AB}+\frac12\mathrm{Tr}\chi\,\slashed g_{AB},
    \qquad
    \underline\chi_{AB}=\hat{\underline\chi}_{AB}+\frac12\mathrm{Tr}\underline\chi\,\slashed g_{AB},
\end{align}
where
\begin{align}
    \mathrm{Tr}\chi=\slashed g^{AB}\chi_{AB}, \qquad
    \mathrm{Tr}\underline\chi=\slashed g^{AB}\underline\chi_{AB}.
\end{align}
The trace-free parts $\hat\chi$ and $\hat{\underline\chi}$ measure the shear of the null congruences, while the traces measure their null expansions.
\begin{prop}\label{formula-conn-coeff}
    The connection $\nabla$ satisfies the following formulas:
    \begin{align}
        \begin{aligned}
            \nabla_{e_A} L = \chi_{AB}e_B-\eta_AL, \ \nabla_{\underline L}L = 2\eta_Ae_A, \\
            \nabla_{e_A}\underline L = \underline{\chi}_{AB}e_B+\eta_A\underline L, \ \nabla_L\underline L = 2\underline\eta_Ae_A, \\
            \nabla_Le_A = \slashed{\nabla}_Le_A+\underline\eta_AL, \ \nabla_{\underline L}e_A = \slashed{\nabla}_{\underline L}e_A+\eta_A\underline L, \\
            \nabla_{e_B}e_A = \slashed{\nabla}_{e_B}e_A+\frac12\chi_{AB}L+\frac12\underline\chi_{AB}\underline L,
        \end{aligned}
    \end{align}
    where $\slashed\nabla$ is the induced connection on the $2$-sphere, i.e., $\slashed{\nabla}_XY = \sum_{A=1,2} \langle \nabla_XY,e_A\rangle$.
\end{prop}
We refer the reader to \cite{alinhac} for the proof.
\subsection{Lie derivatives}
We first recall that the Lie derivative is defined by $\mathcal L_Xf=Xf$, and $\mathcal L_XY=[X,Y]$, for any scalar $f$ and vector fields $X,Y$ and $[X,Y]=\nabla_XY-\nabla_YX$. For any $n$-tensor $T$, the Lie derivative is extended via the identity:
\begin{align*}
	X \left( T(Y_1,\cdots,Y_n) \right) = (\mathcal L_XT)(Y_1,\cdots,Y_n)+\sum_{i=1}^nT(Y_1,\cdots, [X,Y_i],\cdots,Y_n).
\end{align*}
In particular, for $2$-form $F$, we have
\begin{align}
(\mathcal L_ZF)_{\mu\nu} = Z^\lambda\nabla_\lambda F_{\mu\nu}+\nabla_\mu Z^\lambda F_{\lambda\nu}+\nabla_\nu Z^\lambda F_{\mu\lambda}.
\end{align}
\subsection{Auxiliary estimates}
We make use of the foliation via the double null foliation $\mathcal C_u$ and $\underline{\mathcal C}_{v}$. We define the $L^p$-norm of functions on the foliations by
\begin{align*}
	\|f\|_{L^p(\mathcal C_u)}^p := \int_{\mathcal C_u}|f|^p\,d\sigma_{\mathcal C_u} , \  \|f\|_{L^p(\underline{\mathcal C}_{v})}^p := \int_{\underline{\mathcal C}_{v}}|f|^p\,d\sigma_{\underline{\mathcal C}_{v}}.
\end{align*}
\begin{prop}[Sobolev inequalities on $S_{u,v}$]
Let $f$ be a sufficiently regular scalar function on $S_{u,v}$. Then
\begin{align}
    \|f\|_{L^4(S_{u,v})}
    \lesssim r^{-1/2}\|f\|_{L^2(S_{u,v})}^{1/2}
    \left(\|f\|_{L^2(S_{u,v})}+r\|\slashed\nabla f\|_{L^2(S_{u,v})}\right)^{1/2},
\end{align}
and
\begin{align}
    \|f\|_{L^\infty(S_{u,v})}
    \lesssim r^{-1}\sum_{k\le 2} r^k\|\slashed\nabla^k f\|_{L^2(S_{u,v})}.
\end{align}
Equivalently, in the asymptotically flat region, the last estimate may be written schematically as
\begin{align}
    \|f\|_{L^\infty(S_{u,v})}
    \lesssim r^{-1}\sum_{|I|\le 2}\|\mathcal L_{\Omega}^I f\|_{L^2(S_{u,v})},
\end{align}
where $\Omega$ denotes the rotational vector fields. The same estimates hold componentwise for $S_{u,v}$-tangent tensor fields.
\end{prop}
\begin{prop}[Hardy inequality on the null foliations]\label{Hardy}
	For any sufficiently regular scalar function $\phi$, we have
	\begin{align*}
		\|r^{-1}\phi\|^2_{L^2(\mathcal C_u)} \lesssim \int_{\mathcal C_u}|\mathcal L_{L}\phi|^2\,d\sigma_{\mathcal C_u} +\frac1{r^2}\sum_{i<j}\int_{\mathcal C_u} |\mathcal L_{\Omega_{ij}}\phi|^2\,d\sigma_{\mathcal C_u},
	\end{align*}
	and
	\begin{align*}
	\|r^{-1}\phi\|^2_{L^2(\underline{\mathcal C}_v)} \lesssim \int_{\underline{\mathcal C}_v}|\mathcal L_{\underline L}\phi|^2\,d\sigma_{\underline{\mathcal C}_v} +\frac1{r^2}\sum_{i<j}\int_{\underline{\mathcal C}_v} |\mathcal L_{\Omega_{ij}}\phi|^2\,d\sigma_{\underline{\mathcal C}_v},
	\end{align*}
    where $\mathcal L$ is the Lie derivative.
\end{prop}

Consider a $1$-form $\xi$ satisfying
\begin{align}
    \slashed{\mathrm{div}} \xi = f, \quad \slashed{\mathrm{curl}} \xi = *f,
\end{align}
where $f,f_*$ are given scalar functions on a compact Riemannian manifold $(S,\slashed{g})$ and the operators $\slashed{\mathrm{div}}$ and $\slashed{\mathrm{curl}}$ are defined as
\begin{align}
    \slashed{\mathrm{div}} \xi = \slashed{\nabla}^{e_A} \xi_A, \quad \slashed{\mathrm{curl}} \xi = \epsilon^{AB}\slashed{\nabla}_{e_A}\xi_B.
\end{align}
\begin{prop}[Proposition 2.2.1 of \cite{christoklai}]
    On an arbitrary compact Riemannian manifold $(S,\slashed{g})$, with the Gauss curvature $K$ of $S$, the following estimate holds:
    \begin{align}
        \int_{S} |\slashed\nabla\xi|^2+K|\xi|^2\,d\sigma = \int_{S}|f|^2+|*f|^2\,d\sigma.
    \end{align}
\end{prop}
We shall use a standard Morawetz estimate.
To do this, we let $\Box_g\phi=F$ be an inhomogeneous wave equation. For simplicity we assume $g=m$, i.e., the Minkowski metric for a moment. We define the modified energy current
\begin{align}
J^{X,w}_{\mu}[\phi] = T_{\mu\nu}[\phi]X^\nu+\frac12 w \partial_\mu(\phi^2)-\frac12 \partial_\mu w \phi^2,
\end{align}
where $X=f(r)\partial_r$ is a radial vector field and $w$ is a scalar function of $r$. We choose $w(r)=\frac{f(r)}{r}$.
Then we see that
\begin{align*}
	\nabla^\mu J^{X,w}_{\mu}[\phi] = F ( X\phi+w\phi) +K^X[\phi] +w \partial^\lambda\phi\partial_\lambda\phi -\frac12 \Box w\phi^2,
\end{align*}
where $K^X[\phi]=\frac12 T_{\mu\nu}[\phi] {}^{(X)}\pi^{\mu\nu}$. A straightforward computation gives
\begin{align*}
	K^X[\phi] = \left( \frac12 f'(r)+\frac{f(r)}{r} \right)|\partial_t\phi|^2+ \left( \frac12 f'(r)-\frac{f(r)}{r} \right)|\partial_r\phi|^2 - \frac12 f'(r)|\slashed\nabla\phi|^2.
\end{align*}
Then we have
\begin{align*}
	\nabla^\mu J^{X,w}_{\mu}[\phi]  & = \left( \frac12 f'(r)+\frac{f(r)}{r}-w(r) \right)|\partial_t\phi|^2+\left( \frac12 f'(r)-\frac{f(r)}{r}+w(r) \right)|\partial_r\phi|^2+\left( -\frac12 f'(r)+w(r) \right)|\slashed\nabla\phi|^2 \\
	& \qquad +F ( f(r)\partial_r\phi+w(r)\phi )-\frac12 \Box w \phi^2.
\end{align*}
As one tries to extend the above computation to a spacetime with a generally curved metric $g$, one encounters additional bulk terms of the form: $O(\partial g)|\nabla\phi|^2$. Since $|\partial g| \lesssim \langle r\rangle^{-2}$ on an asymptotically flat spacetime, if we choose $f$ and $w$ so that $f'(r)\approx\langle r\rangle^{-1-\delta}$, such an additional bulk term can be absorbed into the left-hand side and hence we are allowed to exploit the following Morawetz estimates:
\begin{prop}[Morawetz estimates]\label{Morawetz}
    Let an inhomogeneous wave equation $\Box_g\phi=F$ be given. Then
    \begin{align}
        \int_{\mathcal D} \frac1{\langle r\rangle^{1+\delta}}\left( |\nabla_L\phi|^2+|\nabla_{\underline L}\phi|^2+|\slashed\nabla\phi|^2 \right)\,dV \lesssim \int_{\partial D}\mathcal J^T[\phi]\,d\sigma + \left|\int_{\mathcal D} F ( \partial_r\phi+\frac1r\phi) \,dV \right| .
    \end{align}
\end{prop}

\section{Geometric analysis of the Dirac equation}\label{sec:Dirac}
\subsection{Geometric formulation of the Dirac equation}
In this section, we present the geometric formalism of the Dirac operator and invoke several facts from the Clifford structure. We refer the reader to the literature, including Lawson--Michelsohn \cite{lawson}, Wernli \cite{wernli}, Treude \cite{treude}, and LeFloch--Ma--Zhang \cite{leflochmazhang}. The formalism of the Dirac operator is developed through the following steps:
\begin{align*}
	\textrm{Clifford algebra} \to \textrm{Spin group} \to \textrm{Spin structure} \to \textrm{Spinor bundle} \to \textrm{Clifford representation}.
\end{align*}
Since the underlying manifold is diffeomorphic to $\mathbb R^4$, we assume throughout that $\mathcal M$ is equipped with a fixed spin structure compatible with the chosen orientation and time-orientation.

From now on we use the convention of the Lorentzian signature $(-1,+1,\cdots,+1)$.
\begin{defn}[Clifford algebra]
	Let $(V,g)$ be a real vector space equipped with a Lorentzian metric $g$.
The Clifford algebra $\mathrm{Cl}(V,g)$ is the unital associative algebra generated by $V$
subject to the relation: for $v,w\in V$ we have
\begin{align*}
	v\cdot w+w\cdot v = -2g(v,w)\mathbf 1.
\end{align*}
\end{defn}
\begin{defn}[Even Clifford algebra and group of units]
Let $(V,g)$ be a real Lorentzian vector space and let $\mathrm{Cl}(V,g)$ be the associated
real Clifford algebra.
The even Clifford algebra $\mathrm{Cl}^0(V,g)$ is the subalgebra of $\mathrm{Cl}(V,g)$
generated by products of an even number of vectors in $V$.

We denote by $\mathrm{Cl}^0(V,g)^\times$ the group of units of $\mathrm{Cl}^0(V,g)$, namely
\[
\mathrm{Cl}^0(V,g)^\times
:= \{\, a\in \mathrm{Cl}^0(V,g) \mid \exists\, b\in \mathrm{Cl}^0(V,g)
\text{ such that } ab = ba = 1 \,\}.
\]
\end{defn}

\begin{defn}[Spin group]
Let $(V,\boldsymbol\theta)$ be the Lorentzian vector space of dimension $d$ with signature $(-,+,\cdots,+)$. We denote by $SO^+(V,\boldsymbol\theta)$ the identity component of the Lorentz group, namely the subgroup preserving both orientation and time-orientation. The spin group $\mathrm{Spin}^+(V,\boldsymbol\theta)$ is the connected double cover of $SO^+(V,\boldsymbol\theta)$.
\end{defn}
\begin{prop}
There exists a surjective group homomorphism
\[
\rho:\mathrm{Spin}^+(V,\boldsymbol\theta)\to SO^+(V,\boldsymbol\theta),
\]
whose kernel is $\{\pm 1\}$.
\end{prop}
\begin{defn}[Spin structure]
	Let $(\mathcal M,g)$ be an oriented, time-oriented Lorentzian manifold of dimension $d$.
A spin structure on $\mathcal M$ is a principal $\mathrm{Spin}^+(V,\boldsymbol\theta)$-bundle
\begin{align*}
	\mathbf P\to \mathcal M,
\end{align*}
together with a $\rho$-equivariant bundle morphism
\begin{align*}
	\mathbf P\to F_{SO^+}(\mathcal M),
\end{align*}
where $F_{SO^+}(\mathcal M)$ denotes the oriented and time-oriented orthonormal frame bundle and
$\rho:\mathrm{Spin}^+(V,\boldsymbol\theta)\to SO^+(V,\boldsymbol\theta)$ is the double covering map.
\end{defn}
\begin{defn}[Spinor module]
Let $(V,\boldsymbol\theta)$ be the Lorentzian vector space.
We denote by $\mathrm{Cl}_\mathbb C(V,\boldsymbol\theta):=\mathrm{Cl}(V,\boldsymbol\theta)\otimes_\mathbb R\mathbb C$
the complexified Clifford algebra.
A complex vector space $\Delta$ is called a (complex) spinor module
if it carries an irreducible representation of $\mathrm{Cl}_\mathbb C(V,\boldsymbol\theta)$.
\end{defn}
We refer the reader to Appendix \ref{appendix-dirac} for the definition of an irreducible representation.
\begin{defn}[Spin representation]
Let $\Delta$ be a complex spinor module for $\mathrm{Cl}_\mathbb C(V,\boldsymbol\theta)$.
The spin representation
\[
\kappa:\mathrm{Spin}^+(V,\boldsymbol\theta)\to GL(\Delta)
\]
is defined by restricting the Clifford action of $\mathrm{Cl}_\mathbb C(V,\boldsymbol\theta)$
to the subgroup $\mathrm{Spin}^+(V,\boldsymbol\theta)\subset \mathrm{Cl}^0(V,\boldsymbol\theta)^\times$.
\end{defn}

\begin{defn}[Spinor bundle]
	Let $\kappa:\mathrm{Spin}^+(V,\boldsymbol\theta)\to GL(\Delta)$ be the complex spin representation.
The associated spinor bundle is defined by
\begin{align*}
	\mathbf S = \mathbf P\times_\kappa\Delta.
\end{align*}
A smooth section $\psi\in\Gamma(\mathbf S)$ is called a spinor field.
\end{defn}
\begin{defn}[Clifford representation]
	Let $\mathbf S$ be the spinor bundle over $(\mathcal M,g)$.
The Clifford representation is a bundle map
\begin{align*}
	\gamma: T\mathcal M\to \mathrm{End}(\mathbf S)
\end{align*}
which is fiberwise linear and satisfies the following relation:
\begin{align*}
	\gamma(v)\gamma(w)+\gamma(w)\gamma(v)=-2g(v,w)Id_{\mathbf S},
\end{align*}
for all tangent vectors $v,w$.\\
By universality, $\gamma$ extends uniquely to a representation of the Clifford bundle
\begin{align*}
	\gamma: \mathrm{Cl}(T\mathcal M,g)\to \mathrm{End}(\mathbf S).
\end{align*}
\end{defn}

\begin{defn}[Dirac pairing and positive spinor product]
Let $\Delta$ be the complex spinor module. We fix a non-degenerate Hermitian
form $h_D$ on $\Delta$, called the Dirac pairing, which is invariant under the
spin representation and satisfies
\[
h_D(\gamma(v)\psi,\phi)=h_D(\psi,\gamma(v)\phi),
\quad v\in V.
\]
It induces a fiberwise non-degenerate Hermitian pairing
\[
\langle\cdot,\cdot\rangle_D:\mathbf S_x\times\mathbf S_x\to\mathbb C.
\]
This pairing is not positive definite in Lorentzian signature.

If $T$ is a future-directed unit timelike vector field, we define the associated
positive spinor product by
\[
\langle\psi,\phi\rangle_T
:=\langle\psi,\gamma(T)\phi\rangle_D.
\]
We choose the sign convention for $\langle\cdot,\cdot\rangle_D$ so that
$\langle\psi,\psi\rangle_T\ge0$.
\end{defn}

\begin{rem}
In the standard Minkowski frame with $T=e_0$, the Dirac pairing is represented by
\[
\langle\psi,\phi\rangle_D=\psi^\dagger\gamma^0\phi,
\]
while the associated positive product is
\[
\langle\psi,\phi\rangle_{e_0}
=
\langle\psi,\gamma(e_0)\phi\rangle_D
=
\psi^\dagger\phi.
\]
\end{rem}

\begin{defn}[Spin connection]
	We let $\nabla$ denote the Levi-Civita connection on the Lorentzian manifold $(\mathcal M,g)$. This connection is naturally lifted to the connection on the spinor bundle via the spin structure on $\mathcal M$. Thus we have the connection
	\begin{align*}
		\nabla^{\mathbf S}:\Gamma(\mathbf S)\to \Gamma(T^*\mathcal M\otimes\mathbf S).
	\end{align*}
	We call $\nabla^{\mathbf S}$ the spin connection.
In a local oriented orthonormal frame $\{e_a\}_{a=0}^{d-1}$, it is given by
\begin{align}\label{local-spin-connection}
    \nabla^{\mathbf S}_{e_\mu}\psi
    =e_\mu(\psi)+\frac14\omega_{\mu ab}\gamma(e_a)\gamma(e_b)\psi,
    \qquad
    \omega_{\mu ab}=g(\nabla_{e_\mu}e_a,e_b).
\end{align}
\end{defn}
We refer the reader to Appendix \ref{appendix-dirac} for the explicit definition of the natural lift of the Levi-Civita connection.
\begin{defn}[Adjoint with respect to a timelike observer]
Let $T$ be a future-directed unit timelike vector field and let $\langle\cdot,\cdot\rangle_T$ be the associated positive spinor product. For an endomorphism $A\in\mathrm{End}(\mathbf S)$, we denote by $A^{\dagger}$ the adjoint with respect to $\langle\cdot,\cdot\rangle_T$, i.e.
\[
\langle A\phi,\psi\rangle_T=\langle\phi,A^{\dagger}\psi\rangle_T
\]
for all spinor fields $\phi,\psi$.
\end{defn}

Now we list some important properties concerning the Clifford representation and the spin connection. We refer to \cite{leflochmazhang,treude} for the proof.
\begin{prop}
	We have the following properties:
	\begin{enumerate}
		\item Linearity: \begin{align*}
			\gamma(v+w) = \gamma(v)+\gamma(w), \quad v,w\in T\mathcal M, \\
			\gamma(fv) = f\gamma(v), \quad f\in C^\infty(\mathcal M).
		\end{align*}
		\item Compatibility with the Dirac pairing: Clifford multiplication is formally symmetric with respect to $\langle\cdot,\cdot\rangle_D$, namely
		\begin{align*}
			\langle\gamma(v)\phi,\psi\rangle_D=\langle\phi,\gamma(v)\psi\rangle_D, \qquad v\in T\mathcal M.
		\end{align*}
		\item Covariant compatibility: Let $\nabla^{\mathbf S}$ be the spin connection induced from the Levi-Civita connection.
Then for any vector fields $X,v$,
\begin{align*}
	\nabla^{\mathbf S}_X(\gamma(v)) = \gamma(\nabla_Xv).
\end{align*}
 \item Compatibility with the spin connection: for any vector field $X$ and spinor fields $\phi,\psi\in\Gamma(\mathbf S)$, we have \begin{align*}
 	X\langle\phi,\psi\rangle_D = \langle\nabla_X^{\mathbf S}\phi,\psi\rangle_D
+\langle\phi,\nabla_X^{\mathbf S}\psi\rangle_D.
 \end{align*}
	\end{enumerate}
\end{prop}
\begin{rem}[Clifford bundle versus gamma matrices]
We point out that the Clifford representation
\[
\gamma : T\mathcal M \to \mathrm{End}(\mathbf S)
\]
is a bundle-level, frame-independent object.
In particular, $\gamma(v)$ is intrinsically associated to the tangent vector $v\in T\mathcal M$
and $\gamma(v)$ does not depend on any choice of local coordinates or frames.

The commonly used gamma matrices arise only after fixing a local orthonormal frame
$\{e_\mu\}_{\mu=0}^{d-1}$, in which case one writes
\[
\gamma_\mu := \gamma(e_\mu).
\]
Indeed, we may write $\gamma(v)=\gamma(v^\mu e_\mu)=v^\mu \gamma(e_\mu)=v^\mu \gamma_\mu$.
These matrices satisfy the Clifford relations
\[
\gamma_\mu \gamma_\nu + \gamma_\nu \gamma_\mu
= -2 g(e_\mu,e_\nu)\mathrm{Id}_{\mathbf S}.
\]
However, it should be regarded merely as a local representation of the Clifford representation.

Under a change of orthonormal frame, the matrices $\gamma_\mu$ transform accordingly,
while the underlying Clifford action $\gamma$ remains invariant.
All geometric identities in this framework are formulated at the level of the Clifford bundle
and are therefore independent of any particular choice of gamma matrices.
\end{rem}
\begin{prop}[Curvature of the spin connection]
For any spinor field $\psi$, the curvature of the spin connection is given by
\begin{align}\label{spin-curvature-commutator}
    [\nabla^{\mathbf S}_\mu,\nabla^{\mathbf S}_\nu]\psi
    =\frac14 R_{\mu\nu ab}\gamma(e_a)\gamma(e_b)\psi.
\end{align}
\end{prop}

Now we are able to define the Dirac operator.
\begin{defn}[Dirac operator]
Let $\{e_a\}_{a=0}^{d-1}$ be a local oriented orthonormal frame on $(\mathcal M,g)$ and set $\varepsilon_a=g(e_a,e_a)$, so that $\varepsilon_0=-1$ and $\varepsilon_a=1$ for $a\ge1$. We define the raised Clifford action by
\begin{align*}
    \gamma^a:=\varepsilon_a\gamma(e_a).
\end{align*}
The Dirac operator is
\[
\mathscr D\psi := \sum_{a=0}^{d-1}\gamma^a\nabla^{\mathbf S}_{e_a}\psi.
\]
Equivalently, in local coordinates, $\mathscr D\psi=\gamma^\mu\nabla^{\mathbf S}_\mu\psi$.
\end{defn}
The above definition is independent of the choice of local oriented orthonormal frame.

\subsection{Spinor-adapted Lie derivative}
\begin{defn}[Kosmann Lie derivative]
Let $X$ be a smooth vector field on $(\mathcal M,g)$.
The Kosmann Lie derivative of a spinor field $\psi\in\Gamma(\mathbf S)$
along $X$ is defined by
\begin{align*}
	\mathscr L_X^{}\psi
	:= \nabla_X^{\mathbf S}\psi
	- \frac14 \left(\nabla_\rho X_\sigma \right)
	\gamma^\rho\gamma^\sigma \psi,
\end{align*}
where $\nabla$ denotes the Levi-Civita connection
and $\gamma^\mu$ denotes the raised Clifford action with respect to a local orthonormal frame, as in the definition of the Dirac operator.
\end{defn}
It is well-known that the Lie derivative is not linear in $X$. For the Kosmann Lie derivative, we have
\begin{align}\label{kosmann-fX}
\mathscr L_{fX}\psi = f\mathscr L_X\psi -\frac14\left( (\partial_\rho f)X_\sigma -(\partial_\sigma f)X_\rho \right)\gamma^\rho \gamma^\sigma\psi.
\end{align}

\begin{prop}[Commutator identity with Clifford action]
For any vector field $X$ and any tangent vector $v$,
the Kosmann Lie derivative $\mathscr L$ satisfies
\begin{align*}
	\mathscr L_X (\gamma(v))
	= \gamma(\mathcal L_X v),
\end{align*}
where $\mathcal L_X$ is the ordinary Lie derivative acting on vector fields.
\end{prop}
From now on, we allow the abuse of the notation and simply denote the spinorial connection by $\nabla_\mu:=\nabla_{e_\mu}$.
\begin{prop}
   Let $Z$ be a vector field. We have
    \begin{align}\label{comm-dirac}
    \begin{aligned}
          [ \gamma^\mu\nabla_\mu, \mathscr L_Z  ] \psi &=  \frac12 \gamma^\rho {\pi_{\rho}}^{\! \mu} \nabla_\mu \psi + \frac12 \gamma^\mu Z^\nu \gamma^\rho\gamma^\sigma R_{\mu\nu\rho\sigma}\psi -\frac18 \gamma^\mu\gamma^\rho\gamma^\sigma \left( \nabla_\mu \pi_{\rho\sigma}+\nabla_{\rho}\pi_{\mu\sigma}-\nabla_\sigma\pi_{\mu\rho} \right) \psi,
    \end{aligned}
\end{align}
where $\pi:={}^{(Z)}\pi$ is the deformation tensor of the vector field $Z$.
\end{prop}
In particular, the commutator vanishes when the vector field $Z$ is Killing.
\begin{proof}
    We recall the definition of the spinor-adapted Lie derivative:
\begin{align*}
    \mathscr L_Z\psi := Z^\nu \nabla_\nu \psi -\frac14 (\nabla_\rho Z_\sigma)\gamma^\rho\gamma^\sigma\psi.
\end{align*}
Then the commutator of the Dirac operator and the Lie derivative gives
\begin{align*}
    [ \gamma^\mu\nabla_\mu,\mathscr L_Z ] \psi & = [ \gamma^\mu \nabla_\mu, Z^\nu\nabla_\nu  ] \psi -\frac14 [ \gamma^\mu\nabla_\mu, (\nabla_\rho Z_\sigma) \gamma^\rho \gamma^\sigma  ] \psi.
\end{align*}
We are concerned with the first commutator term of the right-hand side. We write
\begin{align*}
    [ \gamma^\mu \nabla_\mu, Z^\nu\nabla_\nu  ] \psi & = \gamma^\mu \nabla_\mu(Z^\nu \nabla_\nu\psi) - Z^\nu \nabla_\nu (\gamma^\mu\nabla_\mu\psi) \\
    & = \gamma^\mu(\nabla_\mu Z^\nu) \nabla_\nu \psi +\gamma^\mu Z^\nu \nabla_\mu(\nabla_\nu\psi) -Z^\nu (\nabla_\nu \gamma^\mu)\nabla_\mu\psi -Z^\nu \gamma^\mu \nabla_\nu(\nabla_\mu\psi) \\
    & = \gamma^\mu (\nabla_\mu Z^\nu) \psi +\gamma^\mu Z^\nu [\nabla_\mu,\nabla_\nu ]\psi \\
    & = \gamma^\mu (\nabla_\mu Z^\nu) \psi+\frac14 R_{\mu\nu\rho\sigma}\gamma^\rho \gamma^\sigma \psi,
\end{align*}
where we used the fact that $\nabla_\nu \gamma^\mu \equiv0$ and the commutator identity for the spin connection:
\begin{align}
    [\nabla_\mu,\nabla_\nu]\psi = \frac14 R_{\mu\nu\rho\sigma}\gamma^\rho\gamma^\sigma \psi.
\end{align}
On the other hand, we have
\begin{align*}
    [ \gamma^\mu\nabla_\mu, (\nabla_\rho Z_\sigma)\gamma^\rho\gamma^\sigma ] \psi & = \gamma^\mu (\nabla_\mu \nabla_\rho Z_\sigma) \gamma^\rho\gamma^\sigma \psi + \gamma^\mu (\nabla_\rho Z_\sigma) (\nabla_\mu \gamma^\rho\gamma^\sigma) \psi \\
    & \qquad + (\nabla_\rho Z_\sigma) \gamma^\mu\gamma^\rho \gamma^\sigma \nabla_\mu\psi - (\nabla_\rho Z_\sigma) \gamma^\rho \gamma^\sigma (\gamma^\mu\nabla_\mu\psi).
\end{align*}
Then we conclude that
\begin{align}
    \begin{aligned}
        [ \gamma^\mu\nabla_\mu, \mathscr L_Z  ] \psi &= \gamma^\mu (\nabla_\mu Z^\nu) \psi+\frac14 R_{\mu\nu\rho\sigma}\gamma^\rho \gamma^\sigma \psi \\
        & \qquad -\frac14 \gamma^\mu (\nabla_\mu \nabla_\rho Z_\sigma) \gamma^\rho\gamma^\sigma \psi-\frac14 (\nabla_\rho Z_\sigma) [ \gamma^\mu, \gamma^\rho\gamma^\sigma ] \nabla_\mu \psi.
    \end{aligned}
\end{align}
By using the Clifford algebra, we see that
\begin{align*}
    [\gamma^\mu, \gamma^\rho\gamma^\sigma] & = (-2g^{\mu\rho}-\gamma^\rho\gamma^\mu) \gamma^\sigma - \gamma^\rho (-2 g^{\mu\sigma}-\gamma^\mu\gamma^\sigma) \\
    & = 2(g^{\mu\sigma}\gamma^\rho -g^{\mu\rho}\gamma^\sigma).
\end{align*}
Thus combining the first and the fourth terms of the above identity gives
\begin{align*}
    \gamma^\mu (\nabla_\mu Z^\nu) \psi-\frac14 (\nabla_\rho Z_\sigma) [ \gamma^\mu, \gamma^\rho\gamma^\sigma ] \nabla_\mu \psi & = \frac12 \gamma^\rho {\pi_{\rho}}^{\! \mu} \nabla_\mu \psi,
\end{align*}
where $\pi:={}^{(Z)}\pi$ is the deformation tensor of the vector field $Z$.
Now we invoke the identity for the derivative $\nabla\nabla Z$:
\begin{align}
    \nabla_\mu\nabla_\rho Z_\sigma & = R_{\sigma\rho\mu\nu}Z^\nu +\frac12 \left( \nabla_\mu \pi_{\rho\sigma}+\nabla_{\rho}\pi_{\mu\sigma}-\nabla_\sigma\pi_{\mu\rho} \right).
\end{align}
Finally we have
\begin{align}
    \begin{aligned}
          [ \gamma^\mu\nabla_\mu, \mathscr L_Z  ] \psi &=  \frac12 \gamma^\rho {\pi_{\rho}}^{\! \mu} \nabla_\mu \psi + \frac12 \gamma^\mu Z^\nu \gamma^\rho\gamma^\sigma R_{\mu\nu\rho\sigma}\psi -\frac18 \gamma^\mu\gamma^\rho\gamma^\sigma \left( \nabla_\mu \pi_{\rho\sigma}+\nabla_{\rho}\pi_{\mu\sigma}-\nabla_\sigma\pi_{\mu\rho} \right) \psi.
    \end{aligned}
\end{align}
\end{proof}
\begin{prop}\label{prop-dirac-comm-LLb}
    If $Z=L$,
    \begin{align}
        \begin{aligned}
            [\gamma^\mu\nabla_\mu,\mathscr L_L] & = - (2\eta_A\gamma^{e_A}\nabla_{\underline L}+\eta_A \gamma^{\underline L}\nabla_{e_A})\psi.
        \end{aligned}
    \end{align}
    If $Z=\underline L$,
    \begin{align}
        \begin{aligned}
            [\gamma^\mu\nabla_\mu, \mathscr L_{\underline L}] & =  -(2\eta_A\gamma^{e_A}\nabla_L+\eta_A\gamma^{\underline L}\nabla_{e_A})\psi,
        \end{aligned}
    \end{align}
    up to lower order terms.
\end{prop}

We also need the commutator of the Kosmann Lie derivatives:
\begin{lem}\label{kosmann-commutator}
    We have the following identity
    \begin{align}
        [\mathscr L_X,\mathscr L_Y]\psi & = \frac14 \gamma^\rho\gamma^\sigma X^\mu Y^\nu R_{\rho\sigma\mu\nu}\psi +\mathcal E\psi,
    \end{align}
    where $\mathcal E$ is not harmful error terms:
    \begin{align*}
        \mathcal E & = \frac14 Y^\lambda\nabla_\lambda\left( (\nabla_\rho X_\sigma-\nabla_\sigma X_\rho)\gamma^\rho\gamma^\sigma \right)-\frac14 X^\lambda\nabla_\lambda \left( (\nabla_\mu Y_\nu-\nabla_\nu Y_\mu)\gamma^\mu\gamma^\nu \right) \\
        & \qquad +\frac1{16}(\nabla_\rho X_\sigma-\nabla_\sigma X_\rho)(\nabla_\mu Y_\nu-\nabla_\nu Y_\mu)[\gamma^\rho\gamma^\sigma,\gamma^\mu\gamma^\nu].
    \end{align*}
\end{lem}

\subsection{Null decomposition for the Dirac equation}
We deal with the double null foliation introduced in Section~\ref{sec:prelim-geometry}.
In particular, we use optical functions $(u,\underline u)$ with corresponding outgoing/ingoing null hypersurfaces
$\mathcal C_u$ and $\underline{\mathcal C}_{v}$,
and we denote by $(L,\underline L)$ the associated normalised null pair.
We also write $S_{u,v}:=\mathcal C_u\cap\underline{\mathcal C}_{v}$ for the $2$-spheres and fix an orthonormal frame $\{e_A\}_{A=1,2}$ tangent to $S_{u,v}$, so that $\{e_A,L,\underline L\}$ is a null frame.

Now we define the projection operators:
\begin{align*}
	\Pi_+ = \frac14 \gamma(\underline L)\gamma(L), \ \Pi_- =\frac14 \gamma(L)\gamma(\underline L).
\end{align*}
\begin{prop}
The operators $\Pi_\pm$ satisfy the following properties:
\begin{align*}
    \Pi_++\Pi_- = I, \ \Pi_\pm \Pi_\mp = 0.
\end{align*}
\end{prop}
\begin{proof}
The proof follows by the Clifford algebra. Indeed, we see that
\begin{align*}
	\Pi_++\Pi_- &= \frac14 \left(\gamma(\underline L)\gamma(L)+\gamma(L)\gamma(\underline L) \right) = -\frac12 g_{L\underline L}I= -\frac12 (-2) I = I,
    \end{align*}
    and
    \begin{align*}
	\Pi_+^2 &= \frac1{16}\gamma(\underline L)\gamma(L)\gamma(\underline L)\gamma(L) = \frac1{16}\gamma(\underline L)\left( \gamma(L)\gamma(\underline L)+\gamma(\underline L)\gamma(L)-\gamma(\underline L)\gamma(L) \right)\gamma(L) \\
	& = \frac1{16}\gamma(\underline L)\left( 4I\right)\gamma(L)-\frac1{16}\gamma(\underline L)^2\gamma(L)^2 \\
	& = \frac14 \gamma(\underline L)\gamma(L)+0 = \Pi_+.
\end{align*}
On the other hand, we have
\begin{align*}
    \Pi_+ \Pi_- &= \frac1{16}\gamma(\underline L)\gamma(L)\gamma(L)\gamma(\underline L)  = 0.
\end{align*}
The computations of $\Pi_-^2$ and $\Pi_-\Pi_+$ are similar, and we omit the details.
\end{proof}
In what follows, we decompose
$$
\psi = \psi_++\psi_-,
$$
where
$$\psi_+= \Pi_+\psi , \quad \psi_-= \Pi_-\psi.
$$
Now we put
\begin{align*}
	T = \frac12 (L+\underline L), \ N=\frac12(L-\underline L).
\end{align*}
Then we have
\begin{align*}
	g(T,T)=-1, \ g(N,N) =1,
\end{align*}
which implies that
\begin{align*}
	\gamma(T)^2 = I, \ \gamma(N)^2 =-I.
\end{align*}
Now we write the Dirac equation in terms of the null frame:
\begin{align*}
	i\gamma(L)\nabla_L\psi +i\gamma(\underline L)\nabla_{\underline L}\psi +i\gamma(e_A)\nabla_{e_A}\psi-m\psi = 0.
\end{align*}
By multiplying $\gamma(\underline L)$ we have
\begin{align*}
	i\gamma(\underline L)\gamma(L)\nabla_L\psi =-i\gamma(\underline L)\gamma(e_A)\nabla_{e_A}\psi+m\gamma(\underline L)\psi.
\end{align*}
Since $g(e_A,\underline L)=0$, we have $\gamma(\underline L)\gamma(e_A)=-\gamma(e_A)\gamma(\underline L)$ and
\begin{align*}
	\gamma(\underline L) = \gamma(T)^2\gamma(\underline L) = \frac12\gamma(T)(\gamma(L)+\gamma(\underline L))\gamma(\underline L) = \frac12\gamma(T)\gamma(L)\gamma(\underline L).
\end{align*}
Therefore, we see that
\begin{align*}
	i\Pi_+\nabla_L\psi = \frac i2\gamma(e_A)\gamma(T)\Pi_-\nabla_{e_A}\psi+\frac12 m \gamma(T)\Pi_-\psi,
\end{align*}
and similarly, we have
\begin{align*}
	i\Pi_-\nabla_{\underline L}\psi = \frac i2\gamma(e_A)\gamma(T)\Pi_+\nabla_{e_A}\psi+\frac12 m\gamma(T)\Pi_+\psi.
\end{align*}
Then we have
\begin{align}
    i\nabla_{L}\psi_+ = -\frac i2 \gamma(T)\gamma(e_A)\nabla_{e_A}\psi_-+\frac12 m \gamma(T)\psi_- + \mathcal R_+(\psi), \\
    i\nabla_{\underline L}\psi_- = -\frac i2 \gamma(T)\gamma(e_A)\nabla_{e_A}\psi_++\frac12 m\gamma(T)\psi_+ +\mathcal R_-(\psi),
\end{align}
where
\begin{align}
    \mathcal R_+(\psi) := i [\nabla_L,\Pi_+]\psi +\frac i2 \gamma(T)\gamma(e_A) [\nabla_{e_A},\Pi_-]\psi, \\
    \mathcal R_-(\psi) := i [\nabla_{\underline L},\Pi_-]\psi +\frac i2 \gamma(T)\gamma(e_A) [\nabla_{e_A},\Pi_+]\psi.
\end{align}
In particular, for the Dirac equation $i\gamma^\mu\nabla_\mu\psi -m\psi = F $, with a source $F$, we have
\begin{align*}
    i\nabla_L\psi_+ = -\frac i2\gamma(T)\gamma(e_A)\nabla_{e_A}\psi_-+\frac12 m\gamma(T)\psi_-+\frac12 \gamma(T)\Pi_- F+\mathcal R_+(\psi), \\
    i\nabla_{\underline L}\psi_- = -\frac i2\gamma(T)\gamma(e_A)\nabla_{e_A}\psi_++\frac12 m\gamma(T)\psi_++\frac12 \gamma(T)\Pi_+ F+\mathcal R_-(\psi).
\end{align*}
We give the explicit computations for the commutator terms $\mathcal R_+(\psi)$ and $\mathcal R_-(\psi)$:
\begin{lem}
    We have
    \begin{align*}
        -i\mathcal R_+(\psi) &= \frac1{2r}\psi +O(r^{-2})\psi, \\
        -i\mathcal R_-(\psi) &= -\frac1{2r}\psi +O(r^{-2})\psi.
    \end{align*}
\end{lem}
\begin{proof}
We have
\begin{align*}
    [ \nabla_{e_A} ,\Pi_- ] & = \frac14 [ \nabla_{e_A}, \gamma(L)\gamma(\underline L) ] \\
    & = \frac14 \nabla_{e_A}( \gamma(L)\gamma(\underline L) ) \\
    & = \frac14 \left( \gamma(\nabla_{e_A}L)\gamma(\underline L) + \gamma(L)\gamma(\nabla_{e_A}\underline L) \right) \\
    & = \frac14 \left( \chi_{AB}\gamma(e_B)\gamma(\underline L) -\eta_A\gamma(L)\gamma(\underline L) +\underline\chi_{AB}\gamma(L) \gamma(e_B)+\eta_A\gamma(L)\gamma(\underline L) \right) \\
    & = \frac14 \left( \chi_{AB}\gamma(e_B)\gamma(\underline L)+\underline\chi_{AB}\gamma(L)\gamma(e_B) \right) \\
    & = \frac14 \chi_{AB}\gamma(e_B) \left( \gamma(\underline L)+\gamma(L) \right) + O(r^{-2}) \\
    & = \frac12\chi_{AB}\gamma(e_B)\gamma(T) +O(r^{-2}).
\end{align*}
We note that
\begin{align*}
    \gamma(T)\gamma(e_A)\gamma(e_B)\gamma(T) = \gamma(e_A)\gamma(e_B)\gamma(T)^2 = \gamma(e_A)\gamma(e_B).
\end{align*}
Then
\begin{align*}
    \frac12 \gamma(T)\gamma(e_A)[\nabla_{e_A},\Pi_-] & = \frac14 \gamma(e_A)\gamma(e_B)\chi_{AB}+O(r^{-2}) \\
    & = \frac14 \mathrm{Tr}\chi +\gamma(e_1)\gamma(e_2) ( \chi_{12}-\chi_{21})+O(r^{-2}).
\end{align*}
On the other hand, we have
\begin{align*}
    [\nabla_L,\Pi_+] & = \frac14 \nabla_L \left( \gamma(\underline L)\gamma(L) \right) \\
    & = \frac14 \gamma(\nabla_L \underline L) \gamma(L) \\
    & = \frac12 \underline\eta_A \gamma(e_A)\gamma(L).
    \end{align*}
    The computation for the commutators $[\nabla_{e_A},\Pi_+]$ and $[\nabla_{\underline L},\Pi_-]$ is also similar. We have
    \begin{align*}
      \frac12 \gamma(T)\gamma(e_A)  [\nabla_{e_A},\Pi_+] & = -\frac14 \mathrm{Tr}\chi -\frac12\gamma(e_1)\gamma(e_2) ( \chi_{12}-\chi_{21}) +O(r^{-2}), \\
        [\nabla_{\underline L}, \Pi_- ] & = \frac12 \eta_A \gamma(e_A)\gamma(\underline L).
    \end{align*}
\end{proof}
\begin{prop}\label{prop-null-decomp-dirac}
    The Dirac equation $(i\gamma^\mu\nabla_\mu-m)\psi=F$ has the null decomposition of the form:
    \begin{align}
    \nabla_L\psi_+ = -\frac12\gamma(T)\gamma(e_A)\nabla_{e_A}\psi_--\frac i2 m\gamma(T)\psi_--\frac i2 \gamma(T)\Pi_- F+\frac1{2r}\psi +O(r^{-2})\psi, \\
    \nabla_{\underline L}\psi_- = -\frac12\gamma(T)\gamma(e_A)\nabla_{e_A}\psi_+-\frac i2 m\gamma(T)\psi_+-\frac i2 \gamma(T)\Pi_+ F -\frac1{2r}\psi +O(r^{-2})\psi.
\end{align}
\end{prop}
\subsection{Dirac current and energy identity}
\begin{defn}[Dirac current]
Let $\psi\in\Gamma(\mathbf S)$ be a spinor field.
The Dirac current associated to $\psi$ is the vector field $J^\mu(\psi)$ defined by
\begin{align*}
	J^\mu(\psi) := \langle \psi, \gamma^\mu \psi\rangle_D.
\end{align*}
where $\langle\cdot,\cdot\rangle_D$ denotes the Dirac pairing on $\mathbf S$.
\end{defn}
In a local orthonormal frame $\{e_\mu\}$ with $e_0$ future-directed timelike, the components of the Dirac current read
\[
J^\mu = \psi^\dagger\gamma(e_0)\gamma^\mu\psi.
\]
\begin{prop}[Conservation of Dirac current]
Let $\psi$ be a smooth solution to the free Dirac equation
\[
\mathscr D\psi := (i\gamma^\mu\nabla_\mu-m)\psi = 0.
\]
Then the associated Dirac current is divergence-free:
\begin{align*}
	\nabla_\mu J^\mu(\psi) = 0.
\end{align*}
\end{prop}
\begin{prop}[Variation of the Dirac current]
For a spinor field $\psi$ and a vector field $v$, define
\[
J(\psi)(v):=\langle \psi,\gamma(v)\psi\rangle.
\]
Then for any vector field $X$,
\[
\mathcal L_X \big(J(\psi)(v)\big)
= \langle \mathscr L_X\psi,\gamma(v)\psi\rangle
+ \langle \psi,\gamma(v)\mathscr L_X\psi\rangle
+ \langle \psi,\gamma(\mathcal L_X v)\psi\rangle,
\]
where $\mathscr L_X$ is the Kosmann Lie derivative acting on spinors, and $\mathcal L_X$ is the usual Lie derivative acting on vectors and tensors.
\end{prop}
A fundamental property of the Dirac current is its causality. More precisely, for any spinor field $\psi$, $J^\mu(\psi)$ is future-directed causal vector, i.e., we have
\begin{align*}
	g_{\mu\nu}J^\mu(\psi) J^\nu(\psi) \le 0,
\qquad J(\psi)(n) \ge 0 \,\text{ for any future-directed timelike } n.
\end{align*}
In consequence, for any future-directed timelike normal vector $n$, we have
\begin{align*}
	J(\psi)(n) := J^\mu(\psi)n_\mu \ge0.
\end{align*}
Due to the above positivity, it is reasonable to define the norm of the spinor $\psi$ associated to the vector $n$ by
\begin{align*}
	|\psi|_n = \sqrt{J(\psi)(n)}.
\end{align*}
\begin{prop}[Lemma 5.2 of \cite{leflochmazhang}]
	Let $n$ be a future-directed timelike unit vector field. For any spinor fields $\psi,\phi$, we have
	\begin{align*}
		|\langle\psi,\gamma(n)\phi\rangle | \lesssim |\psi|_n|\phi|_n, 
	\end{align*}
	and for any space-like vector $X$, we have
\begin{align*}
	|J(\psi)(X)| \le |X| J(\psi)(n).
\end{align*}
\end{prop}
We refer the reader to Lemmas 5.2 and 5.3 of \cite{leflochmazhang}.
\begin{proof}
The proof is somewhat straightforward. Indeed, we already know that $J(\psi)(n)$ is non-negative for any spinor field $\psi$ and any future-directed timelike vector $n$. We let $\psi,\phi\in\Gamma(\mathbf S)$ be spinor fields and $n$ be a future-directed timelike vector field. Then we consider the functions $f$ and $g$ given by
\begin{align*}
    f(t) = J(\phi-t\psi)(n), \quad g(t)= J(\phi-it\psi)(n).
\end{align*}
The two functions $f$ and $g$ are nonnegative, which is equivalent to the inequalities
\begin{align*}
    \Re\langle\phi,\gamma(n)\psi\rangle^2 \le |\phi|_n |\psi|_n, \quad \Im\langle\phi,\gamma(n)\psi\rangle^2 \le |\phi|_n|\psi|_n.
\end{align*}
To prove the second inequality, we first decompose the vector $X=X^\perp +\langle X,n\rangle n$, where $X^\perp$ satisfies $\langle X^\perp,n\rangle=0$. Then
\begin{align*}
    |J(\psi)(X)| &= |\langle J,X\rangle| \\
    & \le |\langle J,X^\perp\rangle| + |\langle J,n\rangle| |\langle X,n\rangle|.
\end{align*}
To control the first term of the last inequality, we decompose the current $J=\langle J,n\rangle n+\langle J,\widehat{X^\perp}\rangle \widehat{X^\perp}$, where $\widehat X^\perp=\frac{X^\perp}{|X^\perp|}$. Since the vector $J$ is causal and $X^\perp$ is spacelike, we have the inequality $|\langle J,n\rangle|\ge |\langle J,\widehat{X^\perp}\rangle|$. Thus we have
\begin{align*}
    |\langle J,X^\perp\rangle| \le |\langle J,n\rangle||X^\perp|.
\end{align*}
Combining the inequalities, we conclude that
\begin{align*}
    |J(\psi)(X)| \le |\langle J,n\rangle||X| = |X||J(\psi)(n)|.
\end{align*}
\end{proof}

\section{Linear analysis for the spinor fields}\label{sec:linear-dirac}

\subsection{Wave equation for the spinor fields}
Given $\tau\ge0$, we consider the hypersurface $\Sigma_\tau$ given by the union of the spacelike time slice $\{ t= \tau \}\cap \{r\le R\} $ and the null hypersurface $ \{u=\tau-R\} \cap \{v \ge \tau+R \} $. We square the Dirac operator to get the wave-type equation.
\begin{prop}
    We have
    \begin{align}
        \gamma^\mu\nabla_\mu(\gamma^\nu\nabla_\nu\psi) = -g^{\mu\nu}\nabla_\mu\nabla_\nu\psi -\frac14\mathcal R\psi,
    \end{align}
    where $\mathcal R$ is the scalar curvature.
\end{prop}
\begin{proof}
We see that
\begin{align*}
   i\gamma^\mu\nabla_\mu(i\gamma^\nu\nabla_\nu\psi) & = -\gamma^\mu\gamma^\nu \nabla_\mu \nabla_\nu\psi \\
   & = -\frac12 ( \gamma^\mu\gamma^\nu \nabla_\mu\nabla_\nu + \gamma^\nu\gamma^\mu \nabla_\nu \nabla_\mu)\psi \\
   & = -\frac12 ( \gamma^\mu\gamma^\nu \nabla_\mu\nabla_\nu + \gamma^\nu\gamma^\mu \nabla_\mu\nabla_\nu +\gamma^\nu\gamma^\mu \nabla_\nu \nabla_\mu-\gamma^\nu\gamma^\mu \nabla_\mu\nabla_\nu)\psi \\
   & = -\frac12 ( \gamma^\mu\gamma^\nu+\gamma^\nu\gamma^\mu)\nabla_\mu\nabla_\nu\psi + \frac12 \gamma^\nu\gamma^\mu [ \nabla_\mu,\nabla_\nu ]\psi \\
   & = g^{\mu\nu}\nabla_\mu\nabla_\nu \psi+ \frac18 \gamma^\nu \gamma^\mu \gamma^\rho\gamma^\sigma R_{\mu\nu\rho\sigma}\psi \\
   & = g^{\mu\nu}\nabla_\mu\nabla_\nu \psi+ \frac14 \mathcal R\psi,
\end{align*}
where we used the identity
\begin{align}\label{eq-clifford-comm}
    [ \nabla_\mu, \nabla_\nu ]\psi = \frac14 \gamma^\rho \gamma^\sigma R_{\mu\nu\rho\sigma}\psi, \quad \gamma^\mu\gamma^\nu \gamma^\rho\gamma^\sigma R_{\mu\nu\rho\sigma} = -2\mathcal R,
\end{align}
where $\mathcal R$ is the scalar curvature. See \cite{alcu,parker} for the proof of \eqref{eq-clifford-comm}.
\end{proof}
Therefore, by applying the Dirac operator on both sides of the Dirac equation:
\begin{align}
    i\gamma^\mu\nabla_\mu(i\gamma^\nu\nabla_\nu\psi) - m i\gamma^\mu\nabla_\mu\psi = 0,
\end{align}
we obtain the following wave-type equation:
\begin{align}
    (g^{\mu\nu}\nabla_\mu \nabla_\nu -m^2) \psi = -\frac14\mathcal R\psi.
\end{align}
Note that the above equation consists of the principal part of the wave operator.
We now consider the energy-momentum tensor of spinor fields satisfying the above wave-type equation:
\begin{align}
    T_{\mu\nu}[\psi] := \Re \langle \nabla_\mu\psi, \gamma(T)\nabla_\nu\psi\rangle -\frac12  g_{\mu\nu} ( \langle \nabla^\lambda\psi, \gamma(T)\nabla_\lambda\psi\rangle +m^2 \langle \psi,\gamma(T)\psi\rangle).
\end{align}
We recall that the Dirac current $J^\mu [\psi] = \langle \psi,\gamma^\mu\psi\rangle $ is the causal vector. Thus $J^T[\psi]$ is positive-definite, where $T$ is the future-directed unit normal vector field to the time slice $\{t=constant\}$.

We define the energy current for the spinor field in terms of wave:
\begin{align}
    \mathcal J^X_\mu [\psi] := T_{\mu\nu}[\psi] X^\nu
\end{align}
and we put
\begin{align}
    K^X[\psi] := T_{\mu\nu}[\psi] {}^{(X)}\pi^{\mu\nu},
\end{align}
where $\pi:={}^{(X)}\pi$ is the deformation tensor of the vector field $X$. We also have
\begin{align}
    \nabla^\mu \mathcal J^X_{\mu}[\psi ] = K^X[\psi]
\end{align}
up to error terms involving $\nabla^\mu\gamma^T$.

\begin{prop}[Morawetz estimates for the spinor fields]\label{morawetz-dirac}
    Given the linear massless Dirac equation $i\gamma^\mu\nabla_\mu\psi = 0$, we have
    \begin{align}
        \begin{aligned}
            \int_{\mathcal D}\frac1{\langle r\rangle^{1+\delta}} \left( \langle \nabla_L\psi,\gamma^T\nabla_L\psi\rangle+\langle \nabla_{\underline L}\psi,\gamma^T\nabla_{\underline L}\psi\rangle + \langle \slashed\nabla\psi,\gamma^T\slashed\nabla\psi\rangle  \right)\,dV & \lesssim \int_{\partial D}\mathcal J^T[ \psi ]\,d\sigma,
        \end{aligned}
    \end{align}
    for some small $\delta>0$.
    In particular, we have
    \begin{align}
         \int_{\mathcal D}\frac1{\langle r\rangle^{ } }  \langle \slashed\nabla\psi,\gamma^T\slashed\nabla\psi\rangle \,dV  & \lesssim \int_{\partial D}\mathcal J^T[ \psi ]\,d\sigma,
    \end{align}
    and
    \begin{align}
        \int_{\mathcal D} \frac1{\langle r\rangle^{3+\delta}} \langle \psi,\gamma^T\psi\rangle\,dV \lesssim \int_{\partial D}\mathcal J^T[ \psi ]\,d\sigma.
    \end{align}
\end{prop}
\begin{proof}
    We recall that squaring the Dirac equation yields a wave-type equation up to the scalar curvature $\mathcal R\psi$. Since $|\mathcal R|\lesssim \langle r\rangle^{-3}$, this lower-order term can be absorbed into the left-hand side via the following inequality:
    \begin{align*}
        \left| \int_{\mathcal D} \mathcal R \langle \psi, \nabla\psi\rangle \,dV \right| & \lesssim \left( \int_{\mathcal D}\frac1{\langle r\rangle^{1+\delta}} |\nabla \psi|^2\,dV \right)^\frac12 \left( \int_{\mathcal D} \frac1{\langle r\rangle^{5-\delta}} |\psi|^2\,dV \right)^\frac12.
    \end{align*}
   It is remarkable that another error terms appear not only due to the equation itself, but also from the energy-momentum tensor.

In fact, the energy-momentum tensor for the spinor fields $\psi$ in terms of the wave equation is given by the inner product $\langle \cdot, \gamma^T\cdot\rangle$, and hence its divergence yileds error due to the derivative $\nabla T$. Nevertheless, this error term is acceptable, since we have $\nabla T \approx \partial g \le c \langle r\rangle^{-2}$, with a sufficiently small constant $c>0$ and hence
    \begin{align*}
       \int_{\mathcal D} |\nabla\psi|^2 |\nabla T|\,dV  \le c \int_{\mathcal D} \frac1{\langle r\rangle^2}|\nabla\psi|^2\,dV,
    \end{align*}
    which can be absorbed into the left-hand side of the Morawetz inequality.
\end{proof}

\subsection{$r^p$-method for the linear Dirac equation}\label{subsec:rp-lin-dirac}
We apply the $r^p$ method introduced by Dafermos-Rodnianski \cite{DR}. We consider the linear Dirac equation
\begin{align}
    (i\gamma^\mu\nabla_\mu-m)\psi =0,
\end{align}
where $m\ge0$ is the mass, and we assume that the initial data $\psi_0$ are supported on the initial hypersurface $\Sigma_{t=0}$.

By an application of the standard argument by Morawetz we obtain
\begin{align}
    \int^\infty_\tau \int_{r\le R} \mathcal J^T_{\mu}[\psi]n^\mu_{\Sigma_\tau}  \le C_R \int_{\Sigma_\tau} \mathcal J^T_\mu [\psi] n^\mu_{\Sigma_\tau}.
\end{align}
Note that the above integrated local energy decay-type estimate still remains valid in the presence of the mass term, in that the integration appears only on a compact region $\{ r \le R\}$.

From now on we rewrite the squared Dirac equation in terms of the null frame. Indeed, we have
\begin{align*}
g^{\mu\nu}\nabla_\mu\nabla_\nu & = g^{L\underline L}\nabla_L\nabla_{\underline L} + g^{\underline LL}\nabla_{\underline L}\nabla_L + \slashed{g}^{AB}\nabla_{e_A}\nabla_{e_B} \\
    & = -\frac12 \nabla_{L}\nabla_{\underline L} -\frac12 \nabla_{\underline L}\nabla_L +\slashed{\Delta} \\
    & = -\nabla_{\underline L}\nabla_{ L} -\frac12 [\nabla_{ L},\nabla_{\underline L} ] +\slashed{\Delta} \\
    & = -\nabla_{\underline L}\nabla_{ L} +\slashed{\Delta}-\frac18 \gamma^\rho\gamma^\sigma R_{L\underline L\rho\sigma}.
\end{align*}
Hence the wave equation can be rewritten as follows:
\begin{align}\label{dirac-wave-LLb}
-\nabla_{\underline L}\nabla_{ L}\psi +\slashed{\Delta}\psi = -\frac14\mathcal R \psi + \frac18 \gamma^\rho\gamma^\sigma  R_{L\underline L\rho\sigma}\psi.
\end{align}
Now we shall use the integration by parts after multiplication on both sides by $r^p( \nabla_L\psi)^\dagger$ as the argument by \cite{DR} to obtain the $r^p$-type weighted energy inequality. 

In order to take the positive-definite inner product, we have to exploit the future-directed time-like unit vector $T$, and hence we have
\begin{align}
    -r^p\langle  \nabla_L\psi, \gamma(T)\nabla_{\underline L}\nabla_L\psi\rangle + r^p\langle \nabla_L\psi, \gamma(T)\slashed\Delta \psi\rangle = r^p\mathcal R\langle \nabla_L\psi, \gamma(T)\psi\rangle,
\end{align}
where the right-hand side consists of lower-order terms and $\mathcal R$ is now an abbreviation of the curvature tensors.

On the other hand, one can also obtain the similar identity to \eqref{dirac-wave-LLb} by applying the adjoint $\dagger$:
\begin{align}\label{dirac-wave-LLb-a}
    -\nabla_{\underline L}\nabla_{ L}\psi^\dagger +\slashed{\Delta}\psi^\dagger = -\frac14\mathcal R \psi + \frac18 \gamma^\rho\gamma^\sigma  R_{L\underline L\rho\sigma}\psi^\dagger.
\end{align}
Then by multiplying on both sides by $r^p \nabla_L\psi$, we obtain
\begin{align}
    -r^p \langle \nabla_{\underline L}\nabla_L\psi, \gamma(T)\nabla_L\psi\rangle + r^p \langle \slashed\Delta\psi, \gamma(T) \nabla_L\psi\rangle = r^p\mathcal R \langle \psi, \gamma(T) \nabla_L\psi\rangle,
\end{align}
where we also note that $\gamma(T)^\dagger= \gamma(T)$.
Now the addition of the two identities gives
\begin{align}
    -r^p \Re \langle \nabla_L\psi, \gamma(T)\nabla_{\underline L}\nabla_L\psi\rangle + r^p \Re \langle \nabla_L\psi, \gamma(T)\slashed\Delta\psi\rangle =r^p \mathcal R \Re \langle \nabla_L\psi,\gamma(T)\psi\rangle,
\end{align}
where $\Re z$ is the real part of $z$.
Then we see that
\begin{align}
\begin{aligned}
    -r^p \Re \langle \nabla_L\psi, \gamma(T)\nabla_{\underline L}\nabla_L\psi\rangle & = -\frac12 \nabla_{\underline L} \left( r^p \Re \langle \nabla_L\psi, \gamma(T)\nabla_L\psi\rangle \right) - \frac{p}{2}r^{p-1} \Re \langle \nabla_L\psi, \gamma(T)\nabla_L\psi\rangle \\
    & \qquad +\frac12 r^p \Re \langle \nabla_L\psi, \gamma(\nabla_{\underline L}T)\nabla_L\psi\rangle .
    \end{aligned}
\end{align}
We recall that the inner product $\langle \nabla_L\psi, \gamma(T)\nabla_L\psi\rangle= J^T[\nabla_L\psi] $, and it is positive-definite since the Dirac current is the causal vector and $T$ is time-like. We may put $J^T[\psi]:= |\psi|_T^2$. Let the domain $\mathcal D^{\tau_2}_{\tau_1}$ be the domain bounded by the null hypersurface $\{u=\tau_2-R\}$ and $ \{u=\tau_1-R\} $ and the surface $\{r=R\}$. Then the use of the integration by parts on the domain $\mathcal D^{\tau_2}_{\tau_1}$ yields the following weighted energy inequality:
  \begin{align}
    \begin{aligned}
      &  \int_{ \substack{ u = \tau_2-R \\ v\ge \tau_2+R } } r^p |\nabla_L\psi|_T^2\,d\sigma + \int_{\mathcal D^{\tau_2}_{\tau_1}} r^{p-1} ( p |\nabla_L\psi|_T^2-pm^2|\psi|_T^2  + (2-p) |\slashed\nabla\psi|_T^2)\,dV +\int_{\mathcal I^{\tau_2-R}_{\tau_1-R}} r^p |\slashed{\nabla}\psi|_T^2\,d\sigma \\
      & \le  \int_{ \substack{ u = \tau_1-R \\ v\ge \tau_1+R } } r^p |\nabla_L\psi|_T^2\,d\sigma + \left|\int_{\tau_1}^{\tau_2} r^p ( |\slashed\nabla\psi|_T^2 - |\nabla_L\psi|_T^2 +m^2 |\psi|_T^2 + \mathcal R
      |\psi|_T^2) \bigg|_{r=R}d\sigma d\tau\right| \\
      & \qquad + \int_{\mathcal I^{\tau_2-R}_{\tau_1-R}} r^p \mathcal R |\psi|_T^2 \,d\sigma + \int_{\mathcal D^{\tau_2}_{\tau_1}} | p r^{p-1} \mathcal R |\psi|_T^2 + \nabla \mathcal R |\psi|_T^2 \,dV  + \int_{\mathcal D^{\tau_2}_{\tau_1} } \langle r\rangle^{p-2} ( |\nabla_L\psi|_T^2+|\slashed\nabla\psi|_T^2)\,dV ,
    \end{aligned}
\end{align}
where the second spacetime integral of the right-hand side of the above inequality appears due to the error term $\gamma(\nabla_{\underline L}T)$ and $\gamma(\slashed\nabla T)$. 


From now on we exclusively consider the massless case, i.e., $m=0$. The surface integral along the null infinity $\mathcal I^{\tau_2-R}_{\tau_1-R}$ on the right-hand side tends to zero due to the good weight $r^p \mathcal R \lesssim \langle r\rangle^{-1}$. If we are concerned with a curved background close to the Minkowski spacetime, the first spacetime integral involving the curvature terms can be absorbed into the second spacetime integral after using the Hardy inequality along the null hypersurface $\Sigma_\tau$. Furthermore, the control of these bulk terms is obvious. Indeed, due to the range of $p$, it suffices to give good control of the integral
\begin{align}
    \int_{\mathcal D^{\tau_2}_{\tau_1} } |\nabla_L\psi|_T^2+|\slashed\nabla\psi|_T^2\,dV.
\end{align}
However, the above inequality with $p=1$ implies that
\begin{align}
    \int_{\mathcal D^{\tau_2}_{\tau_1} } |\nabla_L\psi|_T^2+|\slashed\nabla\psi|_T^2\,dV \lesssim \int_{\substack{u=\tau_1-R \\ v\ge \tau_1+R}} r |\nabla_L\psi|_T^2\,d\sigma+  \langle R\rangle^{-1}\int_{\mathcal D^{\tau_2}_{\tau_1} } |\nabla_L\psi|_T^2+|\slashed\nabla\psi|_T^2\,dV ,
\end{align}
which can be absorbed into the left-hand side.

Thus we obtain
\begin{align}
    \begin{aligned}
         &  \int_{ \substack{ u = \tau_2-R \\ v\ge \tau_2+R } } r^p |\nabla_L\psi|_T^2\,d\sigma + \int_{\mathcal D^{\tau_2}_{\tau_1}} r^{p-1} ( p |\nabla_L\psi|_T^2 + (2-p) |\slashed\nabla\psi|_T^2)\,dV  \\
      & \lesssim  \int_{ \substack{ u = \tau_1-R \\ v\ge \tau_1+R } } r^p |\nabla_L\psi|_T^2\,d\sigma + \int_{\tau_1}^{\tau_2} r^p ( |\slashed\nabla\psi|_T^2 - |\nabla_L\psi|_T^2  + \mathcal R
      |\psi|_T^2) \bigg|_{r=R}d\tau. 
    \end{aligned}
\end{align}
The integral along the surface $\{r=R\}$ is controlled via the spacetime integral on the interior domain $\{r\le R\}$ and hence we further have
\begin{align}\label{p-we}
    \begin{aligned}
         &  \int_{ \substack{ u = \tau_2-R \\ v\ge \tau_2+R } } r^p |\nabla_L\psi|_T^2\,d\sigma + \int_{\mathcal D^{\tau_2}_{\tau_1}} r^{p-1} ( p |\nabla_L\psi|_T^2 + (2-p) |\slashed\nabla\psi|_T^2)\,dV  \\
      & \le  \int_{ \substack{ u = \tau_1-R \\ v\ge \tau_1+R } } r^p |\nabla_L\psi|_T^2\,d\sigma + C \int_{\Sigma_{\tau_1}} \mathcal J^T [\psi] d\sigma_{\Sigma_{\tau}}.
    \end{aligned}
\end{align}
Now the remaining task is a repetition of the argument by \cite{DR}.
Indeed, we apply \eqref{p-we} with $p=2$ and get
\begin{align}
    \begin{aligned}
         &  \int_{ \substack{ u = \tau_2-R \\ v\ge \tau_2+R } } r^2 |\nabla_L\psi|^2\,d\sigma + \int_{\mathcal D^{\tau_2}_{\tau_1}} r^{} |\nabla_L\psi|^2 \,dV  \\
      & \le  \int_{ \substack{ u = \tau_1-R \\ v\ge \tau_1+R } } r^2 |\nabla_L\psi|^2\,d\sigma + C \int_{\Sigma_{\tau_1}} \mathcal J^T d\sigma.
    \end{aligned}
\end{align}
This also obviously implies that
\begin{align}
    \begin{aligned}
         &  \int_{\tau_1}^{\tau_2}\int_{\Sigma_\tau} r^{} |\nabla_L\psi|^2 \,d\sigma_{\Sigma_\tau}d\tau   \le  \int_{ \substack{ u = \tau_1-R \\ v\ge \tau_1+R } } r^2 |\nabla_L\psi|^2\,d\sigma + C \int_{\Sigma_{\tau_1}} \mathcal J^T d\sigma.
    \end{aligned}
\end{align}
Now we consider a famialy of intervals $[\tau_{i-1},\tau_i]$, $i=1,2,\cdots,N$ of dyadic scales. Then for each interval $[\tau_{i-1},\tau_i]$, one can always choose a $\tau^*\in[\tau_{i-1},\tau_i]$ such that
\begin{align}
    \begin{aligned}
         &  \int_{\Sigma_\tau} r^{} |\nabla_L\psi|^2 \,d\sigma_{\Sigma_\tau}d\tau   \le  (\tau^*)^{-1}\left(\int_{ \substack{ u = \tau_{i-1}-R \\ v\ge \tau_{i-1}+R } } r^2 |\nabla_L\psi|^2\,d\sigma + C \int_{\Sigma_{\tau_1}} \mathcal J^T d\sigma\right).
    \end{aligned}
\end{align}
In particular, we have
\begin{align}
    \begin{aligned}
         &  \int_{\Sigma_\tau} r^{} |\nabla_L\psi|^2 \,d\sigma_{\Sigma_\tau}d\tau   \le  (\tau^*)^{-1}\left(\int_{ \substack{ u = \tau_{N-1}-R \\ v\ge \tau_{N-1}+R } } r^2 |\nabla_L\psi|^2\,d\sigma + C \int_{\Sigma_{\tau_1}} \mathcal J^T d\sigma\right).
    \end{aligned}
\end{align}
for some $\tau^*\in [ \tau_{N-1},\tau_N ]$ and by using the above energy bound we obtain
\begin{align}
    \begin{aligned}
         &  \int_{ \substack{ u=\tau_N-R \\ v\ge \tau_N+R } } r^{} |\nabla_L\psi|^2 \,d\sigma_{\Sigma_\tau}d\tau   \le  (\tau_N)^{-1}\left(\int_{ \substack{ u = \tau_{1}-R \\ v\ge \tau_{1}+R } } r^2 |\nabla_L\psi|^2\,d\sigma + C \int_{\Sigma_{\tau_1}} \mathcal J^T d\sigma\right),
    \end{aligned}
\end{align}
where we used the fact $\tau^*\approx \tau_N$ for $\tau^*\in [\tau_{N-1},\tau_N]$. Now we apply the inequality \eqref{p-we} with $p=1$ on the domain $\mathcal D^{\tau_N}_{\tau_{N-1}}$ to get
\begin{align}
    \begin{aligned}
      &  \int_{ \substack{ u=\tau_{N}-R \\ v\ge \tau_N+R } } r |\nabla_L\psi|^2\,d\sigma + \int_{\mathcal D^{\tau_N}_{\tau_{N-1}}  } |\nabla_L\psi|^2+|\slashed\nabla\psi|^2\,dV \\
      & \le \int_{ \substack{ u=\tau_{N-1}-R \\ v\ge \tau_{N-1}+R } } r |\nabla_L\psi|^2\,d\sigma + C\int_{\Sigma_{N-1}} \mathcal J^T[\psi] d\sigma \\
      & \le (\tau_N)^{-1}\left(\int_{ \substack{ u = \tau_{1}-R \\ v\ge \tau_{1}+R } } r^2 |\nabla_L\psi|^2\,d\sigma + C \int_{\Sigma_{\tau_1}} \mathcal J^T d\sigma\right) \\
      & \qquad + C\int_{\Sigma_{N-1}} \mathcal J^T[\psi] d\sigma.
    \end{aligned}
\end{align}
Now we add a multiple of the Morawetz-type inequality to get
\begin{align}
    \begin{aligned}
         &  \int_{\tau_{N-1}}^{\tau_N}\int_{\Sigma_\tau\cap \{r\le R\} } \mathcal J^T_\mu [\psi] n^\mu_{\Sigma_\tau} + \int_{ \substack{ u=\tau_{N}-R \\ v\ge \tau_N+R } } r |\nabla_L\psi|^2\,d\sigma + \int_{\mathcal D^{\tau_N}_{\tau_{N-1}}  } |\nabla_L\psi|^2+|\slashed\nabla\psi|^2\,dV \\
         & \le (\tau_N)^{-1}\left(\int_{ \substack{ u = \tau_{1}-R \\ v\ge \tau_{1}+R } } r^2 |\nabla_L\psi|^2\,d\sigma + C \int_{\Sigma_{\tau_1}} \mathcal J^T d\sigma \right) + C\int_{\Sigma_{\tau_{N-1}}} \mathcal J^T[\psi]  d\sigma.
    \end{aligned}
\end{align}
Finally we obtain
\begin{align}
    \int_{\tau_{N-1}}^{\tau_N}\int_{\Sigma_\tau } \mathcal J^T [\psi] d\sigma \le (\tau_N)^{-1}\left(\int_{ \substack{ u = \tau_{1}-R \\ v\ge \tau_{1}+R } } r^2 |\nabla_L\psi|^2\,d\sigma + C \int_{\Sigma_{\tau_1}} \mathcal J^T d\sigma\right) + C\int_{\Sigma_{\tau_{N-1}}} \mathcal J^T [\psi] d\sigma.
\end{align}
Now we want to show that
\begin{align}
    \int_{\Sigma_\tau} \mathcal J^T d\sigma \le C  \int_{\Sigma_{\tau'}} \mathcal J^Td\sigma
\end{align}
for all $\tau\ge \tau'$. To show this we need to control the curvature terms as a source term. This is obvious when a spacetime is near the Minkowski. Indeed, by the Hardy inequality, we see that
\begin{align*}
    \left| \int_{\mathcal D^{\tau}_{\tau'}} \mathcal R \langle \psi, \nabla\psi\rangle\,dV \right| & \lesssim \int_{\tau'}^\tau \int_{\Sigma_\tau} \frac{\epsilon}{\langle r\rangle^3} | \langle \psi,\nabla \psi\rangle|\,dV \\
    & \lesssim \int_{\mathcal D^{\tau}_{\tau'}} \epsilon \langle r\rangle^{-2} (|\nabla_L\psi|^2+ |\slashed\nabla\psi|^2) \,dV,
\end{align*}
and hence the integrated local energy decay estimate gives the control.

From now on we apply several vector fields $Z\in \{ L,\underline L, \Omega_{ij} \} $, $i,j=1,2,3$. We start with the (massless) linear Dirac equation:
\begin{align}
    i\gamma^\mu\nabla_\mu\psi = 0.
\end{align}
By applying the spinor-adapted Lie derivative $\mathscr L_Z$ we have
\begin{align}
    i\gamma^\mu\nabla_\mu\mathscr L_Z\psi = [ i\gamma^\mu\nabla_\mu, \mathscr L_Z ] \psi.
\end{align}
We apply the Dirac operator once again and obtain the wave-type equation:
\begin{align}
    g^{\mu\nu}\nabla_\mu\nabla_\nu\mathscr L_Z\psi = -\frac14\mathcal R\mathscr L_Z\psi -\gamma^\lambda\nabla_\lambda( [ \gamma^\mu\nabla_\mu , \mathscr L_Z ]  \psi ).
\end{align}
By repeating the previous argument we obtain the following inequality:
\begin{align}\label{p-we-dirac-Z}
    \begin{aligned}
       & \int_{ \substack{ u = \tau_2-R \\ v \ge \tau_2+R } } r^p |\nabla_L\mathscr L_Z\psi|^2\,d\sigma + \int_{\mathcal D^{\tau_2}_{\tau_1} } r^{p-1} ( p |\nabla_L\mathscr L_Z\psi|^2 + (2-p) |\slashed{\nabla}\mathscr L_Z\psi|^2\,dV \\
        & \le \int_{ \substack{ u = \tau_1-R \\ v\ge \tau_1+R } } r^p |\nabla_L\mathscr L_Z\psi|^2\,d\sigma + \int_{\tau_1}^{\tau_2} r^p ( |\slashed\nabla\mathscr L_Z\psi|^2 - |\nabla_L\mathscr L_Z\psi|^2+\mathcal R|\mathscr L_Z\psi|^2) \bigg|_{r=R} \,d\tau \\
        & \qquad + \left| \int_{ \mathcal D^{\tau_2}_{\tau_1} } r^p \langle \gamma^\lambda\nabla_\lambda[ \gamma^\mu\nabla_\mu, \mathscr L_Z ]\psi, \nabla_L \mathscr L_Z\psi\rangle \,dV \right|.
    \end{aligned}
\end{align}
In fact, in order to control the second integral of the right-hand side, we apply the Morawetz estimate Proposition \ref{morawetz-dirac} on the interior region $\{r\le R\}$:
\begin{align}
    \begin{aligned}
& \int_{\tau_1}^{\tau_2}\int_{\Sigma_\tau\cap \{r\le R\} } r^{p-1} ( |\nabla_L\mathscr L_Z\psi|^2+|\nabla_{\underline L}\mathscr L_Z\psi|^2+|\slashed\nabla\mathscr L_Z\psi|^2) \,dV \\
& \lesssim_R \int_{\Sigma_{\tau_1} }  \mathcal J^T[\mathscr L_Z\psi] d\sigma + \int_{\tau_1}^{\tau_2} \int_{\Sigma_\tau \cap \{r\le R\} } r^p | \langle \gamma^\lambda\nabla_\lambda[ \gamma^\mu\nabla_\mu, \mathscr L_Z ]\psi, \nabla \mathscr L_Z\psi\rangle | \, dV.
    \end{aligned}
\end{align}
Then we have for any vector field $Z\in \{ L,\underline L, \Omega_{ij} \} $,
\begin{align*}
    \int_{\tau_1}^{\tau_2} \int_{\Sigma_\tau \cap\{r\le R\} } r^p | \langle \gamma^\lambda\nabla_\lambda[ \gamma^\mu\nabla_\mu, \mathscr L_Z ]\psi, \nabla \mathscr L_Z\psi\rangle | \, dV & \lesssim \epsilon \int_{\tau_1}^{\tau_2} \int_{\Sigma_\tau \cap\{r\le R\}} r^{p-2} |\nabla \mathscr L_Z\psi|^2\,dV \\
    & \lesssim \epsilon  \int_{\tau_1}^{\tau_2} \int_{\Sigma_\tau \cap\{r\le R\}}  r^{p-1} |\nabla\mathscr L_Z\psi|^2\,dV,
\end{align*}
and hence it can be absorbed into the left-hand side provided that $\epsilon>0$ is sufficiently small. Then we see that
\begin{align*}
    & \int_{\tau_1}^{\tau_2} r^p ( |\slashed\nabla\mathscr L_Z\psi|^2 - |\nabla_L\mathscr L_Z\psi|^2+\mathcal R|\mathscr L_Z\psi|^2) \bigg|_{r=R} \,d\tau \\
    &\le C_R \int_{\tau_1}^{\tau_2}\int_{ R-1\le r\le R } r^{p-1} |\nabla \mathscr L_Z\psi|^2\,dV \\
    & \lesssim C_R'\int_{\Sigma_{\tau_1}} \mathcal J^T[\mathscr L_Z\psi] d\sigma.
\end{align*}
Now we need to control the third integral on the right-hand side of the inequality \eqref{p-we-dirac-Z}, which appears due to the commutator term. Indeed, we have the following:
\begin{prop}\label{dirac-comm-rp}
Let $0\le p\le 2$. We consider the metric $g=m+h$, where $m$ is the Minkowski metric and $h$ satisfies the decomposition $h=h^{\rm (rad)}(r)+h^{\rm (ang)}(r,\omega)$ with
\begin{align}
h^{\rm (rad)}(r) = O(r^{-1}), \quad h^{\rm (ang)}(r,\omega) = O(r^{-1-\eta}),
\end{align}
for some $\eta>0$
and the coefficients of $h$ are small enough so that the metric $g$ is close enough to the Minkowski metric $m$. Given the linear massless Dirac equation $i\gamma^\mu\nabla_\mu\psi=0$,
we have for $Z=L,\underline L$,
\begin{align}
\begin{aligned}
    & \left| \int_{\mathcal D^{\tau_2}_{\tau_1} } r^p \langle \gamma^\lambda\nabla_\lambda ( [\gamma^\mu\nabla_\mu, \mathscr L_{Z} ]\psi), \gamma^T \nabla_L\mathscr L_{Z}\psi \rangle \,dV   \right|   \lesssim \sum_{Z\in \{L,\underline L,\Omega\} } \begin{cases} {}^{(p-1)} \mathcal E^D[ \mathscr L_Z\psi]^2(\tau_1), \quad p =1+\delta >1, \\
    {}^{(p)}\mathcal E^D[\mathscr L_Z\psi]^2(\tau_1) , \quad p\le1,
        \end{cases}
    \end{aligned}
\end{align}
and
\begin{align}
\begin{aligned}
    & \left| \int_{\mathcal D^{\tau_2}_{\tau_1} } r^p \langle \gamma^\lambda\nabla_\lambda ( [\gamma^\mu\nabla_\mu, \mathscr L_{\Omega_{ij}} ]\psi), \gamma^T \nabla_L\mathscr L_{\Omega_{}}\psi \rangle \,dV   \right| \lesssim  \sum_{Z\in \{L,\underline L,\Omega\} } {}^{(p-\eta)}\mathcal E^D_{}[\mathscr L_Z\psi]^2(\tau_1).
        \end{aligned}
\end{align}
In particular, if $\eta=1$, we have
\begin{align}
\begin{aligned}
    & \left| \int_{\mathcal D^{\tau_2}_{\tau_1} } r^p \langle \gamma^\lambda\nabla_\lambda ( [\gamma^\mu\nabla_\mu, \mathscr L_{\Omega_{}} ]\psi), \gamma^T \nabla_L\mathscr L_{\Omega_{ij}}\psi \rangle \,dV   \right| \lesssim  \sum_{Z\in \{L,\underline L,\Omega\} } {}^{(p-1)}\mathcal E^D_{}[\mathscr L_Z\psi]^2(\tau_1),
        \end{aligned}
\end{align}
where $\Omega$ denotes rotation vector fields satisfying $|r\slashed\nabla\psi|\lesssim |\Omega\psi|$.
\end{prop}
\begin{proof}

\noindent\textit{Case 1. $Z=L$:}
We first consider the integral
\begin{align}
    \int_{\mathcal D^{\tau_2}_{\tau_1} } r^p | \gamma^\lambda\nabla_\lambda ( \eta_A \gamma^{e_A}\nabla_{\underline L}\psi , \gamma^T \nabla_L \mathscr L_L\psi\rangle| \,dV.
\end{align}
Since $\eta = O(r^{-2})$, we see that
\begin{align}\label{pf-comm-L1}
\begin{aligned}
     \int_{\mathcal D^{\tau_2}_{\tau_1} } r^p | \langle\gamma^\lambda\nabla_\lambda ( \eta_A \gamma^{e_A}\nabla_{\underline L}\psi) , \gamma^T \nabla_L \mathscr L_L\psi\rangle| \,dV & \lesssim  \int_{\mathcal D^{\tau_2}_{\tau_1} } \langle r\rangle^{p-2} | \langle \gamma^{L}\nabla_{L} ( \gamma^{e_A}\nabla_{\underline L}\psi) , \gamma^T \nabla_L \mathscr L_L\psi\rangle| \,dV \\
     & \qquad + \int_{\mathcal D^{\tau_2}_{\tau_1} } \langle r\rangle^{p-2} | \langle \gamma^{e_A}\nabla_{e_A} ( \gamma^{e_A}\nabla_{\underline L}\psi) , \gamma^T \nabla_L \mathscr L_L\psi\rangle| \,dV \\
     & \qquad + \int_{\mathcal D^{\tau_2}_{\tau_1} } \langle r\rangle^{p-2} | \langle \gamma^{\underline L}\nabla_{\underline L} ( \gamma^{e_A}\nabla_{\underline L}\psi) , \gamma^T \nabla_L \mathscr L_L\psi\rangle| \,dV.
     \end{aligned}
\end{align}
Here the third integral seems problematic. However, in view of the Dirac equation $\gamma^L\nabla_L\psi + \gamma^{\underline L}\nabla_{\underline L}\psi + \gamma^{e_A}\nabla_{e_A}\psi = 0$, one can rewrite
\begin{align*}
     \int_{\mathcal D^{\tau_2}_{\tau_1} } \langle r\rangle^{p-2} | \langle \gamma^{\underline L}\nabla_{\underline L} ( \gamma^{e_A}\nabla_{\underline L}\psi) , \gamma^T \nabla_L \mathscr L_L\psi\rangle| \,dV & \lesssim  \int_{\mathcal D^{\tau_2}_{\tau_1} } \langle r\rangle^{p-2} | \langle\gamma^{\underline L}\nabla_{\underline L} ( \mathscr L_{\underline L}\psi) , \gamma^T \nabla_L \mathscr L_L\psi\rangle| \,dV \\
     & \qquad+  \int_{\mathcal D^{\tau_2}_{\tau_1} } \langle r\rangle^{p-3} | \langle \gamma^{\underline L}\nabla_{\underline L} \psi , \gamma^T \nabla_L \mathscr L_L\psi\rangle| \,dV,
\end{align*}
and
\begin{align*}
    \int_{\mathcal D^{\tau_2}_{\tau_1} } \langle r\rangle^{p-2} | \langle\gamma^{\underline L}\nabla_{\underline L} ( \mathscr L_{\underline L}\psi) , \gamma^T \nabla_L \mathscr L_L\psi\rangle| \,dV & \lesssim \int_{\mathcal D^{\tau_2}_{\tau_1} } \langle r\rangle^{p-2} |\langle \gamma^{ L}\nabla_{ L} ( \mathscr L_{\underline L}\psi) , \gamma^T \nabla_L \mathscr L_L\psi\rangle| \,dV \\
    & \qquad + \int_{\mathcal D^{\tau_2}_{\tau_1} } \langle r\rangle^{p-2} | \langle \gamma^{e_A }\nabla_{e_A} ( \mathscr L_{\underline L}\psi) , \gamma^T \nabla_L \mathscr L_L\psi\rangle| \,dV \\
    & \qquad + \int_{\mathcal D^{\tau_2}_{\tau_1} } \langle r\rangle^{p-2} | \langle [ \gamma^\mu\nabla_\mu, \mathscr L_{\underline L} ]\psi , \gamma^T \nabla_L \mathscr L_L\psi\rangle| \,dV,
\end{align*}
where the third integral can be again absorbed somewhere into \eqref{pf-comm-L1}. Thus we are only concerned with the first and the second integrals of the inequality \eqref{pf-comm-L1}. Now we deal with the second integral of \eqref{pf-comm-L1}. Since $p\le2$, we write
\begin{align*}
    \int_{\mathcal D^{\tau_2}_{\tau_1} } \langle r\rangle^{p-2} | \langle \gamma^{e_A }\nabla_{e_A} ( \mathscr L_{\underline L}\psi) , \gamma^T \nabla_L \mathscr L_L\psi\rangle| \,dV & \lesssim \sum_{A=1,2}\int_{\mathcal D^{\tau_2}_{\tau_1} }  \langle r\rangle^{p-2}|\nabla_{e_A}  \mathscr L_{\underline L}\psi|  | \nabla_L \mathscr L_L\psi | \,dV \\
    & \lesssim \int_{\mathcal D^{\tau_2}_{\tau_1} } \langle r\rangle^{p-2}|\slashed\nabla  \mathscr L_{\underline L}\psi|^2\,dV + \int_{\mathcal D^{\tau_2}_{\tau_1} } \langle r\rangle^{p-2} | \nabla_L \mathscr L_L\psi |^2 \,dV .
\end{align*}
If $p=1+\delta>1$, then
we apply \eqref{p-we-dirac-Z} with $p'=\delta$ and get
\begin{align*}
    \int_{\mathcal D^{\tau_2}_{\tau_1} } \langle r\rangle^{-1+\delta}|\slashed\nabla  \mathscr L_{\underline L}\psi|^2\,dV & \le \int_{ \substack{ u=\tau_1-R \\ v \ge \tau_1+R } } r^{\delta} |\nabla_L \mathscr L_{\underline L}\psi|^2\,d\sigma + \int_{\mathcal D^{\tau_2}_{\tau_1} } r^\delta | \langle \gamma^\lambda\nabla_\lambda [ \gamma^\mu\nabla_\mu, \mathscr L_{\underline L} ] \psi, \gamma^T \nabla_L \mathscr L_{\underline L}\psi\rangle | \,dV,
\end{align*}
and
\begin{align*}
    \int_{\mathcal D^{\tau_2}_{\tau_1} } r^\delta | \langle \gamma^\lambda\nabla_\lambda [ \gamma^\mu\nabla_\mu, \mathscr L_{\underline L} ] \psi, \gamma^T \nabla_L \mathscr L_{\underline L}\psi\rangle | \,dV & \lesssim \int_{\mathcal D^{\tau_2}_{\tau_1} } \langle r\rangle^{-2+\delta} | \nabla_L^2 \psi | |\nabla_L\mathscr L_{\underline L}\psi| \,dV \\
    & \qquad + \int_{\mathcal D^{\tau_2}_{\tau_1} } \langle r\rangle^{-2+\delta} | \nabla_{\underline L} \nabla_L \psi | |\nabla_L\mathscr L_{\underline L}\psi| \,dV  \\
    & \qquad + \int_{\mathcal D^{\tau_2}_{\tau_1} } \langle r\rangle^{-2+\delta} | \nabla_{e_A} \nabla_L \psi | |\nabla_L\mathscr L_{\underline L}\psi| \,dV .
\end{align*}
The first integral can be absorbed into the very first inequality \eqref{pf-comm-L1}, since we assumed here $p=1+\delta>1$. The second spacetime integral can be absorbed into the left-hand side of the inequality \eqref{p-we-dirac-Z} with $p=\delta$. 
To control the third integral, we write
\begin{align*}
    \int \langle r\rangle^{-2+\delta} | \nabla_{e_A} \nabla_L \psi | |\nabla_L\mathscr L_{\underline L}\psi| \,dV & \lesssim  \int \langle r\rangle^{-2+\delta} | \slashed\nabla \nabla_L \psi |^2\,dV + \int \langle r\rangle^{-2+\delta} |\nabla_L\mathscr L_{\underline L}\psi|^2 \,dV.
\end{align*}
Then the first spacetime integral of the above inequality is controlled via the inequality \eqref{p-we-dirac-Z} with $Z=L$ and $p=\delta$. The second spacetime integral involving $\nabla_L\mathscr L_{\underline L}\psi$ is controlled by the inequality \eqref{p-we-dirac-Z} with $Z=\underline L$ and $p=\delta$. Similarly, we see that
\begin{align*}
    \int_{\mathcal D^{\tau_2
    }_{\tau_1}} \langle r\rangle^{-1+\delta} |\nabla_L\mathscr L_L\psi|^2\,dV & \lesssim \int_{ \substack{ u=\tau_1-R \\ v\ge \tau_1+R } } r^\delta |\nabla_L\mathscr L_L\psi|^2\,d\sigma +  \int_{\mathcal D^{\tau_2
    }_{\tau_1}} r^\delta |\langle\gamma^\lambda\nabla_\lambda [ \gamma^\mu\nabla_\mu, \mathscr L_L ]\psi, \gamma^T\nabla_L\mathscr L_L\psi\rangle|\,dV,
\end{align*}
and hence the spacetime integral of the commutator term can be absorbed into the left-hand side of the very first inequality \eqref{pf-comm-L1}.
In order to complete the control of the commutator $[\gamma^\mu\nabla_\mu, \mathscr L_L]$, we are left to consider the integral
\begin{align}\label{pf-comm-L-int1}
    \int_{\mathcal D^{\tau_2}_{\tau_1} } r^p | \gamma^\lambda\nabla_\lambda ( \eta_A \gamma^{\underline L}\nabla_{e_A}\psi, \gamma^T \nabla_L \mathscr L_L\psi\rangle | \,dV.
\end{align}
Then we have
\begin{align*}
     \int_{\mathcal D^{\tau_2}_{\tau_1} } r^p | \gamma^\lambda\nabla_\lambda ( \eta_A \gamma^{\underline L}\nabla_{e_A}\psi, \gamma^T \nabla_L \mathscr L_L\psi\rangle | \,dV & \lesssim  \sum_{A=1,2} \int_{\mathcal D^{\tau_2}_{\tau_1} } r^{p-2} | \gamma^\lambda \gamma^{\underline L}\nabla_\lambda \nabla_{e_A}\psi | | \nabla_L \mathscr L_L\psi | \,dV \\
     & \qquad + \int_{\mathcal D^{\tau_2}_{\tau_1} } r^{p-3} | \nabla_{e_A}\psi | | \nabla_L \mathscr L_L\psi | \,dV.
\end{align*}
From now on we focus on the first integral of the right-hand side. Due to the Clifford algebra, the integral with $\lambda=\underline L$ vanishes and hence we are left to treat the following integrals:
\begin{align*}
     \int_{\mathcal D^{\tau_2}_{\tau_1} } r^{p-2} | \gamma^{L}\gamma^{\underline L} \nabla_L \nabla_{e_A}\psi | | \nabla_L \mathscr L_L\psi | \,dV + \int_{\mathcal D^{\tau_2}_{\tau_1} } r^{p-2} | \gamma^{e_A}\gamma^{\underline L} \nabla_{e_A} \slashed\nabla\psi | | \nabla_L \mathscr L_L\psi | \,dV.
\end{align*}
The control of the first integral is given in the previous argument. We focus on the second integral involving the term $|\slashed\nabla^2\psi|$. For $p=1+\delta>1$, we see that
\begin{align*}
   & \int_{\mathcal D^{\tau_2}_{\tau_1}} r^{-1+\delta} |\slashed\nabla^2\psi| |\nabla_L\mathscr L_L\psi|\,dV \\
    & \lesssim   \int_{\mathcal D^{\tau_2}_{\tau_1}} r^{-3+\delta} |\slashed\nabla\mathscr L_{\Omega}\psi|^2\,dV +  \int_{\mathcal D^{\tau_2}_{\tau_1}} r^{-1+\delta} |\nabla_L\mathscr L_L\psi|^@\,dV \\
    & \lesssim \int_{ \substack{ u=\tau_1-R \\ v\ge \tau_1+R } } r^\delta |\nabla_L\mathscr L_{\Omega}\psi|^2\,d\sigma + \int_{ \substack{ u=\tau_1-R \\ v\ge \tau_1+R } } r^{\delta}|\nabla_L\mathscr L_L\psi|^2\,dV \\
    & \qquad + \int_{\mathcal D^{\tau_2}_{\tau_1}}  r^{\delta}| \langle \gamma^\lambda\nabla_\lambda [  \gamma^\mu\nabla_\mu, \mathscr L_{\Omega}]\psi , \gamma^T \nabla_L\mathscr L_{\Omega}\psi\rangle| \,dV + \int_{\mathcal D^{\tau_2}_{\tau_1}}  r^{\delta}| \langle \gamma^\lambda\nabla_\lambda [  \gamma^\mu\nabla_\mu, \mathscr L_{L}]\psi , \gamma^T \nabla_L\mathscr L_{\Omega}\psi\rangle| \,dV ,
\end{align*}
where the spacetime integrals of the commutator terms can be absorbed into the spacetime integral with the weight $r^p$, since $p=1+\delta>1$.
Thus we conclude that for $p=1+\delta>1$,
\begin{align}
\begin{aligned}
	& \int_{\mathcal D^{\tau_2}_{\tau_1}}  r^{p}| \langle \gamma^\lambda\nabla_\lambda [  \gamma^\mu\nabla_\mu, \mathscr L_{L}]\psi , \gamma^T \nabla_L\mathscr L_{L}\psi\rangle| \,dV \lesssim  \sum_{Z\in \{L,\underline L,\Omega\} } \int_{ \substack{ u=\tau_1-R \\ v\ge \tau_1+R } }  r^{p-1} |\nabla_L \mathscr L_Z\psi|^2\,d\sigma.
\end{aligned}
\end{align}
For $p\le 1$, a repetition of the above argument gives
\begin{align}
\begin{aligned}
	& \int_{\mathcal D^{\tau_2}_{\tau_1}}  r^{p}| \langle \gamma^\lambda\nabla_\lambda [  \gamma^\mu\nabla_\mu, \mathscr L_{L}]\psi , \gamma^T \nabla_L\mathscr L_{L}\psi\rangle| \,dV \lesssim  \sum_{Z\in \{L,\underline L,\Omega\} } \int_{ \substack{ u=\tau_1-R \\ v\ge \tau_1+R } }  r^{p} |\nabla_L \mathscr L_Z\psi|^2\,d\sigma.
\end{aligned}
\end{align}

\noindent\textit{Case 2. $Z=\underline L$:}
We need to control the integral:
\begin{align}
\begin{aligned}
    \int_{\mathcal D^{\tau_2}_{\tau_1} } r^p | \langle \gamma^\lambda\nabla_\lambda [ \gamma^\mu\nabla_\mu, \mathscr L_{\underline L} ] \psi , \gamma^T \nabla_L \mathscr L_{\underline L}\psi \rangle | \,dV & \le  \int_{\mathcal D^{\tau_2}_{\tau_1} } r^p | \langle \gamma^\lambda\nabla_\lambda ( \eta_A \gamma^{e_A}\nabla_L\psi) , \gamma^T \nabla_L \mathscr L_{\underline L}\psi \rangle |\,dV \\
    & \qquad + \int_{\mathcal D^{\tau_2}_{\tau_1} } r^p | \langle \gamma^\lambda\nabla_\lambda ( \eta_A \gamma^{\underline L} \nabla_{e_A} ), \gamma^T \nabla_L \mathscr L_{\underline L}\psi \rangle \,dV .
    \end{aligned}
\end{align}
Then we have
\begin{align*}
    & \int_{\mathcal D^{\tau_2}_{\tau_1} } r^p | \langle \gamma^\lambda\nabla_\lambda ( \eta_A \gamma^{e_A}\nabla_L\psi) , \gamma^T \nabla_L \mathscr L_{\underline L}\psi \rangle |\,dV \\
    & \lesssim  \int_{\mathcal D^{\tau_2}_{\tau_1} } r^{p-2} | \nabla_L^2\psi | |\nabla_L \mathscr L_{\underline L}\psi| \,dV  +  \int_{\mathcal D^{\tau_2}_{\tau_1} } r^{p-2} | \nabla_{\underline L}\nabla_L \psi| |\nabla_L \mathscr L_{\underline L}\psi| \,dV + \int_{\mathcal D^{\tau_2}_{\tau_1} } r^{p-2} |\slashed\nabla \nabla_L \psi| |\nabla_L \mathscr L_{\underline L}\psi| \,dV \\
    & \lesssim \int_{\mathcal D^{\tau_2}_{\tau_1} } |\nabla_L \mathscr L_L\psi|^2 + |\nabla_L \mathscr L_{\underline L}\psi|^2 + |\slashed\nabla\mathscr L_L\psi|^2+ |\nabla_L\psi|^2+|\slashed\nabla\psi|^2\, dV ,
\end{align*}
where all the terms have already been controlled in the previous argument. We omit the further details.

\noindent\textit{Case 3. $Z=\Omega_{ij}$:} In advance to the proof we invoke the metric $g=m+h$, where $m$ is the Minkowski metric and $h= O(r^{-1})$, which is small enough so that the metric $g$ is near the Minkowski metric. We write
\begin{align}
    h (r, \omega) = h^{\rm (rad)}(r) + h^{\rm (ang)}(r,\omega),
\end{align}
where $\omega\in \mathbb S^2$ and assume that
\begin{align}
    | h^{\rm (ang)}| \le O(r^{-1-\eta}).
\end{align}
By the commutator identity \eqref{comm-dirac}, we have
\begin{align}
  |  [ \gamma^\mu\nabla_\mu, \mathscr L_{\Omega_{ij}} ] \psi | \lesssim |{}^{(\Omega)}\pi| |\slashed\nabla\psi| \lesssim r^{-1-\eta}  |\slashed\nabla\psi|,
\end{align}
up to lower-order terms.
This implies that
\begin{align*}
   & \int_{\mathcal D^{\tau_2}_{\tau_1} } r^p | \langle \gamma^\lambda\nabla_\lambda [ \gamma^\mu\nabla_\mu, \mathscr L_{\Omega_{ij}} ]\psi, \gamma^T \nabla_L\mathscr L_{\Omega_{ij}}\psi\rangle| \,dV \\
    & \lesssim \int_{\mathcal D^{\tau_2}_{\tau_1} } r^{p-1-\eta} | \nabla_L \slashed\nabla\psi| |\nabla_L\mathscr L_{\Omega_{ij}}\psi| \,dV + \int_{\mathcal D^{\tau_2}_{\tau_1} } r^{p-1-\eta} | \slashed\nabla \slashed\nabla\psi| |\nabla_L\mathscr L_{\Omega_{ij}}\psi| \,dV \\
    & \qquad + \int_{\mathcal D^{\tau_2}_{\tau_1} } r^{p-1-\eta} | \nabla_{\underline L} \slashed\nabla\psi| |\nabla_L\mathscr L_{\Omega_{ij}}\psi| \,dV.
\end{align*}
We only deal with the third integral on the right-hand side of the above inequality. The control of the remaining integrals is followed in a similar way. We write
\begin{align*}
    \int_{\mathcal D^{\tau_2}_{\tau_1} } r^{p-1-\eta} | \nabla_{\underline L} \slashed\nabla\psi| |\nabla_L\mathscr L_{\Omega_{ij}}\psi| \,dV& \lesssim \int_{\mathcal D^{\tau_2}_{\tau_1} } r^{p-1-\eta} |\slashed\nabla\nabla_{\underline L}\psi| |\nabla_L\mathscr L_{\Omega_{ij}}\psi|\,dV + \int_{\mathcal D^{\tau_2}_{\tau_1} } r^{p-3-\eta} |\psi| |\nabla_L \mathscr L_{\Omega_{ij}}\psi|\,dV \\
    & \lesssim \int_{\mathcal D^{\tau_2}_{\tau_1} } r^{p-1-\eta} |\slashed\nabla\mathscr L_{\underline L}\psi| |\nabla_L\mathscr L_{\Omega_{ij}}\psi|\,dV + \int_{\mathcal D^{\tau_2}_{\tau_1} } r^{p-2-\eta} |\slashed\nabla\psi| |\nabla_L \mathscr L_{\Omega_{ij}}\psi|\,dV \\ &\qquad + \int_{\mathcal D^{\tau_2}_{\tau_1} } r^{p-3-\eta} |\psi| |\nabla_L \mathscr L_{\Omega_{ij}}\psi|\,dV.
\end{align*}
If $p=1+\delta>1$, an application of \eqref{p-we-dirac-Z} with $p=1+\delta-\eta$ gives
\begin{align*}
    \int_{\mathcal D^{\tau_2}_{\tau_1} } r^{p-1-\eta} | \nabla_L \mathscr L_{\Omega_{ij}}\psi|^2\,dV 
    & \lesssim \int_{ \substack{ u=\tau_1-R \\ v\ge \tau_1+R } } r^{p-\eta} |\nabla_L \mathscr L_{\Omega_{ij}}\psi|^2\,d\sigma + \int_{\mathcal D^{\tau_2}_{\tau_1} }r^{p-\eta} | \langle \gamma^\lambda\nabla_\lambda [ \gamma^\mu\nabla_\mu, \mathscr L_{\Omega_{ij}} ]\psi, \gamma^T \nabla_L\mathscr L_{\Omega_{ij}}\psi\rangle | \,dV ,
\end{align*}
and hence we see the better weight $r^{p-\eta}$, which implies that the above spacetime integral can be absorbed somewhere.
By the Hardy inequality, we have
\begin{align*}
    \int_{\mathcal D^{\tau_2}_{\tau_1} } r^{p-3-\eta} |\psi|^2\,dV & \lesssim \int_{\mathcal D^{\tau_2}_{\tau_1} } r^{p-1-\eta} (|\nabla_L\psi|^2+|\slashed\nabla\psi|^2)\,dV,
\end{align*}
and hence the control of the remaining integrals is also obvious. Therefore we obtain
\begin{align}
\begin{aligned}
	\int_{\mathcal D^{\tau_2}_{\tau_1} } r^p | \langle \gamma^\lambda\nabla_\lambda [ \gamma^\mu\nabla_\mu, \mathscr L_{\Omega}]\psi, \gamma^T\nabla_L\mathscr L_{\Omega}\psi\rangle|\,dV  \lesssim \sum_{Z\in \{ L, \underline L, \Omega \} } \int_{ \substack{ u=\tau_1-R \\ v\ge \tau_1+R } } r^{p-\eta} |\nabla_L\mathscr L_Z\psi|^2\,d\sigma .
\end{aligned}
\end{align}
 For $p\le 1$, we simply repeat the previous argument and obtain the same estimates as above. This compltes the control of the commutator terms $[ \gamma^\mu\nabla_\mu, \mathscr L_Z] $ for $Z\in \{ L,\underline L, \Omega_{ij} \} $.
\end{proof}
Finally, we conclude that for $0\le p \le 2$,
\begin{align}
    \begin{aligned}
       & \int_{ \substack{ u = \tau_2-R \\ v \ge \tau_2+R } } r^p |\nabla_L\mathscr L_Z\psi|^2\,d\sigma + \int_{\mathcal D^{\tau_2}_{\tau_1} } r^{p-1} ( p |\nabla_L\mathscr L_Z\psi|^2 + (2-p) |\slashed{\nabla}\mathscr L_Z\psi|^2\,dV \\
        & \lesssim \int_{ \substack{ u = \tau_1-R \\ v\ge \tau_1+R } } r^p |\nabla_L\mathscr L_Z\psi|^2\,d\sigma + C \int_{\Sigma_{\tau_1}} \mathcal J^T [\mathscr L_Z\psi] d\sigma .
    \end{aligned}
\end{align}
From now on, we exclusively consider the case $\eta=1$, i.e., the non-spherical part decays like $r^{-2}$, which corresponds to a Kerr-type asymptotics.

The remaining task is identical as the linear case, and hence we obtain
\begin{align}
    \int_{\tau_{N-1}}^{\tau_N}\int_{\Sigma_\tau } \mathcal J^T [\mathscr L_Z\psi] d\sigma \le (\tau_N)^{-1}\left(\int_{ \substack{ u = \tau_{1}-R \\ v\ge \tau_{1}+R } } r^2 |\nabla_L\mathscr L_Z\psi|^2\,d\sigma + C \int_{\Sigma_{\tau_1}} \mathcal J^T [\mathscr L_Z\psi] d\sigma\right) + C\int_{\Sigma_{N-1}} \mathcal J^T[\mathscr L_Z\psi]d\sigma .
\end{align}
This implies that
\begin{align}
    \int_{\Sigma_\tau } \mathcal J^T [\mathscr L_Z\psi] d\sigma \le C\tau^{-2} \left( \int_{\substack{ u=\tau_1-R \\ v\ge \tau_1+R } } r^2 |\nabla_L \mathscr L_Z\psi|^2\,d\sigma + \int_{\Sigma_{\tau_1}} \mathcal J^T[ \mathscr L_Z\psi ] d\sigma \right).
\end{align}
It is remarkable that in the aforementioned $r^p$ method for the Dirac equation, we have not distinguished the $\psi_+$ and $\psi_-$. Even though the energy of $\psi_+$ is defined along the outgoing cone $\mathcal C_u$, while the energy of $\psi_-$ is given along the ingoing cone $\underline{\mathcal C}_v$, its good derivative $\nabla_L\psi_+$ and $\slashed\nabla\psi_-$ has well-defined energy along the cone $\mathcal C_u$.

\subsection{Energy of the spinor fields}
Now we control the energy of the spinor fields in terms of the Dirac current. We recall that the Dirac current is given by
\begin{align}
    J^\mu [ \psi ] := \langle \psi,\gamma^\mu \psi\rangle.
\end{align}
We also note that the Dirac current is divergence-free for the linear case:
\begin{align}
    \nabla_\mu J^\mu[\psi ] = 0.
\end{align}
Now we consider an application of the vector fields $Z$ in terms of the spinor-adapted Lie derivative $\mathscr L_Z$:
\begin{align}
    i\gamma^\mu\nabla_\mu (\mathscr L_Z\psi) = i[ \gamma^\mu\nabla_\mu, \mathscr L_Z ] \psi.
\end{align}
Then the divergence of the Dirac current $ J^\mu [ \mathscr L_Z\psi ] $ is in general non-zero. We have
\begin{align}
    \nabla_\mu J^\mu [ \mathscr L_Z\psi ] = \Im \langle [ i\gamma^\mu\nabla_\mu, \mathscr L_Z ]\psi , \mathscr L_Z\psi \rangle.
\end{align}
Thus an application of the divergence theorem on the region $D^{\tau}_{\tau_0}$ bounded by two hypersurfaces $\Sigma_{\tau} $ and $\Sigma_{\tau_0} $ gives
\begin{align}
    \begin{aligned}
    &   \int_{\Sigma_{\tau}\cap\{r\le R\} }\langle \mathscr L_Z\psi, \gamma(T)\mathscr L_Z\psi\rangle \,d\sigma  +\int_{ \substack{ u=\tau-R \\ v\ge \tau+R } } \langle \mathscr L_Z\psi, \gamma(L)\mathscr L_Z\rangle\,d\sigma+ \int_{\mathcal I^{\tau-R}_{\tau_0-R} } \langle \mathscr L_Z\psi, \gamma(\underline L)\mathscr L_Z\psi\rangle\,d\sigma \\
    & =\int_{\Sigma_{\tau_0}\cap\{r\le R\} }\langle \mathscr L_Z\psi, \gamma(T)\mathscr L_Z\psi\rangle \,d\sigma + \int_{ \substack{ u=\tau_0-R \\ v\ge \tau_0+R } } \langle \mathscr L_Z\psi, \gamma(L)\mathscr L_Z\rangle\,d\sigma  \\
    & \qquad+ \int_{D^{\tau}_{\tau_0} } \Im \langle [ i\gamma^\mu\nabla_\mu, \mathscr L_Z ]\psi , \mathscr L_Z\psi \rangle\,dV.
    \end{aligned}
\end{align}
Then we have to deal with the spacetime integral due to the commutator terms. The control of the integral on the interior domain is rather obvious.
Hence we focus on the control of the integral of the commutator terms.
\begin{prop}[Basic linear estimates for spinor fields]
    We consider the linear massless Dirac equation $i\gamma^\mu\nabla_\mu\psi=0$ on the exterior domain $\mathcal D^{\tau}_{\tau_0}$ bounded by two hypersurfaces $\Sigma_\tau$ and $\Sigma_{\tau_0}$ with $r\ge R$. For $Z\in \{L,\underline L, \Omega \} $, we have
    \begin{align}
        \begin{aligned}
           & \int_{\Sigma_\tau} \langle \mathscr L_Z\psi, \gamma(n) \mathscr L_Z\psi \rangle \,d\sigma + \int_{\mathcal I^{\tau-R}_{\tau_0-R} } \langle \mathscr L_Z\psi, \gamma(\underline L)\mathscr L_Z\psi\rangle\,d\sigma \\
           & \qquad \lesssim  \mathscr E^D[\mathscr L_Z\psi]^2(\tau_0)+\sup_{v:\tau\le \tau(v)}\mathscr F^D[\mathscr L_Z\psi]^2(v,\tau_0,\tau)+ {}^{(\delta)}\mathcal E^D[\psi]^2(\tau_0),
        \end{aligned}
    \end{align}
    for some fixed $\delta>0$. More generally, we have the following linear energy estimates:
    \begin{align}
        \mathscr E^D_{\le k}[\psi]^2(\tau) + \mathscr F^D_{\le k}[\psi]^2(v=\infty,\tau_0,\tau) \lesssim \mathscr E^D_{\le k}[\psi]^2(\tau_0)+\sup_{v: \tau\le\tau(v)}\mathscr F^D_{\le k}[\psi]^2(v,\tau_0,\tau)+{}^{(\delta)}\mathcal E^D_{\le k-1}[\psi]^2(\tau_0).
    \end{align}
\end{prop}
\begin{proof}
    We recall that the divergence theorem for the Dirac current on the domain $ D^{\,\tau}_{\tau_0}$ gives
    \begin{align}
        \begin{aligned}
            & \int_{\Sigma_\tau} \langle \mathscr L_Z\psi, \gamma(n) \mathscr L_Z\rangle \,d\sigma + \int_{\mathcal I^{\tau-R}_{\tau_0-R} } \langle \mathscr L_Z\psi, \gamma(\underline L)\mathscr L_Z\psi\rangle\,d\sigma \\
            & \qquad = \int_{\Sigma_{\tau_0}}\langle \mathscr L_Z\psi, \gamma(n) \mathscr L_Z\rangle \,d\sigma + \int_{D^{\tau}_{\tau_0}} \langle [ \gamma^\mu\nabla_\mu, \mathscr L_Z ]\psi, \mathscr L_Z\psi\rangle \,dV.
        \end{aligned}
    \end{align}
    By Proposition \ref{comm-dirac}, we have
    \begin{align}
        \begin{aligned}
           \sum_{Z\in \{L,\underline L,\Omega \} } \left| \int_{D^{\tau}_{\tau_0}} \langle [ \gamma^\mu\nabla_\mu, \mathscr L_Z ]\psi, \mathscr L_Z\psi\rangle \,dV \right| & \lesssim \int_{D^{\tau}_{\tau_0}}  \frac1{\langle r\rangle^2} ( |\nabla_{\underline L}\psi| + |\slashed\nabla\psi| ) |\mathscr L_{L}\psi|\,dV \\
           & \qquad + \int_{D^{\tau}_{\tau_0}}  \frac1{\langle r\rangle^2} ( |\nabla_{ L}\psi| + |\slashed\nabla\psi| ) |\mathscr L_{\underline L}\psi|\,dV  \\
           & \qquad\qquad + \int_{D^{\tau}_{\tau_0}}  \frac1{\langle r\rangle^2} |\slashed\nabla\psi|  |\mathscr L_{\Omega}\psi|\,dV.
        \end{aligned}
    \end{align}
    By the definition of the Kosmann Lie derivative, we have
    \begin{align}
        |\mathscr L_L\psi| \lesssim |\nabla_L\psi|+\frac1r|\psi|, \ |\mathscr L_{\underline L}\psi| \lesssim |\nabla_{\underline L}\psi|+\frac1r|\psi|, \ |\mathscr L_\Omega\psi| \lesssim |\nabla_\Omega \psi|+ |\psi|.
    \end{align}
    Then we see that for $Z=L$,
    \begin{align*}
        \int_{\mathcal D^\tau_{\tau_0}}   \frac1{\langle r\rangle^2} ( |\nabla_{\underline L}\psi| + |\slashed\nabla\psi| ) |\mathscr L_{L}\psi|\,dV & \lesssim  \int_{\mathcal D^\tau_{\tau_0}}   \frac1{\langle r\rangle^2} ( |\nabla_{\underline L}\psi| + |\slashed\nabla\psi| )|\nabla_L\psi|\,dV \\
        & \qquad + \int_{\mathcal D^\tau_{\tau_0}}   \frac1{\langle r\rangle^3} ( |\nabla_{\underline L}\psi| + |\slashed\nabla\psi| )|\psi|\,dV \\
        & \lesssim  \left(\int_{\tau_0}^{\tau} {}_{(-1-\delta)}\mathcal E^D[\mathscr L_Z\psi]^2(\tau)\,d\tau \right)^\frac12 \left( \int_{\tau_0}^{\tau}{}_{(-3-\delta)}\mathcal E^D[\psi]^2(\tau)\,d\tau \right)^\frac12 
        & 
    \end{align*}
    and hence an application of the Morawetz estimates Proposition \ref{morawetz-dirac} gives the desired estimates. The estimate for $Z=\underline L$ is very similar and hence we omit the details.

    For $Z=\Omega$, we see that
    \begin{align*}
        \int_{D^{\tau}_{\tau_0}}  \frac1{\langle r\rangle^2} |\slashed\nabla\psi|  |\mathscr L_{\Omega}\psi|\,dV & \lesssim \int_{\mathcal D^\tau_{\tau_0}} \frac1{\langle r\rangle} |\slashed\nabla\psi|^2\,dV +  \int_{\mathcal D^\tau_{\tau_0}} \frac1{\langle r\rangle^2} |\slashed\nabla\psi||\psi|\,dV. 
    \end{align*}
    Then the first integral is controlled via the Morawetz estimates Proposition \ref{morawetz-dirac}. For the second integral, using the Hardy inequality, we are left to deal with the integral
    \begin{align*}
    	\int_{\mathcal D^\tau_{\tau_0}}  \frac1{\langle r\rangle} |\nabla_L\psi|^2\,dV,
    \end{align*}
    which is bounded by ${}^{(\delta)}\mathcal E^D[\psi]^2(\tau_0)$ due to the Morawetz estimates Proposition \ref{Morawetz} with an obvious modification of weight. The linear estimates for general $k\ge1$ follows in an obvious way.
\end{proof}
\begin{rem}\label{rem-dirac-wave}
The preceding argument may give the impression that the distinction between the Dirac current energy $\mathscr E^D$ and the wave energy $\mathcal E^D$ is unessential. Indeed, for the linear Dirac equation, the Morawetz estimates associated with the squared Dirac equation provide the spacetime control needed to estimate the commutator errors, and hence allow one to recover the energy bounds for $\mathscr L_Z\psi$.

This phenomenon is, however, restricted to the linear setting. In the nonlinear problem, the wave energy alone is generally insufficient to close the bootstrap argument, since the nonlinear terms cannot be controlled solely through estimates for the squared equation. It is therefore essential to combine the $L^2$ energy associated with the Dirac current, which provides the fundamental energy estimate for the spinor field itself, with the wave-based Morawetz estimates that yield the necessary spacetime control.

The nonlinear analysis in the next section demonstrates precisely where the wave energy alone fails to close and, consequently, why the Dirac current energy is indispensable.
\end{rem}



\section{Application to tensorial wave-Dirac system}\label{sec:tensor-dirac}
We consider the following coupled semilinear system:
\begin{align}\label{tensor-dirac}
    \begin{aligned}
        i\gamma^\mu\nabla_\mu\psi = F_{\mu\nu} \gamma^\mu\gamma^\nu\psi, \\
        \nabla^\mu F_{\mu\nu} = \langle \psi, \gamma_\nu\psi\rangle, \\
        dF = 0,
    \end{aligned}
\end{align}
where $F$ is an antisymmetric $2$-tensor. We also define the null components of $F$:
\begin{align}
    \alpha_A := F(e_A,L), \ \underline\alpha_A := F(e_A,\underline L), \ \rho := \frac12 F(L,\underline L), \ \sigma := \frac12 \epsilon^{AB}F(e_A,e_B),
\end{align}
where $\epsilon_{AB}$ is the antisymmetric tensor such that $\epsilon^{12}=1$. We refer the reader to Appendix \ref{sec:null-tensor} for details of the null decomposition of the antisymmetric $2$-form $F$. See also \cite{pasqualotto}. 

We begin by establishing the local well-posedness of \eqref{tensor-dirac}. To do this, we introduce the notation:
\begin{align}
\begin{aligned}
 \mathcal X^D_{k}[\psi]^2(\tau) := {}_{\rm int}\mathcal E^D_{\le k-1}[\psi]^2(\tau) + \mathscr E^D_{\le k}[\psi]^2(\tau)+\sup_{v:\tau\le \tau(v)} \mathscr F^D_{\le k}[\psi]^2(v,0,\tau), \\
  \mathcal X^T_{k}[F]^2(\tau) := {}_{\rm int}\mathcal E^T_{\le k-1}[F]^2(\tau) + \mathscr E^T_{\le k}[F]^2(\tau)+\sup_{v:\tau\le \tau(v)} \mathscr F^T_{\le k}[F]^2(v,0,\tau).
  \end{aligned}
\end{align}
In particular, if $\tau=0$, then the energy along the ingoing cone vanishes and hence $\mathcal X^D_{k}[\psi]^2(0) := {}_{\rm int}\mathcal E^D_{\le k-1}[\psi]^2(0) + \mathscr E^D_{\le k}[\psi]^2(0)$.
\begin{thm}[Local well-posedness] \label{thm-lwp}
    Let $N\ge11$ and $0<\epsilon_0<1$ be given. We consider the Cauchy problem for the tensorial wave-Dirac system \eqref{tensor-dirac} with the initial data $(\psi,F)|_{\Sigma_{\tau=0}}:=(\psi_0,F_0)$ satisfying the smallness condition:
     \begin{align}
       \mathcal X^D_{N}[\psi]^2(0)+\mathcal X^T_{N}[F]^2(0) \le (\epsilon_0)^2.
    \end{align}
    Then the system \eqref{tensor-dirac} is locally well-posed. More precisely, there exists a $\tau=\tau(\epsilon_0)>0$ such that the system \eqref{tensor-dirac} admits a unique solution $(\psi,F)$ satisfying
    \begin{align}
        \psi \in C( [ 0, \tau]; H^N)\cap C^1((0,\tau); H^{N-1}) ,\quad F\in C( [0,\tau]; H^N) \cap C^1((0,\tau); H^{N-1})
    \end{align}
    with the energy estimates
    \begin{align}
      \sup_{0\le\tau'\le\tau}\left(   \mathcal X^D_{N}[\psi]^2(\tau')+\mathcal X^T_{N}[F]^2(\tau')\right)\le 2(\epsilon_0)^2.
    \end{align}
\end{thm}
The proof follows the standard iteration argument based on the energy estimates, which is postponed to Appendix \ref{sec:lwp}.

\begin{prop}
    Let $\mathcal T^*>0$ be a maximal existence time of local solutions to \eqref{tensor-dirac}. If $\mathcal T^*$ is finite, then the energy blows up, i.e., we have
    \begin{align*}
        \limsup_{\tau  \to \mathcal T^*} \left( \mathcal X^D_{N}[\psi]^2(\tau)+\mathcal X^T_{N}[F]^2(\tau) \right) = \infty.
    \end{align*}
    Equivalently,  if
     \begin{align*}
         \sup_{ 0\le \tau < \mathcal T^*} \left( \mathcal X^D_{N}[\psi]^2(\tau)+\mathcal X^T_{N}[F]^2(\tau) \right) < \infty,
     \end{align*}
     then solutions can be extended beyond $\mathcal T^*$.
\end{prop}

The continuation criterion above shows that the proof of global existence reduces to establishing uniform energy bounds. We are now ready to state the main theorem, the full version of Theorem \ref{main-thm-informal}.
\begin{thm}\label{main-thm-formal}
    We consider the tensorial wave-Dirac system \eqref{tensor-dirac} on the $(1+3)$-Lorentzian manifold $(\mathcal M,g)$, which is foliated by the family of the hypersurfaces $\Sigma_\tau$. We fix $N\ge 11$ and a small $0<\epsilon_0<1$. We also fix $0<\delta<\frac1{20}$. We let the initial data are given by $(\psi,F)|_{\Sigma_{\tau=0}}:=(\psi_0,F_0) $ and satisfy the smallness condition:
    \begin{align}
        \mathcal X^D_{N}[\psi]^2(0)+\mathcal X^T_{N}[F]^2(0) \le (\epsilon_0)^2.
    \end{align}
    Then there exists a small $0\le \varepsilon<\epsilon_0$ and a sufficiently large $C\ge1$ satisfying the following energy estimates:
    \begin{align}
    \mathcal X^D_{N}[\psi]^2(\tau)+\mathcal X^T_{N}[F]^2(\tau) \le C^2\varepsilon^2.
    \end{align}
    We have the boundedness of $r^p$-type weighted energy for the spinor fields
    \begin{align}
      {}^{(2-\delta)}\mathcal E^D_{\le N-1} [ \psi ]^2(\tau) \le C^2\varepsilon^2,
    \end{align}
    and a logarithmic growth of the weighted wave energy for the extreme components $\alpha,\underline\alpha$ of the tensor fields:
    \begin{align}
        {}^{(2-\delta)}\mathcal E^T_{\le N-2}[\alpha]^2(\tau) +  {}^{(2-\delta)}\mathcal E^T_{\le N-2}[\underline\alpha]^2(\tau) \le C^2\varepsilon^2 \tau^{\frac\delta2}.
    \end{align}
    The further decay properties of the spinor field $\psi$ and the tensor fields $F$ is listed in Section \ref{sec:iteration}.
\end{thm}


\subsection{Bootstrap assumptions}\label{subsec:bootstrap}
By Theorem \ref{thm-lwp}, there exists a unique solution $(\psi,F)$ to the system \eqref{tensor-dirac}. Let $\mathcal T^*>0$ denote the maximal existence time of the solution $(\psi,F)$. Now we consider the domain $D^{\mathcal T^*}_0$ defined by
\begin{align}
    \mathcal D^{\mathcal T^*}_0 := \bigcup_{0\le \tau< \mathcal T^*}\Sigma_{\tau}.
\end{align}
We now introduce the bootstrap assumptions. Let
\begin{align*}
    \mathcal I :=
\left\{ T\in(0,\mathcal T^*) :
\text{the bootstrap assumptions below hold on }D_0^T
\right\},
\end{align*}
and define
\begin{align*}
    \tau_{\max} := \sup\,\mathcal I.
\end{align*}
By the local well-posedness and the smallness of the initial data,
the bootstrap assumptions hold for sufficiently small times.
Hence the set $\mathcal I$ is nonempty, and therefore we have $\tau_{\max}>0$.

We let $|I|=k$ and $0\le k\le N$. Fix a small $0<\delta<\frac1{20}$ and a large constant $C>1$. We suppose that the following bootstrap assumptions hold on the domain $D^{\tau_{\max}}_{0}$.
\subsubsection{Assumed estimates for the spinor fields}
We assume the following:
\begin{align}
    \begin{aligned}
        {}^{(2-\delta)}\mathcal E^D_{\le N-4}[\psi]^2(\tau) \le C^2\varepsilon^2, \quad   {}^{(2-\delta)}\mathcal E^D_{\le N-1}[\psi]^2(\tau) \le C^2\varepsilon^2\tau^{\delta}, \\
        {}^{(1)}\mathcal E^D_{\le N-4}[\psi]^2(\tau) \le C^2\varepsilon^2 \tau^{-1+\delta}, \quad {}^{(1)}\mathcal E^D_{\le N-1}[\psi]^2(\tau) \le C^2\varepsilon^2.
    \end{aligned}
\end{align}
We assume that
\begin{align}
    \int_{\tau}^{2\tau} {}^{(0)}\mathcal E^D_{\le N-4}[\psi]^2(\tau)\,d\tau \le C^2\varepsilon^2 \tau^{-1+\delta}, \quad  \int_{\tau}^{2\tau} {}^{(1-\delta)}\mathcal E^D_{\le N-1}[\psi]^2(\tau)\,d\tau \le C^2\varepsilon^2 ,
\end{align}
where $2\tau \le \tau_{\max}$.
We assume that
\begin{align}
    \begin{aligned}
        \mathscr E^D_{\le N}[\psi]^2(\tau) +\sup_{v:\tau\le \tau(v)} \mathscr F^D_{\le N}[\psi]^2(v,\tau,2\tau) \le C^2\varepsilon^2 \tau^{\epsilon},
    \end{aligned}
\end{align}
where $\epsilon>0$ is chosen so that $\epsilon\ll \delta$.
On the interior domain, we assume that
\begin{align}
    {}_{\rm int}\mathcal E^D_{\le N-4}[\psi]^2(\tau) \le C^2\varepsilon^2 \tau^{-1+\delta}, \quad  {}_{\rm int}\mathcal E^D_{\le N-1}[\psi]^2(\tau) \le C^2\varepsilon^2.
\end{align}
\subsubsection{Assumed estimates for the tensor fields}
We also present the bootstrap assumptions for the tensor fields $h$:
\begin{align}
    \begin{aligned}
        {}^{(2-\delta)}\mathcal E^T_{\le N-2}[\alpha]^2(\tau) \le C^2\varepsilon^2 \tau^{3\delta}, \ {}^{(2-\delta)}\mathscr E^T_{\le N-2}[\underline\alpha]^2(\tau) \le C^2\varepsilon^2 \tau^{2\delta}, \\
         {}^{(1)}\mathcal E^T_{\le N-4}[\alpha]^2(\tau) \le C^2\varepsilon^2 \tau^{-\frac12+4\delta}, \ {}^{(1)}\mathscr E^T_{\le N-4}[\underline\alpha]^2(\tau) \le C^2\varepsilon^2 \tau^{-\frac12+3\delta},
    \end{aligned}
\end{align}
and
\begin{align}
    \begin{aligned}
     {}^{(1+\delta)}   \mathscr E^T_{\le N}[\alpha]^2(\tau) \le C^2\varepsilon^2 \tau^{3\delta}, \ \sup_{v:\tau\le \tau(v)}{}^{(0)}\mathscr F^T_{\le N}[\underline\alpha]^2(v,\tau,2\tau) \le C^2\varepsilon^2 \tau^{\delta}, \\
      {}^{(\delta)}   \mathscr E^T_{\le N}[\alpha]^2(\tau) \le C^2\varepsilon^2 , \ \sup_{v:\tau\le \tau(v)} {}^{(-1+\delta)}\mathscr F^T_{\le N}[\underline\alpha]^2(v,\tau,2\tau) \le C^2\varepsilon^2 .
    \end{aligned}
\end{align}
On the interior domain, we assume that
\begin{align}
    {}_{\rm int}\mathscr E^T_{\le N-4}[F]^2(\tau) \le C^2\varepsilon^2 \tau^{-1+\delta}, \quad {}_{\rm int}\mathscr E^T_{\le N}[F]^2(\tau) \le C^2\varepsilon^2.
\end{align}
\subsubsection{Integrated local energy decay estimates}
We assume that
\begin{align}\label{iled-ab}
    \int_{\mathcal D^{2\tau }_\tau } r^{-2+\delta} |\mathcal L_Z^{\le N}\underline\alpha|^2\,dV \le C^2\varepsilon^2 ,
\end{align}
where $\mathcal D^{2\tau}_{\tau}=D^{2\tau}_{\tau}\cap \{r\ge R\} $,
and
\begin{align}\label{iled-a}
    \int_{\tau }^{2\tau} {}^{(\delta)}\mathscr E^T_{\le N}[\alpha,\rho,\sigma]^2(\tau)\,d\tau  \le C^2\varepsilon^2 \tau^{3\delta} .
\end{align}

\subsection{Bootstrap hierarchy}
It would be instructive to present the energy bootstrap hierarchy. From the following bootstrap assumptions for the spinor fields:
\begin{align}
\boxed{
\begin{aligned}
{}^{(2-\delta)}\mathcal E^D_{\le N-4}[\psi]^2 \le C^2\varepsilon^2 , \ {}^{(2-\delta)}\mathcal E^D_{\le N-1}[\psi]^2 \le C^2\varepsilon^2 \tau^{\delta}  \\
{}^{(1)}\mathcal E^D_{\le N-4}[\psi]^2 \le C^2\varepsilon^2 \tau^{-1+\delta}, \ {}^{(1)}\mathcal E^D_{\le N-1}[\psi]^2 \le C^2\varepsilon^2 .
\end{aligned}
}
\end{align}
and the tensor fields:
\begin{align}
\boxed{
\begin{aligned}
{}^{(2-\delta)}\mathcal E^T_{\le N-2}[\alpha]^2 \le C^2\varepsilon^2 \tau^{3\delta} , \ {}^{(2-\delta)}\mathcal E^T_{\le N-2}[\underline\alpha]^2 \le C^2\varepsilon^2 \tau^{2\delta} \\
{}^{(1)}\mathcal E^T_{\le N-4}[\alpha]^2 \le C^2\varepsilon^2 \tau^{-\frac12+4\delta}, \ {}^{(1)}\mathcal E^T_{\le N-4}[\underline\alpha]^2 \le C^2\varepsilon^2 \tau^{-\frac12+3\delta},
\end{aligned}
}
\end{align}
we improve a logarithmic growth of the top-order weighted energy:
\begin{align}
	 {}^{(2-\delta)}\mathcal E^D_{\le N-1}[\psi]^2 \le C^2\varepsilon^2 \tau^{\frac12\delta+\epsilon}.
\end{align}
This yields
\begin{align}
	{}^{(2-\delta)}\mathcal E^T_{\le N-2}[\alpha]^2 \le C^2\varepsilon^2 \tau^{\frac52\delta} , \ {}^{(2-\delta)}\mathcal E^T_{\le N-2}[\underline\alpha]^2 \le C^2\varepsilon^2 \tau^{\delta}, \\
	{}^{(2-\delta)}\mathcal E^T_{\le N-3}[\alpha]^2, \, {}^{(2-\delta)}\mathcal E^T_{\le N-3}[\underline\alpha]^2 \le C^2\varepsilon^2 .
\end{align}
From the bootstrap assumptions:
\begin{align}
\boxed{
\begin{aligned}
{}^{(1+\delta)}\mathscr E^T_{\le N}[\alpha]^2 \le C^2\varepsilon^2 \tau^{3\delta}, \ {}^{(0)}\mathscr F^T_{\le N}[\underline\alpha]^2 \le C^2\varepsilon^2 \tau^{\delta},
\end{aligned}
}
\end{align}
we improve a logarithmic growth of the top-order weighted energy:
\begin{align}
\begin{aligned}
{}^{(1+\delta)}\mathscr E^T_{\le N}[\alpha]^2 \le C^2\varepsilon^2 \tau^{\frac52\delta+\epsilon}, \ {}^{(0)}\mathscr F^T_{\le N}[\underline\alpha]^2 \le C^2\varepsilon^2 \tau^{\frac\delta2}.
\end{aligned}
\end{align}
This also gives
\begin{align}
\begin{aligned}
	{}^{(\delta)}\mathscr E^T_{\le N-2}[\alpha]^2+ \int_{\tau}^{2\tau}{}^{(-1+\delta)}\mathscr E^T_{\le N-2}[\rho,\sigma]^2\,d\tau \le C^2\varepsilon^2 \tau^{-\frac12+\delta}.
\end{aligned}
\end{align}
The above improved decay implies the following decay of the weighted energy:
\begin{align}
\begin{aligned}
	{}^{(1)}\mathcal E^T_{\le N-2}[\alpha]^2 \le C^2\varepsilon^2 \tau^{-\frac12+3\delta}, \ {}^{(1)}\mathcal E^T_{\le N-2}[\underline\alpha]^2 \le C^2\varepsilon^2 \tau^{-\frac12+2\delta},
\end{aligned}
\end{align}
which gives
\begin{align}
\begin{aligned}
	{}^{(1)}\mathcal E^D_{\le N-2}[\psi]^2 \le C^2\varepsilon^2 \tau^{-\frac14+\delta}.
\end{aligned}
\end{align}
Combining the aforementioned improved decay estimates, we obtain
\begin{align}
{}^{(0)}\mathcal E^D_{\le N-6}[\psi]^2 \le C^2\varepsilon^2 \tau^{-\frac54+\delta}.
\end{align}

\subsection{Dyadic decomposition}
In order to improve the above bootstrap assumptions, we fix an arbitrary time $T<\tau_{\max}$. We let $\tau_1>0$ be fixed and define
\begin{align*}
    \tau_i := 2^{i-1}\tau_1, \quad i\ge1.
\end{align*}
We also choose $n$ such that $\tau_n \le T \le 2\tau_n$. Then we have
\begin{align*}
    [0,T] = [0,\tau_1] \cup \bigcup_{i=1}^{n-1} [\tau_i,\tau_{i+1}] \cup [\tau_n,T].
\end{align*}
Accordingly, we obviously have the dyadic decomposition of the domain $D^{T}_0$ as follows:
\begin{align}
    D^T_0 = D^{\,\tau_1}_0 \cup \bigcup_{i=1}^{n-1} D^{\,\tau_{i+1}}_{\tau_i} \cup D^{T}_{\tau_n}.
\end{align}
Therefore, it is enough to establish the desired bootstrap improvement on each dyadic slab $D^{\tau_{i+1}}_{\tau_i}$. (See for instance Section \ref{sec:bdd-weight-spinor}, \ref{sec:bdd-weight-tensor}.) For the final slab $D^{T}_{\tau_n}$, since $\tau_n\le T\le 2\tau_n$, we have $D^{T}_{\tau_n}\subset D^{2\tau_n}_{\tau_n}$. Then we have either $2\tau_n<\tau_{\max}$ or $\tau_{\max}<2\tau_n$. If $2\tau_n< \tau_{\max}$, then the final slab is treated exactly same as the preceding slabs. For the latter case, we can repeat the previous procedure on the truncated slab $D^{\min(2\tau_n,\tau_{\max})}_{\tau_n}$. Assuming the desired bootstrap improvement, the continuity argument implies that the bootstrap interval can be extended beyond $\tau_{\max}$, which contradicts the definition of $\tau_{\max}$.

In consequence, we conclude that $\tau_{\max}=\mathcal T^*$, which implies, by the continuity criterion, $\mathcal T^*=\infty$, and hence establish the global existence of solutions.

Moreover, the above discussion allows us to focus on only a single dyadic slab $D^{\tau_{i+1}}_{\tau_i}$, for some $1<i<n-1$.

\subsection{Main strategy}
The main scheme of the proof of Theorem \ref{main-thm-formal} is as follows:
\begin{enumerate}
    \item[(1)] Close of the nonlinearity with $r^{2-\delta}$-weight: \\
The first step is to close the $r^{2-\delta}$-weighted higher-order energy for $\psi$. Concerning the control of the energy for the spinor fields, the most serious interaction appears in the integral of the form:
\begin{align}
    \int_{\mathcal D^{\tau_{i+1}}_{\tau_i} } r^{2-\delta} Z^{\le N-1}\underline\alpha \cdot\nabla_L\psi_- \cdot \nabla_L Z^{\le N-1}\psi_-\,dV,
\end{align}
with $p=2-\delta$.
In this case, we cannot apply the Sobolev inequality to $\underline\alpha$ and hence we are not able to get any decay from $\underline\alpha$. Even worse, we cannot expect any further decay from the $\nabla_L\psi_-$, since it already has the derivative $\nabla_{L}$. Instead, we foliate the region via the ingoing null cones $\underline{\mathcal C}_v$ and apply the weighted Sobolev inequality to the lower-order spinor $\psi_-$ to get the derivative $\nabla_{\underline L}$. In doing this we can exploit the null decomposition for the spinor fields and gain a factor $r^{-1}$. It turns out that the nonlinearity with the weight $r^{2-\delta}$ can be closed even for the most serious interaction, and gives the bound $C^3\varepsilon^3 \tau^{\epsilon} $, which indicates that a logarithmic growth of the top-order energy can be improved. Other possibly serious interaction also appears in the integral of the form:
\begin{align}
    \int_{\mathcal D^{\tau_{i+1}}_{\tau_i}} r^{2-\delta} \nabla_{\underline L}Z^{N-1}\alpha \cdot\psi_+ \cdot \nabla_L Z^{\le N-1}\psi_+ \,dV, \quad  \int_{\mathcal D^{\tau_{i+1}}_{\tau_i}} r^{2-\delta} \alpha \cdot \nabla_{\underline L}Z^{N-1}\psi_+ \cdot \nabla_L Z^{\le N-1}\psi_+ \,dV.
\end{align}
The key is to integrate by parts with respect to the derivative $\nabla_{\underline L}$. In doing so, we obtain several boundary integrals along the null hypersurfaces $\Sigma_\tau$ and the timelike surface $\{r=R\}$, whose decay is easier to control than that of the spacetime integrals. The spacetime integral in which $\nabla_{\underline L}$ acts on the lower-order term is also easier to control. When $\nabla_{\underline L}$ acts on the top-order term $\nabla_L Z^{N-1}\psi$, we encounter a derivative loss. However, this is harmless, since the wave equation associated with the Dirac equation ensures that $\nabla_{\underline L}\nabla_L Z^{\le N-1}\psi$ can be rewritten as $\slashed\Delta Z^{\le N-1}\psi$ together with commutator terms and additional nonlinearities. Thus, we can avoid derivative loss by integrating by parts once again with respect to the angular variables.

Moreover, by using the inequality $|\slashed\nabla Z^{\le k} \psi|\lesssim |Z^{\le k+1}\psi|$, we gain an additional factor $r^{-1}$ and can close the nonlinear estimate with the weight $r^{2-\delta}$ via the weighted Sobolev inequality and the bootstrap assumptions.

\item[(2)] It is remarkable that in the control of the top-order weighted energy for the spinor fields, $\mathcal E^D_{\le N-1}[\psi]$, we do not use any top-order weighted energy for $\alpha$ and $\underline\alpha$. Instead, only relatively lower-order weighted energy ${}^{(2-\delta)}\mathcal E^T_{\le N-4}[\alpha,\underline\alpha]$ is used throughout the control of ${}^{(2-\delta)}\mathcal E^D_{\le N-1}[\psi]$. This observation is significant, in that we do not make any bootstrap assumption for the top-order weighted energy for $\alpha,\underline\alpha$.
In particular, it turns out that $ {}^{(2-\delta)}\mathcal E^T_{\le N-3}[\alpha,\underline\alpha]^2\le C^2\varepsilon^2$ and ${}^{(2-\delta)}\mathcal E^T_{\le N-2}[\alpha]^2\le C^2\varepsilon^2 \tau^{2\delta}$ and ${}^{(2-\delta)}\mathcal E^T_{\le N-2}[\underline\alpha]^2\le C^2\varepsilon^2 \tau^{\delta}$.

\item[(3)] The next step is then to control the weighted energy for the tensor fields $F_{\mu\nu}$: ${}^{(1+\delta)}\mathscr E^T_{\le N}[\alpha] $ and ${}^{(0)}\mathscr E^T_{\le N}[\underline\alpha]$. 
It turns out that one can obtain a decay for
${}^{(\delta)}\mathscr E^T_{\le N-3}[\alpha]$. This also implies the integrated local energy decay estimates as follows:
\begin{align}
    \int_{\tau_i}^{\tau_{i+1}} {}^{(-1+\delta)}\mathscr E^T_{\le N-3}[\rho,\sigma]^2(\tau)\,d\tau \le C^2\varepsilon^2 (\tau_i)^{-\frac12+\delta}.
\end{align}

\item[(4)] Then an improved decay for the weighted energy ${}^{(1)}\mathcal E^T_{\le N-2}[\alpha,\underline\alpha]$ can be established. Although we started with merely a boundedness for ${}^{(1)}\mathcal E^T_{\le N-2}[\alpha,\underline\alpha]$, we are able to establish an improved decay:
\begin{align}
    {}^{(1)}\mathcal E^T_{\le N-2}[\alpha]^2(\tau) \le C^2\varepsilon^2 \tau^{-\frac12+3\delta}, \quad  {}^{(1)}\mathcal E^T_{\le N-2}[\underline\alpha]^2(\tau) \le C^2\varepsilon^2 \tau^{-\frac12+2\delta}.
\end{align}
By using the above improved decay, we can also obtain a decay for the weighted energy of the spinor fields of higher-order:
\begin{align}
{}^{(1)}\mathcal E^D_{\le N-2}[\psi]^2(\tau) \le C^2\varepsilon^2 \tau^{-\frac14+\delta}.
\end{align}
Combining with the improved decay of $\mathcal E^T_{\le N-2}[\alpha,\underline\alpha]$ and the above integrated local energy decay estimates for $\rho$ and $\sigma$, we are able to improve the decay ${}^{(0)}\mathcal E^D_{\le N-6}[\psi]^2 \lesssim \tau^{-1-\delta'}$, for some $\delta'>\delta$.
Then we prove the decay of the weighted energy ${}^{(1)}\mathcal E^D_{\le N-4}[\psi]^2\lesssim \tau^{-1+\delta}$.

\item[(5)] We would like to highlight that a repetition of the argument of closing $r^{2-\delta}$-weighted energy is enough to obtain an improved decay. Indeed, the control of the weighted energy ${}^{(2-\delta)}\mathcal E^D_{\le N-1}[\psi]$ heavily relies on the null decomposition for the spinor and tensor fields.
\end{enumerate}
\begin{rem}
Undoubtedly, there is still room for a possible improvement of the decay of the weighted energy for the spinor and tensor fields. For example, by iterating the above improved decay, the energy ${}^{(1)}\mathcal E^T_{\le N-4}[\alpha,\underline\alpha]$ can be improved further and one might be able to obtain 	${}^{(\delta)}\mathscr E^T_{\le N-4}[\alpha]^2\lesssim \tau^{-1+\delta}$. One can also establish almost optimal decay for the unweighted energy of the spinor of lower-order, i.e., ${}^{(0)}\mathcal E^D_{\le N-6}[\psi]^2\lesssim \tau^{-2+\delta}$. However, we do not pursue establishing a possibly sharp decay. The main concern of this paper is to deduce nonlinear stability of the tensorial wave-Dirac system with a much weaker decay assumption and null geometry.
\end{rem}

\section{The $r^p$-hierarchy}\label{sec:rp-method}
\subsection{$r^p$-weighted energy for the spinor fields}\label{subsec:rp-dirac}
We recall the Dirac part of the system \eqref{eq-tensor-dirac}:
\begin{align}
    i\gamma^\mu\nabla_\mu\psi = F_{\mu\nu}\gamma^\mu\gamma^\nu \psi.
\end{align}
In order to establish an $r^p$-weighted energy hierarchy for the spinor fields, we apply the Dirac operator $\gamma^\lambda\nabla_\lambda$ once again, which yields
\begin{align*}
    g^{\mu\nu}\nabla_\mu\nabla_\nu\psi = \mathcal R\psi + i\gamma^\lambda\nabla_\lambda( F_{\mu\nu}\gamma^\mu\gamma^\nu \psi) .
\end{align*}
In general by commuting the vector fields $\mathscr L_Z^{\le k}$, we obtain
\begin{align}
    g^{\mu\nu}\nabla_\mu\nabla_\nu \mathscr L_Z^{\le k}\psi = \mathcal R(\mathscr L_Z^{\le k}\psi) + i\gamma^\lambda\nabla_\lambda ( [\gamma^\mu\nabla_\mu,\mathscr L_Z^{\le k}]\psi) +i\gamma^\lambda\nabla_\lambda \mathscr L_Z^{\le k}( F_{\mu\nu}\gamma^\mu\gamma^\nu \psi) .
\end{align}
By repeating the argument of Section \ref{subsec:rp-lin-dirac}, we obtain the following inequality on the exterior region $\mathcal D^{\tau_2}_{\tau_1}$, which is bounded by two null hypersurfaces $\Sigma_{\tau_2}$ and $\Sigma_{\tau_1}$: for $0\le k\le N-1$,
\begin{align}
    \begin{aligned}
        &\int_{\substack{u=\tau_2-R \\ v\ge \tau_2+R }} r^p |\nabla_L\mathscr L_Z^{\le k}\psi|^2_T\,d\sigma + \int_{\mathcal D^{\tau_2}_{\tau_1}} r^{p-1} \left( p |\nabla_L\mathscr L_Z^{\le k}\psi|^2+(2-p)|\slashed\nabla\mathscr L_Z^{\le k}\psi|^2 \right)\,dV \\
        & \lesssim \int_{\substack{u=\tau_1-R \\ v\ge \tau_1+R }} r^p |\nabla_L\mathscr L_Z^{\le k}\psi|^2_T\,d\sigma + {}_{\rm int}\mathcal E^D_{\le k}[\psi]^2(\tau_1) + \left| \int_{\mathcal D^{\tau_2}_{\tau_1}} r^p \langle \gamma^\lambda\nabla_\lambda \mathscr L_Z^{\le k}(F_{\mu\nu}\gamma^\mu\gamma^\nu \psi),\gamma^T\nabla_L\mathscr L_Z^{\le k}\psi\rangle\,dV \right|.
    \end{aligned}
\end{align}
We shall use the weighted energy of the spinor fields with $p=2-\delta \to 1\to 0$. To do this, the first step is to control the spacetime integrals of the nonlinearities with the weight $r^{2-\delta}$:
\begin{align}
    \left| \int_{\mathcal D^{\tau_2}_{\tau_1}} r^{2-\delta} \langle \gamma^\lambda\nabla_\lambda \mathscr L_Z^{\le k}(F_{\mu\nu}\gamma^\mu\gamma^\nu \psi),\gamma^T\nabla_L\mathscr L_Z^{\le k}\psi\rangle\,dV \right|.
\end{align}
We remark that only a restricted combination of $(\lambda,\mu,\nu)$ appears due to the Clifford algebra. For instance, if $F=\alpha$, i.e, $(\mu,\nu)=(A,L)$, the integral vanishes unless $\lambda=B$ or $\lambda=\underline L$. In a similar way, it turns out that we need to deal with the integrals of the form: if $F=\alpha$,
\begin{align}
    \begin{aligned}
        & \int_{\mathcal D^{\tau_2}_{\tau_1}} r^p \langle \gamma^{\underline L}\nabla_{\underline L} \mathscr L_Z^{\le k}(\alpha_A\gamma^{e_A}\gamma^{L} \psi),\gamma^T\nabla_L\mathscr L_Z^{\le k}\psi\rangle\,dV  + \int_{\mathcal D^{\tau_2}_{\tau_1}} r^p \langle \gamma^{e_B}\nabla_{e_B} \mathscr L_Z^{\le k}(\alpha_A\gamma^{e_A}\gamma^{L} \psi),\gamma^T\nabla_L\mathscr L_Z^{\le k}\psi\rangle\,dV ,
    \end{aligned}
\end{align}
if $F=\underline\alpha$,
\begin{align}
    \begin{aligned}
         & \int_{\mathcal D^{\tau_2}_{\tau_1}} r^p \langle \gamma^{ L}\nabla_{L}  \mathscr L_Z^{\le k}(\underline\alpha_A\gamma^{e_A}\gamma^{\underline L} \psi),\gamma^T\nabla_L\mathscr L_Z^{\le k}\psi\rangle\,dV  + \int_{\mathcal D^{\tau_2}_{\tau_1}} r^p \langle \gamma^{e_B}\nabla_{e_B} \mathscr L_Z^{\le k}(\underline\alpha_A\gamma^{e_A}\gamma^{L} \psi),\gamma^T\nabla_L\mathscr L_Z^{\le k}\psi\rangle\,dV ,
    \end{aligned}
\end{align}
if $F=\rho$,
\begin{align}
    \begin{aligned}
         & \int_{\mathcal D^{\tau_2}_{\tau_1}} r^p \langle \gamma^{ L}\nabla_{ L} \mathscr L_Z^{\le k}(\rho\gamma^{L}\gamma^{\underline L} \psi),\gamma^T\nabla_L\mathscr L_Z^{\le k}\psi\rangle\,dV  + \int_{\mathcal D^{\tau_2}_{\tau_1}} r^p \langle \gamma^{e_B}\nabla_{e_B} \mathscr L_Z^{\le k}(\rho\gamma^{L}\gamma^{\underline L} \psi),\gamma^T\nabla_L\mathscr L_Z^{\le k}\psi\rangle\,dV \\
         +&  \int_{\mathcal D^{\tau_2}_{\tau_1}} r^p \langle \gamma^{\underline L}\nabla_{\underline L} \mathscr L_Z^{\le k}(\rho\gamma^{\underline L}\gamma^{ L} \psi),\gamma^T\nabla_L\mathscr L_Z^{\le k}\psi\rangle\,dV  + \int_{\mathcal D^{\tau_2}_{\tau_1}} r^p \langle \gamma^{e_B}\nabla_{e_B} \mathscr L_Z^{\le k}(\rho\gamma^{\underline L}\gamma^{ L} \psi),\gamma^T\nabla_L\mathscr L_Z^{\le k}\psi\rangle\,dV ,
    \end{aligned}
\end{align}
if $F=\sigma$,
\begin{align}
    \begin{aligned}
       & \int_{\mathcal D^{\tau_2}_{\tau_1}} r^p \langle \gamma^{ L}\nabla_{ L} \mathscr L_Z^{\le k}(\sigma\gamma^{e_A}\gamma^{e_B} \psi),\gamma^T\nabla_L\mathscr L_Z^{\le k}\psi\rangle\,dV  + \int_{\mathcal D^{\tau_2}_{\tau_1}} r^p \langle \gamma^{e_C}\nabla_{e_C} \mathscr L_Z^{\le k}(\sigma\gamma^{e_A}\gamma^{e_B} \psi),\gamma^T\nabla_L\mathscr L_Z^{\le k}\psi\rangle\,dV  \\
       & \qquad + \int_{\mathcal D^{\tau_2}_{\tau_1}} r^p \langle \gamma^{\underline L}\nabla_{\underline L} \mathscr L_Z^{\le k}(\sigma\gamma^{e_A}\gamma^{e_B} \psi),\gamma^T\nabla_L\mathscr L_Z^{\le k}\psi\rangle\,dV.
    \end{aligned}
\end{align}
We present the detailed control of the above integrals with $F=\alpha,\underline\alpha$ in Section \ref{sec:bdd-weight-spinor} as the control of the integrals including the scalar components is then obvious. After proving the boundedness of the $r^{2-\delta}$-weighted energy for the spinor fields, the improved decay estimates for the energy ${}^{(1)}\mathcal E^D_{\le k}[\psi]$ and ${}^{(0)}\mathcal E^D_{\le k}[\psi]$ can be obtained via a typical argument using the pigeonhole principle on the dyadic region.
\subsection{$r^p$-weighted energy for the null components}\label{subsec:rp-null-h}
We now turn to the quantitative analysis of the tensor fields. In the following two subsections, we consider the system
\begin{align}
\nabla^\mu F_{\mu\nu} = J_\nu, \
\nabla_{[\lambda}F_{\mu\nu]}=0.
\end{align}
Our objective is to establish an $r^p$-weighted energy hierarchy for the null components of $F$. Following the physical-space approach by \cite{holzegel10,pasqualotto}, we derive a family of weighted energy inequalities which will later provide the quantitative decay estimates required in the nonlinear bootstrap argument.

From the null decomposition for the tensor fields, we obtain the following inequality:
\begin{align}\label{p-we-maxwell}
\begin{aligned}
   & \int_{ \substack{u=\tau_2-R \\v\ge \tau_2+R } } r^p (\rho^2+\sigma^2)\,\sin\theta \,d\theta d\phi dv + \int_{\mathcal I^{\tau_2-R}_{\tau_1-R}} r^p |\underline\alpha|^2\,\sin\theta \,d\theta d\phi du + \int_{\mathcal D^{\tau_2}_{\tau_1} } r^{p-1} (2-p) |\underline\alpha|^2\,\sin\theta \,d\theta d\phi dvdu \\
   & =\int_{ \substack{u=\tau_1-R \\v\ge \tau_1+R } } r^p (\rho^2+\sigma^2)\,\sin\theta \,d\theta d\phi dv+ \int_{\tau_1}^{\tau_2} r^p ( \rho^2+\sigma^2- |\underline\alpha|^2) \bigg|_{r=R} \,\sin\theta \,d\theta d\phi d\tau \\
   & \qquad+ \int_{\mathcal D^{\tau_2}_{\tau_1} } r^{p-1}(p-4) (\rho^2+\sigma^2) \,\sin\theta \,d\theta d\phi dvdu + \int_{\mathcal D^{\tau_2}_{\tau_1} } r^p \underline\alpha J_{\underline L}  \,\sin\theta d\theta d\phi dudv,
   \end{aligned}
\end{align}
and
\begin{align}\label{q-we-maxwell}
    \begin{aligned}
       & \int_{ \substack{u=\tau_2-R \\v\ge \tau_2+R } } r^q |\alpha|^2 \,\sin\theta \,d\theta d\phi dv+ \int_{\mathcal I^{\tau_2-R}_{\tau_1-R} } r^q (\rho^2+\sigma^2)\,\sin\theta d\theta d\phi du \\
       & \qquad + \int_{\mathcal D^{\tau_2}_{\tau_1} } r^{q-1} ( (4-q)(\rho^2+\sigma^2) + (q-2)|\alpha|^2) \,\sin\theta \,d\theta d\phi dvdu \\
        & = \int_{ \substack{u=\tau_1-R \\v\ge \tau_1+R } } r^q |\alpha|^2 \,\sin\theta \,d\theta d\phi dv+ \int_{\tau_1}^{\tau_2} r^q ( |\alpha|^2 - (\rho^2+\sigma^2) ) \bigg|_{r=R} \!\!\sin\theta \,d\theta d\phi d\tau \\
        & \qquad + \int_{\mathcal D^{\tau_2}_{\tau_1} } r^q \alpha J_{L}\,\sin\theta \,d\theta d\phi dudv .
    \end{aligned}
\end{align}
We also have the Morawetz-type estimates on the interior domain $r\le R$:
\begin{align}
    \begin{aligned}
        \int_{\tau}^{\infty} \int_{\Sigma_\tau \cap \{r\le R\} } |\alpha|^2+|\underline\alpha|^2+\rho^2+\sigma^2\,dV \le C_R \int_{\Sigma_\tau} \mathcal J^T[ F ] d\sigma.
    \end{aligned}
\end{align}
From now on we restrict ourselves into $0\le p\le 2$ and $2\le q\le4$.
We apply \eqref{p-we-maxwell} with $p=2$ and \eqref{q-we-maxwell} with $q=3+\delta$ respectively. Then we have
\begin{align}
    \begin{aligned}
        & \int_{ \substack{u=\tau_2-R \\v\ge \tau_2+R } }  (\rho^2+\sigma^2)\,d\sigma+ \int_{\mathcal I^{\tau_2-R}_{\tau_1-R}} |\underline\alpha|^2\,d\sigma  \\
   & \le \int_{ \substack{u=\tau_1-R \\v\ge \tau_1+R } }  (\rho^2+\sigma^2)\,d\sigma+ C \int_{\Sigma_{\tau_1}} \mathcal J^T[F] \,d\sigma + \int_{\mathcal D^{\tau_2}_{\tau_1} } r^{-1} (\rho^2+\sigma^2) \,dV \\
   & \qquad +\left| \int_{\mathcal D^{\tau_2}_{\tau_1} }  \underline \alpha  J_{\underline L} \,dV \right|,
    \end{aligned}
\end{align}
and
\begin{align}
    \begin{aligned}
        & \int_{ \substack{u=\tau_2-R \\v\ge \tau_2+R } } r^{1+\delta} |\alpha|^2 \,d\sigma + \int_{\mathcal D^{\tau_2}_{\tau_1} } r^{\delta} (|\alpha|^2 +\rho^2+\sigma^2) \,dV \\
        & \le  \int_{ \substack{u=\tau_1-R \\v\ge \tau_1+R } } r^{1+\delta} |\alpha|^2 \,d\sigma + C \int_{\Sigma_{\tau_1}} \mathcal J^T [F] \,d\sigma + \left| \int_{\mathcal D^{\tau_2}_{\tau_1} } r^{1+\delta} \alpha \cdot J_L \,dV \right|
    \end{aligned}
\end{align}
In general, by commuting vector fields, we have the inequalities: for $0\le k\le N$,
\begin{align}\label{p-we-ab-Z}
    \begin{aligned}
        & \int_{ \substack{u=\tau_2-R \\v\ge \tau_2+R } }  |\mathcal L_Z^{\le k}(\rho,\sigma)|^2 \,d\sigma + \int_{\mathcal I^{\tau_2-R}_{\tau_1-R}} |\mathcal L_Z^{\le k}\underline\alpha|^2\,d\sigma \\
   & \le \int_{ \substack{u=\tau_1-R \\v\ge \tau_1+R } }  |\mathcal L_Z^{\le k}(\rho,\sigma)|^2\,d\sigma+ C \int_{\Sigma_{\tau_1}} \mathcal J^T[\mathcal L_Z^{\le k}F] \,d\sigma + \int_{\mathcal D^{\tau_2}_{\tau_1} } r^{-1} |\mathcal L_Z^{\le k}(\rho,\sigma)|^2 \,dV \\
   & \qquad +\left| \int_{\mathcal D^{\tau_2}_{\tau_1} } \mathcal L_Z^{\le k} \underline \alpha\cdot  \mathcal L_Z^{\le k}J_{\underline L} \,dV \right|,
    \end{aligned}
\end{align}
and
\begin{align}\label{q-we-a-Z}
    \begin{aligned}
        & \int_{ \substack{u=\tau_2-R \\v\ge \tau_2+R } } r^{1+\delta} |\mathcal L_Z^{\le k}\alpha|^2 \,d\sigma + \int_{\mathcal D^{\tau_2}_{\tau_1} } r^{\delta} (|\mathcal L_Z^{\le k}\alpha|^2 +|\mathcal L_Z^{\le k}(\rho,\sigma)|^2) \,dV \\
        & \le  \int_{ \substack{u=\tau_1-R \\v\ge \tau_1+R } } r^{1+\delta} |\mathcal L_Z^{\le k}\alpha|^2 \,d\sigma + C \int_{\Sigma_{\tau_1}} \mathcal J^T [\mathcal L_Z^{\le k}F] \,d\sigma + \left| \int_{\mathcal D^{\tau_2}_{\tau_1} } r^{1+\delta}\mathcal L_Z^{\le k} \alpha \cdot \mathcal L_Z^{\le k}J_L \,dV \right|,
    \end{aligned}
\end{align}
where we omit the spacetime integral of the commutator terms which appears on the right-hand side, since it can be absorbed in an obvious way.

Therefore, the first task to close the $r^p$-weighted energy hierarchy is to control the spacetime integrals due to the nonlinearities:
\begin{align}
    \left| \int_{\mathcal D^{\tau_2}_{\tau_1} } \mathcal L_Z^{\le k} \underline \alpha\cdot  \mathcal L_Z^{\le k}J_{\underline L} \,dV \right| , \quad \left| \int_{\mathcal D^{\tau_2}_{\tau_1} } r^{1+\delta}\mathcal L_Z^{\le k} \alpha \cdot \mathcal L_Z^{\le k}J_L \,dV \right|.
\end{align}
The additional bulk term on the right-hand side of the inequality \eqref{p-we-ab-Z} is not problematic, in that it can be controlled via a bulk term with a stronger weight. More precisely, we see that
\begin{align*}
    \int_{\mathcal D^{\tau_2}_{\tau_1} } r^{-1} |\mathcal L_Z^{\le k}(\rho,\sigma)|^2 \,dV  \lesssim \langle R\rangle^{-\delta} \int_{\mathcal D^{\tau_2}_{\tau_1} } r^{-1+\delta} |\mathcal L_Z^{\le k}(\rho,\sigma)|^2 \,dV ,
\end{align*}
and the integral of the right-hand side of the above integral can be controlled via the $r^q$-weighted energy with $q=2+\delta$.

In fact, after the close of the $r^p$ and $r^q$-weighted inequality with $p=2$ and $q=3+\delta$, we apply the inequalities with $p=1+\delta$ and $q=2+\delta$ to obtain the decay estimates. To be precise, we see that
\begin{align}\label{p-we-ab-Z1}
    \begin{aligned}
        & \int_{ \substack{u=\tau_2-R \\v\ge \tau_2+R } } r^{-1+\delta} |\mathcal L_Z^{\le k}(\rho,\sigma)|^2 \,d\sigma + \int_{\mathcal I^{\tau_2-R}_{\tau_1-R}}r^{-1+\delta} |\mathcal L_Z^{\le k}\underline\alpha|^2\,d\sigma + \int_{\mathcal D^{\tau_2}_{\tau_1}} r^{-2+\delta}|\mathcal L_Z^{\le k}\underline\alpha|^2\,dV \\
   & \le \int_{ \substack{u=\tau_1-R \\v\ge \tau_1+R } } r^{-1+\delta} |\mathcal L_Z^{\le k}(\rho,\sigma)|^2\,d\sigma+ C \int_{\Sigma_{\tau_1}} \mathcal J^T[\mathcal L_Z^{\le k}F] \,d\sigma + \int_{\mathcal D^{\tau_2}_{\tau_1} } r^{-2+\delta} |\mathcal L_Z^{\le k}(\rho,\sigma)|^2 \,dV \\
   & \qquad +\left| \int_{\mathcal D^{\tau_2}_{\tau_1} } r^{-1+\delta} \mathcal L_Z^{\le k} \underline \alpha\cdot  \mathcal L_Z^{\le k}J_{\underline L} \,dV \right|,
    \end{aligned}
\end{align}
and
\begin{align}\label{q-we-a-Z1}
    \begin{aligned}
        & \int_{ \substack{u=\tau_2-R \\v\ge \tau_2+R } } r^{\delta} |\mathcal L_Z^{\le k}\alpha|^2 \,d\sigma + \int_{\mathcal D^{\tau_2}_{\tau_1} } r^{-1+\delta}  |\mathcal L_Z^{\le k}(\rho,\sigma)|^2 \,dV \\
        & \le  \int_{ \substack{u=\tau_1-R \\v\ge \tau_1+R } } r^{\delta} |\mathcal L_Z^{\le k}\alpha|^2 \,d\sigma + C \int_{\Sigma_{\tau_1}} \mathcal J^T [\mathcal L_Z^{\le k}F] \,d\sigma + \left| \int_{\mathcal D^{\tau_2}_{\tau_1} } r^{\delta}\mathcal L_Z^{\le k} \alpha \cdot \mathcal L_Z^{\le k}J_L \,dV \right|.
    \end{aligned}
\end{align}
A standard approach of using the dyadic argument and the pigeonhole principle yields the decay estimates for the energy of the null components. We present the details in Section \ref{subsec:imp-null-h}.

\subsection{$r^p$ method for the extreme components $\alpha,\underline\alpha$ as wave}
To compensate for the lack of decay of $\alpha$ and $\underline\alpha$, we consider the wave equation of $\alpha,\underline\alpha$.
We recall the tensor equations:
\begin{align}
    \nabla^\mu F_{\mu\nu} = J_\nu.
\end{align}
Then we apply the derivative $\nabla$ once again and get the wave equations of the form:
\begin{align}
    \Box_g F_{\mu\nu} = -{R_{\mu\nu}}^{\!\rho\sigma}F_{\rho\sigma} - ( \nabla_\mu J_\nu - \nabla_\nu J_\mu).
\end{align}
Then we put $(\mu,\nu)=(A,L)$ and obtain the wave equation for $\alpha_A$:
\begin{align}\label{eq-a-wave}
    \Box_g \alpha_A & = -{R_{AL}}^{\! \rho\sigma}F_{\rho\sigma} - (\nabla_{e_A}J_L -\nabla_L J_{e_A}).
\end{align}
From now on we put $\mathfrak a_A=r\alpha_A$. Then we have
\begin{align}
    \Box_g \alpha_A = -\nabla_{\underline L}\nabla_L\mathfrak a_A +\slashed\Delta\mathfrak a_A + O(\frac1{r^2}) \nabla \mathfrak a + O(\frac1{r^3})\mathfrak a_A.
\end{align}
Now we repeat the previous argument and get the identity:
\begin{align}
    \begin{aligned}
      &  \int_{\substack{ u=\tau_2-R \\ v\ge \tau_2+R } } r^p|\nabla_L \mathfrak a|^2\,\sin\theta\,d\theta d\phi dv + \int_{\mathcal D^{\tau_2}_{\tau_1} } r^{p-1} ( p |\nabla_L \mathfrak a|^2+ (2-p) |\slashed\nabla\mathfrak a|^2) \,\sin\theta\,d\theta d\phi dv du \\
      & \qquad + \int_{\mathcal I^{\tau_2-R}_{\tau_1-R} } r^p |\slashed\nabla \mathfrak a|^2\,\sin\theta\,d\theta d\phi du \\
      & = \int_{\substack{u=\tau_1-R \\ v\ge \tau_1+R } } r^p |\nabla_L\mathfrak a|^2\,\sin\theta\,d\theta d\phi dv + \int_{\tau_1}^{\tau_2} r^p ( |\slashed\nabla\mathfrak a|^2- |\nabla_L\mathfrak a|^2)\bigg|_{r=R}\,\sin\theta\,d\theta d\phi d\tau \\
      & \qquad + \int_{\mathcal D^{\tau_2}_{\tau_1} } r^p ( O(r^{-2}) \nabla \mathfrak a + O(r^{-3})\mathfrak a) \cdot \nabla_L \mathfrak a\,\sin\theta\,d\theta d\phi dvdu + \int_{\mathcal D^{\tau_2}_{\tau_1} } r^p {R_{AL} }^{\! \rho\sigma}F_{\rho\sigma}\nabla_L \mathfrak a \,\sin\theta\,d\theta d\phi dvdu \\
      & \qquad\qquad + \int_{\mathcal D^{\tau_2}_{\tau_1} } r^{p+1} (\nabla_{e_A}J_L-\nabla_L J_{e_A}) \nabla_L \mathfrak a \,\sin\theta\,d\theta d\phi dvdu.
    \end{aligned}
\end{align}
In a similar way, we obtain the wave equation for $\underline\alpha_A$:
\begin{align}\label{eq-ab-wave}
    \Box_g \underline\alpha_A = - {R_{A\underline L} }^{\rho\sigma} F_{\rho\sigma}- ( \nabla_{e_A}J_{\underline L}- \nabla_{\underline L} J_{e_A}).
\end{align}
We put $\underline{\mathfrak a}_A=r\underline\alpha_A$ and obtain the identity
\begin{align}
    \begin{aligned}
      &  \int_{\substack{ u=\tau_2-R \\ v\ge \tau_2+R } } r^p|\nabla_L \underline{\mathfrak a}|^2\,\sin\theta\,d\theta d\phi dv + \int_{\mathcal D^{\tau_2}_{\tau_1} } r^{p-1} ( p |\nabla_L \underline{\mathfrak a}|^2+ (2-p) |\slashed\nabla\underline{\mathfrak a}|^2) \,\sin\theta\,d\theta d\phi dv du \\
      & \qquad + \int_{\mathcal I^{\tau_2-R}_{\tau_1-R} } r^p |\slashed\nabla \underline{\mathfrak a}|^2\,\sin\theta\,d\theta d\phi du \\
      & = \int_{\substack{u=\tau_1-R \\ v\ge \tau_1+R } } r^p |\nabla_L\underline{\mathfrak a}|^2\,\sin\theta\,d\theta d\phi dv + \int_{\tau_1}^{\tau_2} r^p ( |\slashed\nabla\underline{\mathfrak a}|^2- |\nabla_L\underline{\mathfrak a}|^2)\bigg|_{r=R}\,\sin\theta\,d\theta d\phi d\tau \\
      & \qquad + \int_{\mathcal D^{\tau_2}_{\tau_1} } r^p ( O(r^{-2}) \nabla \underline{\mathfrak a} + O(r^{-3})\underline{\mathfrak a}) \cdot \nabla_L \underline{\mathfrak a}\,\sin\theta\,d\theta d\phi dvdu + \int_{\mathcal D^{\tau_2}_{\tau_1} } r^p {R_{AL} }^{\! \rho\sigma}F_{\rho\sigma}\nabla_L \underline{\mathfrak a} \,\sin\theta\,d\theta d\phi dvdu \\
      & \qquad\qquad + \int_{\mathcal D^{\tau_2}_{\tau_1} } r^{p+1} (\nabla_{e_A}J_{\underline L}-\nabla_{\underline L} J_{e_A}) \nabla_L \underline{\mathfrak a} \,\sin\theta\,d\theta d\phi dvdu.
    \end{aligned}
\end{align}
By commuting the vector fields for the wave equations for the extreme components $\alpha$ and $\underline\alpha$ \eqref{eq-a-wave} and \eqref{eq-ab-wave}, we obtain the following weighted energy inequalities: for $0\le k\le N-2$,
\begin{align}
    \begin{aligned}
           &  \int_{\substack{ u=\tau_2-R \\ v\ge \tau_2+R } } r^p|\nabla_L \mathcal L_Z^{\le k}\alpha |^2\,d\sigma + \int_{\mathcal D^{\tau_2}_{\tau_1} } r^{p-1} ( p |\nabla_L \mathcal L_Z^{\le k}\alpha|^2+ (2-p) |\slashed\nabla\mathcal L_Z^{\le k}\alpha|^2) \,dV \\ 
      & \lesssim \int_{\substack{u=\tau_1-R \\ v\ge \tau_1+R } } r^p |\nabla_L\mathcal L_Z^{\le k}\alpha|^2\,d\sigma + {}_{\rm int}\mathcal E^T_{\le k}[\alpha]^2(\tau_1)   
    +\left| \int_{\mathcal D^{\tau_2}_{\tau_1} } r^{p} (\slashed\nabla_{}\mathcal L_Z^{\le k}J_L-\nabla_L \mathcal L_Z^{\le k}\slashed J_{}) \nabla_L \mathcal L_Z^{\le k}\alpha \,dV \right|,
    \end{aligned}
\end{align}
and
\begin{align}
    \begin{aligned}
           &  \int_{\substack{ u=\tau_2-R \\ v\ge \tau_2+R } } r^p|\nabla_L \mathcal L_Z^{\le k}\underline\alpha |^2\,d\sigma + \int_{\mathcal D^{\tau_2}_{\tau_1} } r^{p-1} ( p |\nabla_L \mathcal L_Z^{\le k}\underline\alpha|^2+ (2-p) |\slashed\nabla\mathcal L_Z^{\le k}\underline\alpha|^2) \,dV \\ 
      & \lesssim \int_{\substack{u=\tau_1-R \\ v\ge \tau_1+R } } r^p |\nabla_L\mathcal L_Z^{\le k}\underline\alpha|^2\,d\sigma + {}_{\rm int}\mathcal E^T_{\le k}[\underline\alpha]^2(\tau_1)   
    +\left| \int_{\mathcal D^{\tau_2}_{\tau_1} } r^{p} (\slashed\nabla_{}\mathcal L_Z^{\le k}J_{\underline L}-\nabla_{\underline L} \mathcal L_Z^{\le k}\slashed J_{}) \nabla_L \mathcal L_Z^{\le k}\underline\alpha \,dV \right|.
    \end{aligned}
\end{align}
We omit the spacetime integrals which appear due to the commutator terms and the lower-order error terms from the wave equations. These can be, in fact, controlled by a standard argument using the bootstrap assumptions together with the integrated local energy decay estimates. Furthremore, these estimates are identical to those appearing in the tensorial wave analysis of \cite{holzegel10,pasqualotto}.

Therefore, we concentrate ourselves into the spacetime integrals of the nonlinearities due to the Dirac current:
\begin{align}
    \begin{aligned}
        \left| \int_{\mathcal D^{\tau_2}_{\tau_1} } r^{2-\delta} (\slashed\nabla_{}\mathcal L_Z^{\le k}J_L-\nabla_L \mathcal L_Z^{\le k}\slashed J_{}) \nabla_L \mathcal L_Z^{\le k}\alpha \,dV \right|, \quad \left| \int_{\mathcal D^{\tau_2}_{\tau_1} } r^{2-\delta} (\slashed\nabla_{}\mathcal L_Z^{\le k}J_{\underline L}-\nabla_{\underline L} \mathcal L_Z^{\le k}\slashed J_{}) \nabla_L \mathcal L_Z^{\le k}\underline\alpha \,dV \right|.
    \end{aligned}
\end{align}
We are concerned with the above integrals in Section \ref{subsec:bdd-weight-aab}.
\begin{rem}
    Throughout this paper, we do not use the top-order weighted wave energy ${}^{(2-\delta)}\mathcal E^T_{\le N-1}[\alpha]$ and ${}^{(2-\delta)}\mathcal E^T_{\le N-1}[\underline\alpha]$  for $\alpha$ and $\underline\alpha$. In other words, we do not need to control the above spacetime integrals for $k=N-1$.
\end{rem}
\subsection{Auxiliary estimates}
Before proving Theorem \ref{main-thm-formal}, we record some useful identities and estimates.

\begin{prop}[Null decomposition for the spinor fields]\label{decomp-dirac-null-tensor}
The Dirac part satisfies the following decomposition:
\begin{align}
    \begin{aligned}
        \nabla_L\psi_+ & = -\frac12 \gamma^T \gamma^{e_A}\nabla_{e_A}\psi_-+\frac1{2r}\psi + \alpha_A \gamma^{e_A}\psi_+ + (\rho,\sigma)\psi_\pm , \\
                \nabla_{\underline L}\psi_- & = -\frac12 \gamma^T \gamma^{e_A}\nabla_{e_A}\psi_+-\frac1{2r}\psi + \underline\alpha_A \gamma^{e_A}\psi_- + (\rho,\sigma)\psi_\pm.
    \end{aligned}
\end{align}
\end{prop}
\begin{proof}
    The proof is obvious after a direct application of Proposition \ref{prop-null-decomp-dirac}.
\end{proof}

\begin{prop}
We have the following decomposition for the spinor fields:
\begin{align}\label{decomp-LLbar}
\begin{aligned}
\nabla_L\nabla_{\underline L}(\mathscr L_Z^I\psi) = \slashed\Delta \mathscr L_Z^I\psi + \gamma^\lambda\nabla_\lambda ( [ \gamma^\mu\nabla_\mu, \mathscr L_Z^I]\psi) + \mathcal R \mathscr L_Z^I\psi + \gamma^\lambda\nabla_\lambda \mathscr L_Z^I ( F_{\rho\sigma}\gamma^\rho\gamma^\sigma \psi).
\end{aligned}
\end{align}
\end{prop}
\begin{proof}
    This is merelt a reformulation of the squared Dirac equation.
\end{proof}

\begin{prop}
We have
\begin{align}\label{decomp-alpha-wave}
\begin{aligned}
	\nabla_{L}\nabla_{\underline L}\mathcal L_Z^I(r\alpha_A) &=  \slashed\Delta(r\alpha_A) + r\mathcal R \mathcal L_Z^I(\alpha,\underline\alpha,\rho,\sigma) + \nabla_L\mathcal L_Z^I\alpha+ \nabla_{\underline L}\mathcal L_Z^I\alpha + r [ \Box_g, \mathcal L_Z^I]\alpha_A   \\ & \qquad  + r ( \nabla_L J_{e_A}-\nabla_{e_A}J_L), \\
    \nabla_{L}\nabla_{\underline L}\mathcal L_Z^I(r\underline\alpha_A) &=  \slashed\Delta(r\underline\alpha_A) + r\mathcal R \mathcal L_Z^I(\alpha,\underline\alpha,\rho,\sigma) + \nabla_L\mathcal L_Z^I\underline\alpha+ \nabla_{\underline L}\mathcal L_Z^I\underline\alpha + r [ \Box_g, \mathcal L_Z^I]\underline\alpha_A   \\ & \qquad  + r ( \nabla_{\underline L} J_{e_A}-\nabla_{e_A}J_{\underline L}),
    \end{aligned}
\end{align}
where $\mathcal R$ is an abbreviation of the curvature tensor.
\end{prop}

\section{Boundedness of the weighted energy for the spinor fields}\label{sec:bdd-weight-spinor}
This section is devoted to the proof of the boundedness of the weighted energy ${}^{(2-\delta)}\mathcal E^D_{\le N-1}[\psi]$. 
We recall the coupled system:
\begin{align}
    \begin{aligned}
        g^{\mu\nu}\nabla_\mu \nabla_\nu \psi = \mathcal R\psi + i\gamma^\lambda\nabla_\lambda ( F_{\mu\nu}\gamma^\mu\gamma^\nu\psi) ,
    \end{aligned}
\end{align}
where we present the squared form, since we are here concerned with the energy in terms of the wave. The goal of this section is to prove the following:
\begin{prop}
Let $N\ge11$ and $0<\delta<\frac1{20}$. Then we have
\begin{align}
    \int_{\mathcal D^{\tau_{i+1}}_{\tau_i} } r^{2-\delta} \langle \gamma^\lambda\nabla_\lambda \mathscr L_Z^{\le N-1}(F_{\mu\nu}\gamma^\mu\gamma^\nu\psi) , \gamma^T \nabla_L\mathscr L_Z^{\le N-1}\psi \rangle \,dV \lesssim C^3\varepsilon^3 (\tau_i)^\epsilon .
\end{align}
where $\epsilon>0$ is small so that $\epsilon\ll\delta$.
\end{prop}
From the above estimate, one can infer that the bootstrap assumptions could be improved and by iterating, a logarithmic growth of the higher-order weighted energy turns out to be bounded.

\subsection{Control of $\alpha\cdot\psi$}
If $F=\alpha$, we need to consider the following integral:
\begin{align}
    \sum_{ |I_1|+|I_2| \le |I| }\int_{\mathcal D^{\tau_{i+1}}_{\tau_i} } r^{2-\delta} \langle  \mathcal L_Z^{I_1} \alpha_A \gamma^{\underline L}\gamma^{e_A}\gamma^{L} \nabla_{\underline L}\mathscr L_Z^{I_2} \psi , \gamma^T \nabla_L \mathscr L_Z^{I} \psi \rangle \,dV.
\end{align}
For the high-low interaction, i.e., $|I_1|>|I_2|$, the Sobolev inequality along the null hypersurface $\Sigma_\tau$ allows us to use the decomposition \eqref{decomp-LLbar}. On the other hand, when $|I_1|\le |I_2|$, in which case we cannot apply the Sobolev inequality to the spinor, we use the integration by parts in the $\underline L$-direction, which gives the following integrals:
\begin{align}
    \begin{aligned}
        & \int_{\mathcal D^{\tau_{i+1}}_{\tau_i} } r^{2-\delta} \langle  \mathcal L_Z^{I_1} \alpha_A \gamma^{\underline L}\gamma^{e_A}\gamma^{L} \nabla_{\underline L}\mathscr L_Z^{I_2} \psi , \gamma^T \nabla_L \mathscr L_Z^{I} \psi \rangle \,dV \\
        & = \int_{\substack{ u= \tau_{i+1}-R \\ v \ge \tau_{i+1}+R  } } r^{2-\delta} \langle  \mathcal L_Z^{I_1} \alpha_A \gamma^{\underline L}\gamma^{e_A}\gamma^{L} \mathscr L_Z^{I_2} \psi , \gamma^T \nabla_L \mathscr L_Z^{I} \psi \rangle \,d\sigma + \int_{\substack{ u= \tau_{i}-R \\ v \ge \tau_{i}+R  } } r^{2-\delta} \langle  \mathcal L_Z^{I_1} \alpha_A \gamma^{\underline L}\gamma^{e_A}\gamma^{L} \mathscr L_Z^{I_2} \psi , \gamma^T \nabla_L \mathscr L_Z^{I} \psi \rangle \,d\sigma \\
        & \qquad + \int_{r=R } r^{2-\delta} \langle  \mathcal L_Z^{I_1} \alpha_A \gamma^{\underline L}\gamma^{e_A}\gamma^{L} \mathscr L_Z^{I_2} \psi , \gamma^T \nabla_L \mathscr L_Z^{I} \psi \rangle \,d\sigma \\
        & \qquad - \int_{\mathcal D^{\tau_{i+1}}_{\tau_i} } r^{2-\delta} \langle  \nabla_{\underline L}\mathcal L_Z^{I_1} \alpha_A \gamma^{\underline L}\gamma^{e_A}\gamma^{L} \mathscr L_Z^{I_2} \psi , \gamma^T \nabla_L \mathscr L_Z^{I} \psi \rangle \,dV \\
        & \qquad - \int_{\mathcal D^{\tau_{i+1}}_{\tau_i} } r^{2-\delta} \langle  \mathcal L_Z^{I_1} \alpha_A \gamma^{\underline L}\gamma^{e_A}\gamma^{L} \mathscr L_Z^{I_2} \psi , \gamma^T \nabla_{\underline L}\nabla_L \mathscr L_Z^{I} \psi \rangle \,dV
    \end{aligned}
\end{align}
\subsubsection{Control of the surface integral along the null hypersurface:}
For the control of the first and the second boundary terms, it is enough to consider only the first integral, since $\tau_i$ and $\tau_{i+1}$ are dyadic numbers and $\tau_{i+1}=2\tau_i$. We recall that $|I_1| \le |I_2|$. Then an application of the weighted Sobolev inequality to $\alpha$ gives
\begin{align*}
   & \left| \int_{\substack{ u= \tau_{i+1}-R \\ v \ge \tau_{i+1}+R  } } r^{2-\delta} \langle  \mathcal L_Z^{I_1} \alpha_A \gamma^{\underline L}\gamma^{e_A}\gamma^{L} \mathscr L_Z^{I_2} \psi , \gamma^T \nabla_L \mathscr L_Z^{I} \psi \rangle \,d\sigma \right|  \\
   & \lesssim \int_{\substack{ u= \tau_{i+1}-R \\ v \ge \tau_{i+1}+R  } } r^{2-\delta} | \mathcal L_Z^{\le N-8} \alpha| | \mathscr L_Z^{\le N-1} \psi_+| | \nabla_L \mathscr L_Z^{\le N-1} \psi_+ |\,d\sigma \\
& \lesssim  {}^{(1)}\mathcal E^T_{\le N-6}[\alpha](\tau_{i+1}) \mathscr E^D_{\le N-1}[\psi](\tau_{i+1})  \, {}^{(2-\delta)}\mathcal E^D_{\le N-1}[\psi](\tau_{i+1}) \\
& \lesssim C^3\varepsilon^3 (\tau_i)^{-\frac14+\frac52\delta+\epsilon}.
\end{align*}
This give the required bound for the surface integral along the null hypersurfaces.

\subsubsection{Control of the surface integral along the time-like surface: }
For the integral along the time-like surface $\{r=R\}$, we note that this is controlled by the interior domain. We consider the spacetime integral on the interior:
\begin{align}
   \int_{\tau_i}^{\tau_{i+1} }  \int_{ R-1\le r\le R } r^{2-\delta} \langle  \mathcal L_Z^{I_1} \alpha_A \gamma^{\underline L}\gamma^{e_A}\gamma^{L} \mathscr L_Z^{I_2} \psi , \gamma^T \nabla_L \mathscr L_Z^{I} \psi \rangle \,d\sigma d\tau .
\end{align}
Indeed, if the above spacetime integral is bounded by $C(\tau)$, then we would have
\begin{align}
    \left| \int_{r=R } r^{2-\delta} \langle  \mathcal L_Z^{I_1} \alpha_A \gamma^{\underline L}\gamma^{e_A}\gamma^{L} \mathscr L_Z^{I_2} \psi , \gamma^T \nabla_L \mathscr L_Z^{I} \psi \rangle \,d\sigma \right| \lesssim C(\tau) (\tau_*)^{-1},
\end{align}
for some $ \tau_* $ with $\tau_i\le \tau_*\le \tau_{i+1} $ by an averaging argument. To show this, we recall that the interior domain is foliated by the time-slice. We also note that $|\mathscr L_{\Omega}\psi|\le |\nabla_\Omega\psi|+|\psi| \le R |\slashed\nabla\psi|+|\psi|$ on the interior regime $\{r \le R\}$. Then we write
\begin{align*}
   & \left| \int_{\tau_i}^{\tau_{i+1} }  \int_{ R-1\le r\le R } r^{2-\delta} \langle  \mathcal L_Z^{I_1} \alpha_A \gamma^{\underline L}\gamma^{e_A}\gamma^{L} \mathscr L_Z^{I_2} \psi , \gamma^T \nabla_L \mathscr L_Z^{I} \psi \rangle \,d\sigma d\tau  \right|  \\
   & \lesssim  C_R C\varepsilon  \int_{\tau_i}^{\tau_{i+1} } {}_{\rm int}\mathcal E^T_{\le N-6}[\alpha] {}_{\rm int}\mathcal E^D_{\le N-1}[\psi]^2    d\tau  \\
   & \lesssim_R  C^3\varepsilon^3  \int_{\tau_i}^{\tau_{i+1} }  \langle \tau_{}\rangle^{-\frac12}   \, d\tau  \\
   & \lesssim_R  C^3\varepsilon^3 \langle \tau_{i}\rangle^{\frac12},
\end{align*}
which implies that
\begin{align*}
     \left| \int_{r=R } r^{2-\delta} \langle  \mathcal L_Z^{I_1} \alpha_A \gamma^{\underline L}\gamma^{e_A}\gamma^{L} \mathscr L_Z^{I_2} \psi , \gamma^T \nabla_L \mathscr L_Z^{I} \psi \rangle \,d\sigma \right| \lesssim C^3\varepsilon^3 (\tau_i)^{-\frac12},
\end{align*}
and hence we obtain
the desired control for the surface integrals.

\subsubsection{Control of the spacetime integral without derivative loss}
From now on, we are concerned with the integral:
\begin{align}
    \int_{\mathcal D^{\tau_{i+1}}_{\tau_i} } r^{2-
    \delta} \langle  \nabla_{\underline L}\mathcal L_Z^{I_1} \alpha_A \gamma^{\underline L}\gamma^{e_A}\gamma^{L} \mathscr L_Z^{I_2} \psi , \gamma^T \nabla_L \mathscr L_Z^{I} \psi \rangle \,dV,
\end{align}
where the derivative $\nabla_{\underline L}$ acts on the tensor field $\alpha$. We recall that $|I_1|\le |I_2|$, in which case we can apply the Sobolev inequality to $\mathcal L_Z^{I_1}\alpha$. By the decomposition \eqref{decomp-alpha-wave}, we have
\begin{align}\label{int-apsi-wo-dloss}
\begin{aligned}
    & \int_{\tau_i}^{\tau_{i+1}}r^{\frac32-\delta} \|\nabla_L \nabla_{\underline L}\mathcal L_Z^{\le N-6} \alpha \|_{L^2(\Sigma_\tau)} \| \mathscr L_Z^{\le N-1} \psi_+\|_{L^2(\Sigma_\tau)} \| \nabla_L \mathscr L_Z^{\le N-1} \psi_+\|_{L^2(\Sigma_\tau)}\,d\tau \\
    & \lesssim \int_{\tau_i}^{\tau_{i+1}}r^{\frac12-\delta} \|\nabla_L \nabla_{\underline L}\mathcal L_Z^{\le N-6}(r \alpha) \|_{L^2(\Sigma_\tau)} \| \mathscr L_Z^{\le N-1} \psi_+\|_{L^2(\Sigma_\tau)} \| \nabla_L \mathscr L_Z^{\le N-1} \psi_+\|_{L^2(\Sigma_\tau)}\,d\tau  \\
    & \qquad + \int_{\tau_i}^{\tau_{i+1}}r^{\frac12-\delta} \|\nabla_L \mathcal L_Z^{\le N-6} \alpha \|_{L^2(\Sigma_\tau)} \| \mathscr L_Z^{\le N-1} \psi_+\|_{L^2(\Sigma_\tau)} \| \nabla_L \mathscr L_Z^{\le N-1} \psi_+\|_{L^2(\Sigma_\tau)}\,d\tau \\
    & \qquad + \int_{\tau_i}^{\tau_{i+1}}r^{\frac12-\delta} \|\nabla_{\underline L} \mathcal L_Z^{\le N-6} \alpha \|_{L^2(\Sigma_\tau)} \| \mathscr L_Z^{\le N-1} \psi_+\|_{L^2(\Sigma_\tau)} \| \nabla_L \mathscr L_Z^{\le N-1} \psi_+\|_{L^2(\Sigma_\tau)}\,d\tau  .
    \end{aligned}
\end{align}
Then we see that
\begin{align*}
    & \int_{\tau_i}^{\tau_{i+1}}r^{\frac12-\delta} \|\nabla_{\underline L} \mathcal L_Z^{\le N-6} \alpha \|_{L^2(\Sigma_\tau)} \| \mathscr L_Z^{\le N-1} \psi_+\|_{L^2(\Sigma_\tau)} \| \nabla_L \mathscr L_Z^{\le N-1} \psi_+\|_{L^2(\Sigma_\tau)}\,d\tau  \\
    & \lesssim \int_{\tau_i}^{\tau_{i+1}}r^{\frac12-\delta} \|\nabla_{ L} \mathcal L_Z^{\le N-6} \underline\alpha \|_{L^2(\Sigma_\tau)} \| \mathscr L_Z^{\le N-1} \psi_+\|_{L^2(\Sigma_\tau)} \| \nabla_L \mathscr L_Z^{\le N-1} \psi_+\|_{L^2(\Sigma_\tau)}\,d\tau \\
    & \qquad + \int_{\tau_i}^{\tau_{i+1}}r^{-\frac12-\delta} \| \mathcal L_Z^{\le N-5} \rho \|_{L^2(\Sigma_\tau)} \| \mathscr L_Z^{\le N-1} \psi_+\|_{L^2(\Sigma_\tau)} \| \nabla_L \mathscr L_Z^{\le N-1} \psi_+\|_{L^2(\Sigma_\tau)}\,d\tau ,
\end{align*}
and hence we have
\begin{align*}
  &  \int_{\tau_i}^{\tau_{i+1}}r^{\frac12-\delta} \|\nabla_{ L} \mathcal L_Z^{\le N-6} \underline\alpha \|_{L^2(\Sigma_\tau)} \| \mathscr L_Z^{\le N-1} \psi_+\|_{L^2(\Sigma_\tau)} \| \nabla_L \mathscr L_Z^{\le N-1} \psi_+\|_{L^2(\Sigma_\tau)}\,d\tau \\
  & \lesssim \left( \int_{\tau_i}^{\tau_{i+1}} r^{-\delta}\, {}^{(0)} \mathcal E^T_{\le N-6}[\underline\alpha]^2 \mathscr E^D_{\le N-1}[\psi]^2\,d\tau \right)^\frac12 \left( \int_{\tau_i}^{\tau_{i+1}} {}^{(1-\delta)}\mathcal E^D_{\le N-1}[\psi]^2\,d\tau \right)^\frac12 \\
  & \lesssim C^2\varepsilon^2 (\tau_i)^{\frac\delta2} \mathscr E^D_{\le N-1}[\psi](\tau_i)  \, {}^{(1)}\mathcal E^T_{\le N-6}[\underline\alpha](\tau_i) \\
  & \lesssim C^3\varepsilon^3 \tau_i^{-\frac14+\frac52\delta+\epsilon},
\end{align*}
and by the integrated local energy decay estimates \eqref{iled-a} for $\rho$,
\begin{align*}
    & \int_{\tau_i}^{\tau_{i+1}}r^{-\frac12-\delta} \| \mathcal L_Z^{\le N-5} \rho \|_{L^2(\Sigma_\tau)} \mathscr E^D_{\le N-1}[\psi] \mathcal E^D_{\le N-1}[\psi] \,d\tau  \\
    & \lesssim \left( \int_{\tau_i}^{\tau_{i+1}}r^{-1+\delta} \| \mathcal L_Z^{\le N-5} \rho \|_{L^2(\Sigma_\tau)}^2\,d\tau  \right)^\frac12 \left(  \int_{\tau_i}^{\tau_{i+1}}r^{-3\delta}  \mathscr E^D_{\le N-1}[\psi]^2 \mathcal E^D_{\le N-1}[\psi]^2\,d\tau \right)^\frac12 \\
    & \lesssim C^2\varepsilon^2 \mathscr E^D_{\le N-1}[\psi](\tau_i)   \left(  \int_{\tau_i}^{\tau_{i+1}}r^{-3\delta}\,{}^{(0)}  \mathcal E^D_{\le N-1}[\psi]^2(\tau)\,d\tau \right)^\frac12 \\
    & \lesssim C^3\varepsilon^3 (\tau_{i+1})^{\epsilon} .
\end{align*}
The second integral of \eqref{int-apsi-wo-dloss} is controlled in an obvious way.
Thus we are left to control the first integral of the right-hand side of the inequality \eqref{int-apsi-wo-dloss}. Using the decomposition \eqref{decomp-alpha-wave}, we have
\begin{align}
    \begin{aligned}
        & \int_{\tau_i}^{\tau_{i+1}}r^{\frac12-\delta} \|\nabla_L \nabla_{\underline L}\mathcal L_Z^{\le N-6}(r \alpha) \|_{L^2(\Sigma_\tau)} \| \mathscr L_Z^{\le N-1} \psi_+\|_{L^2(\Sigma_\tau)} \| \nabla_L \mathscr L_Z^{\le N-1} \psi_+\|_{L^2(\Sigma_\tau)}\,d\tau\\
        & \lesssim  \int_{\tau_i}^{\tau_{i+1}}r^{-\frac12-\delta} \|\mathcal L_Z^{\le N-4} \alpha \|_{L^2(\Sigma_\tau)} \mathscr E^D_{\le N-1}[\psi] \mathcal E^D_{\le N-1}[\psi] \,d\tau \\
        & \qquad + \int_{\tau_i}^{\tau_{i+1}}r^{-\frac32-\delta} \| \mathcal L_Z^{\le N-6} (\underline\alpha,\rho,\sigma) \|_{L^2(\Sigma_\tau)} \mathscr E^D_{\le N-1}[\psi] \mathcal E^D_{\le N-1}[\psi]\,d\tau \\
        & \qquad + \int_{\tau_i}^{\tau_{i+1}}r^{\frac32-\delta} \|\slashed\nabla\mathcal L_Z^{\le N-6}  J_L\|_{L^2(\Sigma_\tau)} \mathscr E^D_{\le N-1}[\psi] \mathcal E^D_{\le N-1}[\psi]\,d\tau \\
        & \qquad + \int_{\tau_i}^{\tau_{i+1}}r^{\frac32-\delta} \|\nabla_L\mathcal L_Z^{\le N-6}  \slashed{J}\|_{L^2(\Sigma_\tau)} \mathscr E^D_{\le N-1}[\psi] \mathcal E^D_{\le N-1}[\psi]\,d\tau.
    \end{aligned}
\end{align}
Then we see that
\begin{align*}
    & \int_{\tau_i}^{\tau_{i+1}}r^{-\frac12-\delta} \|\mathcal L_Z^{\le N-4} \alpha \|_{L^2(\Sigma_\tau)} \mathscr E^D_{\le N-1}[\psi] \mathcal E^D_{\le N-1}[\psi] \,d\tau\\
    & \lesssim C\varepsilon (\tau_{i+1})^{\epsilon} \left( \int_{\tau_i}^{\tau_{i+1}} r^{-1-\delta}\|\mathcal L_Z^{\le N-4}\alpha\|_{L^2(\Sigma_\tau)}^2 \right)^\frac12 \left( \int_{\tau_i}^{\tau_{i+1}} r^{-\delta}\,{}^{(0)} \mathcal E^D_{\le N-1}[\psi]^2\,d\tau  \right)^\frac12 \\
    & \lesssim  C^3\varepsilon^3(\tau_{i+1})^{\epsilon} ,
\end{align*}
and
\begin{align*}
   &  \int_{\tau_i}^{\tau_{i+1}}r^{-\frac32-\delta} \| \mathcal L_Z^{\le N-6} (\underline\alpha,\rho,\sigma) \|_{L^2(\Sigma_\tau)} \mathscr E^D_{\le N-1}[\psi] \mathcal E^D_{\le N-1}[\psi]\,d\tau \\
   & \lesssim \left( \int_{\mathcal D^{\tau_{i+1}}_{\tau_i} }r^{-3-2\delta} | \mathcal L_Z^{\le N-6} (\underline\alpha,\rho,\sigma)|^2\,dV   \right)^\frac12 \left( \int_{\tau_i}^{\tau_{i+1}}   \mathscr E^D_{\le N-1}[\psi]^2\,{}^{(0)}  \mathcal E^D_{\le N-1}[\psi]^2 \,d\tau  \right)^\frac12 \\
   & \lesssim C^3\varepsilon^3 (\tau_{i+1})^{\epsilon}.
\end{align*}
Now we are left to treat the integrals containing additional Dirac current. The first integral is rather easier. We only deal with the second integral:
\begin{align*}
    & \int_{\tau_i}^{\tau_{i+1}}r^{\frac32-\delta} \|\nabla_L\mathcal L_Z^{\le N-6}  \slashed{J}\|_{L^2(\Sigma_\tau)} \mathscr E^D_{\le N-1}[\psi] \mathcal E^D_{\le N-1}[\psi]\,d\tau \\
    & \lesssim \int_{\tau_i}^{\tau_{i+1}}r^{1-\delta} \mathcal E^D_{\le N-6}[\psi] \mathcal E^D_{\le N-4}[\psi] \mathscr E^D_{\le N-1}[\psi] \mathcal E^D_{\le N-1}[\psi]\,d\tau \\
    & \lesssim C\varepsilon \tau_i^\epsilon \left( \int_{\tau_i}^{\tau_{i+1}} \mathcal E^D_{\le N-6}[\psi]^2\, {}^{(1-\delta)} \mathcal E^D_{\le N-1}[\psi]^2\,d\tau \right)^\frac12 \left( \int_{\tau_i}^{\tau_{i+1}} {}^{(1-\delta)} \mathcal E^D_{\le N-4}[\psi]^2\,d\tau  \right)^\frac12 \\
    & \lesssim C^4\varepsilon^4 (\tau_i)^{-\frac12+\frac32\delta+\epsilon}.
\end{align*}

\subsubsection{Control of the spacetime integral with derivative loss}

Finally we arrive at the integral:
\begin{align}
    \int_{\mathcal D^{\tau_{i+1}}_{\tau_i} } r^{2-\delta} \langle  \mathcal L_Z^{I_1} \alpha_A \gamma^{\underline L}\gamma^{e_A}\gamma^{L} \mathscr L_Z^{I_2} \psi , \gamma^T \nabla_{\underline L}\nabla_L \mathscr L_Z^{I} \psi \rangle \,dV.
\end{align}
Since $|I|=N-1$, we now encounter the derivative order $N+1$, which cannot be absorbed as the energy. However, the decomposition \eqref{decomp-LLbar} shows that this derivative loss becomes not harmful. Indeed, we write
\begin{align}
    \begin{aligned}
        & \int_{\mathcal D^{\tau_{i+1}}_{\tau_i} } r^{2-\delta} \langle  \mathcal L_Z^{I_1} \alpha_A \gamma^{\underline L}\gamma^{e_A}\gamma^{L} \mathscr L_Z^{I_2} \psi , \gamma^T \nabla_{\underline L}\nabla_L \mathscr L_Z^{I} \psi \rangle \,dV \\
        & = \int_{\mathcal D^{\tau_{i+1}}_{\tau_i} } r^{2-\delta} \langle  \mathcal L_Z^{I_1} \alpha_A \gamma^{\underline L}\gamma^{e_A}\gamma^{L} \mathscr L_Z^{I_2} \psi , \gamma^T \slashed\Delta \mathscr L_Z^{I} \psi \rangle \,dV \\
        & \qquad + \int_{\mathcal D^{\tau_{i+1}}_{\tau_i} } r^{2-\delta} \langle  \mathcal L_Z^{I_1} \alpha_A \gamma^{\underline L}\gamma^{e_A}\gamma^{L} \mathscr L_Z^{I_2} \psi , \gamma^T \gamma^\lambda\nabla_\lambda ( [ \gamma^\mu\nabla_\mu , \mathscr L_Z^{I} ] \psi ) \rangle \,dV \\
        & \qquad + \int_{\mathcal D^{\tau_{i+1}}_{\tau_i} } r^{2-\delta} \langle  \mathcal L_Z^{I_1} \alpha_A \gamma^{\underline L}\gamma^{e_A}\gamma^{L} \mathscr L_Z^{I_2} \psi , \gamma^T \mathcal R \mathscr L_Z^{I} \psi \rangle \,dV \\
        & \qquad + \int_{\mathcal D^{\tau_{i+1}}_{\tau_i} } r^{1+\delta} \langle  \mathcal L_Z^{I_1} \alpha_A \gamma^{\underline L}\gamma^{e_A}\gamma^{L} \mathscr L_Z^{I_2} \psi , \gamma^T \gamma^\lambda\nabla_\lambda \mathscr L_Z^I( h_{\rho\sigma}\gamma^\rho \gamma^\sigma \psi) \rangle \,dV.
    \end{aligned}
\end{align}
We first consider the integral containing $\slashed\Delta$. We use the integration by parts with respect to the angular variables, which gives us
\begin{align}
    \begin{aligned}
       & \left| \int_{\mathcal D^{\tau_{i+1}}_{\tau_i} } r^{2-\delta} \langle  \mathcal L_Z^{I_1} \alpha_A \gamma^{\underline L}\gamma^{e_A}\gamma^{L} \mathscr L_Z^{I_2} \psi , \gamma^T \slashed\Delta \mathscr L_Z^{I} \psi \rangle \,dV \right| \\
       & \lesssim  \int_{\mathcal D^{\tau_{i+1}}_{\tau_i} } r^{1-\delta}|  \mathcal L_Z^{\le N-7} \alpha | | \mathscr L_Z^{\le N-1} \psi_+|  | \slashed\nabla \mathscr L_Z^{\le N-1} \psi_+ | \,dV \\
      & \qquad + \int_{\mathcal D^{\tau_{i+1}}_{\tau_i} } r^{2-\delta}|  \mathcal L_Z^{\le N-8} \alpha | |\slashed\nabla \mathscr L_Z^{\le N-1} \psi_+|  | \slashed\nabla \mathscr L_Z^{\le N-1} \psi_+ | \,dV 
    \end{aligned}
\end{align}
Then
\begin{align*}
    & \int_{\mathcal D^{\tau_{i+1}}_{\tau_i} } r^{1-\delta}|  \mathcal L_Z^{\le N-7} \alpha | | \mathscr L_Z^{\le N-1} \psi_+|  | \slashed\nabla \mathscr L_Z^{\le N-1} \psi_+ | \,dV \\
    & \lesssim \int_{\tau_i}^{\tau_{i+1} } r^{\frac12-\delta} \mathcal E^T_{\le N-5}[\alpha] \mathscr E^D_{\le N-1}[\psi] \mathcal E^D_{\le N-1}[\psi] \,d\tau \\
    & \lesssim C\varepsilon (\tau_{i+1})^{\epsilon}\left(\int_{\tau_i}^{\tau_{i+1} } r^{-\delta}\,{}^{(0)} \mathcal E^T_{\le N-5}[\alpha]^2\,d\tau \right)^\frac12 \left( \int_{\tau_i}^{\tau_{i+1} } {}^{(1-\delta)}\mathcal E^D_{\le N-1}[\psi]^2 \,d\tau \right)^\frac12 \\
    & \lesssim C^2\varepsilon^2\, {}^{(1)}\mathcal E^T_{\le N-5}[\alpha](\tau_i) \tau_i^{\frac\delta2+\epsilon} \\
    & \lesssim C^3\varepsilon^3 (\tau_{i})^{-\frac14+\frac52\delta+\epsilon} ,
\end{align*}
and
\begin{align*}
	& \int_{\mathcal D^{\tau_{i+1}}_{\tau_i} } r^{2-\delta}|  \mathcal L_Z^{\le N-8} \alpha | |\slashed\nabla \mathscr L_Z^{\le N-1} \psi_+|  | \slashed\nabla \mathscr L_Z^{\le N-1} \psi_+ | \,dV  \\
	& \lesssim \int_{\tau_i}^{\tau_{i+1} } r^{\frac32-\delta} \mathcal E^T_{\le N-6}[\alpha] \mathcal E^D_{\le N-1}[\psi]^2\,d\tau \\
	& \lesssim \left( \int_{\tau_i}^{\tau_{i+1} } {}^{(1)}\mathcal E^T_{\le N-6}[\alpha]^2 \, {}^{(1-\delta)}\mathcal E^D_{\le N-1}[\psi]^2\,d\tau   \right)^\frac12 \left( \int_{\tau_i}^{\tau_{i+1} }{}^{(1-\delta)}\mathcal E^D_{\le N-1}[\psi]^2\,d\tau  \right)^\frac12 \\
	& \lesssim C^2\varepsilon^2\tau_i^{\delta}\, {}^{(1)}\mathcal E^T_{\le N-6}[\alpha](\tau_i)  \\
    & \lesssim C^3\varepsilon^3 (\tau_{i})^{-\frac14+3\delta} ,
\end{align*}
Now we consider the integral including the commutator terms:
\begin{align}
     \int_{\mathcal D^{\tau_{i+1}}_{\tau_i} } r^{2-\delta} \langle  \mathcal L_Z^{I_1} \alpha_A \gamma^{\underline L}\gamma^{e_A}\gamma^{L} \mathscr L_Z^{I_2} \psi , \gamma^T \gamma^\lambda\nabla_\lambda ( [ \gamma^\mu\nabla_\mu , \mathscr L_Z^{I} ] \psi ) \rangle \,dV.
\end{align}
Then we see that
\begin{align}
    \begin{aligned}
        & \left| \int_{\mathcal D^{\tau_{i+1}}_{\tau_i} } r^{2-\delta} \langle  \mathcal L_Z^{I_1} \alpha_A \gamma^{\underline L}\gamma^{e_A}\gamma^{L} \mathscr L_Z^{I_2} \psi , \gamma^T \gamma^\lambda\nabla_\lambda ( [ \gamma^\mu\nabla_\mu , \mathscr L_Z^{I} ] \psi ) \rangle \,dV \right| \\
        & \lesssim \left| \int_{\mathcal D^{\tau_{i+1}}_{\tau_i} } r^{2-\delta} \langle r\rangle^{-2} \langle  \mathcal L_Z^{I_1} \alpha_A \gamma^{\underline L}\gamma^{e_A}\gamma^{L} \mathscr L_Z^{I_2} \psi , \gamma^T \gamma^L \nabla_L  \mathscr L_Z^{\le |I|}\psi \rangle \,dV \right| \\
        & \qquad +\left| \int_{\mathcal D^{\tau_{i+1}}_{\tau_i} } r^{2-\delta} \langle r\rangle^{-2} \langle  \mathcal L_Z^{I_1} \alpha_A \gamma^{\underline L}\gamma^{e_A}\gamma^{L} \mathscr L_Z^{I_2} \psi , \gamma^T \gamma^{e_B} \nabla_{e_B}  \mathscr L_Z^{\le |I|}\psi \rangle \,dV \right| \\
        & \qquad +\left| \int_{\mathcal D^{\tau_{i+1}}_{\tau_i} } r^{2-\delta} \langle r\rangle^{-2} \langle  \mathcal L_Z^{I_1} \alpha_A \gamma^{\underline L}\gamma^{e_A}\gamma^{L} \mathscr L_Z^{I_2} \psi , \gamma^T \gamma^{\underline L} \nabla_{\underline L} \mathscr L_Z^{\le |I|}\psi \rangle \,dV \right|.
    \end{aligned}
\end{align}
Then we have
\begin{align*}
    & \left|  \int_{\mathcal D^{\tau_{i+1}}_{\tau_i} } r^{2-\delta} \langle r\rangle^{-2} \langle  \mathcal L_Z^{I_1} \alpha_A \gamma^{\underline L}\gamma^{e_A}\gamma^{L} \mathscr L_Z^{I_2} \psi , \gamma^T \gamma^L \nabla_L \mathscr L_Z^{\le |I|}\psi \rangle \,dV \right| \\
    & \qquad + \left| \int_{\mathcal D^{\tau_{i+1}}_{\tau_i} } r^{2-\delta} \langle r\rangle^{-2} \langle  \mathcal L_Z^{I_1} \alpha_A \gamma^{\underline L}\gamma^{e_A}\gamma^{L} \mathscr L_Z^{I_2} \psi , \gamma^T \gamma^{e_B} \nabla_{e_B} \mathscr L_Z^{\le |I|}\psi \rangle \,dV  \right| \\
    & \lesssim  \int_{\tau_i}^{\tau_{i+1}}  r^{-\frac12-\delta}  \mathcal E^T_{\le N-6}[\alpha]  \mathscr E^D_{\le N-1}[\psi]  \mathcal E^D_{\le N-1}[\psi] \,d\tau  \\
    & \lesssim \left( \int_{\tau_i}^{\tau_{i+1}} {}^{(0)}\mathcal E^T_{\le N-6}[\alpha]^2\,d\tau  \right)^\frac12 \left( \int_{\tau_i}^{\tau_{i+1}} r^{-1-2\delta}\mathscr E^D_{\le N-1}[\psi]^2\, {}^{(0)}\mathcal E^D_{\le N-1}[\psi]^2\,d\tau   \right)^\frac12\\
    & \lesssim C^3\varepsilon^3 (\tau_{i})^{-\frac14+2\delta+\epsilon}.
\end{align*}

We also see that the integral containing the derivative $\gamma^{\underline L}\nabla_{\underline L}$ is not harmful. Indeed, the Clifford relation shows that
\begin{align*}
  &  \left| \int_{\mathcal D^{\tau_{i+1}}_{\tau_i} } r^{2-\delta} \langle r\rangle^{-2} \langle  \mathcal L_Z^{I_1} \alpha_A \gamma^{\underline L}\gamma^{e_A}\gamma^{L} \mathscr L_Z^{I_2} \psi , \gamma^T \gamma^{\underline L} \nabla_{\underline L}  \mathscr L_Z^{\le |I|}\psi \rangle \,dV \right| \\
    & \lesssim  \int_{\mathcal D^{\tau_{i+1}}_{\tau_i} } r^{-\delta}  | \mathcal L_Z^{\le N-8} \alpha|  |\mathscr L_Z^{\le N-1} \psi_+|  | \nabla_{\underline L}  \mathscr L_Z^{\le N-1}\psi_-|\,dV.
\end{align*}
Thus Proposition \ref{decomp-dirac-null-tensor} ensures that the term $\nabla_{\underline L}\psi_-$ can be replaced by $\slashed\nabla\psi_+$ together with commutator terms and an additional nonlinearity $F\cdot\psi$. The control of the integral involving $\slashed\nabla\psi_+$ then follows from the preceding estimates for integrals involving $\gamma^L\nabla_L$ or $\gamma^{e_B}\nabla_{e_B}$. The remaining integrals are handled similarly, and we omit the details.

The integral including the curvature tensor $\mathcal R$ also follows the previous integral, after the use of the Hardy inequality. We omit the details.

\subsubsection{Control of the spacetime integrals including additional $F\cdot\psi$}
In order to complete the control, we are left to treat the integrals containing additional nonlinearities. This can be written as follows:
\begin{align}
\begin{aligned}
   & \int_{\mathcal D^{\tau_{i+1}}_{\tau_i} } r^{2-\delta} \langle  \mathcal L_Z^{I_1} \alpha_A \gamma^{\underline L}\gamma^{e_A}\gamma^{L} \mathscr L_Z^{I_2} \psi , \gamma^T \gamma^\lambda\nabla_\lambda \mathscr L_Z^I( F_{\rho\sigma}\gamma^\rho \gamma^\sigma \psi) \rangle \,dV \\
   & = \int_{\mathcal D^{\tau_{i+1}}_{\tau_i} } r^{2-\delta} \langle  \mathcal L_Z^{I_1} \alpha_A \gamma^{\underline L}\gamma^{e_A}\gamma^{L} \mathscr L_Z^{I_2} \psi , \gamma^T \gamma^{\underline L}\nabla_{\underline L} \mathscr L_Z^I( \alpha_B\gamma^{e_B} \gamma^L \psi) \rangle \,dV \\
   & + \int_{\mathcal D^{\tau_{i+1}}_{\tau_i} } r^{2-\delta} \langle  \mathcal L_Z^{I_1} \alpha_A \gamma^{\underline L}\gamma^{e_A}\gamma^{L} \mathscr L_Z^{I_2} \psi , \gamma^T \gamma^{e_C}\nabla_{e_C} \mathscr L_Z^I( \underline\alpha_B\gamma^{e_B} \gamma^{\underline L} \psi) \rangle \,dV \\
   & + \int_{\mathcal D^{\tau_{i+1}}_{\tau_i} } r^{2-\delta} \langle  \mathcal L_Z^{I_1} \alpha_A \gamma^{\underline L}\gamma^{e_A}\gamma^{L} \mathscr L_Z^{I_2} \psi , \gamma^T \gamma^{\underline L} \nabla_{\underline L}\mathscr L_Z^I ( \rho \gamma^L \gamma^{\underline L}\psi ) \rangle \,dV \\
   & +  \int_{\mathcal D^{\tau_{i+1}}_{\tau_i} } r^{2-\delta} \langle  \mathcal L_Z^{I_1} \alpha_A \gamma^{\underline L}\gamma^{e_A}\gamma^{L} \mathscr L_Z^{I_2} \psi , \gamma^T \gamma^{e_B} \nabla_{e_B}\mathscr L_Z^I ( \rho \gamma^{\underline L} \gamma^{L}\psi ) \rangle \,dV \\
   &  + \int_{\mathcal D^{\tau_{i+1}}_{\tau_i} } r^{2-\delta} \langle  \mathcal L_Z^{I_1} \alpha_A \gamma^{\underline L}\gamma^{e_A}\gamma^{L} \mathscr L_Z^{I_2} \psi , \gamma^T \gamma^{\underline L}\nabla_{\underline L} \mathscr L_Z^I( \sigma\gamma^{e_B} \gamma^{e_C} \psi) \rangle \,dV \\
   & + \int_{\mathcal D^{\tau_{i+1}}_{\tau_i} } r^{2-\delta} \langle  \mathcal L_Z^{I_1} \alpha_A \gamma^{\underline L}\gamma^{e_A}\gamma^{L} \mathscr L_Z^{I_2} \psi , \gamma^T \gamma^{e_D}\nabla_{e_D} \mathscr L_Z^I( \sigma\gamma^{e_B} \gamma^{e_C} \psi) \rangle \,dV.
   \end{aligned}
\end{align}
For the integral whose integrand contains additional $\alpha\cdot\psi$, it is enough to consider the following integrals:
\begin{align}
\begin{aligned}
&   \int_{\mathcal D^{\tau_{i+1}}_{\tau_i} } r^{2-\delta}|\mathcal L_Z^{\le N-8} \alpha| | \mathscr L_Z^{\le N-1} \psi_+| |\nabla_{\underline L}\mathcal L_Z^{\le N-8}\alpha| |\mathscr L_Z^{\le N-1}\psi_+| \,dV \\
& \qquad+ \int_{\mathcal D^{\tau_{i+1}}_{\tau_i} } r^{2-\delta}|\mathcal L_Z^{\le N-8} \alpha| | \mathscr L_Z^{\le N-1} \psi_+| |\mathcal L_Z^{\le N-8}\alpha| |\nabla_{\underline L}\mathscr L_Z^{\le N-1}\psi_+| \,dV,
\end{aligned}
\end{align}
and
\begin{align}
\begin{aligned}
&  \int_{\mathcal D^{\tau_{i+1}}_{\tau_i} } r^{2-\delta}|\mathcal L_Z^{\le N-8} \alpha| | \mathscr L_Z^{\le N-1} \psi_+| |\nabla_{\underline L}\mathcal L_Z^{\le N-1}\alpha| |\mathscr L_Z^{\le N-8}\psi_+| \,dV \\
& \qquad + \int_{\mathcal D^{\tau_{i+1}}_{\tau_i} } r^{2-\delta}|\mathcal L_Z^{\le N-8} \alpha| | \mathscr L_Z^{\le N-1} \psi_+| |\mathcal L_Z^{\le N-1}\alpha| |\nabla_{\underline L}\mathscr L_Z^{\le N-8}\psi_+| \,dV.
\end{aligned}
\end{align}
Then we write
\begin{align*}
    & \int_{\mathcal D^{\tau_{i+1}}_{\tau_i} } r^{2-\delta}|\mathcal L_Z^{\le N-8} \alpha| | \mathscr L_Z^{\le N-1} \psi_+| |\nabla_{\underline L}\mathcal L_Z^{\le N-8}\alpha| |\mathscr L_Z^{\le N-1}\psi_+| \,dV \\
    & \lesssim \int_{\tau_i}^{\tau_{i+1}} r^{1-\delta} \mathcal E^T_{\le N-6}[\alpha]\mathscr E^D_{\le N-1}[\psi]^2 \|\nabla_{L}\nabla_{\underline L}\mathcal L_Z^{\le N-6}\alpha\|_{L^2(\Sigma_\tau)}\,d\tau .
\end{align*}
Then the decomposition \eqref{decomp-alpha-wave} yields
\begin{align*}
    & \int_{\tau_i}^{\tau_{i+1}} r^{1-\delta} \mathcal E^T_{\le N-6}[\alpha]\mathscr E^D_{\le N-1}[\psi]^2 \|\nabla_{L}\nabla_{\underline L}\mathcal L_Z^{\le N-6}\alpha\|_{L^2(\Sigma_\tau)}\,d\tau \\
    & \lesssim \int_{\tau_i}^{\tau_{i+1}} r^{-\delta} \mathcal E^T_{\le N-6}[\alpha] \mathcal E^T_{\le N-5}[\alpha] \mathscr E^D_{\le N-1}[\psi]^2 \,d\tau \\
    & \qquad + \int_{\tau_i}^{\tau_{i+1}} r^{-1-\delta} \mathcal E^T_{\le N-6}[\alpha]\mathscr E^D_{\le N-1}[\psi]^2 \|\mathcal L_Z^{\le N-5}( \underline\alpha,\rho,\sigma)\|_{L^2(\Sigma_\tau)}  \,d\tau \\
    & \qquad + \int_{\tau_i}^{\tau_{i+1}} r^{1-\delta} \mathcal E^T_{\le N-6}[\alpha]\mathscr E^D_{\le N-1}[\psi]^2 \mathcal E^D_{\le N-4}[\psi]^2 \,d\tau.
\end{align*}
Then the remaining task is to use the bootstrap assumptions and we omit the details.

Now we consider the integral:
\begin{align*}
    & \int_{\mathcal D^{\tau_{i+1}}_{\tau_i} } r^{2-\delta}|\mathcal L_Z^{\le N-8} \alpha| | \mathscr L_Z^{\le N-1} \psi_+| |\mathcal L_Z^{\le N-8}\alpha| |\nabla_{\underline L}\mathscr L_Z^{\le N-1}\psi_+| \,dV \\
    & \lesssim \int_{\tau_i}^{\tau_{i+1}} r^{1-\delta} \mathcal E^T_{\le N-6}[\alpha]^2 \mathscr E^D_{\le N}[\psi]^2\,d\tau \\
    & \lesssim C^2\varepsilon^2 \tau^{2\epsilon} \int_{\tau_i}^{\tau_{i+1}} r^{-\delta} \, {}^{(1)} \mathcal E^T_{\le N-6}[\alpha]^2  \,d\tau \\
    & \lesssim C^4\varepsilon^4 \tau_i^{-\frac12+4\delta+2\epsilon}.
\end{align*}
Similarly, we see that
\begin{align*}
    & \int_{\mathcal D^{\tau_{i+1}}_{\tau_i} } r^{2-\delta}|\mathcal L_Z^{\le N-8} \alpha| | \mathscr L_Z^{\le N-1} \psi_+| |\nabla_{\underline L}\mathcal L_Z^{\le N-1}\alpha| |\mathscr L_Z^{\le N-8}\psi_+| \,dV \\
    & \lesssim \int_{\tau_i}^{\tau_{i+1}} r^{1-\delta} \mathcal E^T_{\le N-6}[\alpha] \mathcal E^D_{\le N-6}[\psi] \mathscr E^D_{\le N-1}[\psi] \mathscr E^T_{\le N}[\alpha]\,d\tau \\
    & \lesssim C\varepsilon \tau^{\epsilon} \int_{\tau_i}^{\tau_{i+1}} r^{1-\delta} \mathcal E^T_{\le N-6}[\alpha] \mathcal E^D_{\le N-6}[\psi]  \mathscr E^T_{\le N}[\alpha]\,d\tau  \\
    & \lesssim C\varepsilon \tau^\epsilon \left( \int_{\tau_i}^{\tau_{i+1}} {}^{(\delta)}\mathscr E^T_{\le N}[\alpha]^2 \,d\tau \right)^\frac12 \left( \int_{\tau_i}^{\tau_{i+1}} r^{-2\delta} \, {}^{(1-\delta)}\mathcal E^T_{\le N-6}[\alpha]^2 \, {}^{(1)}\mathcal E^D_{\le N-6}[\psi]^2\,d\tau  \right)^\frac12 \\
    & \lesssim C^4\varepsilon^4 \tau_i^{-\frac12+3\delta+\epsilon}.
\end{align*}
The fourth integral is very similar to the first integral. Indeed, the weighted Sobolev inequality allows us to exploit the decomposition \eqref{decomp-LLbar}.


From now on we consider the integral where the additional nonlinearity $\underline\alpha\cdot\psi$ appears:
\begin{align}
    \int_{\mathcal D^{\tau_{i+1}}_{\tau_i} } r^{2-\delta} \langle  \mathcal L_Z^{I_1} \alpha_A \gamma^{\underline L}\gamma^{e_A}\gamma^{L} \mathscr L_Z^{I_2} \psi , \gamma^T \gamma^{e_C}\nabla_{e_C} \mathscr L_Z^I( \underline\alpha_B\gamma^{e_B} \gamma^{\underline L} \psi) \rangle \,dV.
\end{align}
In fact, thanks to the angular derivative $\nabla_{e_C}$, it suffices to control the following integrals:
\begin{align*}
 &    \int_{\mathcal D^{\tau_{i+1}}_{\tau_i} } r^{1-\delta} | \mathcal L_Z^{\le N-8}\alpha| |\mathscr L_Z^{\le
 N-1}\psi_+| |\mathcal L_Z^{\le N}\underline\alpha| |\mathscr L_Z^{\le N-8}\psi_-| \,dV \\
  & \qquad +  \int_{\mathcal D^{\tau_{i+1}}_{\tau_i} } r^{1-\delta} | \mathcal L_Z^{\le N-8}\alpha| |\mathscr L_Z^{\le
 N-1}\psi_+| |\mathcal L_Z^{\le N-8}\underline\alpha| |\mathscr L_Z^{\le N}\psi_-| \,dV.
\end{align*}
Then we see that
\begin{align*}
    &  \int_{\mathcal D^{\tau_{i+1}}_{\tau_i} } r^{1-\delta} | \mathcal L_Z^{\le N-8}\alpha| |\mathscr L_Z^{\le
 N-1}\psi_+| |\mathcal L_Z^{\le N}\underline\alpha| |\mathscr L_Z^{N-8}\psi_-| \,dV  \\
 & \lesssim \int_{\tau_i}^{\tau_{i+1}} r^{-\delta} \mathcal E^T_{\le N-6}[\alpha]\mathcal E^D_{\le N-6}[\psi]  \mathscr E^D_{\le N-1}[\psi]  \|\mathcal L_Z^{\le N}\underline\alpha\|_{L^2(\Sigma_\tau)}\,d\tau \\
 & \lesssim  C\varepsilon \tau^{\epsilon} \left( \int_{\tau_i}^{\tau_{i+1}} {}^{(1-\delta)}\mathcal E^T_{\le N-6}[\alpha]^2 \, {}^{(1)}\mathcal E^D_{\le N-6}[\psi]^2\,d\tau \right)^\frac12 \left( \int_{\tau_i}^{\tau_{i+1}}  r^{-2+\delta} \|\mathcal L_Z^{\le N}\underline\alpha\|_{L^2(\Sigma_\tau)}^2\,d\tau \right)^\frac12 \\
 & \lesssim C^4\varepsilon^4 \tau_i^{-\frac12+\delta+\epsilon},
\end{align*}
and
\begin{align*}
    &  \int_{\mathcal D^{\tau_{i+1}}_{\tau_i} } r^{1-\delta} | \mathcal L_Z^{\le N-8}\alpha| |\mathscr L_Z^{\le
 N-1}\psi_+| |\mathcal L_Z^{\le N-8}\underline\alpha| |\mathscr L_Z^{\le N}\psi_-| \,dV \\
 & \lesssim \int_{\tau_i}^{\tau_{i+1}} r^{-\delta} \mathcal E^T_{\le N-6}[\alpha] \mathcal E^T_{\le N-6}[\underline\alpha] \mathscr E^D_{\le N-1}[\psi] \|\mathscr L_Z^{\le N}\psi_-\|_{L^2(\Sigma_\tau)}\,d\tau .
\end{align*}
Here we consider the term $\mathscr L_Z^{\le N}\psi_-$. If the configuration of the vector fields $\mathscr L_Z^{\le N}\psi_-$ contains at least one $\underline L$, then we can replace the term by $\slashed\nabla \mathscr L_Z^{\le N-1}\psi_+ $. If it includes at least one $L$, then it becomes the energy $\mathcal E^D_{\le N-1}[\psi]$. Now we are left to consider the case when the vector fields $\mathscr L_Z^{\le N}\psi_-= \mathscr L_{\Omega}^{\le N}\psi_-$. In this case, we have
\begin{align*}
    & \int_{\tau_i}^{\tau_{i+1}} r^{-\delta} \mathcal E^T_{\le N-6}[\alpha] \mathcal E^T_{\le N-6}[\underline\alpha] \mathscr E^D_{\le N-1}[\psi] \|\mathscr L_Z^{\le N}\psi_-\|_{L^2(\Sigma_\tau)}\,d\tau \\
    & \lesssim \int_{\tau_i}^{\tau_{i+1}} r^{1-\delta} \mathcal E^T_{\le N-6}[\alpha] \mathcal E^T_{\le N-6}[\underline\alpha] \mathscr E^D_{\le N-1}[\psi] \|\slashed\nabla\mathscr L_Z^{\le N-1}\psi_-\|_{L^2(\Sigma_\tau)}\,d\tau \\
    & \lesssim C\varepsilon \tau^{\epsilon} \left( \int_{\tau_i}^{\tau_{i+1}} {}^{(1)}\mathcal E^T_{\le N-6}[\alpha]^2 \,d\tau \right)^\frac12 \left( \int_{\tau_i}^{\tau_{i+1}} r^{-2\delta} \mathcal E^D_{\le N-6}[\psi]^2 \, {}^{(0)} \mathcal E^D_{\le N-1}[\psi]^2\,d\tau \right)^\frac12 \\
    & \lesssim C^4\varepsilon^4 \tau_i^{-\frac12+4\delta+\epsilon}.
\end{align*}
To complete the control of the spacetime integral with derivative loss, we are left to treat the integral including the additional nonlinearity $\rho\cdot\psi$ and $\sigma\cdot\psi$. Here we are only concerned with the integral:
\begin{align}
    \int_{\mathcal D^{\tau_{i+1}}_{\tau_i} } r^{2-\delta} \langle  \mathcal L_Z^{I_1} \alpha_A \gamma^{\underline L}\gamma^{e_A}\gamma^{L} \mathscr L_Z^{I_2} \psi , \gamma^T \gamma^{\underline L} \nabla_{\underline L}\mathscr L_Z^I ( \rho \gamma^L \gamma^{\underline L}\psi ) \rangle \,dV .
\end{align}
The control of the remaining integrals will be obvious. It is enough to consider the following integrals:
\begin{align}\label{int-apsi-rhompsi}
    \begin{aligned}
       & \int_{\mathcal D^{\tau_{i+1}}_{\tau_i} } r^{2-\delta} |  \mathcal L_Z^{\le N-8} \alpha | | \mathscr L_Z^{\le N-1} \psi_+| | \nabla_{\underline L}\mathcal L_Z^{\le N-1} \rho | | \mathscr L_Z^{\le N-8} \psi_-  |\,dV  \\
        & \qquad + \int_{\mathcal D^{\tau_{i+1}}_{\tau_i} } r^{2-\delta} |  \mathcal L_Z^{\le N-8} \alpha | | \mathscr L_Z^{\le N-1} \psi_+| | \mathcal L_Z^{\le N-1} \rho | | \nabla_{\underline L}\mathscr L_Z^{\le N-8} \psi_-  |\,dV \\
        & \qquad + \int_{\mathcal D^{\tau_{i+1}}_{\tau_i} } r^{2-\delta} |  \mathcal L_Z^{\le N-8} \alpha | | \mathscr L_Z^{\le N-1} \psi_+| | \nabla_{\underline L}\mathcal L_Z^{\le N-8} \rho | | \mathscr L_Z^{\le N-1} \psi_-  |\,dV \\
        & \qquad + \int_{\mathcal D^{\tau_{i+1}}_{\tau_i} } r^{2-\delta} |  \mathcal L_Z^{\le N-8} \alpha | | \mathscr L_Z^{\le N-1} \psi_+| | \mathcal L_Z^{\le N-8} \rho | | \nabla_{\underline L}\mathscr L_Z^{\le N-1} \psi_-  |\,dV.
    \end{aligned}
\end{align}
For the first integral, the null decomposition for the tensor fields shows that
\begin{align*}
    & \int_{\mathcal D^{\tau_{i+1}}_{\tau_i} } r^{2-\delta} |  \mathcal L_Z^{\le N-8} \alpha | | \mathscr L_Z^{\le N-1} \psi_+| | \nabla_{\underline L}\mathcal L_Z^{\le N-1} \rho | | \mathscr L_Z^{\le N-8} \psi_-  |\,dV  \\
    & \lesssim \int_{\mathcal D^{\tau_{i+1}}_{\tau_i} } r^{1-\delta} |  \mathcal L_Z^{\le N-8} \alpha | | \mathscr L_Z^{\le N-1} \psi_+| | \mathcal L_Z^{\le N-1} \rho | | \mathscr L_Z^{\le N-8} \psi_-  |\,dV \\
    & \qquad + \int_{\mathcal D^{\tau_{i+1}}_{\tau_i} } r^{1-\delta} |  \mathcal L_Z^{\le N-8} \alpha | | \mathscr L_Z^{\le N-1} \psi_+| |\mathcal L_Z^{\le N} \underline\alpha | | \mathscr L_Z^{\le N-8} \psi_-  |\,dV \\
    & \qquad + \int_{\mathcal D^{\tau_{i+1}}_{\tau_i} } r^{2-\delta} |  \mathcal L_Z^{\le N-8} \alpha | | \mathscr L_Z^{\le N-1} \psi_+| | \mathcal L_Z^{\le N-1} J_{\underline L} | | \mathscr L_Z^{\le N-8} \psi_-  |\,dV .
\end{align*}
Then we see that
\begin{align*}
    & \int_{\mathcal D^{\tau_{i+1}}_{\tau_i} } r^{1-\delta} |  \mathcal L_Z^{\le N-8} \alpha | | \mathscr L_Z^{\le N-1} \psi_+| | \mathcal L_Z^{\le N-1} \rho | | \mathscr L_Z^{\le N-8} \psi_-  |\,dV  \\
    & \lesssim \int_{\tau_i}^{\tau_{i+1}} r^{-\delta} \mathcal E^T_{\le N-6}[\alpha] \mathcal E^D_{\le N-6}[\psi] \mathscr E^D_{\le N-1}[\psi]\mathscr E^T_{\le N-1}[\rho]\,d\tau \\
    & \lesssim C\varepsilon \tau_i^{\epsilon} \left( \int_{\tau_i}^{\tau_{i+1}} {}^{(1)}\mathcal E^T_{\le N-6}[\alpha]^2 {}^{(0)}\mathcal E^D_{\le N-6}[\psi]^2\,d\tau \right)^\frac12 \left( \int_{\tau_i}^{\tau_{i+1}} {}^{(-1+\delta)}\mathscr E^T_{\le N-1}[\rho]^2\,d\tau  \right)^\frac12 \\
    & \lesssim C^4\varepsilon^4 \tau_i^{-\frac12+\delta},
\end{align*}
and
\begin{align*}
    &  \int_{\mathcal D^{\tau_{i+1}}_{\tau_i} } r^{1-\delta} |  \mathcal L_Z^{\le N-8} \alpha | | \mathscr L_Z^{\le N-1} \psi_+| |\mathcal L_Z^{\le N} \underline\alpha | | \mathscr L_Z^{\le N-8} \psi_-  |\,dV \\
    & \lesssim \int_{\tau_i}^{\tau_{i+1}} r^{-\delta} \mathcal E^T_{\le N-6}[\alpha]\mathcal E^D_{\le N-6}[\psi] \mathscr E^D_{\le N-1}[\psi] \|\mathcal L_Z^{\le N}\underline\alpha\|_{L^2(\Sigma_\tau)}\,d\tau \\
    & \lesssim C\varepsilon \tau_i^{\epsilon} \left( \int_{\tau_i}^{\tau_{i+1}} {}^{(1-\delta)}\mathcal E^T_{\le N-6}[\alpha]^2 \, {}^{(1)}\mathcal E^D_{\le N-6}[\psi]^2\,d\tau  \right)^\frac12 \left( \int_{\tau_i}^{\tau_{i+1}} r^{-2+\delta} \|\mathcal L_Z^{\le N}\underline\alpha\|_{L^2(\Sigma_\tau)}^2\,d\tau \right)^\frac12 \\
    & \lesssim C^4\varepsilon^4 \tau_i^{-\frac12+\delta},
\end{align*}
and
\begin{align*}
    & \int_{\mathcal D^{\tau_{i+1}}_{\tau_i} } r^{2-\delta} |  \mathcal L_Z^{\le N-8} \alpha | | \mathscr L_Z^{\le N-1} \psi_+| | \mathcal L_Z^{\le N-1} J_{\underline L} | | \mathscr L_Z^{\le N-8} \psi_-  |\,dV \\
    & \lesssim \int_{\tau_i}^{\tau_{i+1}}r^{\frac12-\delta} \mathcal E^T_{\le N-6}[\alpha]\mathcal E^D_{\le N-6}[\psi]^2 \mathscr E^D_{\le N-1}[\psi]^2\,d\tau \\
    & \lesssim C^2\varepsilon^2 \tau_i^{2\epsilon} \left( \int_{\tau_i}^{\tau_{i+1}} {}^{(1-\delta)}\mathcal E^T_{\le N-6}[\alpha]^2 \,{}^{(0)}\mathcal E^D_{\le N-6}[\psi]^2\,d\tau \right)^\frac12 \left(  \int_{\tau_i}^{\tau_{i+1}} {}^{(0)}\mathcal E^D_{\le N-6}[\psi]^2\,d\tau  \right)^\frac12 \\
    & \lesssim C^4\varepsilon^4 \tau_i^{-1+\delta+2\epsilon}.
\end{align*}
Now we are concerned with the second integral of \eqref{int-apsi-rhompsi}. We apply Proposition \ref{prop-null-decomp-dirac} to $\nabla_{\underline L}\mathscr L_Z^{\le N-8}\psi_-$ and get
\begin{align*}
    & \int_{\mathcal D^{\tau_{i+1}}_{\tau_i} } r^{2-\delta} |  \mathcal L_Z^{\le N-8} \alpha | | \mathscr L_Z^{\le N-1} \psi_+| | \mathcal L_Z^{\le N-1} \rho | | \nabla_{\underline L}\mathscr L_Z^{\le N-8} \psi_-  |\,dV \\
    & \lesssim \int_{\mathcal D^{\tau_{i+1}}_{\tau_i} } r^{2-\delta} |  \mathcal L_Z^{\le N-8} \alpha | | \mathscr L_Z^{\le N-1} \psi_+| | \mathcal L_Z^{\le N-1} \rho | | \slashed\nabla \mathscr L_Z^{\le N-8} \psi_+  |\,dV \\
    & \qquad + \int_{\mathcal D^{\tau_{i+1}}_{\tau_i} } r^{2-\delta} |  \mathcal L_Z^{\le N-8} \alpha | | \mathscr L_Z^{\le N-1} \psi_+| | \mathcal L_Z^{\le N-1} \rho | | \mathscr L_Z^{\le N-8}(\underline\alpha\cdot \psi_-)  |\,dV \\
    & \qquad + \int_{\mathcal D^{\tau_{i+1}}_{\tau_i} } r^{2-\delta} |  \mathcal L_Z^{\le N-8} \alpha | | \mathscr L_Z^{\le N-1} \psi_+| | \mathcal L_Z^{\le N-1} \rho | | \mathscr L_Z^{\le N-8} \left( (\rho,\sigma) \cdot \psi_\pm \right)  |\,dV.
\end{align*}
Then we see that
\begin{align*}
    &  \int_{\mathcal D^{\tau_{i+1}}_{\tau_i} } r^{2-\delta} |  \mathcal L_Z^{\le N-8} \alpha | | \mathscr L_Z^{\le N-1} \psi_+| | \mathcal L_Z^{\le N-1} \rho | | \slashed\nabla \mathscr L_Z^{\le N-8} \psi_+  |\,dV \\
    & \lesssim  \int_{\mathcal D^{\tau_{i+1}}_{\tau_i} } r^{1-\delta} |  \mathcal L_Z^{\le N-8} \alpha | | \mathscr L_Z^{\le N-1} \psi_+| | \mathcal L_Z^{\le N-1} \rho | |  \mathscr L_Z^{\le N-7} \psi_+  |\,dV \\
    & \lesssim \int_{\tau_i}^{\tau_{i+1}}r^{-\delta}  \mathcal E^T_{\le N-6}[\alpha] \mathcal E^D_{\le N-5}[\psi] \mathscr E^D_{\le N-1}[\psi]\mathscr E^T_{\le N-1}[\rho]\,d\tau \\
    & \lesssim C\varepsilon \tau^{\epsilon} \left( \int_{\tau_i}^{\tau_{i+1}} {}^{(-1+\delta)}\mathscr E^T_{\le N-1}[\rho]^2 \, {}^{(1)} \mathcal E^D_{\le N-5}[\psi]^2\,d\tau \right)^\frac12 \left( \int_{\tau_i}^{\tau_{i+1}} {}^{(0)}\mathcal E^T_{\le N-6}[\alpha]^2\,d\tau  \right)^\frac12 \\
    & \lesssim C^4\varepsilon^4 \tau^{-\frac34+4\delta+\epsilon},
\end{align*}
and
\begin{align*}
    & \int_{\mathcal D^{\tau_{i+1}}_{\tau_i} } r^{2-\delta} |  \mathcal L_Z^{\le N-8} \alpha | | \mathscr L_Z^{\le N-1} \psi_+| | \mathcal L_Z^{\le N-1} \rho | | \mathscr L_Z^{\le N-8}(\underline\alpha\cdot \psi_-)  |\,dV \\
    & \lesssim \int_{\tau_i}^{\tau_{i+1}} r^{\frac12-\delta} \mathcal E^T_{\le N-6}[\alpha]\mathcal E^T_{\le N-6}[\underline\alpha]\mathcal E^D_{\le N-6}[\psi] \mathscr E^T_{\le N-1}[\rho] \mathscr E^D_{\le N-1}[\psi]\,d\tau ,
\end{align*}
and therefore one can obtain even better decay than the previous integral. The control of the third integral involving the additional $(\rho,\sigma)\cdot\psi_\pm$ also follows the previous integral, which we omit the details.
Now we consider the third integral of \eqref{int-apsi-rhompsi}. As the first integral of \eqref{int-apsi-rhompsi}, we write
\begin{align*}
    & \int_{\mathcal D^{\tau_{i+1}}_{\tau_i} } r^{2-\delta} |  \mathcal L_Z^{\le N-8} \alpha | | \mathscr L_Z^{\le N-1} \psi_+| | \nabla_{\underline L}\mathcal L_Z^{\le N-8} \rho | | \mathscr L_Z^{\le N-1} \psi_-  |\,dV \\
    & \lesssim \int_{\mathcal D^{\tau_{i+1}}_{\tau_i} } r^{1-\delta} |  \mathcal L_Z^{\le N-8} \alpha | | \mathscr L_Z^{\le N-1} \psi_+| | \mathcal L_Z^{\le N-8} \rho | | \mathscr L_Z^{\le N-1} \psi_-  |\,dV \\
    & \qquad + \int_{\mathcal D^{\tau_{i+1}}_{\tau_i} } r^{1-\delta} |  \mathcal L_Z^{\le N-8} \alpha | | \mathscr L_Z^{\le N-1} \psi_+| | \mathcal L_Z^{\le N-7} \underline\alpha  | | \mathscr L_Z^{\le N-1} \psi_-  |\,dV \\
    & \qquad + \int_{\mathcal D^{\tau_{i+1}}_{\tau_i} } r^{2-\delta} |  \mathcal L_Z^{\le N-8} \alpha | | \mathscr L_Z^{\le N-1} \psi_+| | \mathcal L_Z^{\le N-8} J_{\underline L} | | \mathscr L_Z^{\le N-1} \psi_-  |\,dV.
\end{align*}
The remaining task is obvious. In fact, for the first integral we use in order the H\"older inequality to get the product $L^\infty\times L^2\times L^\infty \times L^2$, the Sobolev embedding on the outgoing null $\Sigma_\tau$, which gives a factor $r^{-1}$. Here we observe that $\|\mathscr L_Z^{\le N-1}\psi_-\|_{L^2(\Sigma_\tau)}\lesssim r\mathcal E^D_{\le N-2}[\psi]$. Then one can obtain the required decay by using the integrated local energy decay estimates for $\rho$ or $\underline\alpha$. We omit the details. Then we are left to consider the fourth integral of \eqref{int-apsi-rhompsi}:
\begin{align*}
    & \int_{\mathcal D^{\tau_{i+1}}_{\tau_i} } r^{2-\delta} |  \mathcal L_Z^{\le N-8} \alpha | | \mathscr L_Z^{\le N-1} \psi_+| | \mathcal L_Z^{\le N-8} \rho | | \nabla_{\underline L}\mathscr L_Z^{\le N-1} \psi_-  |\,dV  \\
    & \lesssim \int_{\mathcal D^{\tau_{i+1}}_{\tau_i} } r^{2-\delta} |  \mathcal L_Z^{\le N-8} \alpha | | \mathscr L_Z^{\le N-1} \psi_+| | \mathcal L_Z^{\le N-8} \rho | | \slashed\nabla\mathscr L_Z^{\le N-1} \psi_+  |\,dV \\
    & \qquad + \int_{\mathcal D^{\tau_{i+1}}_{\tau_i} } r^{2-\delta} |  \mathcal L_Z^{\le N-8} \alpha | | \mathscr L_Z^{\le N-1} \psi_+| | \mathcal L_Z^{\le N-8} \rho | | \mathscr L_Z^{\le N-1} ( \underline\alpha\cdot\psi_-)  |\,dV \\
    & \qquad + \int_{\mathcal D^{\tau_{i+1}}_{\tau_i} } r^{2-\delta} |  \mathcal L_Z^{\le N-8} \alpha | | \mathscr L_Z^{\le N-1} \psi_+| | \mathcal L_Z^{\le N-8} \rho | | \mathscr L_Z^{\le N-1}\left( (\rho,\sigma)\cdot \psi_\pm \right)  |\,dV .
\end{align*}
Then we see that
\begin{align*}
    &  \int_{\mathcal D^{\tau_{i+1}}_{\tau_i} } r^{2-\delta} |  \mathcal L_Z^{\le N-8} \alpha | | \mathscr L_Z^{\le N-1} \psi_+| | \mathcal L_Z^{\le N-8} \rho | | \slashed\nabla\mathscr L_Z^{\le N-1} \psi_+  |\,dV \\
    & \lesssim \int_{\tau_i}^{\tau_{i+1}} r^{1-\delta} \mathcal E^T_{\le N-6}[\alpha] \|\nabla_L \mathcal L_Z^{\le N-6}\rho\|_{L^2(\Sigma_\tau)} \mathscr E^D_{\le N-1}[\psi] \mathcal E^D_{\le N-1}[\psi]\,d\tau \\
    & \lesssim \int_{\tau_i}^{\tau_{i+1}} r^{-\delta} \mathcal E^T_{\le N-6}[\alpha] \| \mathcal L_Z^{\le N-6}\rho\|_{L^2(\Sigma_\tau)} \mathscr E^D_{\le N-1}[\psi] \mathcal E^D_{\le N-1}[\psi]\,d\tau \\
    & \qquad + \int_{\tau_i}^{\tau_{i+1}} r^{1-\delta} \mathcal E^T_{\le N-6}[\alpha] \mathcal E^T_{\le N-6}[\alpha]  \mathscr E^D_{\le N-1}[\psi] \mathcal E^D_{\le N-1}[\psi]\,d\tau \\
    & \qquad + \int_{\tau_i}^{\tau_{i+1}} r^{1-\delta} \mathcal E^T_{\le N-6}[\alpha] \| \mathcal L_Z^{\le N-6}J_L\|_{L^2(\Sigma_\tau)} \mathscr E^D_{\le N-1}[\psi] \mathcal E^D_{\le N-1}[\psi]\,d\tau,
\end{align*}
and the required decay is followed via the integrated local energy decay estimates for $\rho$. The control of the remaining integrals is also very similar, which we omit the details.

This completes the control of the spacetime integral with derivative loss.

Finally, we are left to deal with the integral when the derivative $\nabla_{\underline L}$ acts on $\alpha$:
\begin{align}
    \begin{aligned}
        \int_{\mathcal D^{\tau_{i+1}}_{\tau_i} } r^{2-\delta} \langle \nabla_{\underline L}\mathcal L_Z^{I_1}\alpha_A \gamma^{\underline L}\gamma^{e_A}\gamma^L \mathscr L_Z^{I_2}\psi, \gamma^T \nabla_L \mathscr L_Z^I\psi\rangle\,dV,
    \end{aligned}
\end{align}
in order to complete the integral of $\alpha\cdot\psi$.

The low-high interaction is already treated in the previous case. For the high-low interaction, we use the integration by parts along the $\underline L$-direction. This will be considered in the control of $\underline\alpha\cdot\psi$.

\subsection{Control of $\underline\alpha\cdot\psi$}
We consider the integral:
\begin{align}
    \int_{\mathcal D^{\tau_{i+1}}_{\tau_i} } r^{2-\delta} \langle \mathcal L_Z^{I_1}\underline\alpha_A \gamma^L \gamma^{e_A}\gamma^{\underline L} \nabla_{L}\mathscr L_Z^{I_2}\psi, \gamma^T \nabla_L \mathscr L_Z^I\psi\rangle \,dV.
\end{align}
For the low-high interaction, we simply use the $L^\infty$-Sobolev embedding for $\underline\alpha$.
Therefore we are only concerned with the high-low interaction, in which case one cannot obtain further decay after an application of the Sobolev embedding on the outgoing null surface to $\mathscr L_Z^{I_2}\psi$.

\subsubsection{Foliation by the ingoing null cones}
Instead, we foliate the region $\mathcal D^{\tau_{i+1}}_{\tau_i}$ by the ingoing null cones $\underline{\mathcal C}_v$. Then we apply the weighted Sobolev inequality on $\underline{\mathcal C}_v$ to $\mathscr L_Z^{I_2}\psi$. Since the Clifford relation implies that two spinors are $\psi_-$ components, the null decomposition for the spinor fields Proposition \ref{prop-null-decomp-dirac} gives the following:
\begin{align*}
   & \int_{\mathcal D^{\tau_{i+1}}_{\tau_i} } r^{2-\delta} |\mathcal L_Z^{\le N-1}\underline\alpha| |\nabla_L\mathscr L_Z^{\le N-8}\psi_-| |\nabla_L \mathscr L_Z^{\le N-1}\psi_-|\,dV  \\
   & \lesssim \int r^{\frac32-\delta} \|\mathcal L_Z^{\le N-1}\underline\alpha\|_{L^2(\underline{\mathcal C}_v)} \|\nabla_{\underline L}\nabla_L\mathscr L_Z^{\le N-6}\psi_-\|_{L^2(\underline{\mathcal C}_v)} \|\nabla_L \mathscr L_Z^{\le N-1}\psi_-\|_{L^2(\underline{\mathcal C}_v)}\,dv \\
   & \lesssim \int r^{\frac12-\delta} \|\mathcal L_Z^{\le N-1}\underline\alpha\|_{L^2(\underline{\mathcal C}_v)} \|\nabla_L\mathscr L_Z^{\le N-5}\psi_+\|_{L^2(\underline{\mathcal C}_v)} \|\nabla_L \mathscr L_Z^{\le N-1}\psi_-\|_{L^2(\underline{\mathcal C}_v)}\,dv \\
   & \qquad + \int r^{\frac32-\delta} \|\mathcal L_Z^{\le N-1}\underline\alpha\|_{L^2(\underline{\mathcal C}_v)} \|\nabla_L\mathscr L_Z^{\le N-6}(\underline\alpha\cdot\psi_-)\|_{L^2(\underline{\mathcal C}_v)} \|\nabla_L \mathscr L_Z^{\le N-1}\psi_-\|_{L^2(\underline{\mathcal C}_v)}\,dv \\
   & \qquad + \int r^{\frac32-\delta} \|\mathcal L_Z^{\le N-1}\underline\alpha\|_{L^2(\underline{\mathcal C}_v)} \|\nabla_L\mathscr L_Z^{\le N-6}((\rho,\sigma)\cdot\psi_\pm)\|_{L^2(\underline{\mathcal C}_v)} \|\nabla_L \mathscr L_Z^{\le N-1}\psi_-\|_{L^2(\underline{\mathcal C}_v)}\,dv \\
   & \qquad + \int r^{\frac32-\delta} \|\mathcal L_Z^{\le N-1}\underline\alpha\|_{L^2(\underline{\mathcal C}_v)} \|\nabla_L[ \gamma^\mu\nabla_\mu, \mathscr L_Z^{\le N-6}]\psi \|_{L^2(\underline{\mathcal C}_v)} \|\nabla_L \mathscr L_Z^{\le N-1}\psi_-\|_{L^2(\underline{\mathcal C}_v)}\,dv .
\end{align*}
After using the boundedness of the energy $\mathscr F^T_{\le N}[\underline\alpha]$, we see that
\begin{align*}
    &\int r^{\frac12-\delta} \|\mathcal L_Z^{\le N-1}\underline\alpha\|_{L^2(\underline{\mathcal C}_v)} \|\nabla_L\mathscr L_Z^{\le N-5}\psi_+\|_{L^2(\underline{\mathcal C}_v)} \|\nabla_L \mathscr L_Z^{\le N-1}\psi_-\|_{L^2(\underline{\mathcal C}_v)}\,dv \\
    & \lesssim C\varepsilon \tau^{\frac\delta2} \int r^{\frac12-\delta}  \|\nabla_L\mathscr L_Z^{\le N-5}\psi_+\|_{L^2(\underline{\mathcal C}_v)} \|\nabla_L \mathscr L_Z^{\le N-1}\psi_-\|_{L^2(\underline{\mathcal C}_v)}\,dv  \\
    & \lesssim C\varepsilon \tau^{\frac\delta2}  \left( \int_{\mathcal D^{\tau_{i+1}}_{\tau_i} } r^{-\delta}|\nabla_L \mathscr L_Z^{\le N-5}\psi_+|^2\,dV \right)^\frac12 \left( \int_{\mathcal D^{\tau_{i+1}}_{\tau_i} }r^{1-\delta} |\nabla_L\mathscr L_Z^{\le N-1}\psi_-|^2\,dV \right)^\frac12 \\
    & \lesssim  C\varepsilon \tau^{\frac\delta2} \left( \int_{{\tau_i} }^{\tau_{i+1}} {}^{(0)}\mathcal E^D_{\le N-5}[\psi]^2(\tau)\,d\tau  \right)^\frac12 \left(  \int_{{\tau_i} }^{\tau_{i+1}} {}^{(1-\delta)}\mathcal E^D_{\le N-1}[\psi]^2(\tau)\,d\tau  \right)^\frac12 \\
    & \lesssim C^3\varepsilon^3 (\tau_i)^{-\frac12+\frac32\delta},
\end{align*}
and
\begin{align*}
   & \int r^{\frac32-\delta} \|\mathcal L_Z^{\le N-1}\underline\alpha\|_{L^2(\underline{\mathcal C}_v)} \|\nabla_L\mathscr L_Z^{\le N-6}(\underline\alpha\cdot\psi_-)\|_{L^2(\underline{\mathcal C}_v)} \|\nabla_L \mathscr L_Z^{\le N-1}\psi_-\|_{L^2(\underline{\mathcal C}_v)}\,dv  \\
   & \lesssim C\varepsilon \tau^{\frac\delta2}  \left( \int_{\mathcal D^{\tau_{i+1}}_{\tau_i} } r^{2-\delta} |\nabla_L\mathscr L_Z^{\le N-6}(\underline\alpha\cdot\psi_-)|^2\,dV \right)^\frac12 \left( \int_{\mathcal D^{\tau_{i+1}}_{\tau_i} } r^{1-\delta} |\nabla_L\mathscr L_Z^{\le N-1}\psi_-|^2\,dV \right)^\frac12 \\
   & \lesssim C^2\varepsilon^2 (\tau_i)^{\delta} \left( \int_{{\tau_i} }^{\tau_{i+1}}   {}^{(1-\delta)} \mathcal E^T_{\le N-4}[ \underline\alpha ]^2(\tau)\, {}^{(1)} \mathcal E^D_{\le N-4}[\psi]^2(\tau)\,d\tau   \right)^\frac12 \\
   & \lesssim C^2\varepsilon^2 (\tau_i)^{\delta} \, {}^{(1)} \mathcal E^D_{\le N-4}[\psi] \left( \int_{{\tau_i} }^{\tau_{i+1}}  {}^{(1-\delta)} \mathcal E^T_{\le N-4}[ \underline\alpha ]^2(\tau) \,d\tau   \right)^\frac12 \\
   & \lesssim C^4\varepsilon^4 (\tau_i)^{-\frac12+\frac32\delta},
\end{align*}
and the control of the integral including the additional nonlinearity $(\rho,\sigma)\cdot\psi_\pm$ is also very similar to the control of the previous integral. Here we use the integrated local energy decay estimates \eqref{iled-a} for $\rho$ and $\sigma$, which we omit the details.
The control of the integral including the commutator term follows the control of the first integral.

On the other hand, we consider the integral where the derivative $\nabla_{L}$ acts on $\underline\alpha$ with the high-low interaction. Then the null decomposition for the tensor fields enables us to express the term $\nabla_L\underline\alpha$ as follows:
\begin{align*}
    &  \int_{\mathcal D^{\tau_{i+1}}_{\tau_i} } r^{2-\delta} |\nabla_L\mathcal L_Z^{\le N-1}\underline\alpha| |\mathscr L_Z^{\le N-8}\psi_-| |\nabla_L \mathscr L_Z^{\le N-1}\psi_-|\,dV \\
    & \lesssim \int_{\mathcal D^{\tau_{i+1}}_{\tau_i} } r^{2-\delta} |\nabla_{\underline L}\mathcal L_Z^{\le N-1}\alpha| |\mathscr L_Z^{\le N-8}\psi_-| |\nabla_L \mathscr L_Z^{\le N-1}\psi_-|\,dV \\
    & \qquad + \int_{\mathcal D^{\tau_{i+1}}_{\tau_i} } r^{2-\delta} |\slashed\nabla\mathcal L_Z^{\le N-1}\rho| |\mathscr L_Z^{\le N-8}\psi_-| |\nabla_L \mathscr L_Z^{\le N-1}\psi_-|\,dV.
\end{align*}
For the second integral, we use in order the weighted Sobolev inequality for the lower-order spinor and the integrated local energy decay estimates \eqref{iled-a} to get
\begin{align*}
    & \int_{\mathcal D^{\tau_{i+1}}_{\tau_i} } r^{2-\delta} |\slashed\nabla\mathcal L_Z^{\le N-1}\rho| |\mathscr L_Z^{\le N-8}\psi_-| |\nabla_L \mathscr L_Z^{\le N-1}\psi_-|\,dV \\
    & \lesssim \int_{\mathcal D^{\tau_{i+1}}_{\tau_i} } r^{1-\delta} |\mathcal L_Z^{\le N}\rho| |\mathscr L_Z^{\le N-8}\psi_-| |\nabla_L \mathscr L_Z^{\le N-1}\psi_-|\,dV \\
    & \lesssim \int_{\tau_i}^{\tau_{i+1} } r^{\frac12-\delta} \mathscr E^T_{\le N}[\rho]   \mathcal E^D_{\le N-6}[\psi]\mathcal E^D_{\le N-1}[\psi]\,d\tau \\
    & \lesssim \left( \int_{\tau_i}^{\tau_{i+1} }{}^{(-1+\delta)} \mathscr E^T_{\le N}[\rho]^2  \, {}^{(1)} \mathcal E^D_{\le N-6}[\psi]^2 \,d\tau  \right)^\frac12 \left( \int_{\tau_i}^{\tau_{i+1} } {}^{(1-\delta)}\mathcal E^D_{\le N-1}[\psi]^2\,d\tau  \right)^\frac12 \\
    & \lesssim C^3\varepsilon^3 (\tau_i)^{-\frac12+\delta}.
\end{align*}
For the first integral, we use the integration by parts in the $\underline L$-direction and get
\begin{align*}
    & \int_{\mathcal D^{\tau_{i+1}}_{\tau_i} } r^{2-\delta} |\nabla_{\underline L}\mathcal L_Z^{\le N-1}\alpha| |\mathscr L_Z^{\le N-8}\psi_-| |\nabla_L \mathscr L_Z^{\le N-1}\psi_-|\,dV \\
    & \lesssim \int_{\substack{ u=\tau_{i+1}-R \\ v\ge \tau_{i+1}+R } } r^{2-\delta} |\mathcal L_Z^{\le N-1}\alpha| |\mathscr L_Z^{\le N-8}\psi_-| |\nabla_L \mathscr L_Z^{\le N-1}\psi_-|\,d\sigma + \int_{\substack{ u=\tau_{i}-R \\ v\ge \tau_{i}+R } } r^{2-\delta} |\mathcal L_Z^{\le N-1}\alpha| |\mathscr L_Z^{\le N-8}\psi_-| |\nabla_L \mathscr L_Z^{\le N-1}\psi_-|\,d\sigma \\
    & \qquad + \int_{r=R } r^{2-\delta} |\mathcal L_Z^{\le N-1}\alpha| |\mathscr L_Z^{\le N-8}\psi_-| |\nabla_L \mathscr L_Z^{\le N-1}\psi_-|\,d\sigma \\
    & \qquad + \int_{\mathcal D^{\tau_{i+1}}_{\tau_i} } r^{2-\delta} |\mathcal L_Z^{\le N-1}\alpha| |\nabla_{\underline L}\mathscr L_Z^{\le N-8}\psi_-| |\nabla_L \mathscr L_Z^{\le N-1}\psi_-|\,dV \\
    & \qquad + \int_{\mathcal D^{\tau_{i+1}}_{\tau_i} } r^{2-\delta} |\mathcal L_Z^{\le N-1}\alpha| |\mathscr L_Z^{\le N-8}\psi_-| |\nabla_{\underline L} \nabla_L \mathscr L_Z^{\le N-1}\psi_-|\,dV.
\end{align*}
Then the control of the boundary terms on the null hypersurfaces and the time-like surface $\{r=R\}$ is obvious. Therefore we focus on the spacetime integrals.

\subsubsection{Control of the spacetime integral without derivative loss}
For the first spacetime integral, using Proposition \ref{prop-null-decomp-dirac}, we write
\begin{align*}
    & \int_{\mathcal D^{\tau_{i+1}}_{\tau_i} } r^{2-\delta} |\mathcal L_Z^{\le N-1}\alpha| |\nabla_{\underline L}\mathscr L_Z^{\le N-8}\psi_-| |\nabla_L \mathscr L_Z^{\le N-1}\psi_-|\,dV \\
    & \lesssim \int_{\mathcal D^{\tau_{i+1}}_{\tau_i} } r^{1-\delta} |\mathcal L_Z^{\le N-1}\alpha| |\mathscr L_Z^{\le N-7}\psi_+| |\nabla_L \mathscr L_Z^{\le N-1}\psi_-|\,dV \\
    & \qquad + \int_{\mathcal D^{\tau_{i+1}}_{\tau_i} } r^{2-\delta} |\mathcal L_Z^{\le N-1}\alpha| |\mathscr L_Z^{\le N-8}(\underline\alpha\cdot\psi_-)| |\nabla_L \mathscr L_Z^{\le N-1}\psi_-|\,dV \\
    & \qquad + \int_{\mathcal D^{\tau_{i+1}}_{\tau_i} } r^{2-\delta} |\mathcal L_Z^{\le N-1}\alpha| |\mathscr L_Z^{\le N-8}((\rho,\sigma)\cdot\psi_\pm)| |\nabla_L \mathscr L_Z^{\le N-1}\psi_-|\,dV \\
    & \qquad + \int_{\mathcal D^{\tau_{i+1}}_{\tau_i} } r^{2-\delta} |\mathcal L_Z^{\le N-1}\alpha| | [ \gamma^\mu\nabla_\mu , \mathscr L_Z^{\le N-8}]\psi| |\nabla_L \mathscr L_Z^{\le N-1}\psi_-|\,dV .
\end{align*}
Then we use the integrated local energy decay estimates \eqref{iled-a} for $\alpha$ to get
\begin{align*}
    &  \int_{\mathcal D^{\tau_{i+1}}_{\tau_i} } r^{1-\delta} |\mathcal L_Z^{\le N-1}\alpha| |\mathscr L_Z^{\le N-7}\psi_+| |\nabla_L \mathscr L_Z^{\le N-1}\psi_-|\,dV  \\
    & \lesssim \int_{\tau_i}^{\tau_{i+1} } r^{\frac12-\delta}  \mathscr E^T_{\le N-1}[\alpha]  \mathcal E^D_{\le N-5}[\psi] \mathcal E^D_{\le N-1}[\psi] \,d\tau \\
    & \lesssim \left( \int_{\tau_i}^{\tau_{i+1} } {}^{(\delta)} \mathscr E^T_{\le N-1}[\alpha]^2  \, {}^{(1)} \mathcal E^D_{\le N-5}[\psi]^2\,d\tau  \right)^\frac12 \left( \int_{\tau_i}^{\tau_{i+1} } {}^{(1-\delta)} \mathcal E^D_{\le N-1}[\psi]^2\,d\tau \right)^\frac12 \\
    & \lesssim C^3\varepsilon^3 (\tau_i)^{-\frac12+\frac52\delta},
\end{align*}
and
\begin{align*}
    &  \int_{\mathcal D^{\tau_{i+1}}_{\tau_i} } r^{2-\delta} |\mathcal L_Z^{\le N-1}\alpha| |\mathscr L_Z^{\le N-8}(\underline\alpha\cdot\psi_-)| |\nabla_L \mathscr L_Z^{\le N-1}\psi_-|\,dV  \\
    & \lesssim \int_{\tau_i}^{\tau_{i+1} } r^{1-\delta}\mathscr E^T_{\le N-1}[\alpha]  \mathcal E^T_{\le N-6}[\underline\alpha] \mathcal E^D_{\le N-6}[\psi] \mathcal E^D_{\le N-1}[\psi]\,d\tau \\
    & \lesssim C\varepsilon \left( \int_{\tau_i}^{\tau_{i+1} } r^{-\delta}\, {}^{(\delta)} \mathscr E^T_{\le N-1}[\alpha]^2 \, {}^{(1)}  \mathcal E^T_{\le N-6}[\underline\alpha]^2 \, {}^{(0)}\mathcal E^D_{\le N-6}[\psi]^2 \,d\tau  \right)^\frac12  \left( \int_{\tau_i}^{\tau_{i+1} } {}^{(1-\delta)}\mathcal E^D_{\le N-1}[\psi]^2\,d\tau  \right)^\frac12 \\
    & \lesssim C^4\varepsilon^4 (\tau_{i})^{-\frac34+\frac52\delta}.
\end{align*}
For the integral involving $(\rho,\sigma)$, an application of the Sobolev embedding along the outgoing null gives the derivatve $\nabla_L$, which yields
\begin{align*}
    &  \int_{\mathcal D^{\tau_{i+1}}_{\tau_i} } r^{2-\delta} |\mathcal L_Z^{\le N-1}\alpha| |\mathscr L_Z^{\le N-8}((\rho,\sigma)\cdot\psi_\pm)| |\nabla_L \mathscr L_Z^{\le N-1}\psi_-|\,dV \\
    & \lesssim \int_{\tau_i}^{\tau_{i+1}} r^{-1-\delta} \mathscr E^T_{\le N-1}[\alpha] \mathscr E^T_{\le N-6}[\rho] \mathcal E^D_{\le N-6}[\psi] \mathcal E^D_{\le N-1}[\psi]\,d\tau \\
    & \qquad + \int_{\tau_i}^{\tau_{i+1}} r^{-\delta} \mathscr E^T_{\le N-1}[\alpha] \mathcal E^T_{\le N-6}[\alpha] \mathcal E^D_{\le N-6}[\psi] \mathcal E^D_{\le N-1}[\psi]\,d\tau \\
    & \qquad +\int_{\tau_i}^{\tau_{i+1}} r^{-\delta} \mathscr E^T_{\le N-1}[\alpha] \mathcal E^D_{\le N-5}[\psi]^2  \mathcal E^D_{\le N-6}[\psi] \mathcal E^D_{\le N-1}[\psi]\,d\tau .
\end{align*}
Then the integrated local energy decay estimates \eqref{iled-a} gives the required decay, which we omit the details. The integral containing the commutator terms is also obvious.


\subsubsection{Control of the spacetime integral with derivative loss}
For the integral involving a derivative loss, we write
\begin{align*}
    & \int_{\mathcal D^{\tau_{i+1}}_{\tau_i} } r^{2-\delta} |\mathcal L_Z^{\le N-1}\alpha| |\mathscr L_Z^{\le N-8}\psi_-| |\nabla_{\underline L} \nabla_L \mathscr L_Z^{\le N-1}\psi_-|\,dV \\
    & \lesssim \int_{\mathcal D^{\tau_{i+1}}_{\tau_i} } r^{2-\delta} |\mathcal L_Z^{\le N-1}\alpha| |\mathscr L_Z^{\le N-8}\psi_-| |\slashed\nabla \nabla_L \mathscr L_Z^{\le N-1}\psi_+|\,dV \\
    & \qquad + \int_{\mathcal D^{\tau_{i+1}}_{\tau_i} } r^{2-\delta} |\mathcal L_Z^{\le N-1}\alpha| |\mathscr L_Z^{\le N-8}\psi_-| | \nabla_L \mathscr L_Z^{\le N-1}(\underline\alpha\cdot\psi_-)|\,dV \\
    & \qquad +  \int_{\mathcal D^{\tau_{i+1}}_{\tau_i} } r^{2-\delta} |\mathcal L_Z^{\le N-1}\alpha| |\mathscr L_Z^{\le N-8}\psi_-| | \nabla_L \mathscr L_Z^{\le N-1}((\rho,\sigma)\cdot\psi_\pm)|\,dV \\
    & \qquad +  \int_{\mathcal D^{\tau_{i+1}}_{\tau_i} } r^{2-\delta} |\mathcal L_Z^{\le N-1}\alpha| |\mathscr L_Z^{\le N-8}\psi_-| | \nabla_L [ \gamma^\mu\nabla_\mu, \mathscr L_Z^{\le N-1}] \psi|\,dV.
\end{align*}
 For the first integral we use the integration by parts with the angular variables to get
\begin{align*}
    & \int_{\mathcal D^{\tau_{i+1}}_{\tau_i} } r^{2-\delta} |\mathcal L_Z^{\le N-1}\alpha| |\mathscr L_Z^{\le N-8}\psi_-| |\slashed\nabla \nabla_L \mathscr L_Z^{\le N-1}\psi_-|\,dV \\
    & \lesssim \int_{\mathcal D^{\tau_{i+1}}_{\tau_i} } r^{2-\delta} |\slashed\nabla\mathcal L_Z^{\le N-1}\alpha| |\mathscr L_Z^{\le N-8}\psi_-| | \nabla_L \mathscr L_Z^{\le N-1}\psi_-|\,dV \\
    &  \qquad + \int_{\mathcal D^{\tau_{i+1}}_{\tau_i} } r^{2-\delta} |\mathcal L_Z^{\le N-1}\alpha| |\slashed\nabla\mathscr L_Z^{\le N-8}\psi_-| | \nabla_L \mathscr L_Z^{\le N-1}\psi_-|\,dV.
\end{align*}
Then the remaining task is vey similar to the control of the spacetime integral without derivative loss. We omit the details.

 For the additional nonlinearity $\nabla_L \mathscr L_Z^{\le N-1}(\underline\alpha\cdot\psi_-)$, we restrict ourselves into the case when the resulting product consists of a product of higher-order component $\le N-1$ and lower-order component $\le N-8$. Indeed, the middle case, where the derivative order is around $\le N-5$ can be controlled via the symmetric argument of using $L^4$-Sobolev embedding. Concerning low-high or high-low interactions, if $\nabla_L$ acts on a higher-order factor, we apply the Sobolev inequality to a lower-order factor, which produces an additional decay. On the other hand, if $\nabla_{L}$ acts on a lower-order component, one cannot expect further decay, since we are not allowed to apply the Sobolev embedding to a higher-order factor. Thus it suffices to consider the integrals as follows:
\begin{align*}
    & \int_{\tau_i}^{\tau_{i+1} } r^{1-\delta} \mathscr E^T_{\le N-1}[\alpha]  \mathcal E^D_{\le N-6}[\psi] \|\mathcal L_Z^{\le N-1}\underline\alpha\|_{L^2(\Sigma_\tau)}\mathcal E^D_{\le N-6}[\psi]\,d\tau \\
    & \qquad + \int_{\tau_i}^{\tau_{i+1} } r^{1-\delta} \mathscr E^T_{\le N-1}[\alpha]  \mathcal E^D_{\le N-6}[\psi] \|\mathcal L_Z^{\le N-1}\psi_-\|_{L^2(\Sigma_\tau)}\mathcal E^D_{\le N-6}[\underline\alpha]\,d\tau .
\end{align*}
Then
\begin{align*}
    & \int_{\tau_i}^{\tau_{i+1} } r^{-\delta} \, {}^{(1+\delta)}\mathscr E^T_{\le N-1}[\alpha]\,{}^{(1-\delta)}  \mathcal E^D_{\le N-6}[\psi] \|\mathcal L_Z^{\le N-1}\underline\alpha\|_{L^2(\Sigma_\tau)}\mathcal E^D_{\le N-6}[\psi]\,d\tau \\
    & \lesssim C\varepsilon \tau^{\frac32\delta} \left( \int_{\tau_i}^{\tau_{i+1} }{}^{(1-\delta)}\mathcal E^T_{\le N-6}[\psi]^2\,d\tau  \right)^\frac12 \left( \int_{\tau_i}^{\tau_{i+1} } r^{-1-2\delta} \, {}^{(1)}\mathcal E^D_{\le N-6}[\psi]^2 \|\mathcal L_Z^{\le N-1}\underline\alpha\|_{L^2(\Sigma_\tau)}^2\,d\tau   \right)^\frac12 \\
    & \lesssim C^3\varepsilon^3 \tau^{\frac32\delta} \left( \int_{\tau_i}^{\tau_{i+1} } r^{-1-2\delta} \tau^{-1+\delta}  |\mathcal L_Z^{\le N-1}\underline\alpha|^2\,dV   \right)^\frac12 \\
    & \lesssim C^3\varepsilon^3 \tau^{\frac32\delta} \left( \int_{\tau_i}^{\tau_{i+1} } r^{-1-\delta}   |\mathcal L_Z^{\le N-1}\underline\alpha|^2\,dV   \right)^\frac12 \\
    & \lesssim C^4\varepsilon^4 (\tau_i)^{-\frac14+2\delta},
\end{align*}
and
\begin{align*}
    & \int_{\tau_i}^{\tau_{i+1} } r^{1-\delta} \|\mathcal L_Z^{\le N-1}\alpha\|_{L^2(\Sigma_\tau)} \mathcal E^D_{\le N-6}[\psi] \|\mathcal L_Z^{\le N-1}\psi_-\|_{L^2(\Sigma_\tau)}\mathcal E^T_{\le N-6}[\underline\alpha]\,d\tau \\
    & \lesssim \left( \int_{\tau_i}^{\tau_{i+1} }{}^{(\delta)}\mathscr E^T_{\le N-1}[\alpha]^2\,d\tau  \right)^\frac12 \left( \int_{\tau_i}^{\tau_{i+1} } r^{-3\delta} \, {}^{(1)}\mathcal E^D_{\le N-6}[\underline\alpha]^2\, {}^{(1)} \mathcal E^D_{\le N-6}[\psi]^2 \|\mathscr L_Z^{\le N-1}\psi_-\|_{L^2(\Sigma_\tau)}^2\,d\tau   \right)^\frac12 \\
    & \lesssim C^3\varepsilon^3 \tau^{\frac32\delta} \left( \int_{\tau_i}^{\tau_{i+1} } r^{-3\delta} \tau^{-\frac32+5\delta}  |\mathscr L_Z^{\le N-1}\psi_-|^2\,dV   \right)^\frac12 \\
    & \lesssim C^4\varepsilon^4 (\tau_i)^{-\frac14+4\delta  +\epsilon}.
\end{align*}
\begin{rem}
    It is worthwhile to remark that for the control of the weighted energy for the spinor fields, we only use the energy ${}^{(2-\delta)} \mathcal E^T_{\le N-6}[\alpha]$ and ${}^{(1)}\mathcal E^T_{\le N-4}[(\alpha,\underline\alpha)] $.
\end{rem}

\section{Boundedness and improved decay of the energy for the tensor fields}\label{sec:bdd-weight-tensor}
In this section, we prove the boundedness of the weighted energy for the tensor fields $F$ and then establish improved decay estimates.
\subsection{Boundedness of the wave energy for $\alpha,\underline\alpha$}\label{bdd-weight-wave-aab}\label{subsec:bdd-weight-aab}
We are concerned with the integrals:
\begin{align}
     \int_{\mathcal D^{\tau_{i+1}}_{\tau_i} } r^{2-\delta} \nabla_{L} \mathcal L_Z^{\le k} \slashed J \cdot \nabla_{L}\mathcal L_Z^{\le k}\alpha\,dV ,
\end{align}
and
\begin{align}
     \int_{\mathcal D^{\tau_{i+1}}_{\tau_i} } r^{2-\delta} \nabla_{\underline L} \mathcal L_Z^{\le k} \slashed J \cdot \nabla_L\mathcal L_Z^{\le k}\underline\alpha\,dV,
\end{align}
while the integrals involving the angular derivatives are rather obvious.
We recall the bootstrap assumptions:
\begin{align}
\boxed{
\begin{aligned}
{}^{(2-\delta)}\mathcal E^T_{\le N-2}[\alpha]^2 \le C^2\varepsilon^2 \tau^{3\delta} , \ {}^{(2-\delta)}\mathcal E^T_{\le N-2}[\underline\alpha]^2 \le C^2\varepsilon^2 \tau^{2\delta} \\
{}^{(1)}\mathcal E^T_{\le N-4}[\alpha]^2 \le C^2\varepsilon^2 \tau^{-\frac12+4\delta}, \ {}^{(1)}\mathcal E^T_{\le N-4}[\underline\alpha]^2 \le C^2\varepsilon^2 \tau^{-\frac12+3\delta},
\end{aligned}
}
\end{align}
and
\begin{align}
\boxed{
    \begin{aligned}
     {}^{(1+\delta)}   \mathscr E^T_{\le N}[\alpha]^2(\tau) \le C^2\varepsilon^2 \tau^{3\delta}, \ {}^{(0)}\mathscr F^T_{\le N}[\underline\alpha]^2(v,\tau,2\tau) \le C^2\varepsilon^2 \tau^{\delta}, \\
      {}^{(\delta)}   \mathscr E^T_{\le N}[\alpha]^2(\tau) \le C^2\varepsilon^2 , \ {}^{(-1+\delta)}\mathscr F^T_{\le N}[\underline\alpha]^2(v,\tau,2\tau) \le C^2\varepsilon^2 .
    \end{aligned}
    }
\end{align}
We shall prove the following:
\begin{prop}
    Under the bootstrap assumptions, we have
    \begin{align}
       \left|  \int_{\mathcal D^{\tau_{i+1}}_{\tau_i} } r^{2-\delta} \nabla_{L} \mathcal L_Z^{\le k} \slashed J \cdot \nabla_{L}\mathcal L_Z^{\le k}\alpha\,dV\right| \lesssim \begin{cases}
           C^3\varepsilon^3 \tau_i^{-\delta}, \quad k\le N-3, \\
           C^3\varepsilon^3 \tau_i^{\frac52\delta}, \quad k \le N-2,
       \end{cases}
    \end{align}
    and
    \begin{align}
        \begin{aligned}
            \left|  \int_{\mathcal D^{\tau_{i+1}}_{\tau_i} } r^{2-\delta} \nabla_{\underline L} \mathcal L_Z^{\le k} \slashed J \cdot \nabla_L\mathcal L_Z^{\le k}\underline\alpha\,dV \right| \lesssim \begin{cases}
                C^3\varepsilon^3 \tau_i^{-\delta}, \quad k\le N-3, \\
                C^3\varepsilon^3 \tau_i^{\frac32\delta}, \quad k \le N-2.
            \end{cases}
        \end{aligned}
    \end{align}
    Furthermore, we have the following improved decay:
     \begin{align}
       \left|  \int_{\mathcal D^{\tau_{i+1}}_{\tau_i} } r^{} \nabla_{L} \mathcal L_Z^{\le N-2} \slashed J \cdot \nabla_{L}\mathcal L_Z^{\le N-2}\alpha\,dV\right| \lesssim C^3\varepsilon^3 \tau_i^{-\frac12+\delta},
    \end{align}
    and
    \begin{align}
        \begin{aligned}
            \left|  \int_{\mathcal D^{\tau_{i+1}}_{\tau_i} } r^{} \nabla_{\underline L} \mathcal L_Z^{\le N-2} \slashed J \cdot \nabla_L\mathcal L_Z^{\le N-2}\underline\alpha\,dV \right| \lesssim C^3\varepsilon^3 \tau_i^{-\frac12+\delta},
        \end{aligned}
    \end{align}
    so that we can obtain an improved decay for the weighted wave energy:
    \begin{align}
        \begin{aligned}
            {}^{(1)}\mathcal E^T_{\le N-2}[\alpha]^2(\tau) \le C^2\varepsilon^2 \tau^{-\frac12+\delta}, \  {}^{(1)}\mathcal E^T_{\le N-2}[\underline\alpha]^2(\tau) \le C^2\varepsilon^2 \tau^{-\frac12+\delta}.
        \end{aligned}
    \end{align}
\end{prop}
As we have done in the previous section, it is enough to consider the case when the derivative of the Dirac current $\mathcal L_Z^{\le N-2}\slashed J$ yields a higher-order spinor and a lower-order spinor. Now if the derivative $\nabla_L$ acts on a higer-order spinor, then one simply uses the $L^\infty$-Sobolev embedding to a lower-order spinor and exploit decay from both spinors. Therefore, for the control of the weighted wave energy ${}^{(2-\delta)}\mathcal E_{\le N-2}[\alpha]$, it suffices to consider the case when the derivative $\nabla_L$ acts on a lower-order spinor, in which case one can only apply $L^4$-Sobolev inequality, which does not produce any gain of $r$-weight.

On the other hand, concerning the control of ${}^{(2-\delta)}\mathcal E_{\le N-2}[\underline\alpha]$, as the derivative $\nabla_{\underline L}$ acts on $\psi_-$, one is able to take advantage of the null decomposition for the spinor fields. Thus the most serious interaction for the weighted wave energy for $\underline\alpha$ is when the derivative $\nabla_{\underline L}$ acts on a higher-order spinor, which is $\psi_+$-component.

By the above observation, it suffices to consider the following integrals:
\begin{align}
    \int_{\mathcal D^{\tau_{i+1}}_{\tau_i} } r^{2-\delta} \nabla_L\psi_\pm \cdot \mathscr L_Z^{\le k}\psi_\mp  \cdot \nabla_L \mathcal L_Z^{\le k}\alpha\,dV, \\
     \int_{\mathcal D^{\tau_{i+1}}_{\tau_i} } r^{2-\delta} \psi_-\cdot \nabla_{\underline L}\mathscr L_Z^{\le k}\psi_+ \cdot \nabla_L \mathcal L_Z^{\le k}\underline\alpha\,dV.
\end{align}
We see that
\begin{align*}
    & \int_{\mathcal D^{\tau_{i+1}}_{\tau_i} } r^{2-\delta} \nabla_L \mathscr L_Z^{\le N-8}\psi_\pm \cdot \mathscr L_Z^{\le N-3}\psi_\mp  \cdot \nabla_L \mathcal L_Z^{\le N-3}\alpha\,dV  \\
    & \lesssim \int_{\tau_i}^{\tau_{i+1}} r^{\frac32-\delta} \mathcal E^D_{\le N-6}[\psi] \mathcal E^D_{\le N-1}[\psi] \mathcal E^T_{\le N-3}[\alpha]\,d\tau \\
    & \lesssim \left( \int_{\tau_i}^{\tau_{i+1}} {}^{(1)} \mathcal E^D_{\le N-6}[\psi]^2\,{}^{(1-\delta)} \mathcal E^T_{\le N-3}[\alpha]^2\,d\tau \right)^\frac12 \left( \int_{\tau_i}^{\tau_{i+1}} {}^{(1-\delta)} \mathcal E^D_{\le N-1}[\psi]^2\,d\tau \right)^\frac12 \\
    & \lesssim {}^{(1)} \mathcal E^D[\psi](\tau_i)  \left( \int_{\tau_i}^{\tau_{i+1}}  \,{}^{(1-\delta)}\mathcal E^T_{\le N-3}[\alpha]^2\,d\tau \right)^\frac12 \left( \int_{\tau_i}^{\tau_{i+1}} {}^{(1-\delta)} \mathcal E^D_{\le N-1}[\psi]^2\,d\tau \right)^\frac12 \\
    & \lesssim C^3\varepsilon^3 \tau_i^{-\frac12+\frac52\delta} ,
\end{align*}
and
\begin{align*}
    & \int_{\mathcal D^{\tau_{i+1}}_{\tau_i} } r^{2-\delta} \nabla_L\mathscr L_Z^{\le N-8}\psi_\pm \cdot \mathscr L_Z^{\le N-2}\psi_\mp  \cdot \nabla_L \mathcal L_Z^{\le N-2}\alpha\,dV \\
    & \lesssim \int_{\tau_i}^{\tau_{i+1}}  r^{2-\delta} \mathcal E^D_{\le N-6}[\psi] \mathcal E^D_{\le N-1}[\psi] \mathcal E^T_{\le N-2}[\alpha]\,d\tau \\
    & \lesssim \left( \int_{\tau_i}^{\tau_{i+1}} r \, {}^{(1)}\mathcal E^D_{\le N-6}[\psi]^2 \, {}^{(1-\delta)} \mathcal E^T_{\le N-2}[\alpha]^2\,d\tau  \right)^\frac12 \left( \int_{\tau_i}^{\tau_{i+1}} {}^{(1-\delta)}\mathcal E^D_{\le N-1}[\psi]^2\,d\tau   \right)^\frac12 \\
    & \lesssim \tau_i^\frac12 \,{}^{(1)}\mathcal E^D_{\le N-6}[\psi](\tau) \tau_i^{\frac\delta2}  \left( \int_{\tau_i}^{\tau_{i+1}}{}^{(1-\delta)} \mathcal E^T_{\le N-2}[\alpha]^2\,d\tau  \right)^\frac12 \\
    & \lesssim C^2\varepsilon^2 \tau_i^{\delta}  \left( \int_{\tau_i}^{\tau_{i+1}}{}^{(1-\delta)}  \mathcal E^T_{\le N-2}[\alpha]^2\,d\tau  \right)^\frac12 \\
    & \lesssim C^3\varepsilon^3 \tau_i^{\frac52\delta}.
\end{align*}
Now we control the integral:
\begin{align*}
    \int_{\mathcal D^{\tau_{i+1} }_{\tau_i} } r^{2-\delta} \mathscr L_Z^{\le N-8}\psi_- \cdot \nabla_{\underline L} \mathscr L_Z^{\le N-2}\psi_+ \cdot \nabla_L \mathcal L_Z^{\le N-2}\underline\alpha\,dV.
\end{align*}
We use the integration by parts with respect to the derivative $\nabla_{\underline L}$:
\begin{align}
    \begin{aligned}
        & \int_{\mathcal D^{\tau_{i+1} }_{\tau_i} } r^{2-\delta} \mathscr L_Z^{\le N-8}\psi_- \cdot  \mathscr L_Z^{\le N-2}\psi_+ \cdot \nabla_L \mathcal L_Z^{\le N-2}\underline\alpha\,dV \\
         = & \int_{\substack{u=\tau_{i+1}-R \\ v\ge \tau_{i+1}+R } }  r^{2-\delta} \mathscr L_Z^{\le N-8}\psi_- \cdot \nabla_{\underline L} \mathscr L_Z^{\le N-2}\psi_+ \cdot \nabla_L \mathcal L_Z^{\le N-2}\underline\alpha\,d\sigma - \int_{\substack{u=\tau_{i}-R \\ v\ge \tau_{i}+R } }  r^{2-\delta}\mathscr L_Z^{\le N-8} \psi_- \cdot \mathscr L_Z^{\le N-2}\psi_+ \cdot \nabla_L \mathcal L_Z^{\le N-2}\underline\alpha\,d\sigma \\
         & + \int_{r=R }  r^{2-\delta} \mathscr L_Z^{\le N-8}\psi_- \cdot  \mathscr L_Z^{\le N-2}\psi_+ \cdot \nabla_L \mathcal L_Z^{\le N-2}\underline\alpha\,d\sigma  \\
         & - \int_{\mathcal D^{\tau_{i+1} }_{\tau_i} } r^{2-\delta} \nabla_{\underline L} \mathscr L_Z^{\le N-8}\psi_- \cdot \mathscr L_Z^{\le N-2}\psi_+ \cdot \nabla_L \mathcal L_Z^{\le N-2}\underline\alpha\,dV - \int_{\mathcal D^{\tau_{i+1} }_{\tau_i} } r^{2-\delta}\mathscr L_Z^{\le N-8} \psi_- \cdot  \mathscr L_Z^{\le N-2}\psi_+ \cdot \nabla_{\underline L}\nabla_L \mathcal L_Z^{\le N-2}\underline\alpha\,dV.
    \end{aligned}
\end{align}
For the integral along the null hypersurface, we use the H\"older inequality to get the product $L^4\times L^4\times L^2$ and apply the weighted Sobolev inequality to get
\begin{align*}
   & \left| \int_{\substack{u=\tau_{i+1}-R \\ v\ge \tau_{i+1}+R } }  r^{2-\delta}\mathscr L_Z^{\le N-8} \psi_- \cdot \nabla_{\underline L} \mathscr L_Z^{\le N-2}\psi_+ \cdot \nabla_L \mathcal L_Z^{\le N-2}\underline\alpha\,d\sigma  \right| \\
    & \lesssim {}^{(0)} \mathcal E^D_{\le N-6}[\psi](\tau_{i+1}) {}^{(2-\delta)}\mathcal E^D_{\le N-1}[\psi](\tau_{i+1})\,{}^{(1)} \mathcal E^T_{\le N-2}[\underline\alpha](\tau_{i+1}) \\
    & \lesssim C^3\varepsilon^3 \tau^{-\frac12+\delta} \mathcal E^T_{\le N-2}[\underline\alpha](\tau_{i+1}) \\
    & \lesssim C^3\varepsilon^3 \tau_i^{-\frac12+2\delta} .
\end{align*}
The spacetime integral where the derivative $\nabla_{\underline L}$ acts on the lower-order component $\mathscr L_Z^{\le N-8}\psi_-$ is obvious. We focus on the spacetime integral where the derivative $\nabla_{\underline L}$ acts on the higher-order factor $\nabla_{L}\mathcal L_Z^{\le N-2}\underline\alpha$:
\begin{align}
    \begin{aligned}
        & \left| \int_{\mathcal D^{\tau_{i+1} }_{\tau_i} } r^{2-\delta}\mathscr L_Z^{\le N-8} \psi_- \cdot  \mathscr L_Z^{\le N-2}\psi_+ \cdot \nabla_{\underline L}\nabla_L \mathcal L_Z^{\le N-2}\underline\alpha\,dV \right| \\
        & \lesssim \int_{\mathcal D^{\tau_{i+1} }_{\tau_i} } r^{2-\delta} |\slashed\nabla (\mathscr L_Z^{\le N-8}\psi_- \cdot  \mathscr L_Z^{\le N-2}\psi_+)| | \slashed\nabla \mathcal L_Z^{\le N-2}\underline\alpha|\,dV \\
        & \qquad + \int_{\mathcal D^{\tau_{i+1} }_{\tau_i} } r^{1-\delta} |\mathscr L_Z^{\le N-8}\psi_-| | \mathscr L_Z^{\le N-2}\psi_+ || \nabla_L \mathcal L_Z^{\le N-2}\underline\alpha|\,dV \\
        & \qquad + \int_{\mathcal D^{\tau_{i+1} }_{\tau_i} } r^{1-\delta} |\mathscr L_Z^{\le N-8}\psi_-| | \mathscr L_Z^{\le N-2}\psi_+ || \nabla_{\underline L} \mathcal L_Z^{\le N-2}\underline\alpha|\,dV \\
        & \qquad + \int_{\mathcal D^{\tau_{i+1} }_{\tau_i} } r^{-1-\delta} |\mathscr L_Z^{\le N-8}\psi_-| | \mathscr L_Z^{\le N-2}\psi_+ ||  \mathcal L_Z^{\le N-2}(\alpha,\underline\alpha,\rho,\sigma)|\,dV \\
        & \qquad + \int_{\mathcal D^{\tau_{i+1} }_{\tau_i} } r^{2-\delta} |\mathscr L_Z^{\le N-8}\psi_-| | \mathscr L_Z^{\le N-2}\psi_+ || \mathcal L_Z^{\le N-2}(\nabla_{\underline L}\slashed{J}+\slashed\nabla J_{\underline L} )|\,dV.
    \end{aligned}
\end{align}
Then
\begin{align*}
   & \int_{\mathcal D^{\tau_{i+1} }_{\tau_i} } r^{2-\delta} |\slashed\nabla (\mathscr L_Z^{\le N-8}\psi_- \cdot  \mathscr L_Z^{\le N-2}\psi_+)| | \slashed\nabla \mathcal L_Z^{\le N-2}\underline\alpha|\,dV \\
   & \lesssim \int_{\mathcal D^{\tau_{i+1} }_{\tau_i} } r^{1-\delta} | \mathscr L_Z^{\le N-7}\psi_-|
|\mathscr L_Z^{\le N-2}\psi_+| | \slashed\nabla \mathcal L_Z^{\le N-2}\underline\alpha|\,dV \\
& \qquad + \int_{\mathcal D^{\tau_{i+1} }_{\tau_i} } r^{2-\delta} |\mathscr L_Z^{\le N-8}\psi_-| |\slashed\nabla  \mathscr L_Z^{\le N-2}\psi_+| | \slashed\nabla \mathcal L_Z^{\le N-2}\underline\alpha|\,dV  \\
& \lesssim \int_{\tau_i}^{\tau_{i+1}}r^{\frac32-\delta} \mathcal E^D_{\le N-5}[\psi] \mathcal E^D_{\le N-2}[\psi] \mathcal E^T_{\le N-2}[\underline\alpha]\,d\tau \\
& \lesssim \left(\int_{\tau_i}^{\tau_{i+1}} {}^{(1)}\mathcal E^D_{\le N-5}[\psi]^2 \,{}^{(1-\delta)}\mathcal E^D_{\le N-2}[\psi]^2\,d\tau \right)^\frac12 \left( \int_{\tau_i}^{\tau_{i+1}}{}^{(1-\delta)}\mathcal E^T_{\le N-2}[\underline\alpha]^2\,d\tau \right)^\frac12 \\
& \lesssim C^3\varepsilon^3 \tau_i^{-\frac12+2\delta},
\end{align*}
and
\begin{align*}
    & \int_{\mathcal D^{\tau_{i+1} }_{\tau_i} } r^{1-\delta} |\mathscr L_Z^{\le N-8}\psi_-| | \mathscr L_Z^{\le N-2}\psi_+ || \nabla_L \mathcal L_Z^{\le N-2}\underline\alpha|\,dV \\
    & \lesssim \int_{\tau_i}^{\tau_{i+1}} r^{1-\delta} \mathcal E^D_{\le N-7}[\psi]\mathcal E^D_{\le N-1}[\psi]\mathcal E^T_{\le N-2}[\underline\alpha]\,d\tau \\
    & \lesssim \left( \int_{\tau_i}^{\tau_{i+1}} {}^{(0)}\mathcal E^D_{\le N-7}[\psi]^2 \, {}^{(1-\delta)}\mathcal E^D_{\le N-1}[\psi]^2\,d\tau \right)^\frac12 \left( \int_{\tau_i}^{\tau_{i+1}}{}^{(1-\delta)}\mathcal E^T_{\le N-2}[\underline\alpha]^2\,d\tau \right)^\frac12 \\
    & \lesssim C^3\varepsilon^3 \tau_i^{-\frac12+2\delta}.
\end{align*}
For the third integral, we see that
\begin{align*}
    &  \int_{\mathcal D^{\tau_{i+1} }_{\tau_i} } r^{1-\delta} |\mathscr L_Z^{\le N-8}\psi_-| | \mathscr L_Z^{\le N-2}\psi_+ || \nabla_{\underline L} \mathcal L_Z^{\le N-2}\underline\alpha|\,dV \\
    & \lesssim \int_{\tau_i}^{\tau_{i+1} } r^{1-\delta}\mathcal E^D_{\le N-7}[\psi]\mathcal E^D_{\le N-1}[\psi] \|\mathcal L_Z^{\le N-1}\underline\alpha\|_{L^2(\Sigma_\tau)}\,d\tau \\
    & \lesssim \left( \int_{\tau_i}^{\tau_{i+1} } r^{-\delta}\,\, {}^{(1)}\mathcal E^D_{\le N-7}[\psi]^2 \|\mathcal L_Z^{\le N-1}\underline\alpha\|_{L^2(\Sigma_\tau)}^2\,d\tau  \right)^\frac12 \left( \int_{\tau_i}^{\tau_{i+1} } {}^{(1-\delta)}\mathcal E^D_{\le N-1}[\psi]^2\,d\tau  \right)^\frac12 \\
    & \lesssim C\varepsilon \tau_i^{\frac\delta2} \left( \int_{\tau_i}^{\tau_{i+1} } r^{-\delta}\tau^{-1+\delta}  \|\mathcal L_Z^{\le N-1}\underline\alpha\|_{L^2(\Sigma_\tau)}^2\,d\tau  \right)^\frac12  \\
    & \lesssim C^3\varepsilon^3 \tau_i^{\frac32\delta}.
\end{align*}
The integrals involving the curvature tensor and additional Dirac current are controlled in a similar way to the third integral.

A repetition of the above argument for the weight $r$ yields an improved decay for ${}^{(1)}\mathcal E^T_{\le N-2}[\alpha,\underline\alpha]$. Indeed, we see that
\begin{align*}
    & \int_{\mathcal D^{\tau_{i+1}}_{\tau_i} } r^{} \nabla_L\mathscr L_Z^{\le N-8}\psi_\pm \cdot \mathscr L_Z^{\le N-2}\psi_\mp  \cdot \nabla_L \mathcal L_Z^{\le N-2}\alpha\,dV \\
    & \lesssim \left( \int_{\tau_i}^{\tau_{i+1}}  \, {}^{(1)}\mathcal E^D_{\le N-6}[\psi]^2 \, {}^{(1-\delta)} \mathcal E^D_{\le N-1}[\psi]^2\,d\tau  \right)^\frac12 \left( \int_{\tau_i}^{\tau_{i+1}} {}^{(0)}\mathcal E^T_{\le N-2}[\alpha]^2\,d\tau   \right)^\frac12 \\
    & \lesssim C^2\varepsilon^2 \tau_i^{-\frac12+\delta}  \left( \int_{\tau_i}^{\tau_{i+1}}{}^{(0)} \mathcal E^T_{\le N-2}[\alpha]^2\,d\tau  \right)^\frac12 \\
    & \lesssim C^3\varepsilon^3 \tau_i^{-\frac12+\delta}.
\end{align*}
In order to obtain an improved decay for the weighted wave energy of $\underline\alpha$, it is enough to consider the following integral:
\begin{align}
    \int_{\mathcal D^{\tau_{i+1}}_{\tau_i}} |\mathscr L_Z^{\le N-8}\psi_-| |\mathscr L_Z^{\le N-2}\psi_+| |\nabla_{\underline L}\mathcal L_Z^{\le N-2}\underline\alpha| \,dV,
\end{align}
since the growth of ${}^{(2-\delta)}\mathcal E^T_{\le N-2}[\underline\alpha]$ occurs from the above integral. Then we see that
\begin{align*}
    & \int_{\mathcal D^{\tau_{i+1}}_{\tau_i}} |\mathscr L_Z^{\le N-8}\psi_-| |\mathscr L_Z^{\le N-2}\psi_+| |\nabla_{\underline L}\mathcal L_Z^{\le N-2}\underline\alpha| \,dV \\
    & \lesssim \int_{\tau_i}^{\tau_{i+1}} \mathcal E^D_{\le N-7}[\psi]\mathcal E^D_{\le N-1}[\psi] \|\mathcal L_Z^{\le N-1}\underline\alpha\|_{L^2(\Sigma_\tau)}\,d\tau \\
    & \lesssim \left( \int_{\tau_i}^{\tau_{i+1}} {}^{(1-\delta)}\mathcal E^D_{\le N-1}[\psi]^2\,d\tau \right)^\frac12 \left( \int_{\tau_i}^{\tau_{i+1}} r^{-2+\delta} \,{}^{(1)}\mathcal E^D_{\le N-7}[\psi]^2 \|\mathcal L_Z^{\le N-1}\underline\alpha\|_{L^2(\Sigma_\tau)}^2\,d\tau  \right)^\frac12 \\
    & \lesssim C^2\varepsilon^2 \tau^{-\frac12+\delta}\left( \int_{\tau_i}^{\tau_{i+1}} r^{-2+\delta}  \|\mathcal L_Z^{\le N-1}\underline\alpha\|_{L^2(\Sigma_\tau)}^2\,d\tau  \right)^\frac12 \\
    & \lesssim C^3\varepsilon^3 \tau^{-\frac12+\delta}.
\end{align*}

Then we see that
\begin{align}
    {}^{(1-\delta)}\mathcal E^T_{\le N-2}[\alpha]^2(\tau_i^*) \le (\tau_i)^{-1}( {}^{(2-\delta)}\mathcal E^T_{\le N-2}[\alpha]^2(\tau_i) + C^3\varepsilon^3 \tau^{3\delta}),
\end{align}
and
\begin{align}
    \begin{aligned}
        {}^{(1)}\mathcal E^T_{\le N-2}[\alpha]^2(\tau_{i+1}) &\le  {}^{(1)}\mathcal E^T_{\le N-2}[\alpha]^2(\tau_i^*) + C^3\varepsilon^3 \tau^{-\frac12+3\delta} \\
        & \lesssim {}^{(1-\delta)}\mathcal E^T_{\le N-2}(\tau_i^*) (\tau_i^*)^\delta+C^3\varepsilon^3 \tau^{-\frac12+3\delta}  \\
        & \lesssim C^3\varepsilon^3 \tau^{-\frac12+3\delta}.
    \end{aligned}
\end{align}
A similar argument also gives
\begin{align}
    {}^{(1)}\mathcal E^T_{\le N-2}[\underline\alpha]^2(\tau_{i+1}) \le C^2\varepsilon^2 \tau_i^{-\frac12+2\delta}.
\end{align}
We are also able to obtain a stronger decay for lower-order energy. Indeed, we have the following:
\begin{cor}
From the above improdved decay for the higher-order weighted wave energy,
we obtain the following improved decay:
\begin{align}
{}^{(1)}\mathcal E^T_{\le k}[\alpha]^2(\tau) \le \begin{cases}
    C^2\varepsilon^2 \tau^{-\frac34+\frac32\delta}, \quad k = N-3, \\
    C^2\varepsilon^2 \tau^{-1+\delta}, \quad k \le N-4,
\end{cases}
\quad {}^{(1)}\mathcal E^T_{\le k}[\underline\alpha]^2(\tau) \le \begin{cases}
    C^2\varepsilon^2 \tau^{-\frac34+\delta}, \quad k = N-3, \\
    C^2\varepsilon^2 \tau^{-1+\delta}, \quad k \le N-4.
    \end{cases}
    \end{align}
\end{cor}
\begin{proof}
    We repeat the previous argument. We are only concerned with the integral
    \begin{align*}
        \int_{\mathcal D^{\tau_{i+1}}_{\tau_i}} |\mathscr L_Z^{\le N-8}\psi_-| |\mathscr L_Z^{\le N-3}\psi_+| |\nabla_{\underline L}\mathcal L_Z^{\le N-3}\underline\alpha| \,dV.
    \end{align*}
    After the use of the integration by parts in the $\underline L$-direction, we obtain the boundary terms along the null hypersurfaces and the time-like surfaces, and the spacetime integrals, where the derivative $\nabla_{\underline L}$ acts on the lower-order component $\mathscr L_Z^{\le N-8}\psi_-$ or the higher-order component $\mathscr L_Z^{\le N-3}\psi_+$. By the H\"older inequality to get the product $L^6\times L^6\times L^6$ and the Sobolev embedding, the control of the boundary terms is obvious and gives the desired bound $C^3\varepsilon^3 \tau^{-\frac34+\delta}$. Then it is enough to deal with the spacetime integral, where the derivative $\nabla_{\underline L}$ acts on the higher-order term. By using the H\"older inequality to get the product $L^6\times L^6\times L^6$, we see that
    \begin{align*}
        & \int_{\mathcal D^{\tau_{i+1}}_{\tau_i}} |\mathscr L_Z^{\le N-8}\psi_-| |\nabla_{\underline L}\mathscr L_Z^{\le N-3}\psi_+| |\mathcal L_Z^{\le N-3}\underline\alpha| \,dV \\
        & \lesssim \int_{\tau_i}^{\tau_{i+1}} \mathcal E^D_{\le N-7}[\psi]\mathcal E^D_{\le N-1}[\psi] \mathcal E^T_{\le N-2}[\underline\alpha]\,d\tau \\
        & \lesssim \left( \int_{\tau_i}^{\tau_{i+1}} {}^{(0)}\mathcal E^D_{\le N-7}[\psi]^2\,{}^{(0)}\mathcal E^D_{\le N-1}[\psi]^2\,d\tau \right)^\frac12\left(\int_{\tau_i}^{\tau_{i+1}} {}^{(0)}\mathcal E^T_{\le N-2}[\underline\alpha]^2\,d\tau \right)^\frac12 \\
        & \lesssim C^3\varepsilon^3 \tau_i^{-\frac34+\delta}.
    \end{align*}
    A further improvement for lower-order $\le N-4$ follows a simple repetition of the above argument, which we omit the details.
\end{proof}

\subsection{Boundedness of the energy for the tensor fields}\label{subsec:bdd-weight-tensor}
We also need to control the integral
\begin{align}
    \begin{aligned}
        \int_{\mathcal D^{\tau_{i+1}}_{\tau_i} } r^{1+\delta} \mathcal L_Z^{\le N}\alpha \cdot \mathcal L_Z^{\le N}J_L \,dV, \quad
        \int_{\mathcal D^{\tau_{i+1}}_{\tau_i} }  \mathcal L_Z^{\le N}\underline\alpha \cdot \mathcal L_Z^{\le N}J_{\underline L} \,dV.
    \end{aligned}
\end{align}
\begin{prop}
    Under the bootstrap assumptions, we have
    \begin{align}
        \begin{aligned}
             \int_{\mathcal D^{\tau_{i+1}}_{\tau_i} } r^{1+\delta} \mathcal L_Z^{\le k}\alpha \cdot \mathcal L_Z^{\le k}J_L \,dV \lesssim \begin{cases}
                 C^3\varepsilon^3 \tau_i^{\frac52\delta+\epsilon}, \quad k \le N, \\
                 C^3\varepsilon^3 \tau_i^{-\frac34+\frac52\delta} , \quad k \le N-3,
             \end{cases}
        \end{aligned}
    \end{align}
    and
    \begin{align}
        \begin{aligned}
            \int_{\mathcal D^{\tau_{i+1}}_{\tau_i} }  \mathcal L_Z^{\le k}\underline\alpha \cdot \mathcal L_Z^{\le k}J_{\underline L} \,dV \lesssim  \begin{cases}
                C^3\varepsilon^3 \tau_i^{\frac\delta2+\epsilon} , \quad k \le N, \\
                C^3\varepsilon^3 \tau_i^{-\frac34+\delta} , \quad k \le N-3.
            \end{cases}
        \end{aligned}
    \end{align}
    Furthermore, we have the following improved decay:
    \begin{align}
         \int_{\mathcal D^{\tau_{i+1}}_{\tau_i} } r^{\delta} \mathcal L_Z^{\le k}\alpha \cdot \mathcal L_Z^{\le k}J_L \,dV \lesssim \begin{cases} C^3\varepsilon^3 \tau^{-\frac12+\frac52\delta}, \quad  k \le N, \\
         C^3\varepsilon^3 \tau^{-\frac34+\delta},	\quad k \le N-3
 \end{cases}
    \end{align}
    and
    \begin{align}
        \int_{\mathcal D^{\tau_{i+1}}_{\tau_i} } r^{-1+\delta} \mathcal L_Z^{\le N}\underline\alpha \cdot \mathcal L_Z^{\le N}J_{\underline L} \,dV  \lesssim C^3\varepsilon^3 \tau_i^{-\frac12+\delta}.
    \end{align}
\end{prop}
In fact, it suffices to consider the integrals of the form:
\begin{align}
     \int_{\mathcal D^{\tau_{i+1}}_{\tau_i} } r^{1+\delta} \mathcal L_Z^{\le N}\alpha \cdot \mathscr L_Z^{\le N}\psi_- \cdot \psi_-\,dV,
\end{align}
and
\begin{align}
     \int_{\mathcal D^{\tau_{i+1}}_{\tau_i} }  \mathcal L_Z^{\le N}\underline\alpha \cdot \mathscr L_Z^{\le N}\psi_+ \cdot \psi_+\,dV.
\end{align}
We see that
\begin{align*}
    &  \int_{\mathcal D^{\tau_{i+1}}_{\tau_i} } r^{1+\delta} \mathcal L_Z^{\le N}\alpha \cdot \mathscr L_Z^{\le N}\psi_- \cdot \psi_-\,dV \\
    & \lesssim \int_{\tau_i }^{\tau_{i+1} } r^{\frac12+\delta} \|\mathcal L_Z^{\le N}\alpha\|_{L^2(\Sigma_\tau)} \|\mathscr L_Z^{\le N}\psi_-\|_{L^2(\Sigma_\tau)}  {}^{(0)} \mathcal E^D_{\le 2}[\psi]\,d\tau \\
    & \lesssim \left( \int_{\tau_i }^{\tau_{i+1} } {}^{(\delta)} \mathscr E^T_{\le N}[\alpha]^2\,d\tau \right)^\frac12 \left( \int_{\tau_i }^{\tau_{i+1} } r^\delta \|\mathscr L_Z^{\le N}\psi_-\|_{L^2(\Sigma_\tau)} \, {}^{(1)}\mathcal E^D_{\le2}[\psi]^2\,d\tau  \right)^\frac12 \\
    & \lesssim C^2\varepsilon^2 \tau_i^{\frac32\delta}\tau_i^{\frac\delta2} \left( \int_{\tau_i }^{\tau_{i+1} }  \|\mathscr L_Z^{\le N}\psi_-\|_{L^2(\Sigma_\tau)} \tau^{-1+\delta} \,d\tau  \right)^\frac12 \\
    & \lesssim C^3\varepsilon^3 \tau_i^{\frac52\delta+\epsilon},
\end{align*}
and
\begin{align*}
    & \int_{\mathcal D^{\tau_{i+1}}_{\tau_i} }  \mathcal L_Z^{\le N}\underline\alpha \cdot \mathscr L_Z^{\le N}\psi_+ \cdot \psi_+\,dV \\
    & \lesssim \int_{\tau_i}^{\tau_{i+1}} r^{-\frac12} \|\mathcal L_Z^{\le N}\underline\alpha\|_{L^2(\Sigma_\tau)} \|\mathscr L_Z^{\le N}\psi_+\|_{L^2(\Sigma_\tau)} {}^{(0)}\mathcal E^D_{\le2}[\psi]\,d\tau \\
    & \lesssim \left( \int_{\tau_i}^{\tau_{i+1}} r^{-2+\delta}\|\mathcal L_Z^{\le N}\underline\alpha\|_{L^2(\Sigma_\tau)}^2\,d\tau \right)^\frac12 \left( \int_{\tau_i}^{\tau_{i+1}} r^{-\delta}\|\mathscr L_Z^{\le N}\psi_+\|_{L^2(\Sigma_\tau)}^2 {}^{(1)}\mathcal E^D_{\le2}[\psi]^2  \,d\tau \right)^\frac12  \\
    & \lesssim  C^2\varepsilon^2  \left( \int_{\tau_i}^{\tau_{i+1}} r^{-\delta}\|\mathscr L_Z^{\le N}\psi_+\|_{L^2(\Sigma_\tau)}^2  \tau^{-1+\delta} \,d\tau \right)^\frac12 \\
    & \lesssim C^3\varepsilon^3 \tau_i^{\frac{\delta}{2}+\epsilon}.
\end{align*}

For lower-order cases, we use the Hardy inequality to get $L^6\times L^6\times L^6$ and apply the weighted Sobolev inequality, which yields
\begin{align*}
    &  \int_{\mathcal D^{\tau_{i+1}}_{\tau_i} } r^{1+\delta} \mathcal L_Z^{\le N-3}\alpha \cdot \mathscr L_Z^{\le N-3}\psi_- \cdot \mathscr L_Z^{\le N-8} \psi_-\,dV \\
    & \lesssim \int_{\tau_i}^{\tau_{i+1} } r^{1+\delta} \mathcal E^T_{\le N-2}[\alpha]\mathcal E^D_{\le N-2}[\psi]\mathcal E^D_{\le N-7}[\psi]\,d\tau \\
    & \lesssim \left( \int_{\tau_i}^{\tau_{i+1} } {}^{(0)} \mathcal E^T_{\le N-2}[\alpha]^2\,d\tau \right)^\frac12 \left( \int_{\tau_i}^{\tau_{i+1} } r^{3\delta}\,\, {}^{(1-\delta)}\mathcal E^D_{\le N-2}[\psi]^2 \, {}^{(1)}\mathcal E^D_{\le N-7}[\psi]^2\,d\tau \right)^\frac12 \\
    & \lesssim C^3\varepsilon^3 \tau_i^{-\frac34+\frac52\delta},
\end{align*}
and
\begin{align*}
    & \int_{\mathcal D^{\tau_{i+1}}_{\tau_i} } \mathcal L_Z^{\le N-3}\underline\alpha \cdot \mathscr L_Z^{\le N-3}\psi_+ \cdot \mathscr L_Z^{\le N-8}\psi_+\,dV \\
    & \lesssim \int_{\tau_i}^{\tau_{i+1}} \mathcal E^T_{\le N-2}[\underline\alpha] \mathcal E^D_{\le N-2}[\psi] \mathcal E^D_{\le N-7}[\psi]\,d\tau \\
    & \lesssim \left( \int_{\tau_i}^{\tau_{i+1}} {}^{(0)} \mathcal E^T_{\le N-2}[\underline\alpha]^2\,d\tau \right)^\frac12 \left( \int_{\tau_i}^{\tau_{i+1}}{}^{(0)} \mathcal E^D_{\le N-2}[\psi]^2 \, {}^{(0)} \mathcal E^D_{\le N-7}[\psi]\,d\tau    \right)^\frac12 \\
    & \lesssim C^3\varepsilon^3 \tau_i^{-\frac34+\delta}.
\end{align*}
Now we are concerned with the improved decay. We first deal with the integral
\begin{align*}
   & \int_{\mathcal D^{\tau_{i+1}}_{\tau_i} } r^{\delta} \mathcal L_Z^{\le N}\alpha \cdot \mathscr L_Z^{\le N}\psi_-\cdot \mathscr L_Z^{\le N-8} \psi_- \,dV  \\
   & \lesssim \left( \int_{\mathcal D^{\tau_{i+1}}_{\tau_i} } r^{\delta} |\mathcal L_Z^{\le N}\alpha|^2\,dV \right)^\frac12 \left( \int_{\tau_i}^{\tau_{i+1}}r^\delta r^{-1} \|\mathscr L_Z^{\le N}\psi_-\|_{L^2(\Sigma_\tau)}^2\, {}^{(0)} \mathcal E^D_{\le N-6}[\psi]^2\,d\tau \right)^\frac12.
\end{align*}
Now we consider the configuration of the vector fields $\mathscr L_Z^{\le N}\psi_-$. If it contains at least one $L$ or $\underline L$, then it is automatically absorbed into the energy $\mathcal E^D_{\le N-1}[\psi]$ using Proposition \ref{prop-null-decomp-dirac} up to lower-order terms. If it consists only of rotation fields, then we see that
\begin{align*}
    &  \left( \int_{\mathcal D^{\tau_{i+1}}_{\tau_i} } r^{\delta} |\mathcal L_Z^{\le N}\alpha|^2\,dV \right)^\frac12 \left( \int_{\tau_i}^{\tau_{i+1}}r^\delta r^{-1}\|\mathscr L_Z^{\le N}\psi_-\|_{L^2(\Sigma_\tau)}^2\, {}^{(0)} \mathcal E^D_{\le N-6}[\psi]^2\,d\tau \right)^\frac12 \\
    & \lesssim C\varepsilon \tau_i^{\frac32\delta} \left( \int_{\tau_i}^{\tau_{i+1}}r^{-2+\delta} \|\mathscr L_Z^{\le N}\psi_-\|_{L^2(\Sigma_\tau)}^2\, {}^{(1)} \mathcal E^D_{\le N-6}[\psi]^2\,d\tau \right)^\frac12 \\
    & \lesssim  C\varepsilon \tau_i^{\frac32\delta} \left( \int_{\tau_i}^{\tau_{i+1}}r^{\delta} {}^{(0)}\mathcal E^D_{\le N-1}[\psi]^2 \, {}^{(1)} \mathcal E^D_{\le N-6}[\psi]^2\,d\tau \right)^\frac12 \\
    & \lesssim C^3\varepsilon^3 \tau_i^{-\frac12+\frac52\delta}.
\end{align*}
On the other hand,
\begin{align*}
    & \int_{\mathcal D^{\tau_{i+1}}_{\tau_i} }r^{-1+\delta}  \mathcal L_Z^{\le N}\underline\alpha \cdot \mathscr L_Z^{\le N}\psi_+ \cdot \mathscr L_Z^{\le N-8} \psi_+\,dV \\
    & \lesssim \left( \int_{\tau_i}^{\tau_{i+1}} r^{-2+\delta}\|\mathcal L_Z^{\le N}\underline\alpha\|_{L^2(\Sigma_\tau)}^2\,d\tau \right)^\frac12 \left( \int_{\tau_i}^{\tau_{i+1}} r^{-2+\delta}\|\mathscr L_Z^{\le N}\psi_+\|_{L^2(\Sigma_\tau)}^2 {}^{(1)}\mathcal E^D_{\le N-6}[\psi]^2  \,d\tau \right)^\frac12  \\
    &\lesssim C\varepsilon \left( \int_{\tau_i}^{\tau_{i+1}} r^{-2+\delta}\|\mathscr L_Z^{\le N}\psi_+\|_{L^2(\Sigma_\tau)}^2 {}^{(1)}\mathcal E^D_{\le N-6}[\psi]^2  \,d\tau \right)^\frac12.
\end{align*}
If $\mathscr L_Z^{\le N}\psi_+$ contains at least one $L$ or $\underline L$, then the Morawetz estimates give the bound $C^3\varepsilon^3 \tau^{-\frac12+\frac\delta2+\epsilon}$. If it consists only of rotation fields, we follow the argument used to control the previous integral and obtain the bound $C^3\varepsilon^3 \tau^{-\frac12+\delta}$.

For lower-order case, we repeat the above argument and obtain
We see that
\begin{align*}
	& \int_{\mathcal D^{\tau_{i+1}}_{\tau_i}} r^\delta | \mathcal L_Z^{\le N-3}\alpha| |\mathscr L_Z^{\le N-3}\psi_-| |\mathscr L_Z^{\le N-8}\psi_-|\,dV \\
	& \lesssim \int_{\tau_i}^{\tau_{i+1}} r^\delta \mathcal E^T_{\le N-2}[\alpha] \mathcal E^D_{\le N-2}[\psi]\mathcal E^D_{\le N-7}[\psi] \,d\tau \\
	& \lesssim \left( \int_{\tau_i}^{\tau_{i+1}} {}^{(0)}\mathcal E^T_{\le N-2}[\alpha]^2\,d\tau  \right)^\frac12 \left( \int_{\tau_i}^{\tau_{i+1}}{}^{(0)}\mathcal E^D_{\le N-2}[\psi]^2 {}^{(1)}\mathcal E^D_{\le N-7}[\psi]^2\,d\tau  \right)^\frac12 \\
	& \lesssim C^3\varepsilon^3 \tau_i^{-\frac34+\delta}.
\end{align*}
Repeating the above argument, one also has the following:
\begin{cor}
For the lower-order cases, we have
    \begin{align}
\int_{\mathcal D^{\tau_{i+1}}_{\tau_i}} r^\delta \mathcal L_Z^{\le k}\alpha \mathcal L_Z^{\le k}J_{\underline L}\,dV \lesssim \begin{cases}
    C^3\varepsilon^3 \tau^{-\frac78+\delta}, \quad k = N-4, \\
    C^3\varepsilon^3 \tau^{-1+\delta}, \quad k= N-5.
\end{cases}
\end{align}
\end{cor}

\subsection{Integrated local energy decay estimates for the tensor fields}\label{subsec:iled-tensor}
The aforementioned nonlinear estimates show immediately the following integrated local energy decay estimates for the scalar components of the tensor fields:
\begin{cor}
We have
\begin{align}\label{iled-imp-rho}
\begin{aligned}
	\int_{\mathcal D^{\tau_{i+1}}_{\tau_i}}  r^{-1+\delta}|\mathcal L_Z^{\le k}(\rho,\sigma)|^2 \,dV \lesssim \begin{cases}
		C^2\varepsilon^2 \tau^{-\frac12+\frac52\delta}, \quad k \le N , \\
		C^2\varepsilon^2 \tau^{-\frac34+\delta}, \quad k = N-3, \\
        C^2\varepsilon^2 \tau^{-\frac78+\delta}, \quad k = N -4, \\
        C^2\varepsilon^2 \tau^{-1+\delta}, \quad k \le N-5.
	\end{cases}
\end{aligned}
\end{align}
\end{cor}
\subsection{Decay of the null components for the tensor fields $F_{\mu\nu}$}\label{subsec:imp-null-h}
We deduce the decay of the null components $\alpha,\underline\alpha,\rho,\sigma$ for the tensor fields $F$ via a typical argument of using the dyadic region and the pigeonhole principle.

We first deal with the decay of the outgoing null component $\alpha$. To do this, we apply $r^q$ method for $\alpha$ with $q=3+\delta$ to get
\begin{align}
    \begin{aligned}
        \int_{\mathcal D^{\tau_{i+1}}_{\tau_i}} {}^{(\delta)}\mathscr E^T_{\le N}[\alpha]^2(\tau)\,d\tau \lesssim {}^{(1+\delta)}\mathscr E^T_{\le N}[\alpha]^2(\tau_i) + C^3\varepsilon^3 \tau^{\frac52\delta+\epsilon}.
    \end{aligned}
\end{align}
There exists a $\tau_i^*\in [\tau_i,\tau_{i+1}]$ such that
\begin{align}
    \begin{aligned}
        {}^{(\delta)}\mathscr E^T_{\le N}[\alpha]^2(\tau_i^*)\,d\tau \lesssim \tau_{i}^{-1}\left( {}^{(1+\delta)}\mathscr E^T_{\le N}[\alpha]^2(\tau_i) + C^3\varepsilon^3 \tau^{\frac52\delta+\epsilon} \right),
    \end{aligned}
\end{align}
which yields a good time slice $\Sigma_{\tau_i^*}$ in the dyadic slab $\mathcal D^{\tau_{i+1}}_{\tau_i}$. Then we apply the $r^q$-method for $\alpha$ once again with $q=2+\delta$ on the region $\mathcal D^{\tau_{i+1}}_{\tau_i^*}$ to get
\begin{align}
    \begin{aligned}
        {}^{(\delta)}\mathscr E^T_{\le N}[\alpha]^2(\tau_{i+1}) & \le {}^{(\delta)}\mathscr E^T_{\le N}[\alpha]^2(\tau_i^*) + \left|\int_{\mathcal D^{\tau_{i+1}}_{\tau_i^*}}r^{\delta} \mathcal L_Z^{\le N}\alpha \cdot \mathcal L_Z^{\le N} J_{L}\,dV  \right| \\
        & \lesssim C^3\varepsilon^3 \tau^{-\frac12+\frac52\delta},
    \end{aligned}
\end{align}
which gives the decay of $\alpha$. By repeating this argument, one can also obtain improved decay estimates for $\alpha$ for lower-order $k\le N-1$.

In order to obtain the decay for the scalar components $(\rho,\sigma)$,
we apply the $r^q$ method for $\alpha$ with $q=3+\delta$ to get
\begin{align}
    \begin{aligned}
        \int_{\tau_i}^{\tau_{i+1}} {}^{(\delta)}\mathscr E^T_{\le N}[\rho,\sigma]^2(\tau)\,d\tau \lesssim {}^{(1+\delta)}\mathscr E^T_{\le N}[\alpha]^2(\tau_i) + C^3\varepsilon^3 \tau^{\frac52\delta+\epsilon} .
     \end{aligned}
\end{align}
Then there exists $\tau_i^*\in [\tau_i,\tau_{i+1}] $ such that
\begin{align}
    \begin{aligned}
        {}^{(\delta)}\mathscr E^T_{\le N}[\rho,\sigma]^2(\tau_i^*) \le \tau_i^{-1} \left( {}^{(1+\delta)}\mathscr E^T_{\le N}[\alpha]^2(\tau_i) + C^3\varepsilon^3 \tau^{\frac52\delta+\epsilon}\right).
    \end{aligned}
\end{align}
Then we repeat the above argument and obtain the decay for the scalar components $(\rho,\sigma)$:
\begin{align}
    {}^{(\delta)}\mathscr E^T_{\le N}[\rho,\sigma]^2(\tau) \lesssim C^3\varepsilon^3 \tau^{-\frac12+\frac52\delta}.
\end{align}

Finally we obtain the decay for the ingoing null component $\underline\alpha$. To do this, we apply $r^p$ method for $\underline\alpha$ with $p=1+\delta$ to get
\begin{align}
    \begin{aligned}
        {}^{(-1+\delta)}\mathscr F^T_{\le N}[\underline\alpha]^2(v,\tau_i,\tau_{i+1}) & \le {}^{(-1+\delta)}\mathscr E^T_{\le N}[\rho,\sigma]^2(\tau_i)+ \left| \int_{\mathcal D^{\tau_{i+1}}_{\tau_i}} r^{-2+\delta} |\mathcal L_Z^{\le N}(\rho,\sigma)|^2 \,dV \right| + \left| \int_{\mathcal D^{\tau_{i+1}}_{\tau_i}} r^{-1+\delta} |\mathcal L_Z^{\le N}\underline\alpha\cdot\mathcal L_Z^{\le N}J_{\underline L} |\,dV \right| \\
        & \lesssim C^3\varepsilon^3 \tau^{-\frac12+\frac52\delta}.
    \end{aligned}
\end{align}
Repeatedly, we can establish improved decay estimates for $(\rho,\sigma)$ and $\underline\alpha$ as well as $\alpha$.

Moreover, the above decay estimates for the tensor fields $F$ will be only used to determine the decay of $F$ on the interior domain $r\le R$.
\section{Improved decay of the spinor fields}\label{sec:improv-spinor}

\subsection{Improved decay of the higher-order energy}
\begin{prop}\label{prop-imp-psi-1}
    We have
    \begin{align}
    \begin{aligned}
    	\int_{\mathcal D^{\tau_{i+1}}_{\tau_i}} r \langle \gamma^\lambda\nabla_\lambda \mathscr L_Z^{\le k}( F_{\mu\nu}\gamma^\mu\gamma^\nu \psi), \gamma^T \nabla_L\mathscr L_Z^{\le k}\psi\rangle \,dV \lesssim \begin{cases}
    		C^3\varepsilon^3 \tau^{-\frac14+\delta}, \quad k =  N-1, \\
    		C^3\varepsilon^3 \tau^{-\frac12+\delta}, \quad k =  N-2, \\
            C^3\varepsilon^3 \tau^{-\frac34+\delta}, \quad k=N-3,
    	\end{cases}
    \end{aligned}
    \end{align}
so that we obtain the following improved weighted energy for higher-order:
    \begin{align}
    \begin{aligned}
        {}^{(1)}\mathcal E^D_{\le k}[\psi]^2(\tau) \le \begin{cases}
            C^2\varepsilon^2 \tau^{-\frac14+\delta}, \quad k = N-1, \\
        C^2\varepsilon^2 \tau^{-\frac12+\delta}, \quad k = N-2, \\
        C^2\varepsilon^2 \tau^{-\frac34+\delta}, \quad k = N-3.
        \end{cases}
        \end{aligned}
    \end{align}
    In particular, we have
    \begin{align}
        {}^{(1)}\mathcal E^D_{\le N-4}[\psi]^2 (\tau) \lesssim \tau^{-1+\delta}.
    \end{align}
\end{prop}
Since the main approach of the proof is essentially identical to the control of the weighted energy ${}^{(2-\delta)}\mathcal E^D_{\le k}[\psi]$, we give the proof somewhat schematically.

Let $k=N-1$. If $F=\alpha$, it is enough to consider the following integrals:
\begin{align}\label{int-apsi-imp-lh}
    \int_{\mathcal D^{\tau_{i+1}}_{\tau_i}} r \langle \mathcal L_Z^{\le N-8}\alpha_A \gamma^{\underline L}\gamma^{e_A}\gamma^L \nabla_{\underline L}\mathscr L_Z^{\le N-1}\psi, \gamma^T \nabla_L \mathscr L_Z^{\le N-1}\psi\rangle \,dV
\end{align}
and
\begin{align}\label{int-apsi-imp-hl}
     \int_{\mathcal D^{\tau_{i+1}}_{\tau_i}} r \langle \nabla_{\underline L}\mathcal L_Z^{\le N-1}\alpha_A \gamma^{\underline L}\gamma^{e_A}\gamma^L \mathscr L_Z^{\le N-8}\psi, \gamma^T \nabla_L \mathscr L_Z^{\le N-1}\psi\rangle \,dV,
\end{align}
which correspond to the low-high and high-low interactions. For the first integral \eqref{int-apsi-imp-lh}, we use the integration by parts along the $\underline L$-direction. The boundary terms consisting of the integral along the outgoing null and the time-like surface give the required bound in an obvious way. The control of the spacetime integral without derivative loss, i.e., involving the term $\nabla_{\underline L}\alpha$ is also obvious, after an application of the weighted Sobolev inequality to $\alpha$ and using the decomposition \eqref{decomp-alpha-wave}.

Then we are left to consider the spacetime integral where derivative loss appears:
\begin{align}
     \int_{\mathcal D^{\tau_{i+1}}_{\tau_i}} r \langle \mathcal L_Z^{\le N-8}\alpha_A \gamma^{\underline L}\gamma^{e_A}\gamma^L \mathscr L_Z^{\le N-1}\psi, \gamma^T\nabla_{\underline L} \nabla_L \mathscr L_Z^{\le N-1}\psi\rangle \,dV.
\end{align}
By using the decomposition \eqref{decomp-LLbar}, we use the integration by parts with respect to the angular variables. Then we need to control the integral of the form:
\begin{align}
     \int_{\mathcal D^{\tau_{i+1}}_{\tau_i}} r |\mathcal L_Z^{\le N-8}\alpha| |\slashed\nabla \mathscr L_Z^{\le N-1}\psi_+|  |\slashed\nabla \mathscr L_Z^{\le N-1}\psi_+ | \,dV.
\end{align}
Then we apply the $L^\infty$-Sobolev embedding to $\alpha$ and hence the above integral is bounded by
\begin{align}
    \left(\int_{\tau_i}^{\tau_{i+1}} {}^{(1)}\mathcal E^T_{\le N-6}[\alpha]^2 \, {}^{(0)}\mathcal E^D_{\le N-1}[\psi]^2\,d\tau   \right)^\frac12 \left(\int_{\tau_i}^{\tau_{i+1}} {}^{(0)}\mathcal E^D_{\le N-1}[\psi]^2\,d\tau   \right)^\frac12 \lesssim C^3\varepsilon^3 \tau_i^{-\frac12+\frac\delta2},
\end{align}
which is the required bound. The remaining non-trivial spacetime integral is the following:
\begin{align}
     \int_{\mathcal D^{\tau_{i+1}}_{\tau_i}} r |\mathcal L_Z^{\le N-8}\alpha| | \mathscr L_Z^{\le N-1}\psi_+|  |\gamma^\lambda\nabla_\lambda \mathscr L_Z^{\le N-1}(F_{\mu\nu}\gamma^\mu\gamma^\nu\psi) | \,dV,
\end{align}
which can be also controlled in an obvious way.
Now we deal with the control of the integral \eqref{int-apsi-imp-hl}. After using the integration by parts in the $\underline L$-direction, we focus on the spacetime integral with derivative loss. Using the decomposition \eqref{decomp-LLbar}, we use the integration by parts with respect to the angular variables and then it is enough to consider the integral of the form:
\begin{align}
     \int_{\mathcal D^{\tau_{i+1}}_{\tau_i}}  |\mathcal L_Z^{\le N}\alpha| | \mathscr L_Z^{\le N-8}\psi_+|  |\slashed\nabla \mathscr L_Z^{\le N-1}\psi_+ | \,dV.
\end{align}
Then we use the integrated local energy decay estimates for $\alpha$ and get
\begin{align}
    \begin{aligned}
        \left(\int_{\tau_i}^{\tau_{i+1}} {}^{(\delta)}\mathscr E^T_{\le N}[\alpha]^2 \,d\tau   \right)^\frac12 \left(\int_{\tau_i}^{\tau_{i+1}} {}^{(1)}\mathcal E^D_{\le N-6}[\psi]^2 \,  {}^{(0)}\mathcal E^D_{\le N-1}[\psi]^2\,d\tau   \right)^\frac12 \lesssim C^3\varepsilon^3 \tau_i^{-\frac12+2\delta}.
    \end{aligned}
\end{align}
As the integral \eqref{int-apsi-imp-lh}, we are left to consider the integral:
\begin{align}
     \int_{\mathcal D^{\tau_{i+1}}_{\tau_i}} r |\mathcal L_Z^{\le N-1}\alpha| | \mathscr L_Z^{\le N-8}\psi_+|  |\gamma^\lambda\nabla_\lambda \mathscr L_Z^{\le N-1}(F_{\mu\nu}\gamma^\mu\gamma^\nu\psi) | \,dV,
\end{align}
which can be treated in an obvious way.
If $F=\underline\alpha$, it is enough to consider the following integrals:
\begin{align}\label{int-abpsi-imp-lh}
    \int_{\mathcal D^{\tau_{i+1}}_{\tau_i}} r\langle \nabla_L \mathcal L_Z^{\le N-8}\underline\alpha_A \gamma^L\gamma^{e_A}\gamma^{\underline L}\mathscr L_Z^{\le N-1}\psi, \gamma^T \nabla_L \mathscr L_Z^{\le N-1}\psi\rangle\,dV,
\end{align}
and
\begin{align}\label{int-abpsi-imp-hl}
      \int_{\mathcal D^{\tau_{i+1}}_{\tau_i}} r\langle  \mathcal L_Z^{\le N-1}\underline\alpha_A \gamma^L\gamma^{e_A}\gamma^{\underline L} \nabla_L\mathscr L_Z^{\le N-8}\psi, \gamma^T \nabla_L \mathscr L_Z^{\le N-1}\psi\rangle\,dV.
\end{align}
For the first integral \eqref{int-abpsi-imp-lh}, we replace the term $\nabla_L\underline\alpha$ by $\nabla_{\underline L}\alpha$ with $\slashed\nabla\rho$ by the null decomposition for the tensor fields. Then the integral including the term $\rho$ is bounded by
\begin{align}
    \begin{aligned}
        \int_{\tau_i}^{\tau_{i+1}} r^{-\frac12} \|\nabla_L \mathcal L_Z^{\le N-5}\rho\|_{L^2(\Sigma_\tau)} \|\mathscr L_Z^{\le N-1}\psi_-\|_{L^2(\Sigma_\tau)} \mathcal E^D_{\le N-1}[\psi]\,d\tau .
    \end{aligned}
\end{align}
If $\mathscr L_Z^{\le N-1}\psi_-$ contains at least one $L$ or $\underline L$, then it is automatically absorbed into the energy $\mathcal E^D_{\le N-2}[\psi]$. If it contains only rotation fields, then the above integral is bounded by
\begin{align}
    \begin{aligned}
        \int_{\tau_i}^{\tau_{i+1}} r^{\frac12} \|\nabla_L \mathcal L_Z^{\le N-5}\rho\|_{L^2(\Sigma_\tau)} \mathcal E^D_{\le N-2}[\psi] \mathcal E^D_{\le N-1}[\psi]\,d\tau .
    \end{aligned}
\end{align}
Then the null decomposition for the tensor fields shows that the above integral including $\nabla_L\rho$ is bounded by
\begin{align}
    \begin{aligned}
        & \int_{\tau_i}^{\tau_{i+1}} r^{-\frac12} \| \mathcal L_Z^{\le N-5}\rho\|_{L^2(\Sigma_\tau)} \mathcal E^D_{\le N-2}[\psi] \mathcal E^D_{\le N-1}[\psi]\,d\tau + \int_{\tau_i}^{\tau_{i+1}} r^{-\frac12} \| \mathcal L_Z^{\le N-4}\alpha\|_{L^2(\Sigma_\tau)} \mathcal E^D_{\le N-2}[\psi] \mathcal E^D_{\le N-1}[\psi]\,d\tau \\
         & \qquad\qquad+ \int_{\tau_i}^{\tau_{i+1}} r^{\frac12} \| \mathcal L_Z^{\le N-5} J_L\|_{L^2(\Sigma_\tau)} \mathcal E^D_{\le N-2}[\psi] \mathcal E^D_{\le N-1}[\psi]\,d\tau .
    \end{aligned}
\end{align}
The first integral is controlled via the improved integrated local energy decay estimates \eqref{iled-imp-rho} for $\rho$. For the second integral, if the vector fields of $\mathcal L_Z^{\le N-4}\alpha$ contain at least one $L$ or $\underline L$, it automatically becomes the energy $\mathcal E^T_{\le N-5}[\alpha]$ up to the scalar $\rho$, which is controlled again by \eqref{iled-imp-rho}. The vector fields consist of only rotation fields, then it becomes again $\mathcal E^T_{\le N-5}[\alpha]$ with a factor $r$, which can be absorbed as the energy ${}^{(1)}\mathcal E^T_{\le N-5}[\alpha]$. For the third integral, we simply apply the H\"older inequality to the Dirac current to get the product $L^4\times L^4$ and use the Sobolev embedding and obtain the energy $\mathcal E^D_{\le N-4}[\psi]^2$.

For the control of the integral \eqref{int-abpsi-imp-lh}, we are left to consider the integral of the form:
\begin{align}
     \int_{\mathcal D^{\tau_{i+1}}_{\tau_i}} r\langle \nabla_{\underline L} \mathcal L_Z^{\le N-8}\alpha_A \gamma^L\gamma^{e_A}\gamma^{\underline L}\mathscr L_Z^{\le N-1}\psi, \gamma^T \nabla_L \mathscr L_Z^{\le N-1}\psi\rangle\,dV.
\end{align}
Then we apply the $L^\infty$-Sobolev embedding and use the decomposition \eqref{decomp-alpha-wave} to get the following integrals
\begin{align}
    \begin{aligned}
        & \int_{\tau_i}^{\tau_{i+1}}r^{-\frac12} \mathcal E^T_{\le N-5}[\alpha] \|\mathscr L_Z^{\le N-1}\psi_-\|_{L^2(\Sigma_\tau)}\mathcal E^D_{\le N-1}[\psi]\,d\tau \\
        & \qquad + \int_{\tau_i}^{\tau_{i+1}}r^{-\frac12} \|\nabla_{\underline L}\mathcal L_Z^{\le N-6}\alpha\|_{L^2(\Sigma_\tau)} \|\mathscr L_Z^{\le N-1}\psi_-\|_{L^2(\Sigma_\tau)}\mathcal E^D_{\le N-1}[\psi]\,d\tau \\
        & \qquad + \int_{\tau_i}^{\tau_{i+1}}r^{-\frac32}\|\mathcal L_Z^{\le N-5}(\underline\alpha,\rho,\sigma)\|_{L^2(\Sigma_\tau)} \|\mathscr L_Z^{\le N-1}\psi_-\|_{L^2(\Sigma_\tau)}\mathcal E^D_{\le N-1}[\psi]\,d\tau \\
        & \qquad + \int_{\tau_i}^{\tau_{i+1}}r^{\frac12}\|\mathcal L_Z^{\le N-5} (\nabla_L \slashed{J}+\slashed\nabla J_L) \|_{L^2(\Sigma_\tau)} \|\mathscr L_Z^{\le N-1}\psi_-\|_{L^2(\Sigma_\tau)}\mathcal E^D_{\le N-1}[\psi]\,d\tau,
    \end{aligned}
\end{align}
where the control of the above integrals follow the previous argument. In particular, for the integral including the term $\underline\alpha$, instead of using the integrated local energy decay estimates, which do not produce any decay for $\underline\alpha$, we replace it by $\mathcal E^T_{\le N-6}[\underline\alpha]$ with a factor $r$, since the vector fields of $\mathcal L_Z^{\le N-5}\underline\alpha$ must contain at least one rotation field.

Now we are concerned with the integral \eqref{int-abpsi-imp-hl}. After the foliation by the ingoing null cones $\underline{\mathcal C}_v$, we use the $L^\infty$-Sobolev embedding and the null decomposition Proposition \ref{decomp-dirac-null-tensor} to replace $\nabla_{\underline L}\psi_-$ by $\slashed\nabla\psi_+$ up to lower-order terms. After $\|\mathcal L_Z^{\le N-1}\underline\alpha\|_{L^2(\underline{\mathcal C}_v)}$ is absorbed into the energy ${}^{(0)}\mathscr F^T_{\le N-1}[\underline\alpha]$, the integral \eqref{int-abpsi-imp-hl} is bounded by
\begin{align}
    \begin{aligned}
        {}^{(0)}\mathscr F^T_{\le N-1}[\underline\alpha] \left( \int_{\tau_i}^{\tau_{i+1}} {}^{(0)}\mathcal E^D_{\le N-5}[\psi]^2\,d\tau \right)^\frac12 \left( \int_{\tau_i}^{\tau_{i+1}} r^{-1}\, {}^{(0)}\mathcal E^D_{\le N-1}[\psi]^2\,d\tau  \right)^\frac12 \lesssim C^3\varepsilon^3 \tau^{-\frac12+\delta}.
    \end{aligned}
\end{align}
Here we omit the details of the control of the additional nonlinearities such as $\underline\alpha\cdot\psi_-$ and $(\rho,\sigma)\cdot\psi_\pm$, which appear from Proposition \ref{decomp-dirac-null-tensor}. We note that they produce better decay.

We also omit the control of the nonlinearities $(\rho,\sigma)\cdot\psi_\pm$. Then we are left to establish the improved decay estimates for the energy ${}^{(1)}\mathcal E^D_{\le k}[\psi]$. The remaining task is a standard argument via the pigeonhole principle, the details of which we omit. This completes the proof of Proposition \ref{prop-imp-psi-1}.

\subsection{Improved decay of the lower-order energy}
We have the following improved decay estimates for the lower-order energy:
\begin{prop}\label{prop-imp-psi-0}
    We have
    \begin{align}
        \begin{aligned}
            \int_{\mathcal D^{\tau_{i+1}}_{\tau_i}} \langle \gamma^\lambda\nabla_\lambda \mathscr L_Z^{\le k}( h_{\mu\nu}\gamma^\mu\gamma^\nu \psi), \gamma^T \nabla_L\mathscr L_Z^{\le k}\psi\rangle \,dV \lesssim \begin{cases}
                C^3\varepsilon^3 \tau^{-\frac98+2\delta},  \quad k = N-4, \\
                 C^3\varepsilon^3 \tau^{-\frac54+2\delta},  \quad k = N-5, \\
                 C^3\varepsilon^3 \tau^{-\frac{11}{8}+2\delta}, \quad k = N-6.
            \end{cases}
        \end{aligned}
    \end{align}
    so that we have
    \begin{align}
        {}^{(0)}\mathcal E^D_{\le k}[\psi]^2(\tau) \le \begin{cases}
            C^2\varepsilon^2 \tau^{-\frac98+2\delta}, \quad k= N-4, \\
              C^2\varepsilon^2 \tau^{-\frac54+2\delta}, \quad k= N-5, \\
                C^2\varepsilon^2 \tau^{-\frac{11}8+2\delta}, \quad k= N-6.
        \end{cases}
    \end{align}
\end{prop}
\begin{proof}
    The proof follows that of Proposition \ref{prop-imp-psi-1}. We omit the repetitive details.
\end{proof}
\begin{rem}
    By iteration, one can improve the decay and might be able to establish much stronger decay: ${}^{(0)}\mathcal E^D_{\le N-6}[\psi]^2(\tau) \lesssim \tau^{-2+\delta}$.
\end{rem}

\section{Improved energy estimates for the spinor and tensor fields}\label{sec:imp-energy-est}
Equipped with the improved decay estimates for both the spinor and the tensor fields, we are now able to improve the energy estimates:
\begin{prop}
    We have
    \begin{align}
       \sum_{|I_1|+|I_2|\le N} \int_{\mathcal D^{\tau_{i+1}}_{\tau_i}} \langle \mathscr L_Z^{I_1}F_{\mu\nu} \gamma^\mu \gamma^\nu \mathscr L_Z^{I_2}\psi, \mathscr L_Z^{\le N} \psi\rangle \,dV \lesssim C^3\varepsilon^3 \tau^{-\epsilon},
    \end{align}
    for some small $\epsilon>0$,
    so that we have the improved energy estimates for the spinor fields:
    \begin{align}
        \mathscr E^D_{\le N}[\psi]^2(\tau)+ \sup_{v: \tau\le \tau(v)} \mathscr F^D_{\le N}[\psi]^2(v,\tau_{i},\tau_{i+1}) \le C^2\varepsilon^2.
    \end{align}
\end{prop}
As we have done in the previous sections, we restrict ourselves into the low-high and the high-low interactions. In particular, we may assume that the derivative order of a higher-order component is $\ge N-3$ while the order of a lower-order factor is $\le N-8$. Otherwise, for the middle level, we can apply the H\"older inequality to get the product $L^4\times L^4\times L^2$ and use the $L^4$-Sobolev embedding and then obtain the product of the energy such as $\mathscr E^D_{\le N}\times \mathcal E^D_{\le N-4}\times \mathcal E^T_{\le N-4}$. Then the remaining task is obvious, which gives the required bound. In what follows, we are therefore only concerned with the low-high and high-low cases.

\subsection{$F=\alpha$}
It is enough to consider the integrals of the form:
\begin{align}
    \int_{\mathcal D^{\tau_{i+1}}_{\tau_i}} \mathcal L_Z^{\le N-8}\alpha \cdot \mathscr L_Z^{\le N}\psi_+ \cdot \mathscr L_Z^{\le N}\psi_+ \,dV,
\end{align}
and
\begin{align}
    \int_{\mathcal D^{\tau_{i+1}}_{\tau_i}} \mathcal L_Z^{\le N}\alpha \cdot \mathscr L_Z^{\le N-8}\psi_+ \cdot \mathscr L_Z^{\le N}\psi_+ \,dV.
\end{align}
By using the $L^\infty$-Sobolev embedding, we have
\begin{align*}
    &  \int_{\mathcal D^{\tau_{i+1}}_{\tau_i}} \mathcal L_Z^{\le N-8}\alpha \cdot \mathscr L_Z^{\le N}\psi_+ \cdot \mathscr L_Z^{\le N}\psi_+ \,dV \\
    & \lesssim \int_{\tau_i}^{\tau_{i+1}} r^{-\frac12}\, {}^{(0)}\mathcal E^T_{\le N-6}[\alpha]\mathscr E^D_{\le N}[\psi]^2\,d\tau .
\end{align*}
If the vector fields of $\mathscr L_Z^{\le N}\psi_+$ contain at least one $L$ or $\underline L$, then we can apply the Morawetz estimates for the spinor fields as wave. Indeed, we see that
\begin{align*}
    & \int_{\tau_i}^{\tau_{i+1}} r^{-\frac12}\, {}^{(0)}\mathcal E^T_{\le N-6}[\alpha]\mathscr E^D_{\le N}[\psi]^2\,d\tau  \\
    & \lesssim \int_{\tau_i}^{\tau_{i+1}} r^{-1}\, {}^{(1)}\mathcal E^T_{\le N-6}[\alpha]\mathscr E^D_{\le N}[\psi]^2\,d\tau \\
    & \lesssim \tau^{-\frac12+2\delta}\int_{\mathcal D^{\tau_{i+1}}_{\tau_i}} r^{-1-\delta} \left( |\nabla_L\mathscr L_Z^{\le N-1}\psi_+|^2+|\nabla_{\underline L}\mathscr L_Z^{\le N-1}\psi_+|^2 \right)\,dV \\
    & \lesssim \tau^{-\frac12+2\delta} \left( \mathcal E^D_{\le N-1}[\psi]^2(\tau_{i+1})+\mathcal E^D_{\le N-1}[\psi]^2(\tau_i)+\sup_v\mathcal F^D_{\le N-1}[\psi]^2(v,\tau_i,\tau_{i+1}) \right) \\
    & \lesssim \tau^{-\frac12+3\delta},
\end{align*}
where we omit the boundary term along the timelike surface $\{r=R\}$ and the spacetime integrals that also appear in the Morawetz estimates, since they produce better decay.

If the vector fields in $\mathscr L_Z^{\le N}\psi_+$ consist only of rotation fields $\Omega$, then we see that
\begin{align*}
    &  \int_{\mathcal D^{\tau_{i+1}}_{\tau_i}} \mathcal L_Z^{\le N-8}\alpha \cdot \mathscr L_Z^{\le N}\psi_+ \cdot \mathscr L_Z^{\le N}\psi_+ \,dV  \\
    & \lesssim \int_{\tau_i}^{\tau_{i+1}} {}^{(1)}\mathcal E^T_{\le N-6}[\alpha]\,{}^{(0)} \mathcal E^D_{\le N-1}[\psi] \mathscr E^D_{\le N}[\psi]\,d\tau \\
    & \lesssim \tau^{\frac\delta2} \left( \int_{\tau_i}^{\tau_{i+1}}{}^{(1-\delta)}\mathcal E^T_{\le N-6}[\alpha]^2\,d\tau \right)^\frac12 \left( \int_{\tau_i}^{\tau_{i+1}} \mathscr E^D_{\le N}[\psi]^2 \, {}^{(0)}\mathcal E^D_{\le N-1}[\psi]^2\,d\tau \right)^\frac12 \\
    & \lesssim \tau^{-\frac18+2\delta}.
\end{align*}
On the other hand, we have
\begin{align*}
    & \int_{\mathcal D^{\tau_{i+1}}_{\tau_i}} \mathcal L_Z^{\le N}\alpha \cdot \mathscr L_Z^{\le N-8}\psi_+ \cdot \mathscr L_Z^{\le N}\psi_+ \,dV \\
    & \lesssim \left(  \int_{\mathcal D^{\tau_{i+1}}_{\tau_i}} r^{-1+\delta} |\mathcal L_Z^{\le N}\alpha|^2\,dV \right)^\frac12 \left(  \int_{\tau_i}^{\tau_{i+1}} r^{-\delta} \,\mathscr E^D_{\le N}[\psi]^2\, {}^{(0)}\mathcal E^D_{\le N-6}[\psi]^2\,d\tau   \right)^\frac12 \\
    & \lesssim \tau^{-\frac38+3\delta}.
\end{align*}
\subsection{$F=\underline\alpha$}
We need to consider the following integrals:
\begin{align}
     \int_{\mathcal D^{\tau_{i+1}}_{\tau_i}} \mathcal L_Z^{\le N}\underline\alpha \cdot \mathscr L_Z^{\le N-8}\psi_- \cdot \mathscr L_Z^{\le N}\psi_-\,dV,
\end{align}
and
\begin{align}
     \int_{\mathcal D^{\tau_{i+1}}_{\tau_i}} \mathcal L_Z^{\le N-8}\underline\alpha \cdot \mathscr L_Z^{\le N}\psi_- \cdot \mathscr L_Z^{\le N}\psi_-\,dV.
\end{align}
Then we see that
\begin{align*}
    &  \int_{\mathcal D^{\tau_{i+1}}_{\tau_i}} \mathcal L_Z^{\le N}\underline\alpha \cdot \mathscr L_Z^{\le N-8}\psi_- \cdot \mathscr L_Z^{\le N}\psi_-\,dV \\
    & \lesssim \left( \int_{\mathcal D^{\tau_{i+1}}_{\tau_i}} r^{-1-\delta} |\mathcal L_Z^{\le N}\underline\alpha|^2\,dV \right)^\frac12 \left( \int_{\tau_i}^{\tau_{i+1}} r^{\delta} \, {}^{(0)}\mathcal E^D_{\le N-6}[\psi]^2 \|\mathscr L_Z^{\le N}\psi_-\|_{L^2(\Sigma_\tau)}^2\,d\tau   \right)^\frac12.
\end{align*}
If the vector fields of $\mathscr L_Z^{\le N}\psi_-$ contain at least one $L$ or $\underline L$, it automatically becomes the energy $\mathcal E^D_{\le N-1}[\psi]$. One also encounters additional nonlinear estimates such as $\underline\alpha\cdot\psi_-$ due to $\nabla_{\underline L}\psi_-$. However, it gives better decay and hence we omit it.

On the other hand, if the vector fields consist only of rotation fields, then
\begin{align*}
    &  \int_{\mathcal D^{\tau_{i+1}}_{\tau_i}} \mathcal L_Z^{\le N}\underline\alpha \cdot \mathscr L_Z^{\le N-8}\psi_- \cdot \mathscr L_Z^{\le N}\psi_-\,dV \\
    & \lesssim \left( \int_{\mathcal D^{\tau_{i+1}}_{\tau_i}} r^{-1-\delta} |\mathcal L_Z^{\le N}\underline\alpha|^2\,dV \right)^\frac12 \left( \int_{\tau_i}^{\tau_{i+1}} r^{1+\delta} \, {}^{(1)}\mathcal E^D_{\le N-6}[\psi]^2\, {}^{(0)}\mathcal E^D_{\le N-1}[\psi]^2  \,d\tau   \right)^\frac12 \\
    & \lesssim  \left( \int_{\mathcal D^{\tau_{i+1}}_{\tau_i}} r^{-1-\delta} |\mathcal L_Z^{\le N}\underline\alpha|^2\,dV \right)^\frac12 \left( \int_{\tau_i}^{\tau_{i+1}} r^{2\delta} \, {}^{(1)}\mathcal E^D_{\le N-6}[\psi]^2 \, {}^{(1-\delta)}\mathcal E^D_{\le N-1}[\psi]^2  \,d\tau   \right)^\frac12 \\
    & \lesssim \tau^{-\frac12+2\delta}.
\end{align*}
For the low-high interaction, the argument is very similar to the control of the integral with $h=\alpha$, so we omit the details.

\subsection{$F=(\rho,\sigma)$}
Now we are left to consider the case when $h$ is the scalar components $\rho$ or $\sigma.$ In fact, it suffices to consider only $h=\rho$. Therefore we deal with the following integrals:
\begin{align}
    \int_{\mathcal D^{\tau_{i+1}}_{\tau_i}}  \mathcal L_Z^{\le N-8}\rho \cdot \mathscr L_Z^{\le N}\psi_- \cdot \mathscr L_Z^{\le N}\psi_+\,dV,
\end{align}
and
\begin{align}
    \int_{\mathcal D^{\tau_{i+1}}_{\tau_i}}  \mathcal L_Z^{\le N}\rho \cdot \mathscr L_Z^{\le N-8}\psi_+ \cdot \mathscr L_Z^{\le N}\psi_-\,dV.
\end{align}
Then we see that
\begin{align*}
    &   \int_{\mathcal D^{\tau_{i+1}}_{\tau_i}}  \mathcal L_Z^{\le N-8}\rho \cdot \mathscr L_Z^{\le N}\psi_- \cdot \mathscr L_Z^{\le N}\psi_+\,dV \\
    & \lesssim \int_{\tau_i}^{\tau_{i+1}} r^{-\frac12}\|\nabla_L \mathcal L_Z^{\le N-6}\rho\|_{L^2(\Sigma_\tau)} \|\mathscr L_Z^{\le N}\psi_-\|_{L^2(\Sigma_\tau)} \mathscr E^D_{\le N}[\psi]\,d\tau \\
    & \lesssim \int_{\tau_i}^{\tau_{i+1}} r^{-\frac32}\|\mathcal L_Z^{\le N-6}\rho\|_{L^2(\Sigma_\tau)} \|\mathscr L_Z^{\le N}\psi_-\|_{L^2(\Sigma_\tau)} \mathscr E^D_{\le N}[\psi]\,d\tau \\
    & \qquad + \int_{\tau_i}^{\tau_{i+1}} r^{-\frac32}\| \mathcal L_Z^{\le N-5}\alpha\|_{L^2(\Sigma_\tau)} \|\mathscr L_Z^{\le N}\psi_-\|_{L^2(\Sigma_\tau)} \mathscr E^D_{\le N}[\psi]\,d\tau \\
    & \qquad + \int_{\tau_i}^{\tau_{i+1}} r^{-\frac12}\| \mathcal L_Z^{\le N-6} J_L\|_{L^2(\Sigma_\tau)} \|\mathscr L_Z^{\le N}\psi_-\|_{L^2(\Sigma_\tau)} \mathscr E^D_{\le N}[\psi]\,d\tau .
\end{align*}
As in the previous observation, we may write $\|\mathscr L_Z^{\le N}\psi_-\|_{L^2(\Sigma_\tau)}\lesssim r \mathcal E^D_{\le N-1}[\psi]$, where we omit the additional nonlinearities arising from Proposition \ref{decomp-dirac-null-tensor}, since they give even better decay. The required decay then follows from the improved integrated local energy decay estimates for $\rho$ and ${}^{(0)}\mathcal E^D_{\le N-1}[\psi]$. The integral involving the additional Dirac current also gives better decay. The high-low interaction is straightforward. After applying the $L^\infty$-Sobolev embedding, we use \eqref{iled-imp-rho} to obtain the required decay; we omit the details.

Finally, the control of the nonlinearity on the exterior region $\{r \ge R\}$ is complete after the proof of the following:
\begin{prop}
    We have
    \begin{align}
        \begin{aligned}
            \int_{\mathcal D^{\tau_{i+1}}_{\tau_i}} \mathcal L_Z^{\le N} \alpha \cdot \mathcal L_Z^{\le N}\slashed J \,dV +  \int_{\mathcal D^{\tau_{i+1}}_{\tau_i}} \mathcal L_Z^{\le N} \rho \cdot \mathcal L_Z^{\le N} J_{L} \,dV \lesssim C^3\varepsilon^3 \tau_i^{-\epsilon},
        \end{aligned}
    \end{align}
    for some small $\epsilon>0$.
\end{prop}
We obtain the above control immediately, since we already have sufficient decay with the weight $r^\delta$.

\begin{rem}\label{rem-sharp-decay}
    We would like to highlight that in the proof of the improved energy estimates:
    $$
    \mathscr E^D_{\le N}[\psi]^2+\mathscr F^D_{\le N}[\psi]^2\le C^2\varepsilon^2,
    $$
    we do not rely on any sharp pointwise decay estimates for the spinor fields. Indeed, the nonlinear analysis only requires the relatively weak decay rate: $\tau^{-\frac12-\delta}$, which is significantly weaker decay than an improved decay $\tau^{-1}$.
    This is made possible by systematically exploiting both the null structure inherent in the Dirac equation and the favorable null structure of the tensor-Dirac nonlinear interactions, which provide sufficient spacetime integrability to close the bootstrap argument.

    We also remark that our decay estimates ${}^{(0)}\mathcal E^D_{\le N-6}[\psi](\tau)\lesssim \tau^{-\frac{11}{16}+\delta}$ is weaker than the linear decay $\lesssim \tau^{-\frac56}$ known for the massive Dirac equation on the Kerr black hole background. See also \cite{finster}.
Nevertheless, this stronger linear decay is not required for our nonlinear analysis, since the aforementioned null structures already yield sufficient control of the nonlinear error terms.

\end{rem}

\section{Interior domain}\label{sec:interior}
In the preceding sections, we have been concerned with the exterior domain. Here we give the control for the interior domain.

As we have already established several improved decay estimates,
in view of the Morawetz estimates, we are only left to consider the spacetime integrals of the nonlinearities for the top-order.

We recall the bootstrap assumptions for the energy on the interior region:
\begin{align}
    \boxed{
    \begin{aligned}
         {}_{\rm int}\mathcal E^D_{\le N-4}[\psi]^2(\tau) \le C^2\varepsilon^2 \tau^{-1+\delta}, \quad  {}_{\rm int}\mathcal E^D_{\le N-1}[\psi]^2(\tau) \le C^2\varepsilon^2, \\
          {}_{\rm int}\mathscr E^T_{\le N-3}[F]^2(\tau) \le C^2\varepsilon^2 \tau^{-1+\delta}, \quad  {}_{\rm int}\mathscr E^T_{\le N}[F]^2(\tau) \le C^2\varepsilon^2.
    \end{aligned}
    }
\end{align}
We focus on the following integrals:
\begin{align}
    \int_{\tau_i}^{\tau_{i+1}}\int_{r \le R}  \mathcal L_Z^{\le N}F \cdot \mathscr L_Z^{\le N-7}\psi \cdot \mathscr L_Z^{\le N}\psi \,dV,
\end{align}
and
\begin{align}
    \int_{\tau_i}^{\tau_{i+1}}\int_{r \le R}  \mathcal L_Z^{\le N-7}F \cdot \mathscr L_Z^{\le N}\psi \cdot \mathscr L_Z^{\le N}\psi \,dV.
\end{align}
Then we see that the high-low interaction can controlled as follows:
\begin{align*}
    & \int_{\tau_i}^{\tau_{i+1}}\int_{r \le R}  \mathcal L_Z^{\le N}F \cdot \mathscr L_Z^{\le N-7}\psi \cdot \mathscr L_Z^{\le N}\psi \,dV \\
    & \lesssim  \left( \int_{\tau_i}^{\tau_{i+1}} {}_{\rm int}\mathcal E^D_{\le N-1}[\psi]^2\,d\tau \right)^\frac12 \left( \int_{\tau_i}^{\tau_{i+1}} {}_{\rm int}\mathcal E^D_{\le N-5}[\psi]^2 {}_{\rm int}\mathscr E^T_{\le N}[F]^2\,d\tau  \right)^\frac12 \\
    & \lesssim C^3\varepsilon^3 \tau^{-\frac12+\frac\delta2}.
\end{align*}
The low-high interaction can be treated by a symmetric argument. We omit the details.

A standard approach of using the dyadic argument and the pigeonhole principle gives the decay:
\begin{align}
    {}_{\rm int}\mathcal E^D_{\le k}[\psi]^2(\tau) \lesssim \mathcal E^D_{\le k}[\psi]^2.
\end{align}

Then the integrated local energy decay estimates or the Morawetz estimates on the interior domain gives
\begin{align}
    \int_{\tau_i}^{\tau_{i+1}} {}_{\rm int}\mathcal E^D_{\le N-1}[\psi]^2(\tau)\,d\tau \lesssim {}_{\rm int}\mathcal E^D_{\le N-1}[\psi]^2(\tau_i) + \mathcal B^D_{\le N-1}[\psi]^2(r=R),
\end{align}
where $\mathcal B^D_{\le N-1}[\psi]^2(r=R)$ denotes the boundary term along the time-like surface $\{r=R\}$. This is controlled as follows:
\begin{align}
    \begin{aligned}
        \mathcal B^D_{\le N-1}[\psi]^2( r=R) & \lesssim \int_{\tau_i}^{\tau_{i+1}}\int_{\Sigma_\tau \cap \{R\le r\le 2R \} } \mathcal J^T [ \mathscr L_Z^{\le N-1}\psi  ] \,d\sigma d\tau \\
        & \lesssim \int_{\tau_i}^{\tau_{i+1}}  {}^{(0)}\mathcal E^D_{\le N-1}[\psi]^2(\tau)\,d\tau .
    \end{aligned}
\end{align}
Then we can choose a good time slice $\Sigma_{\tau_i^*}$ such that
\begin{align}
    \begin{aligned}
        {}_{\rm int}\mathcal E^T_{\le N-1}[\psi]^2(\tau_i^*) \lesssim {}^{(0)}\mathcal E^D_{\le N-1}[\psi]^2(\tau_i),
    \end{aligned}
\end{align}
which shows that the decay of the good time slice on the interior domain is determined by the decay of the exterior domain. Moreover, an application of the energy inequality on the interior domain bounded by $\Sigma_{\tau_{i+1}}$, $\Sigma_{\tau_i^*}$ and $\{r=R\}$, we have
\begin{align}
    \begin{aligned}
        {}_{\rm int}\mathcal E^D_{\le N-1}[\psi]^2(\tau_{i+1}) & \lesssim  {}_{\rm int}\mathcal E^D_{\le N-1}[\psi]^2(\tau_{i}^*) + \left|\int_{\tau_i}^{\tau_{i+1}}\int_{r\le R} \langle \mathscr L_Z^{\le N}(F_{\mu\nu}\gamma^\mu\gamma^\nu \psi) , \gamma^T \mathscr L_Z^{\le N}\psi\rangle\,dV \right| \\
        & \lesssim  {}_{\rm int}\mathcal E^D_{\le N-1}[\psi]^2(\tau_{i}^*) + C^3\varepsilon^3 \tau^{-\frac12+\frac\delta2}.
    \end{aligned}
\end{align}
By repeating this argument, one can prove the bootstrap assumption: ${}_{\rm int}\mathcal E^D_{\le N-4}[\psi]^2(\tau)\lesssim \tau^{-1+\delta}$.

Furthermore, the control of the tensor fields $F$ in the interior domain is essentially identical to that of the spinor fields, and we omit the details.
\section{Iteration}\label{sec:iteration}
In this section, we explain how the improved estimates established in the previous sections can be iterated. Since each iteration follows exactly the same procedure as before, we do not repeat the details. Instead, we summarise the estimates obtained after sufficiently many iterations and deduce the final weighted energy bounds.
\subsection{Previous bootstrap assumptions}
We recall our previous bootstrap assumptions:
\begin{align}
    \begin{aligned}
       {}^{(2-\delta)}\mathcal E^D_{\le N-4}[\psi]^2(\tau) \le C^2\varepsilon^2, \quad   {}^{(2-\delta)}\mathcal E^D_{\le N-1}[\psi]^2(\tau) \le C^2\varepsilon^2\tau^{\delta}, \\
        {}^{(1)}\mathcal E^D_{\le N-4}[\psi]^2(\tau) \le C^2\varepsilon^2 \tau^{-1+\delta}, \quad {}^{(1)}\mathcal E^D_{\le N-1}[\psi]^2(\tau) \le C^2\varepsilon^2,
    \end{aligned}
\end{align}
and
\begin{align}
    \begin{aligned}
        {}^{(2-\delta)}\mathcal E^T_{\le N-2}[\alpha]^2(\tau) \le C^2\varepsilon^2 \tau^{3\delta}, \ {}^{(2-\delta)}\mathscr E^T_{\le N-2}[\underline\alpha]^2(\tau) \le C^2\varepsilon^2 \tau^{2\delta}, \\
         {}^{(1)}\mathcal E^T_{\le N-4}[\alpha]^2(\tau) \le C^2\varepsilon^2 \tau^{-\frac12+4\delta}, \ {}^{(1)}\mathscr E^T_{\le N-4}[\underline\alpha]^2(\tau) \le C^2\varepsilon^2 \tau^{-\frac12+3\delta},
    \end{aligned}
\end{align}
and
\begin{align}
    \begin{aligned}
     {}^{(1+\delta)}   \mathscr E^T_{\le N}[\alpha]^2(\tau) \le C^2\varepsilon^2 \tau^{3\delta}, \ {}^{(0)}\mathscr F^T_{\le N}[\underline\alpha]^2(v,\tau,2\tau) \le C^2\varepsilon^2 \tau^{\delta}, \\
      {}^{(\delta)}   \mathscr E^T_{\le N}[\alpha]^2(\tau) \le C^2\varepsilon^2 , \ {}^{(-1+\delta)}\mathscr F^T_{\le N}[\underline\alpha]^2(v,\tau,2\tau) \le C^2\varepsilon^2 ,
    \end{aligned}
\end{align}
and
\begin{align}
    \int_{\mathcal D^{2\tau }_\tau } r^{-2+\delta} |\mathcal L_Z^{\le N}\underline\alpha|^2\,dV \le C^2\varepsilon^2 ,
\end{align}
and
\begin{align}
    \int_{\tau }^{2\tau} {}^{(\delta)}\mathscr E^T_{\le N}[\alpha,\rho,\sigma]^2(\tau)\,d\tau  \le C^2\varepsilon^2 \tau^{3\delta} ,\quad  \int_{\tau }^{2\tau} {}^{(-1+\delta)}\mathscr E^T_{\le N}[\rho,\sigma]^2(\tau)\,d\tau  \le C^2\varepsilon^2
\end{align}
\subsection{Improved estimates}
From the above bootstrap assumptions, we have established the following improved estimates:
\begin{align}
    {}^{(2-\delta)}\mathcal E^D_{\le N-1}[\psi]^2(\tau) \le C^2\varepsilon^2 \tau^{\epsilon},
\end{align}
and
\begin{align}
    \begin{aligned}
         {}^{(1)}\mathcal E^D_{\le k}[\psi]^2(\tau) \le \begin{cases}
            C^2\varepsilon^2 \tau^{-\frac14+\delta}, \quad k = N-1, \\
        C^2\varepsilon^2 \tau^{-\frac12+\delta}, \quad k = N-2, \\
        C^2\varepsilon^2 \tau^{-\frac34+\delta}, \quad k = N-3, \\
        C^2\varepsilon^2 \tau^{-1+\delta}, \quad k\le N-4,
        \end{cases}
    \end{aligned}
\end{align}
and
 \begin{align}
        {}^{(0)}\mathcal E^D_{\le k}[\psi]^2(\tau) \le \begin{cases}
            C^2\varepsilon^2 \tau^{-\frac98+2\delta}, \quad k= N-4, \\
              C^2\varepsilon^2 \tau^{-\frac54+2\delta}, \quad k= N-5, \\
                C^2\varepsilon^2 \tau^{-\frac{11}8+2\delta}, \quad k= N-6,
        \end{cases}
    \end{align}
    and for the tensor fields,
    \begin{align}
        \begin{aligned}
            {}^{(2-\delta)}\mathcal E^T_{\le N-3}[\alpha]^2(\tau) \le C^2\varepsilon^2, \quad  {}^{(2-\delta)}\mathcal E^T_{\le N-2}[\alpha]^2(\tau) \le C^2\varepsilon^2 \tau^{\frac52\delta}, \\
             {}^{(2-\delta)}\mathcal E^T_{\le N-3}[\underline\alpha]^2(\tau) \le C^2\varepsilon^2, \quad  {}^{(2-\delta)}\mathcal E^T_{\le N-2}[\underline\alpha]^2(\tau) \le C^2\varepsilon^2 \tau^{\frac32\delta},
        \end{aligned}
    \end{align}
    and
    \begin{align}
        \begin{aligned}
           {}^{(\delta)}\mathscr E^T_{\le N}[\alpha]^2(\tau) \le C^2\varepsilon^2 , \quad {}^{(1+\delta)}\mathscr E^T_{\le N}[\alpha]^2(\tau) \le C^2\varepsilon^2 \tau^{\frac52\delta+\epsilon}, \\
           {}^{(-1+\delta)}\mathscr F^T_{\le N}[\underline\alpha]^2(v,\tau_i,\tau_{i+1}) \le C^2\varepsilon^2, \quad   {}^{(0)}\mathscr F^T_{\le N}[\underline\alpha]^2(v,\tau_i,\tau_{i+1}) \le C^2\varepsilon^2 \tau^{\frac\delta2+\epsilon},
        \end{aligned}
    \end{align}
    and
    \begin{align}
{}^{(1)}\mathcal E^T_{\le k}[\alpha]^2(\tau) \le \begin{cases}
    C^2\varepsilon^2 \tau^{-\frac34+\frac32\delta}, \quad k = N-3, \\
    C^2\varepsilon^2 \tau^{-1+\delta}, \quad k \le N-4,
\end{cases}
\quad {}^{(1)}\mathcal E^T_{\le k}[\underline\alpha]^2(\tau) \le \begin{cases}
    C^2\varepsilon^2 \tau^{-\frac34+\delta}, \quad k = N-3, \\
    C^2\varepsilon^2 \tau^{-1+\delta}, \quad k \le N-4,
    \end{cases}
    \end{align}
    and
    \begin{align}
\begin{aligned}
	\int_{\mathcal D^{\tau_{i+1}}_{\tau_i}}  r^{-1+\delta}|\mathcal L_Z^{\le k}(\rho,\sigma)|^2 \,dV \lesssim \begin{cases}
		C^2\varepsilon^2 \tau^{-\frac12+\frac52\delta}, \quad k \le N , \\
		C^2\varepsilon^2 \tau^{-\frac34+\delta}, \quad k = N-3, \\
        C^2\varepsilon^2 \tau^{-\frac78+\delta}, \quad k = N -4, \\
        C^2\varepsilon^2 \tau^{-1+\delta}, \quad k \le N-5,
	\end{cases}
\end{aligned}
\end{align}
and finally, we have the following improved energy estimates:
\begin{align}
    \begin{aligned}
        \mathscr E^D_{\le N}[\psi]^2(\tau) + \mathscr F^D_{\le N}[\psi]^2(\tau)+  {}^{(0)}  \mathscr E^T_{\le N}[\alpha]^2(\tau) \le C^2\varepsilon^2. 
    \end{aligned}
\end{align}
\subsection{Iteration}
Since the bootstrap assumptions have now been improved, the continuity argument together with the continuation criterion establishes the global existence of solutions to \eqref{tensor-dirac}. The same argument can then be repeated using the improved estimates in place of the original bootstrap assumptions. As each iteration is completely analogous, we omit the details and state only the resulting estimates:
\begin{align}
    \begin{aligned}
        \mathcal E^D_{\le N-1}[\psi]^2(\tau) \le C^2\varepsilon^2 , \\
    {}^{(2-\delta)}    \mathcal E^T_{\le N-2}[\alpha]^2(\tau) + {}^{(2-\delta)} \mathcal E^T_{\le N-2}[\underline\alpha]^2(\tau) \le C^2\varepsilon^2 \tau^{\frac\delta2}, \\
    {}^{(1+\delta)}\mathscr E^T_{\le N}[\alpha]^2(\tau) \le C^2\varepsilon^2 \tau^{\frac\delta2}, \quad {}^{(0)}\mathscr F^T_{\le N}[\underline\alpha]^2(\tau) \le C^2\varepsilon^2.
    \end{aligned}
\end{align}
In particular, after sufficiently many iterations, the weighted spinor energies become uniformly bounded, whereas the weighted tensor energies exhibit at most a mild logarithmic growth. This completes the proof of Theorem \ref{main-thm-formal}.
\appendix

\section{Formulation of the Dirac equation}\label{appendix-dirac}
 \begin{defn}[Irreducible representation]
Let $A$ be an associative algebra over $\mathbb C$,
and let $\pi:A\to \mathrm{End}(W)$ be a representation on a complex vector space $W$.
The representation $\pi$ is called \emph{irreducible} if
there exists no nontrivial proper subspace $W'\subset W$
such that
\[
\pi(a)W' \subset W' \qquad \text{for all } a\in A,
\]
i.e. the only $A$-invariant subspaces of $W$ are $\{0\}$ and $W$ itself.
\end{defn}
\begin{rem}
In particular, a complex spinor module $\Delta$ is said to be irreducible
if it has no nontrivial proper subspaces invariant under the action
of the complexified Clifford algebra $\mathrm{Cl}_{\mathbb C}(V,\boldsymbol\theta)$.
\end{rem}
\begin{defn}[Spin connection]
Let $\nabla$ denote the Levi-Civita connection on the Lorentzian manifold $(\mathcal M,g)$.
The Levi-Civita connection induces a principal $SO^+(V,\boldsymbol\theta)$-connection
on the oriented and time-oriented orthonormal frame bundle $F_{SO^+}(\mathcal M)$.
Via the spin structure $\mathbf P$ and the covering homomorphism
$\rho:\mathrm{Spin}^+(V,\boldsymbol\theta)\to SO^+(V,\boldsymbol\theta)$,
this connection uniquely lifts to a principal $\mathrm{Spin}^+(V,\boldsymbol\theta)$-connection on $\mathbf P$.

The associated connection on the spinor bundle
\[
\mathbf S=\mathbf P\times_\kappa\Delta
\]
is denoted by
\[
\nabla^{\mathbf S}:\Gamma(\mathbf S)\to\Gamma(T^*\mathcal M\otimes\mathbf S),
\]
and is called the \emph{spin connection}.
\end{defn}
\section{Local well-posedness}\label{sec:lwp}


\begin{thm}\label{appendix-lwp}
    Let $N\ge11$ and $0<\epsilon_0<1$ be given. We consider the Cauchy problems for the tensorial wave-Dirac system \eqref{tensor-dirac} with the initial data $(\psi,F)|_{\Sigma_{\tau=0}}:=(\psi_0,F_0)$ satisfying the smallness condition:
     \begin{align}
       \mathcal X^D_{N}[\psi]^2(0)+\mathcal X^T_{N}[F]^2(0) \le (\epsilon_0)^2.
    \end{align}
    Then the system \eqref{tensor-dirac} is locally well-posed. More precisely, there exists a $\tau=\tau(\epsilon_0)>0$ such that the system \eqref{tensor-dirac} admit a unique solution $(\psi,F)$ satisfying
    \begin{align}
        \psi \in C( [ 0, \tau]; H^N)\cap C^1((0,\tau); H^{N-1}) ,\quad F\in C( [0,\tau]; H^N) \cap C^1((0,\tau); H^{N-1})
    \end{align}
    with the energy estimates
    \begin{align}
      \sup_{0\le\tau'\le\tau}\left(   \mathcal X^D_{N}[\psi]^2(\tau')+\mathcal X^T_{N}[F]^2(\tau')\right)\le 2(\epsilon_0)^2.
    \end{align}
\end{thm}

We consider the domain $D^{\tau^*}_{0}$ bounded by two hypersurfaces $\Sigma_{\tau^*}$ and $\Sigma_{0}$. We apply the divergence theorem to the Dirac current:
\begin{align}
    \begin{aligned}
     {}_{\rm int}\mathcal E^D_{\le N-1}[\psi]^2(\tau^*) +  \mathscr E^D_{\le N}[\psi]^2(\tau^*) + \mathscr F^D_{\le N}[\psi]^2(v,0,\tau^*) &\le \mathscr E^D_{\le N}[\psi]^2(0) +\left|\int_{D^{\tau^*}_0} \langle [ \gamma^\mu\nabla_\mu, \mathscr L_Z^{\le N} ]\psi, \mathscr L_Z^{\le N}\psi\rangle \,dV\right| \\
        & \qquad + \left|\int_{D^{\tau^*}_0} \langle \mathscr L_Z^{\le N}(F_{\mu\nu} \gamma^\mu\gamma^\nu \psi), \mathscr L_Z^{\le N}\psi\rangle\,dV\right|.
    \end{aligned}
\end{align}
On the other hand, we use the divergence theorem for the energy-momentum tensor of the tensor field $F$ and get the following energy inequality:
\begin{align}
    \begin{aligned}
        {}_{\rm int}\mathcal E^T_{\le N-1}[F]^2(\tau^*)+{}^{(0)}\mathscr E^T_{\le N}[F]^2(\tau^*)+{}^{(0)}\mathscr F^T_{\le N}[F]^2(v,0,\tau^*) & \le \mathscr E^T_{\le N}[F]^2(0) + \left| \int_{D^{\tau^*}_0} [ \nabla, \mathcal L_Z^{\le N} ] F \cdot \mathcal L_Z^{\le N} F\,dV  \right| \\
        & \qquad + \left| \int_{D^{\tau^*}_0} \mathcal L_Z^{\le N}F \cdot \mathcal L_Z^{\le N}J   \,dV  \right|,
    \end{aligned}
\end{align}
where $J$ is an abbreviation of the Dirac current $J_\nu$. 
Then we can also rewrite the above energy inequalities in a comprehensive way:
\begin{align}
    \begin{aligned}
        \mathcal X^D_{N}[\psi]^2(\tau^*) + \mathcal X^T_{N}[F]^2(\tau^*) & \le \mathcal X^D_{N}[\psi]^2(0) + \mathcal X^T_{N}[F]^2(0)+\left|\int_{D^{\tau^*}_0} \langle [ \gamma^\mu\nabla_\mu, \mathscr L_Z^{\le N} ]\psi, \mathscr L_Z^{\le N}\psi\rangle \,dV\right| \\ & \qquad +\left| \int_{D^{\tau^*}_0} [ \nabla, \mathcal L_Z^{\le N} ] F \cdot \mathcal L_Z^{\le N} F\,dV  \right| + \left| \int_{D^{\tau^*}_0} \mathcal L_Z^{\le N}F \cdot \mathcal L_Z^{\le N}J   \,dV  \right| \\
        &\qquad\qquad+ \left|\int_{D^{\tau^*}_0} \langle \mathscr L_Z^{\le N}(F_{\mu\nu} \gamma^\mu\gamma^\nu \psi), \mathscr L_Z^{\le N}\psi\rangle\,dV\right|.
    \end{aligned}
\end{align}
The control of the spacetime integrals of the commutator terms is rather obvious. Indeed, we see that
\begin{align*}
    & \left|\int_{D^{\tau^*}_0} \langle [ \gamma^\mu\nabla_\mu, \mathscr L_Z^{\le N} ]\psi, \mathscr L_Z^{\le N}\psi\rangle \,dV\right|+\left| \int_{D^{\tau^*}_0} [ \nabla, \mathcal L_Z^{\le N} ] F \cdot\mathcal L_Z^{\le N} F\,dV  \right| \\
    & \lesssim c\int_0^{\tau^*} \mathcal X^D_{N}[\psi]^2(\tau)+\mathcal X^T_{N}[F]^2(\tau)\,d\tau  \\
    & \lesssim c\tau^* \left( \sup_{0\le \tau\le\tau^*} \mathcal X^D_{N}[\psi]^2(\tau)+\sup_{0\le \tau\le\tau^*} \mathcal X^T_{N}[F]^2(\tau) \right),
\end{align*}
where $c$ is a sufficiently small constant due to the background geometry. The control of the spacetime integrals of the nonlinearities is also obvious. In fact, after using the Sobolev inequality, we have
\begin{align*}
    &  \left| \int_{D^{\tau^*}_0} \mathcal L_Z^{\le N}F \cdot \mathcal L_Z^{\le N}J   \,dV  \right| + \left|\int_{D^{\tau^*}_0} \langle \mathscr L_Z^{\le N}(F_{\mu\nu} \gamma^\mu\gamma^\nu \psi), \mathscr L_Z^{\le N}\psi\rangle\,dV\right| \\
    & \lesssim \int_0^{\tau^*} \mathcal X^D_{N}[\psi]^2(\tau) \mathcal X^T_{N}[F](\tau)\,d\tau \\
    & \lesssim \tau^* \sup_{0\le \tau\le \tau^*}\mathcal X^D_{N}[\psi]^2(\tau)\mathcal X^T_{N}[F](\tau).
\end{align*}
\subsection{a priori closure}
From now on, we put
\begin{align}
    \mathcal Y_{}(\tau):= \mathcal X^D_{N}[\psi]^2(\tau)+ \mathcal X^T_{N}[F]^2(\tau).
\end{align}
Then we have
\begin{align}
    \mathcal Y(\tau^*) \le \mathcal Y(0)+c\tau^* \mathcal Y(\tau^*)+C\tau^* \mathcal Y(\tau^*)^\frac32.
\end{align}
We set
\begin{align}
    \mathcal M(\tau) := \sup_{0\le \tau'\le\tau}\mathcal Y(\tau').
\end{align}
Then we have the inequality
\begin{align}
    \mathcal M(\tau^*) \le \mathcal Y(0)+c\tau^* \mathcal M(\tau^*)+C\tau^* \mathcal M(\tau^*)^\frac32.
\end{align}
Now we put
\begin{align}
    \mathscr X_{\tau} := \{ (\psi,F)\in C( [0,\tau]; H^N(\Sigma_{\tau'} )) , \mathcal M(\tau) \le 2\mathcal Y(0) \},
\end{align}
for some $\tau\le\tau^*\le2$.
For $(\psi,F)\in \mathscr X_\tau$, we have
\begin{align*}
    \mathcal Y(\tau)& \le \mathcal Y(0)+ 2c\tau\mathcal Y(0)+2\sqrt2C\tau\mathcal Y(0)^\frac32 \\
    & \le \mathcal Y(0) ( 1+2c\tau+2\sqrt2 C \tau  \mathcal Y(0)^\frac12).
\end{align*}
We recall that $\tau\le2$ and the perturbation of the background geometry is sufficiently small, and hence $c$ is a sufficiently small constant. Thus, we have
\begin{align}
    2c\tau + 2\sqrt2 C \tau \mathcal Y(0)^\frac12 \le \frac12,
\end{align}
provided that the initial data $\mathcal Y(0)$ is sufficiently small, which implies that
\begin{align*}
    \mathcal M(\tau) \le \frac32\mathcal Y(0).
\end{align*}
\subsection{Picard iteration}
We shall construct the Picard iteration for the system \eqref{eq-tensor-dirac}. To do this, we let $(\psi^{(0)},F^{(0)}) $ be a regular extension of the initial data $(\psi_0,F_0)$ to the domain $D^{\tau}_0$
and we define
\begin{align}
    \begin{aligned}
        i\gamma^\mu\nabla_\mu\psi^{(k+1)} = F_{\mu\nu}^{(k)}\gamma^\mu\gamma^\nu \psi^{(k+1)},\\
        \nabla^\mu F_{\mu\nu}^{(k+1)} = \langle \psi^{(k)},\gamma_n\psi^{(k)}\rangle, \\
        dF^{(k+1)} = 0,
    \end{aligned}
\end{align}
where
\begin{align}
    \psi^{(k+1)}|_{\Sigma_0}=\psi_0, \quad F^{(k+1)}|_{\Sigma_0} = F_0.
\end{align}
Now we define
\begin{align}
    \mathcal M_k(\tau) := \sup_{0\le\tau'\le\tau}\mathcal Y_k(\tau'),
\end{align}
where
\begin{align}
    \mathcal Y_k(\tau) := \mathcal X^D_{N}[\psi^{(k)}](\tau)+\mathcal X^T_{N}[F^{(k)}](\tau).
\end{align}
By repeating the preceding energy estimates, we obtain the inequality:
\begin{align}
    \mathcal M_{k+1}(\tau) & \le \mathcal Y(0)+c\tau \mathcal M_{k+1}(\tau)+ C\tau \mathcal M_k(\tau)^\frac12\mathcal M_{k+1}(\tau).
\end{align}
By an inductive argument, if we have $\mathcal M_k(\tau)\le 2\mathcal Y(0)$, then we have
\begin{align}
    \mathcal M_{k+1}(\tau) \le 2\mathcal Y(0).
\end{align}
Now we shall use the contraction principle. To do this, we define the difference:
\begin{align}
    \delta\psi^{(k)}:= \psi^{(k+1)}-\psi^{(k)}, \quad \delta F^{(k)}:= F^{(k+1)}-F^{(k)}.
\end{align}
Then we have
\begin{align}
    \begin{aligned}
        i\gamma^\mu\nabla_\mu \delta \psi^{(k)} =  F^{(k)}_{\mu\nu}\gamma^\mu\gamma^\nu \delta\psi^{(k)}+\delta F^{(k)}_{\mu\nu}\gamma^\mu\gamma^\nu \psi^{(k)}, \\
        \nabla^\mu \delta F_{\mu\nu}^{(k)} =  \langle \delta\psi^{(k-1)},\gamma_\nu \psi^{(k)}\rangle + \langle \psi^{(k-1)},\gamma_\nu\delta \psi^{(k-1)}\rangle, \\
        d\delta F^{(k)}=0,
    \end{aligned}
\end{align}
with $\delta\psi^{(k)}|_{\Sigma_0}=0$ and $\delta F^{(k)}|_{\Sigma_0}=0$. We put
\begin{align}
    \mathcal D_k(\tau) := \sup_{0\le\tau'\le\tau}\left(  \mathcal X^D_{N}[\delta\psi^{(k)}]^2(\tau')+\mathcal X^T_{N}[\delta F^{(k)}]^2(\tau')  \right).
\end{align}
A repetition of the previous argument gives
\begin{align}
    \mathcal D_k(\tau) \le c\tau \mathcal D_k(\tau_+C\tau \mathcal Y(0)^\frac12 \mathcal D_k(\tau)+C\tau \mathcal Y(0)^\frac12\mathcal D_{k-1}(\tau),
\end{align}
which yields
\begin{align}
    \mathcal D_k(\tau) \le C'\tau \mathcal Y(0)\mathcal D_{k-1}(\tau),
\end{align}
and hence we have
\begin{align}
     \mathcal D_k(\tau) \le \frac12 \mathcal D_{k-1}(\tau),
\end{align}
provided that $\mathcal Y(0)$ is sufficiently small. This shows that the Picard iteration $(\psi^{(k)},F^{(k)})$ is a Cauchy sequence, and hence it converges to a $(\psi,F)$. This ensures the local existence of solutions to the tensorial wave-Dirac system \eqref{eq-tensor-dirac}. Furthermore, this contraction mapping principle yields uniqueness and continuous dependence of solutions as well as existence.

\section{Null decomposition for the Dirac equation}\label{appendix-null-dirac}
This section is devoted to an introduction of the null decomposition of the Dirac equation. Although we have already shown that the null decomposition of the Dirac equation via the double null foliation, we would like to present the derivation of the null decomposition by somewhat a primitive approach. We first choose a local orthonormal frame $(e_0,e_1,e_2,e_3)$ adapted to the usual coordinate system $(x^0,x^1,x^2,x^3)$ and then take a null frame $(e_A,L,\underline L)$. Then we shall again establish a coupled system of transported-type equations along null directions.

We recall the linear Dirac equation
\begin{align*}
	(i\gamma^\mu\nabla_\mu-m)\psi = 0.
\end{align*}
We choose a local orthonormal frame $\{e_0,e_1,e_2,e_3\}$ as the previous section and set the null frame $\{e_A,\underline L,L\}$. According to the choose of the frame we consider the gamma matrices: $\{\gamma({e_0}),\gamma({e_1}),\gamma({e_2}), \gamma({e_3})\}$ with the Clifford algebra
\begin{align*}
	\gamma({e_\mu})\gamma({e_\nu})+\gamma({e_\nu})\gamma({e_\mu}) = -2g(e_\mu,e_\nu)I_{4}
\end{align*}
and also $\{\gamma({e_A}),\gamma({\underline L}),\gamma({L})\}$ satisfying the Clifford algebra. For example, $\gamma({e_0})^2=-g(e_0,e_0)I_4=I_4$ and $\gamma({e_r})^2=-g(e_r,e_r)I_4=-I_4$. Now we consider the matrix $S$ defined by
\begin{align}
S = \gamma({e_0})\gamma({e_r}).
\end{align}
\begin{prop}
The	matrix $S$ is Hermitian, and $S^2=I_4$. Furthermore, we have
\begin{align}
S\gamma(L) = -\gamma(L)S, \ S\gamma(\underline L) = \gamma(\underline L)S,\ S\gamma(e_A)=\gamma(e_A)S.
\end{align}
\end{prop}
\begin{proof}
The Clifford algebra ensures that
\begin{align*}
	S^2 &=\gamma({e_0})\gamma({e_r})\gamma({e_0})\gamma({e_r}) \\
	& = -\gamma({e_0})\gamma({e_0})\gamma({e_r})\gamma({e_r}) \\
	& = -(I_4) (-I_4) = I_4.
\end{align*}
The matrix $S$ is Hermitian. Indeed,
\begin{align*}
	S^\dagger & = (\gamma({e_0})\gamma({e_r}))^\dagger = \gamma({e_r})^\dagger\gamma({e_0})^\dagger  = -\gamma({e_r})\gamma({e_0}) = \gamma({e_0})\gamma({e_r}) = S,
\end{align*}
where we used the Clifford relation $\gamma(e_0)\gamma(e_r)+\gamma(e_r)\gamma(e_0)=-2g(e_0,e_r)I_4=0$.
Thus the matrix $S$ has double eigenvalues $\pm1$. We recall that the Clifford representation is linear and hence we have $\gamma(L)=\gamma({e_0})+\gamma({e_r})$, and this implies that
\begin{align*}
	S\gamma(L) &= \gamma({e_0})\gamma({e_r})\gamma({e_0})+\gamma({e_0})\gamma({e_r})\gamma({e_r}) \\
	& = -\gamma({e_0})\gamma({e_0})\gamma({e_r})-\gamma({e_r})\gamma({e_0})\gamma({e_r}) \\
	& = -(\gamma({e_0})+\gamma({e_r}))\gamma({e_0})\gamma({e_r}) \\
	& = -\gamma(L)S,
\end{align*}
where we used the Clifford relation $\gamma(e_0)\gamma(e_r)+\gamma(e_r)\gamma(e_0)=-2g(e_0,e_r)I_4=0.$
We have $S\gamma({\underline L})=-\gamma({\underline L})S$ in the similar way. On the other hand, using the relation $\gamma(e_0)\gamma(e_A)+\gamma(e_A)\gamma(e_0)=-2g(e_0,e_A)I_4=0$ and $\gamma(e_A)\gamma(e_r)+\gamma(e_r)\gamma(e_A)=-2g(e_A,e_r)I_4=0,$ we see that
\begin{align*}
	S\gamma({e_A}) &= \gamma({e_0})\gamma({e_r})\gamma({e_A}) \\
	& = -\gamma({e_0})\gamma({e_A})\gamma({e_r}) \\
	& = \gamma({e_A})\gamma({e_0})\gamma({e_r}) = \gamma({e_A})S.
\end{align*}
\end{proof}
Now we define the projection
\begin{align}
\Pi_\pm =\frac12(I_4\pm S).
\end{align}
Then we have $\psi = \psi_++\psi_-$, where $\psi_\pm:= \Pi_\pm\psi$.
\begin{prop}\label{prop-pi}
	We have the following:
	\begin{align}
	\Pi_\pm^2 = \Pi_\pm, \ \Pi_\pm\Pi_\mp = 0, \ \Pi_\pm^\dagger = \Pi_\pm, \\
	\Pi_\pm\gamma(L)=\gamma(L)\Pi_\mp, \ \Pi_\pm\gamma(\underline L) = \gamma(\underline L)\Pi_\mp	, \ \Pi_\pm\gamma(e_A) = \gamma(e_A)\Pi_\pm.
	\end{align}
Furthermore, we have
\begin{align}
\langle\psi,\gamma(e_0)\psi\rangle = \langle\psi_+,\gamma(e_0)\psi_+\rangle+\langle\psi_-,\gamma(e_0)\psi_-\rangle.
\end{align}
\end{prop}
\begin{proof}
	The proof immediately follows from the previous proposition.
\end{proof}
Now we rewrite the Dirac equation in terms of the null frame:
\begin{align*}
	i\gamma({L})\nabla_L\psi +i\gamma({\underline L})\nabla_{\underline L}\psi+i\gamma({e_A})\nabla_{e_A}\psi -m\psi =0.
\end{align*}
Applying the projection $\Pi_+$, we have
\begin{align*}
	\Pi_+\left( i\gamma({L})\nabla_L\psi +i\gamma({\underline L})\nabla_{\underline L}\psi+i\gamma({e_A})\nabla_{e_A}\psi -m\psi \right) = 0.
\end{align*}

In consequence,
\begin{align*}
	& \Pi_+\left(  i\gamma({L})\nabla_L\psi +i\gamma({\underline L})\nabla_{\underline L}\psi+i\gamma({e_A})\nabla_{e_A}\psi -m\psi \right)\\
	& = i\gamma(L) \Pi_-\nabla_L\psi+i\gamma({\underline L})\Pi_-\nabla_{\underline L}\psi +i\gamma({e_A})\Pi_+\nabla_{e_A}\psi -m\Pi_+\psi.
\end{align*}
One should be careful that the commutator $[\Pi_\pm,\nabla_L]$ does not generally vanish.
\begin{prop}\label{D-pi-commutator}
	We have the following commutator identities:
	\begin{align}
		[\nabla_L,\Pi_\pm] &= \pm\underline\eta_A\gamma(e_A)\gamma(e_0)\Pi_+, \\
		[\nabla_{\underline L},\Pi_\pm] &=\pm \gamma(e_A)\gamma(e_0)\left( \underline\xi_A\Pi_+-\eta_A\Pi_-  \right)\mp\underline\omega(\Pi_+-\Pi_-), \\
		[\nabla_{e_A},\Pi_\pm] & = \pm\frac12 \gamma(e_B)\gamma(e_0) \left( \underline{\chi}_{AB}\Pi_+-\chi_{AB}\Pi_- \right),
	\end{align}
	where
	\begin{align}
	\begin{aligned}
		\chi_{AB}= \langle \nabla_{e_A}L,e_B\rangle, \ \underline\chi_{AB} = \langle \nabla_{e_A}\underline L,e_B\rangle, \\
		 2\underline\eta_A = \langle \nabla_L \underline L,e_A\rangle, \ 2\eta_A = \langle \nabla_{\underline L}L,e_A\rangle, \ 2\underline\xi_A = \langle\nabla_{\underline L}\underline L,e_A\rangle, \ 4\underline\omega = \langle \nabla_{\underline L}\underline L,L\rangle.
		 \end{aligned}
	\end{align}
	Furthermore, we have
	\begin{align}
		\gamma(L)[\nabla_L,\Pi_\pm] &= 0, \\
		\gamma(\underline L)[\nabla_{\underline L},\Pi_\pm] & = \pm2 \left( \underline\xi_A\gamma(e_A)\Pi_++\underline\omega\gamma(e_0)\Pi_-  \right),\\
		\gamma(e_A)[\nabla_{e_A},\Pi_\pm] & = \pm\frac12\gamma(e_0)\left( \mathrm{Tr}\underline\chi \Pi_+-\mathrm{Tr}\chi \Pi_- \right).
	\end{align}
\end{prop}
\begin{proof}
	It is sufficient to compute the covariant derivative $\nabla_Le_0,\nabla_Le_r,\nabla_{\underline L}e_0,\nabla_{\underline L}e_r$. Indeed, we see that
	\begin{align*}
		[\nabla_L,\Pi_\pm] & = \frac12 [\nabla_L,I\pm\gamma(e_0)\gamma(e_r)] \\
		& = \pm\frac12 \nabla_L(\gamma(e_0)\gamma(e_r)) \\
		& = \pm\frac12 \left(\gamma(\nabla_Le_0)\gamma(e_r)+\gamma(e_0)\gamma(\nabla_Le_r) \right).
	\end{align*}
	We have
	\begin{align*}
		\nabla_Le_0 = \frac12\nabla_L(L+\underline L) = \underline\eta_Ae_A, \quad \nabla_Le_r = \frac12\nabla_L(L-\underline L) = -\underline\eta_A e_A, \\
		\nabla_{\underline L}e_0 = \frac12\nabla_{\underline L}(L+\underline L) = (\eta_A+\underline\xi_A)e_A-\underline\omega \underline L, \quad  \nabla_{\underline L}e_r = \frac12\nabla_{\underline L}(L-\underline L) = (\eta_A-\underline\xi_A)e_A+\underline\omega \underline L, \\
		\nabla_{e_A}e_0 = \frac12 \nabla_{e_A}(L+\underline L) = \frac12(\chi_{AB}e_B-\eta_AL+\underline\chi_{AB}e_B+\eta_A\underline L) = \frac12\left( (\chi_{AB}+\underline\chi_{AB})e_B-2\eta_A e_r \right), \\
		\nabla_{e_A}e_r = \frac12 \nabla_{e_A}(L-\underline L) = \frac12(\chi_{AB}e_B-\eta_AL-\underline\chi_{AB}e_B-\eta_A\underline L) = \frac12\left( (\chi_{AB}-\underline\chi_{AB})e_B-2\eta_A e_0 \right).
	\end{align*}
	Then we obtain
	\begin{align*}
		 \frac12 \left(\gamma(\nabla_Le_0)\gamma(e_r)+\gamma(e_0)\gamma(\nabla_Le_r) \right) & = \frac12\underline\eta_A\left( \gamma(e_A)\gamma(e_r)-\gamma(e_0)\gamma(e_A) \right) \\
		 & = \frac12\underline\eta_A\left( \gamma(e_A)\gamma(e_r)+\gamma(e_A)\gamma(e_0) \right)  \\
		 & =\frac12\underline\eta_A\gamma(e_A)I \left(\gamma(e_0)+\gamma(e_r) \right) \\
		 & = \frac12\underline\eta_A\gamma(e_A)\gamma(e_0)\gamma(e_0) \left(\gamma(e_0)+\gamma(e_r) \right) \\
		 & = \underline\eta_A\gamma(e_A)\gamma(e_0)\Pi_+.
	\end{align*}
	Now we apply the gamma matrice $\gamma(L)$ and see that
	\begin{align*}
		\gamma(L)[\nabla_L,\Pi_\pm] & = \pm\underline\eta_A \gamma(L)\gamma(e_A)\gamma(e_0)\Pi_+ \\
		& = \pm\underline\eta_A\gamma(e_A)\gamma(L)\gamma(e_0)\Pi_+ \\
		& = \pm\underline\eta_A\gamma(e_A)(I-\gamma(e_0)\gamma(e_r))\Pi_+ =0.
	\end{align*}
	Now we consider the commutator of $\nabla_{\underline L}$ with the projection. The proof reduces to the computation
	\begin{align*}
		\frac12\nabla_{\underline L}\left(\gamma(e_0)\gamma(e_r)\right) & = \frac12\left((\eta_A+\underline\xi_A)\gamma(e_A)-\underline\omega\gamma(\underline L)\right)\gamma(e_r)+\frac12\gamma(e_0)\left( (\eta_A-\xi_A)\gamma(e_A)+\underline\omega\gamma(\underline L) \right) \\
		& = \frac12(\eta_A+\underline\xi_A)\gamma(e_A)\gamma(e_r)+\frac12(\eta_A-\underline\xi_A)\gamma(e_0)\gamma(e_A)+\underline\omega (\Pi_--\Pi_+).
	\end{align*}
	We note that
	\begin{align*}
		&  \frac12(\eta_A+\underline\xi_A)\gamma(e_A)\gamma(e_r)+\frac12(\eta_A-\underline\xi_A)\gamma(e_0)\gamma(e_A) \\
		& = \frac12\gamma(e_A)\left( \eta_A \gamma(e_r)+\underline\xi_A\gamma(e_r)+\underline\xi_A\gamma(e_0)-\eta_A\gamma(e_0) \right) \\
		& = \frac12\gamma(e_A) \left( \underline\xi_A(\gamma(e_0)+\gamma(e_r))-\eta_A(\gamma(e_0)-\gamma(e_r)) \right) \\
		& = \gamma(e_A)\gamma(e_0)\left( \underline\xi_A\Pi_+-\eta_A\Pi_-  \right).
	\end{align*}
	Using the identities $\gamma(e_0)\gamma(L)=2\Pi_+$ and $\gamma(e_0)\gamma(\underline L)=2\Pi_-$, we see that
\begin{align*}
	\gamma(\underline L)[\nabla_{\underline L},\Pi_\pm]\psi  &= \pm \gamma(\underline L)\gamma(e_A)\gamma(e_0)\left( \underline\xi_A\Pi_+-\eta_A\Pi_- \right)\mp \underline\omega\gamma(\underline L)(\Pi_+-\Pi_-) \\
	& = \gamma(e_A)\gamma(e_0)\gamma(L)\left( \underline\xi_A\Pi_+-\eta_A\Pi_- \right)\mp \underline\omega\gamma(e_0)\gamma(e_0)\gamma(\underline L)(\Pi_+-\Pi_-) \\
	& = \pm 2\underline\xi_A\gamma(e_A)\Pi_+\pm 2\underline\omega \gamma(e_0)\Pi_-.
\end{align*}
	Finally we have
	\begin{align*}
		\frac12\nabla_{e_A}\left(\gamma(e_0)\gamma(e_r)\right) & = \frac12 \left( \frac{\chi_{AB}+\underline\chi_{AB}}{2}\gamma(e_B)\gamma(e_r)-\eta_A\gamma(e_r)\gamma(e_r)  \right)  \\
		& \qquad + \frac12 \left( \frac{\chi_{AB}-\underline{\chi}_{AB}}{2}\gamma(e_0)\gamma(e_B)-\eta_A\gamma(e_0)\gamma(e_0) \right) \\
		& = \frac14 \left( (\chi_{AB}+\underline\chi_{AB})\gamma(e_B)\gamma(e_r)+(\chi_{AB}-\underline\chi_{AB})\gamma(e_0)\gamma(e_B) \right) \\
		& = \frac14 \gamma(e_B)\left( \underline\chi_{AB}(\gamma(e_0)+\gamma(e_r))-\chi_{AB}(\gamma(e_0)-\gamma(e_r)) \right) \\
		& = \frac12 \gamma(e_B)\gamma(e_0) \left( \underline{\chi}_{AB}\Pi_+-\chi_{AB}\Pi_- \right).
	\end{align*}
	By using the symmetry of $\chi_{AB}$ we have $\chi_{AB}\gamma(e_A)\gamma(e_B)=\frac12\chi_{AB}(\gamma(e_A)\gamma(e_B)+\gamma(e_B)\gamma(e_A))=\chi_{AB}\delta_{AB}=\textrm{Tr}\chi$, and hence we see that
	\begin{align*}
	\gamma(e_A)[\nabla_{e_A},\Pi_\pm] & = 	\pm\frac12 \gamma(e_A)\gamma(e_B)\gamma(e_0) \left( \underline{\chi}_{AB}\Pi_+-\chi_{AB}\Pi_- \right) = \pm\frac12\gamma(e_0)\left( \mathrm{Tr}\underline\chi \Pi_+-\mathrm{Tr}\chi \Pi_- \right).
	\end{align*}
	\end{proof}
	Now we complete the null decomposition for the Dirac spinor.
	\begin{prop}\label{D-decomp}
	The gauge-covariant Dirac equation
	\begin{align*}
		i\gamma(e_\mu)\nabla_{e_\mu}\psi-m\psi =0
	\end{align*}
	is decomposed via the null frame $\{e_A,\underline L,L\}$ as the following system:
			\begin{align}
	i\gamma(\underline L)\nabla_{\underline L}\psi_- & = -i\gamma(e_A)\nabla_{e_A}\psi_++m\psi_++i(2\underline\xi_A+\mathrm{Tr}\underline\chi)	\psi_++i(4\underline\omega\gamma(\underline L)-\mathrm{Tr}\chi)\psi_-, \\
	i\gamma(L)\nabla_L\psi_+ & = -i\gamma(e_A)\nabla_{e_A}\psi_-+m\psi_-+i\mathrm{Tr}\underline\chi\psi_+-i(4\underline\omega\gamma(\underline L)-\mathrm{Tr}\chi)\psi_-.
	\end{align}
	\end{prop}
	In other words, the Dirac equation can be decomposed into the following:
	\begin{align*}
		i\gamma(\underline L)\nabla_{\underline L}\psi_- & = (-i\gamma(e_A)\nabla_{e_A}+m)\psi_++(\textrm{Ricci coeff})\cdot\psi_\pm, \\
	i\gamma(L)\nabla_L\psi_+ & = (-i\gamma(e_A)\nabla_{e_A}+m)\psi_-+(\textrm{Ricci coeff})\cdot\psi_\pm,
	\end{align*}
	where $(\textrm{Ricci coeff})$ are $\chi,\eta,\xi,\omega$, satisfying $\textrm{Tr}\chi=\frac2r+O(\partial h)$ and $\eta,\xi,\omega=O(\partial h)$, with $g=m+h$.
	\begin{proof}
	By applying the projections $\Pi_\pm$ to the equation $i\gamma(L)\nabla_L\psi+i\gamma(\underline L)\nabla_{\underline L}\psi+i\gamma(e_A)\nabla_{e_A}\psi-m\psi=0$, we obtain
	\begin{align}
	\begin{aligned}
	i\gamma(L)\nabla_L\psi_-+i\gamma(\underline L)\nabla_{\underline L}\psi_- & = -i\gamma(e_A)\nabla_{e_A}\psi_++m\psi_++i\gamma(L)[\nabla_L,\Pi_-]\psi+i\gamma(\underline L)[\nabla_{\underline L},\Pi_-]\psi+i\gamma(e_A)[\nabla_{e_A},\Pi_+]\psi, \\
	i\gamma(L)\nabla_L\psi_++i\gamma(\underline L)\nabla_{\underline L}\psi_+ & = -i\gamma(e_A)\nabla_{e_A}\psi_-+m\psi_-+i\gamma(L)[\nabla_L,\Pi_+]\psi+i\gamma(\underline L)[\nabla_{\underline L},\Pi_+]\psi+i\gamma(e_A)[\nabla_{e_A},\Pi_-]\psi.
	\end{aligned}
	\end{align}
	Now we write
	\begin{align*}
		\gamma(L)\nabla_L\psi_- = 2\gamma(e_0)\Pi_+\nabla_L\psi_- = 2\gamma(e_0)\nabla_L\Pi_+\psi_-+2\gamma(e_0)[\Pi_+,\nabla_L]\psi_-=0.
	\end{align*}
	On the other hand,
	\begin{align*}
		\gamma(\underline L)\nabla_{\underline L}\psi_+ &=2\gamma(e_0)\Pi_-\nabla_{\underline L}\psi_+ \\
		& = 2\gamma(e_0)\nabla_{\underline L}\Pi_-\psi_++2\gamma(e_0)[\Pi_-,\nabla_{\underline L}]\psi_+ \\
		& = -2\gamma(e_A)\underline\xi_A\psi_+-2\underline\omega\gamma(e_0)\psi_+.
	\end{align*}
	Then we have
	\begin{align}
	\begin{aligned}
		\nabla_{\underline L}\psi_-& = -\frac12\gamma(e_0)\gamma(e_A)\nabla_{e_A}\psi_+-\frac i2m\gamma(e_0)\psi_++[\nabla_{\underline L},\Pi_-]\psi_-+\Pi_-[\nabla_{\underline L},\Pi_-]\psi+\frac12\gamma(e_0)\gamma(e_A)[\nabla_{e_A},\Pi_+]\psi, \\
		\nabla_L\psi_+ & = -\frac12\gamma(e_0)\gamma(e_A)\nabla_{e_A}\psi_--\frac i2m\gamma(e_0)\psi_-+\Pi_-[\nabla_{\underline L},\Pi_+]\psi+\frac12\gamma(e_0)\gamma(e_A)[\nabla_{e_A},\Pi_-]\psi \\
		& \qquad +(\underline\xi_A\gamma(e_0)\gamma(e_A)+\underline\omega)\psi_++[\nabla_L,\Pi_+]\psi_+.
	\end{aligned}
	\end{align}
\end{proof}
\section{Null decomposition for the Maxwell fields}\label{sec:null-tensor}
The purpose of this section is to derive the equation for the null components of the Maxwell fields:
\begin{align}
\nabla_L\rho+\underline\eta_A\alpha_A+\slashed{\mathrm{div}}\alpha+\left( \frac2r+O(r^{-2}) \right)\rho +\eta_A\alpha_A= J_L, \\
 -\nabla_{\underline L}\rho +\slashed{\mathrm{div}}\underline\alpha+\frac2r \rho = J_{\underline L}, \\
  \frac12\nabla_{\underline L}\alpha_A+\frac12\nabla_L \underline\alpha_A+{\epsilon_A}^{\!B}\slashed{\nabla}_{e_B}\sigma + O(r^{-2})(\alpha+\underline\alpha+\sigma) = J_{e_A}, \\
   -\nabla_{\underline L}\alpha_A+\nabla_L  \underline\alpha_A +2\nabla_{e_A}\rho+2(\underline\eta_B-\eta_B) F(e_A,e_B) = 0, \\
   \nabla_L\sigma-\frac2r\sigma+O(r^{-2})\sigma+\slashed{\mathrm{curl}}\alpha+(\eta+\underline\eta)\wedge\alpha= 0, \\
    \nabla_{\underline L}\sigma+\frac2r\sigma+O(r^{-2})\sigma+\slashed{\mathrm{curl}}\underline\alpha =0, \\
     2\slashed{\nabla}_{e_C}\sigma = \epsilon^{AB}\chi_{BC}\underline\alpha_A+\epsilon^{AB}\underline\chi_{BC}\alpha_A.
\end{align}
We recall that
\begin{align}
    \alpha_A = F(e_A,L), \ \underline\alpha_A = F(e_A,\underline L), \ \rho = \frac12 F(L,\underline L), \ \sigma = F(e_A,e_B) \epsilon_{AB}.
\end{align}
We first consider the equation
\begin{align}
    \nabla^\mu F_{\mu\nu}= J_\nu.
\end{align}
If $\nu = L$, we have $\nabla^\mu F_{\mu L}=\nabla^L F_{LL}+\nabla^{\underline L}F_{\underline LL}+\nabla^{e_A}F_{AL}=J_L$, or equavalently
\begin{align*}
    -\frac12 \nabla_L F_{\underline LL}+\slashed{g}^{AB}\nabla_{e_B}F_{AL} = J_L,
\end{align*}
where $\slashed{g}$ is the induced metric on the $2$-sphere.
We note that for any $2$-tensor $T$ and any vector fields $X,Y,Z$,
\begin{align}
    X ( T(Y,Z)) = (\nabla_X T)(Y,Z) + T(\nabla_X Y,Z) + T(Y, \nabla_XZ).
\end{align}
Then
\begin{align*}
    \nabla_L F_{\underline LL} &= L ( F(\underline L,L)) - F(\nabla_L\underline L,L) - F(\underline L,\nabla_LL) \\
    & = -2\nabla_L\rho -2\underline\eta_A \alpha_A,
\end{align*}
and
\begin{align*}
    \nabla_{e_B}F_{AL} & = e_B ( F(e_A,L)) -F(\nabla_{e_B}e_A,L)-F(e_A,\nabla_{e_B}L) \\
    & = e_B (\alpha_A)- F(\slashed{\nabla}_{e_B}e_A+\frac12\chi_{AB}\underline L+\frac12\underline\chi_{AB}L,L)-F(e_A,\chi_{BC}e_C-\eta_BL) \\
    & = e_B(\alpha_A)-\langle \nabla_{e_B}e_A,e_C\rangle F(e_C,L)-\frac12\chi_{AB}F(\underline L,L)-\chi_{BC}F(e_A,e_C)+\eta_B F(e_A,L).
\end{align*}
We also note that
\begin{align*}
\slashed{\mathrm{div}}\alpha &= \slashed{\nabla}_{e_A}\alpha_A \\
& = e_A(\alpha_A)  -\alpha(\slashed{\nabla}_{e_A}e_A ) \\
& = e_A(\alpha_A) - \langle \nabla_{e_A}e_A,e_C\rangle F(e_C,L).
\end{align*}
Then
\begin{align*}
    \slashed{g}^{AB}\nabla_{e_B}F_{AL} & = \slashed{\mathrm{div}}\alpha-\frac12 \mathrm{Tr}\chi F(\underline L,L) -\chi_{AC}F(e_A,e_C)+\eta_A\alpha_A \\
    & = \slashed{\mathrm{div}}\alpha+ \left( \frac2r+O(r^{-2}) \right)\rho +\eta_A\alpha_A.
\end{align*}
Hence we obtain
\begin{align}
    \nabla_L\rho+\underline\eta_A\alpha_A+\slashed{\mathrm{div}}\alpha+\left( \frac2r+O(r^{-2}) \right)\rho +\eta_A\alpha_A= J_L.
\end{align}
If $\nu = \underline L$, we have $\nabla^\mu F_{\mu\underline L}= \nabla^L F_{L\underline L}+\nabla^{e_A}F_{A\underline L} =J_{\underline L}$, or equivalently
\begin{align*}
    -\frac12 \nabla_{\underline L}F_{L\underline L}+\slashed{g}^{AB}\nabla_{e_B}F_{A\underline L}  = J_{\underline L}.
\end{align*}
Then we write
\begin{align*}
    \nabla_{\underline L}F_{L\underline L} & = \underline L ( F(L,\underline L)) - F(\nabla_{\underline L}L,\underline L) -F(L,\nabla_{\underline L}\underline L) \\
    & = 2\nabla_{\underline L}\rho -2\eta_A\underline\alpha_A,
\end{align*}
and we also have
\begin{align*}
    \nabla_{e_B}F_{A\underline L} & = e_B ( F(e_A,\underline L)) - F(\nabla_{e_B}e_A,\underline L) - F(e_A,\nabla_{e_B}\underline L) \\
    & = e_B(\underline\alpha_A) - F(\slashed{\nabla}_{e_B}e_A+\frac12\chi_{AB}\underline L+\frac12\underline\chi_{AB}L,\underline L) - F(e_A,\underline\chi_{BC}e_C+\eta_B\underline L) \\
    & = e_B (\underline\alpha_A) - \langle \nabla_{e_B}e_A,e_C\rangle F(e_C,\underline L) - \frac12\underline\chi_{AB}F(L,\underline L)-\eta_B F(e_A,\underline L).
\end{align*}
Then
\begin{align*}
    \slashed{g}^{AB}\nabla_{e_B}F_{AL} & = \slashed{\mathrm{div}}\underline\alpha+\frac2r \rho-\eta_A\underline\alpha_A,
\end{align*}
and hence we conclude that
\begin{align}
    -\nabla_{\underline L}\rho +\slashed{\mathrm{div}}\underline\alpha+\frac2r \rho = J_{\underline L}.
\end{align}
If $\nu = A$, we have $\nabla^\mu F_{\mu A}=\nabla^L F_{LA}+\nabla^{\underline L}F_{\underline LA}+\nabla^{e_B}F_{BA}=J_{e_A}$, or equivalently,
\begin{align*}
    -\frac12 \nabla_{\underline L}F_{LA}-\frac12\nabla_L F_{\underline LA}+\slashed{g}^{BC}\nabla_{e_C}F_{BA} = J_{e_A}.
\end{align*}
Now we write
\begin{align*}
    \nabla_{\underline L}F_{LA} & = \underline L ( F(L,e_A)) - F(\nabla_{\underline L}L,e_A) - F(L,\nabla_{\underline L}e_A) \\
    & = -\nabla_{\underline L}\alpha_A-2\eta_C F(e_C,e_A) - \underline{\chi}_{AC} F(L,e_C) -\eta_A F(L,\underline L) \\
    & = -\nabla_{\underline L}\alpha_A +\underline\chi_{AB}\alpha_B-2\eta_C F(e_C,e_A) -2\eta_A\rho,
\end{align*}
and
\begin{align*}
    \nabla_L F_{\underline LA} & = L ( F(\underline L,e_A)) - F(\nabla_L\underline L,e_A)-F(\underline L,\nabla_Le_A) \\
    & = -\nabla_L \underline\alpha_A -2\underline\eta_C F(e_C,e_A) - \langle\nabla_L e_A,e_C\rangle F(\underline L,e_C) -\eta_A F(\underline L,L) \\
    & = -\nabla_L \underline{\alpha}_A +\chi_{AB}\underline\alpha_B-2\underline\eta_C F(e_C,e_A) +2\eta_A \rho,
\end{align*}
and
\begin{align*}
    \nabla_{e_C}F(e_B,e_A) & = e_C ( F(e_B,e_A)) - F(\nabla_{e_C}e_B,e_A) - F(e_B,\nabla_{e_C}e_A) \\
    & = e_C( F(e_B,e_A)) - F( \slashed{\nabla}_{e_C}e_B+\frac12\chi_{BC}\underline L+\frac12\underline{\chi}_{BC}L,e_A) \\
    & \qquad - F(e_B,\slashed{\nabla}_{e_C}e_A+\frac12\chi_{AC}\underline L+\frac12\underline\chi_{AC}L) .
\end{align*}
Since $F_{AB}=\sigma \epsilon_{AB}$,
\begin{align*}
    \slashed{g}^{BC}\nabla_{e_C}F(e_B,e_A) & = {\epsilon_A}^{\!B}\slashed{\nabla}_{e_B}\sigma -\frac12 \mathrm{Tr}\chi
    F(\underline L,e_A) -\frac12\mathrm{Tr}\underline\chi F(L,e_A)-\frac12\chi_{AB}F(e_B,\underline L) -\frac12\underline\chi_{AB}F(e_B,L) \\
    & = {\epsilon_A}^{\!B}\slashed{\nabla}_{e_B}\sigma  +\frac12 \mathrm{Tr}\chi \underline\alpha_A +\frac12 \mathrm{Tr}\underline\chi \alpha_A -\frac12\chi_{AB}\underline\alpha_B -\frac12\underline\chi_{AB}\alpha_B.
\end{align*}
We recall that $\chi_{AB}=\frac1r \slashed{g}_{AB}+O(r^{-2})$ and $\underline\chi_{AB}=-\chi_{AB}+O(r^{-2})$ and combine all the computations to get
\begin{align*}
	-\frac12\nabla_{\underline L}F_{LA}-\frac12\nabla_L F_{\underline LA}+\slashed{g}^{BC}\nabla_{e_C}F_{BA} & = \frac12 \nabla_{\underline L}\alpha_A-\frac12\underline{\chi}_{AB}\alpha_B+\eta_B F(e_B,e_A)+\eta_A \rho \\
	& \qquad + \frac12\nabla_L \underline\alpha_A -\frac12\chi_{AB}\underline\alpha_B +\underline\eta_B F(e_B,e_A) -\eta_A\rho \\
	& \qquad + {\epsilon_A}^{\!B}\slashed{\nabla}_{e_B}\sigma  +\frac12 \mathrm{Tr}\chi \underline\alpha_A +\frac12 \mathrm{Tr}\underline\chi \alpha_A -\frac12\chi_{AB}\underline\alpha_B -\frac12\underline\chi_{AB}\alpha_B \\
	& = \frac12 \nabla_{\underline L}\alpha_A+ \frac12\nabla_L \underline\alpha_A+ {\epsilon_A}^{\!B}\slashed{\nabla}_{e_B}\sigma +\frac12 \mathrm{Tr}\chi \underline\alpha_A +\frac12 \mathrm{Tr}\underline\chi \alpha_A -\chi_{AB}\underline\alpha_B -\underline\chi_{AB}\alpha_B \\
	& \qquad + (\eta_B+\underline\eta_B) F(e_B,e_A) \\
	& = \frac12 \nabla_{\underline L}\alpha_A+ \frac12\nabla_L \underline\alpha_A+ {\epsilon_A}^{\!B}\slashed{\nabla}_{e_B}\sigma  + O(r^{-2}) (\alpha+\underline\alpha+\sigma),
\end{align*}
and hence we obtain
\begin{align}
\begin{aligned}
   \frac12 \nabla_{\underline L}\alpha_A+ \frac12\nabla_L \underline\alpha_A+ {\epsilon_A}^{\!B}\slashed{\nabla}_{e_B}\sigma  + O(r^{-2}) (\alpha+\underline\alpha+\sigma) = J_{e_A}.
    \end{aligned}
\end{align}
From now on we consider the Bianchi identitiy:
\begin{align}
    \nabla_\lambda F_{\mu\nu}+\nabla_\mu F_{\nu\lambda}+\nabla_\nu F_{\lambda\mu}=0.
\end{align}
If $(\lambda,\mu,\nu)=(L,\underline L,A)$, we have
\begin{align}
    \nabla_LF_{\underline LA}+\nabla_{\underline L}F_{AL}+\nabla_{e_A}F_{L\underline L} = 0.
\end{align}
We see that
\begin{align*}
    \nabla_{e_A}F_{L\underline L} & = e_A ( F(L,\underline L)) - F(\nabla_{e_A}L,\underline L)-F(L,\nabla_{e_A}\underline L) \\
    & = 2\nabla_{e_A}\rho -F (\chi_{AB}e_B-\eta_AL,\underline L) - F(L,\underline\chi_{AB}e_B+\eta_A\underline L) \\
    & = 2\nabla_{e_A}\rho - \chi_{AB}\underline\alpha_B+\underline\chi_{AB}\alpha_B.
\end{align*}
Hence the identity is rewritten as follows:
\begin{align}
    \begin{aligned}
 & \underline L ( F(L,e_A)) - F(\nabla_{\underline L}L,e_A) - F(L,\nabla_{\underline L}e_A) \\
    & = -\nabla_{\underline L}\alpha_A-2\eta_C F(e_C,e_A) +\chi_{AB}\underline\alpha_B -2\eta_A \rho \\
 & \qquad +\nabla_L \underline\alpha_A +2\underline\eta_C F(e_C,e_A)  -\underline\chi_{AB}\alpha_B+2\eta_A\rho \\
 & \qquad\qquad  +2\nabla_{e_A}\rho - \chi_{AB}\underline\alpha_B+\underline\chi_{AB}\alpha_B \\
 & =  -\nabla_{\underline L}\alpha_A+\nabla_L +2\nabla_{e_A}\rho\underline\alpha_A +2(\underline\eta_B-\eta_B) F(e_A,e_B).
 \end{aligned}
\end{align}
Then we have
\begin{align}
	 -\nabla_{\underline L}\alpha_A+\nabla_L  \underline\alpha_A +2\nabla_{e_A}\rho+2(\underline\eta_B-\eta_B) F(e_A,e_B) = 0.
\end{align}
If $(\lambda,\mu,\nu)=(L,A,B)$, we have
\begin{align}
    \nabla_L F_{AB}+\nabla_{e_A}F_{BL}+\nabla_{e_B}F_{LA} = 0.
\end{align}
Then we write
\begin{align*}
    \nabla_L F_{AB} & = L ( F(e_A,e_B)) -F(\nabla_Le_A,e_B) - F(e_A,\nabla_Le_B) \\
    & = \epsilon_{AB
    } \nabla_L\sigma + \mathrm{Tr}\chi \epsilon_{AB}-\chi_{AC}F(e_C,e_B)-\underline\eta_A F(L,e_B) - \chi_{BC} F(e_A,e_C)-\underline\eta_B F(e_A,L) \\
    & = \epsilon_{AB} \left( \nabla_L\sigma+\frac2r \sigma+O(r^{-2})\sigma \right)- \chi_{AC}F(e_C,e_B)-\chi_{BC}F(e_A,e_C) - \underline\eta_B\alpha_A+\underline\eta_A\alpha_B,
\end{align*}
where we used the identity
\begin{align}
\nabla_L \epsilon_{AB} = \mathrm{Tr}\chi \epsilon_{AB}.
\end{align}
We also have
\begin{align*}
    \nabla_{e_B}F_{LA} & = e_B ( F(L,e_A))-F(\nabla_{e_B}L,e_A) - F(L,\nabla_{e_B}e_A) \\
    & = -e_B(\alpha_A) -F(\chi_{BC}e_C-\eta_B L,e_A) - F(L,\slashed{\nabla}_{e_B}e_A+\frac12\chi_{AB}\underline L+\frac12\underline\chi_{AB}L)  \\
    & = - \slashed{\nabla}_{e_B}\alpha_A -\chi_{BC}F(e_C,e_A) +\eta_B F(L,e_A) -\frac12 \chi_{AB}F(L,\underline L) \\
    & = - \slashed{\nabla}_{e_B}\alpha_A-\chi_{BC}F(e_C,e_A) -\eta_B \alpha_A-\chi_{AB}\rho,
\end{align*}
and
\begin{align*}
    \nabla_{e_A}F_{BL} & = e_A ( F(e_B,L)) - F(\nabla_{e_A}e_B,L) - F(e_B, \nabla_{e_A}L) \\
    & = e_A(\alpha_B) - F(\slashed{\nabla}_{e_A}e_B +\frac12\chi_{AB}\underline L+\frac12\underline{\chi}_{AB}L,L) - F(e_B,\chi_{AC}e_C-\eta_AL) \\
    & = \slashed{\nabla}_{e_A}\alpha_B +\chi_{AB}\rho-\chi_{AC}F(e_B,e_C) +\eta_A\alpha_B.
\end{align*}
We note that the term $\nabla_L F_{AB}$ vanishes unless $A\neq B$ while for $A=B$,
\begin{align*}
\nabla_{e_B}F_{LA}+\nabla_{e_A}F_{BL} = \nabla_{e_A}F_{LA}+\nabla_{e_A}F_{AL} = \nabla_{e_A}( F_{LA}+F_{AL}) \equiv \nabla_{e_A}0 = 0.
\end{align*}
Thus combining the above computation gives
\begin{align}
    \nabla_L\sigma+\frac2r\sigma+O(r^{-2})\sigma+\epsilon_{AB}\slashed{\nabla}_{e_A}\alpha_B+\epsilon_{AB}(\eta_A+\underline\eta_A)\alpha_B= 0,
\end{align}
or equivalently,
\begin{align}
    \nabla_L\sigma+\frac2r\sigma+O(r^{-2})\sigma+\slashed{\mathrm{curl}}\alpha+(\eta+\underline\eta)\wedge\alpha= 0,
\end{align}
If $(\lambda,\mu,\nu)=(\underline L,A,B)$, we have
\begin{align}
    \nabla_{\underline L}F_{AB}+\nabla_{e_A}F_{B\underline L}+\nabla_{e_B}F_{\underline LA} = 0.
\end{align}
Then we write
\begin{align*}
    \nabla_{\underline L}F_{AB} & = \underline L ( F(e_A,e_B))-F(\nabla_{\underline L}e_A,e_B) -F(e_A,\nabla_{\underline L}e_B) \\
    & = \underline L ( \sigma \epsilon_{AB}) - F(\underline\chi_{AC}e_C+\eta_A\underline L,e_B) - F(e_A, \underline\chi_{BC}e_C+\eta_B\underline L) \\
    & = \epsilon_{AB}\left( \nabla_{\underline L}-\frac2r \sigma +O(r^{-2})\sigma \right)-\underline\chi_{AC}F(e_C,e_B) -\underline\chi_{BC}F(e_A,e_C) -\eta_A F(\underline L,e_B) -\eta_B F(e_A,\underline L),
    \end{align*}
    where we used the identity
\begin{align}
\nabla_{\underline L}\epsilon_{AB} = \mathrm{Tr}\underline\chi \epsilon_{AB}.
\end{align}
We also have
\begin{align*}
    \nabla_{e_A}F_{B\underline L} & = e_A ( F(e_B,\underline L))-F(\nabla_{e_A}e_B,\underline L) - F(e_B,\nabla_{e_A}\underline L) \\
    &  = e_A(\underline\alpha_B)- F(\slashed{\nabla}_{e_A}e_B+\frac12\chi_{AB}\underline L+\frac12\underline{\chi}_{AB}L,\underline L) - F(e_B,\underline\chi_{AC}e_C+\eta_A\underline L) \\
    & = \slashed{\nabla}_{e_A}\underline\alpha_B-\frac12 \underline\chi_{AB}F(L,\underline L)- \underline\chi_{AC}F(e_B,e_C) - \eta_A F(e_B,\underline L) \\
    & = \slashed{\nabla}_{e_A}\underline\alpha_B - \underline\chi_{AB}\rho -\underline\chi_{AC}F(e_B,e_C) -\eta_A \underline\alpha_B,
\end{align*}
and
\begin{align*}
    \nabla_{e_B}F_{\underline LA} & = e_B( F(\underline L,e_A)) - F(\nabla_{e_B}\underline L,e_A) - F( \underline L,\nabla_{e_B}e_A) \\
    & = -e_B(\underline\alpha_A)- F(\underline L,\slashed{\nabla}_{e_B}e_A+\frac12\chi_{AB}\underline L+\frac12\underline\chi_{AB}L) - F(\underline\chi_{BC}e_C+\eta_B\underline L,e_A) \\
    & = -\slashed{\nabla}_{e_B}\underline\alpha_A-\frac12 \underline\chi_{AB}F(\underline L,L) -\underline\chi_{BC}F(e_C,e_A) -\eta_B F(\underline L,e_A) \\
    & = -\slashed{\nabla}_{e_B}\underline\alpha_A +\underline\chi_{AB}\rho -\underline\chi_{BC}F(e_C,e_A) +\eta_B\underline\alpha_A.
\end{align*}
Therefore we deduce that
\begin{align}
    \nabla_{\underline L}\sigma-\frac2r\sigma+O(r^{-2})\sigma+\slashed{\nabla}_{e_A}\underline\alpha_B-\slashed{\nabla}_{e_B}\underline\alpha_A  = 0,
\end{align}
or equivalently,
\begin{align}
    \nabla_{\underline L}\sigma-\frac2r\sigma+O(r^{-2})\sigma+\slashed{\mathrm{curl}}\underline\alpha  = 0.
\end{align}
If $(\lambda,\mu,\nu) = (A,B,C)$, we have
\begin{align}
\nabla_{e_A}F_{BC}+\nabla_{e_B}F_{AC}+\nabla_{e_C}F_{AB} = 0.
\end{align}
Then we write
\begin{align*}
    \nabla_{e_A}F_{BC} & = e_A ( F(e_B,e_C)) - F({\nabla}_{e_A}e_B,e_C) - F(e_B,\nabla_{e_A}e_C) \\
    & = e_A ( F(e_B,e_C))- F( \slashed{\nabla}_{e_A}e_B+\frac12\chi_{AB}\underline L+\frac12 \underline\chi_{AB}L,e_C) - F(e_B,\slashed{\nabla}_{e_A}e_C+\frac12\chi_{AC}\underline L+\frac12 \underline\chi_{AC}L) \\
    & = \epsilon_{BC}\slashed{\nabla}_{e_A}\sigma +\frac12\chi_{AB}\underline\alpha_C+\frac12\underline\chi_{AB}\alpha_C-\frac12\chi_{AC}\underline\alpha_B-\frac12\underline\chi_{AC}\alpha_B,
\end{align*}
and similarly,
\begin{align*}
    \nabla_{e_B}F_{AC} & = \epsilon_{AC}\slashed{\nabla}_{e_B}\sigma+\frac12 \chi_{AB}\underline\alpha_C+\frac12\underline\chi_{AB}\alpha_C-\frac12\chi_{BC}\underline\alpha_A-\frac12\underline\chi_{BC}\alpha_A, \\
    \nabla_{e_C}F_{AB} & = \epsilon_{AB}\slashed{\nabla}_{e_C}\sigma+\frac12\chi_{CA}\underline\alpha_B+\frac12\underline\chi_{CA}\alpha_B-\frac12\chi_{CB}\underline\alpha_A-\frac12\underline\chi_{CB}\alpha_A.
\end{align*}
Then we obtain
\begin{align}
    \epsilon_{BC}\slashed{\nabla}_{e_A}\sigma+\epsilon_{AC}\slashed{\nabla}_{e_B}\sigma+\epsilon_{AB}\slashed{\nabla}_{e_C}\sigma +\chi_{AB}\underline\alpha_C+\underline\chi_{AB}\alpha_C-\chi_{BC}\underline\alpha_A-\underline\chi_{BC}\alpha_A = 0.
\end{align}
We shall further simplify the above identity by multiplying $\epsilon^{AB}$. Indeed, an easy computation yields $\epsilon^{AB}\epsilon_{AB}=2$ and $\epsilon^{AB}\epsilon_{BC}={\delta^A}_{\!C}$, which implies that
\begin{align}
    2\slashed{\nabla}_{e_C}\sigma = \epsilon^{AB}\chi_{BC}\underline\alpha_A+\epsilon^{AB}\underline\chi_{BC}\alpha_A.
\end{align}
\section*{Acknowledgements}
The author would like to express his gratitude to Gustav Holzegel for his guidance and many valuable discussions.
He acknowledges support by the Alexander
von Humboldt Foundation in the framework of the Alexander von Humboldt Professorship endowed by the
Federal Ministry of Education and Research.

This work is also funded by the Deutsche Forschungsgemeinschaft (DFG, German Research Foundation) under Germany's Excellence Strategy EXC 2044/2 -390685587, Mathematics M\"unster: Dynamics-Geometry-Structure.




\begin{thebibliography}{00}

\bibitem{alcu} M. Alcubierre, {\it The Dirac equation in general relativity and $3+1$ formalism}, available in arXiv:2503.03918.

\bibitem{alinhac} S. Alinhac, {\it Geometric Analysis of Hyperbolic
Differential Equations: An Introduction}, London Mathematical Society Lecture Note Series: 374, (2010).

\bibitem{artzcaccia} J. Ben-Artzi, F. Cacciafesta, A.S. de Suzzoni, J. Zhang, {\it Global Strichartz estimates for the Dirac equation on symmetric spaces}, Forum of Mathematics, Sigma, 10, E25, (2022).

\bibitem{batic} D. Batic, {\it Scattering for massive Dirac fields on the Kerr metric}, J. Math. Phys.
\textbf{48}, 022502, (2007).

\bibitem{bar} C. B\"ar, {\it The Dirac operator on hyperbolic manifolds of finite volume}, J. Diff. Geom. \textbf{53}, (1999), 439--488.

\bibitem{behe} I. Bejenaru and S. Herr, {\it The cubic Dirac equation: small initial data in $H^1(\mathbb R^3)$}, Comm. Math. Phys. \textbf{335}(1), (2015), 43--83.

\bibitem{behe1} I. Bejenaru and S. Herr, {\it The cubic Dirac equation: small initial data in $H^{1/2}(\mathbb R^2)$}, Comm. Math. Phys. \textbf{343}(2), (2016), 515--562.

\bibitem{beherr} I. Bejenaru and S. Herr, {\it On global well-posedness and scattering for the massive Dirac-Klein-Gordon system}, J. Eur. Math. Soc. (JEMS), \textbf{19}:8, (2017), 2445--2467.

\bibitem{boucan} N. Bournaveas and T. Candy, {\it Global well-posedness for the massless cubic Dirac equation}, Int. Math. Res. Notices. No. 22, (2016), 6735--6828.

\bibitem{boucanma} N. Bournaveas, T. Candy, S. Machihara. {\it A note on the Chern-Simons-Dirac equations in the Coulomb gauge}, Discrete and Continuous Dynamical Systems, \textbf{34} No. 7, (2014), 2693--2701.

\bibitem{brand} V. Brandling, K. Kr\"oncke, {\it Global existence of Dirac-wave maps with curvature term on expanding spacetimes}, Calc. Var. \textbf{57}, No. 119, (2018)


\bibitem{cacciasu1} F. Cacciafesta and A. S. de Suzzoni, {\it Weak dispersion for the Dirac equation on asymptotically flat and warped products spaces}, Discrete Contin. Dyn. Syst., \textbf{39}(8), (2019), 4359--4398.

\bibitem{cacciasu} F. Cacciafesta and A. S. de Suzzoni, {\it Local in time Strichartz estimates for the Dirac equation on spherically symmetric spaces}, Int. Math. Res. Not., Vol. 2022, Issue 4, (2022), 2729--2771.

\bibitem{caccia1} F. Cacciafesta, A.S. de Suzzoni, and L. Meng, {\it Strichartz estimates for the Dirac equations on asymptotically flat manifolds}, available in doi.org/10.2422/2036-2145.202203-026.

\bibitem{caccia2} F. Cacciafesta, E. Danesi, L. Meng, {\it Strichartz estimates for the half-wave/Klein-Gordon and Dirac equations on compact manifolds without boundary}, Math. Ann., doi.org/10.1007/s00208-023-02716-5.

\bibitem{canhe2} T. Candy, S. Herr, {\it Transference of bilinear restriction estimates to quadratic variation norms and the Dirac-Klein-Gordon system}, Anal. PDE, \textbf{11}, No. 5, (2018), 1171--1240.

\bibitem{canhe1} T. Candy, S. Herr, {\it Conditional large initial data scattering results for the Dirac-Klein-Gordon system}, Forum of Mathematics, Sigma, \textbf{6}, (2018).

\bibitem{canhe} T. Candy, S. Herr, {\it On the Majorana condition for nonlinear Dirac systems}, Ann. I. H. Poincar\'e, Vol. 35, (2018), 1707--1717.

\bibitem{chagla} J. M. Chadam, R. T. Glassey, {\it On certain global solutions of the Cauchy problem for the (classical) coupled Klein-Gordon-Dirac equations in one and three space dimensions}, Arch. Ration. Mech. Anal., \textbf{54}, (1974), 223--237.

\bibitem{chagla1} J. M. Chadam, R. T. Glassey, {\it On the Maxwell-Dirac equations with zero magnetic field and their solution in two space dimensions}, J. Math. Anal. Appl., \textbf{53} (1976), 495--597.

\bibitem{chen} X. Chen, {\it Global stability of Minkowski spacetime for a spin-1/2 field}, Advances in Theoretical and Mathematical Physics, \textbf{29}, No. 2, (2025), 485--556.

\bibitem{chohonglee} Y. Cho, S. Hong, K. Lee, {\it Scattering and nonscattering of the Hartree-type nonlinear Dirac system at critical regularity}, SIAM J. Math. Anal. \textbf{55}, No. 4, (2023).

\bibitem{chohongoz} Y. Cho, S. Hong, T. Ozawa, {\it Charge conjugation approach to scattering for the Hartree type Dirac equations with chirality}, J. Math. Phys. \textbf{64}, 021508, (2023).

\bibitem{chokwonleeyang} Y. Cho, S. Kwon, K. Lee, C. Yang, {\it The modified scattering for Dirac equations of scattering-critical nonlinearity}, Adv. Diff. Equ., 29(3/4), (2024), 179--222.


\bibitem{choklee} Y. Cho, K. Lee, {\it The global dynamics for the Maxwell-Dirac system}, available in arXiv:2406.18887.

\bibitem{choleeoz} Y. Cho, K. Lee, T. Ozawa, {\it Small data scattering of 2D Hartree type Dirac equations}, J. Math. Anal. Appl., 506, 125549, (2022).


\bibitem{christoklai} D. Christodoulou, S. Klainerman, {\it The Global Nonlinear Stability of the Minkowski Space} Princeton University Press, Princeton, (1993).

\bibitem{cloos} C.C. Cloos, {\it On the long-time behavior of the three-dimensional dirac-maxwell equation
with zero magnetic field}, (2020).

\bibitem{daferrod1} M. Dafermos, I. Rodnianski, {\it The redshift effect and radiation decay on black hole spacetimes}, Comm. Pure Appl. Math., \textbf{52}, (2009), 859--919.

\bibitem{daferrod} M. Dafermos, I. Rodnianski, {\it A proof of the uniform boundedness of solutions to the wave equation on slowly rotating Kerr backgrounds}, Invent. Math., \textbf{185}, (2011), 467--559.

\bibitem{dahorod} M. Dafermos, G. Holzegel, I. Rodnianski, {\it Boundedness and decay for the Teukolsky equation on Kerr spacetimes I: the case $|a|\ll m$}, Annals of PDE, \textbf{5}, No. 1, (2019).

\bibitem{dahorod1} M. Dafermos, G. Holzegel, I. Rodnianski,  {\it The linear stability of the Schwarzschild solution to gravitational perturbations}, Acta Math. \textbf{222}, No. 1, (2019), 1--214.



\bibitem{DHRT} M. Dafermos, G. Holzegel, I. Rodnianski, M. Taylor, {\it Quasilinear wave equations on asymptotically flat spacetimes
with applications to Kerr black holes}, Annal. \& PDE., Vol. 19 (2026), No. 5, 909--1028.

\bibitem{DR} M. Dafermos, I. Rodnianski, {\it A new physical-space approach to decay for the wave equation with applications to black hole spacetimes}, XVIth International Congress on Mathematical Physics, pp. 421-432 (2010).

\bibitem{dirac1928} P. A. M. Dirac, {\it The quantum theory of the electron}, Proc. Roy. Soc. London A, \textbf{117} (1928), 610--624.

\bibitem{donglefloch} S. Dong, P. G. LeFloch, Z. Wyatt, {\it Global evolution of $U(1)$ Higgs Boson: nonlinear stability and uniform energy bounds}, Ann. Henri Poincar\'e, 22, (2021), 677--713.

\bibitem{escobedovega} M. Escobedo, L. Vega, {\it A semilinear Dirac equation in $H^s(\mathbf R^3)$ for $s>1$}, SIAM J. Math. Anal. Vol. 28, No. 2, (1997), 338--362.

\bibitem{finster} F. Finster, N. Kamran, J. Smoller, S.-T. Yau, {\it Decay rates and probability estimates for massive Dirac particles in the Kerr-Newman black hole Geometry}, Commun. Math. Phys. \textbf{230}, (2002), 201--244

\bibitem{galyag} A. Galstian, K. Yagdjian, {\it The self-interacting Dirac fields in FLRW spacetime}, Nonlinear Diff. Equ. Appl. (2022) 29:62.

\bibitem{gavrusoh} C. Gavrus, S.-J. Oh, {\it Global well-posedness of high dimensional Maxwell-Dirac for small critical data}, Memoirs. Amer. Math. Soc., \textbf{264}, No. 1279, (2020).

\bibitem{geosha} V. Georgiev, B. Shakarov, {\it Global large data solutions for 2D Dirac equation with
Hartree type interaction}, IMRN, Vol. 2022, Issue 17, (2022), 12803--12820.

\bibitem{geogiev} V. Georgiev, {\it Small amplitude solutions of the Maxwell-Dirac equations}, Indiana University Mathematics Journal, Vol. 40, no. 3, (1991), 845--883.

\bibitem{ginoux} N. Ginoux, O. M\"uller, {\it Global solvability of massless Dirac-Maxwell systems}, Annales de l'Institut Henri Poincar\'e C, Analyse non lin\'eaire, Vol. 35, No. 6, (2018), 1645--1654.

\bibitem{hanicolas} D. H\"afner, J.-P. Nicolas. {\it Scattering of massless Dirac fields by a Kerr black
hole}, Rev. Math. Phys. \textbf{16}. No. 1, (2004), 29--123.

\bibitem{hafnernicolas} D. H\"afner, J.-P. Nicolas, {\it The characteristic Cauchy problem for Dirac fields on
curved backgrounds}, Journal of Hyperbolic Diff. Equ. Vol. 8, No. 3, (2011), 437--483.


\bibitem{herrh} S. Herr, S. Hong, {\it Strichartz estimates for the half Klein-Gordon equation on asymptotically flat backgrounds and applications to cubic Dirac equations}, available in arXiv:2502.13670.

\bibitem{herrifrimspitz} S. Herr, M. Ifrim, M. Spitz, {\it Modified scattering for the three dimensional Maxwell-Dirac system}, Ann. PDE 12, 14 (2026).

\bibitem{herrmaul} S. Herr, C. Maul\'en, C. Mu\~noz, {\it Decay of solutions of nonlinear Dirac equations}, Commun. Math. Phys. 407, 94 (2026).

\bibitem{holzegel10} G. Holzegel, {\it Ultimately Schwarzschildean Spacetimes and the
Black Hole Stability Problem}, available on arXiv:1010.3216.


\bibitem{klai} S. Klainerman, {\it The null condition and global existence to nonlinear wave equations}, Nonlinear systems of partial differential equations in applied mathematics, (1986), 293--326.

\bibitem{kwonleeyang} S. Kwon, K. Lee, C. Yang, {\it The modified scattering of two dimensional semi-relativistic Hartree equations}, J. Evol. Equ. \textbf{24}, No. 64, (2024).

\bibitem{leflochmazhang} P.G. LeFloch, Y. Ma, W. Zhang, {\it The Global nonlinear stability of Minkowski spacetime
with self-gravitating massive Dirac fields}, available on arXiv:2510.20626.

\bibitem{lawson} H. B. Lawson, M.-L. Michelsohn, {\it Spin geometry}, Princeton, NJ:
Princeton University Press, (1989)

\bibitem{klee} K. Lee, {\it Scattering results for the $(1+4)$ dimensional massive Maxwell-Dirac system under Lorenz gauge condition}, available in arXiv:2312.13621.


\bibitem{lizhang} J. Li, Y. Zhang, {\it A vector field method for some nonlinear Dirac models
in Minkowski spacetime}, Journal of Differential Equations, \textbf{273}, (2021): 58--82.

\bibitem{mazhang} S. Ma, L. Zhang, {\it Sharp decay estimates for massless Dirac fields on a Schwarzschild background}, J. Func. Anal. \textbf{282}, (2022), 109375.

\bibitem{machihara} S. Machihara, M. Nakamura, K. Nakanishi, and T. Ozawa, {\it Endpoint Strichartz estimates and global solutions for the nonlinear Dirac equations}, J. Funct. Anal. \textbf{219}(1), (2005), 1--20.

\bibitem{parker} L.E. Parker and D. J. Toms, {\it Quantum field theory in curved spacetime}, Cambridge university press.

\bibitem{pasqualotto} F. Pasqualotto, {\it The Spin $\pm1$ Teukolsky Equations and the Maxwell System on Schwarzschild}, Ann. Henri Poincar\'e, \textbf{20}, (2019), 1263--1323.




\bibitem{psarelli1} M. Psarelli, {\it Maxwell-Dirac equations in four-dimensional Minkowski space}, Comm. PDE. (2005), 30(1–2), 97--119.

\bibitem{strichartz} R. S. Strichartz, {\it Restrictions of Fourier transforms to quadratic surfaces and decay of solutions of wave equations}, Duke Math. J. \textbf{44}(3), (1977), 705--714.

\bibitem{tes} A. Tesfahun, {\it Long-time behavior of solutions to cubic Dirac equation with Hartree type nonlinearity in $\mathbb R^{1+2}$}, Int. Math. Res. Not., 2020, (2020), 6489--6538.

\bibitem{tes1} A. Tesfahun, {\it Small data scattering for cubic dirac equation with hartree type nonlinearity in $\mathbb R^{1+3}$}, SIAM J. Math. Anal., 52 (2020), 2969--3003.

\bibitem{thirring} W. Thirring, {\it A soluble relativistic field theory}, Annals of Physics, \textbf{3}, (1958), 91--112.


\bibitem{treude} J.-H. Treude, {\it Decay in outgoing null directions of
solutions of the massive Dirac equation
in certain asymptotically flat, static
spacetimes}, (Dissertation), available in 10.5283/epub.32344

\bibitem{wang} X. Wang, {\it On global existence of 3D charge critical Dirac-Klein-Gordon system}, Int. Math. Res. Notices, \textbf{2015}, (2015), 10801--10846.

\bibitem{wernli} K. Wernli, {\it Lecture notes on spin geometry}, available on arXiv:1911.09766.

\bibitem{yag} K. Yagdjian, {\it Global in time self-interacting Dirac fields in the de Sitter space}, J. Evol. Equ. (2022) 22:22.


\bibitem{cyang} C. Yang, {\it Scattering results for Dirac Hartree-type equations with small initial data}, Commun. Pure Appl. Anal., \textbf{18}, (2019), 1711--1734.

\bibitem{zhaowu} P. Zhao, X. Wu, {\it On the local existence for the characteristic initial
value problem for the Einstein-Dirac system}, available in arXiv:2509.04167

\end{thebibliography}
\end{document}